\documentclass[reqno,12pt]{amsart}

\usepackage{fullpage}
\usepackage[english]{babel}
\usepackage[T1]{fontenc}
\usepackage{graphicx}
\usepackage{amsmath}
\usepackage{amsfonts}
\usepackage{amssymb}
\usepackage{a4wide}

\newtheorem{theo}{Theorem}
\newtheorem*{theoA}{Theorem A}
\newtheorem*{theoB}{Theorem B}
\newtheorem*{theoC}{Theorem C}
\newtheorem{cor}{Corollary}

\newtheorem{propo}{Proposition}
\newtheorem{Lemma}{Lemma}

\theoremstyle{remark}
\newtheorem*{Remark}{Remark}
\newtheorem*{Remarks}{Remarks}
\newtheorem*{D\'efinition}{D\'efinition}

\renewcommand{\underline}[1]{\langle #1\rangle}

\numberwithin{equation}{section}

\author{E. Delaygue, T. Rivoal and J. Roques}
\title{On Dwork's $p$-adic formal congruences theorem and hypergeometric mirror maps}
\date{}



\thanks{The first and second authors are partially funded by 
the project Holonomix (PEPS CNRS INS2I 2012). The second and 
third authors are partially funded by the project Qdiff (ANR-2010-JCJC-010501).}

\begin{document}
\maketitle

\begin{abstract} Using Dwork's theory, we prove a broad generalisation of his  
famous $p$-adic formal congruences theorem. This enables us to prove certain $p$-adic 
congruences for the generalized hypergeometric series with rational parameters; in particular, they hold for 
{\em any} prime number $p$ and not only for almost all primes. Along the way, using Christol's functions, 
we provide an explicit formula for the ``Eisenstein constant'' of any globally bounded  hypergeometric 
series with rational parameters. As an application of these results, we obtain an arithmetic statement of a new type 
concerning the integrality of Taylor coefficients of the associated mirror maps. It  
essentially contains all the similar univariate integrality results in the literature.
\end{abstract}

\section{Introduction}

Mirror maps are power series  which occur in Mirror Symmetry as the inverse  for 
composition of power series $q(z)=\exp(\omega_2(z)/\omega_1(z))$, called  local $q$-coordinates, 
where $\omega_1$ and $\omega_2$ are particular solutions of the Picard-Fuchs equation associated 
with certain one-parameter families of Calabi-Yau varieties. They can be viewed as higher dimensional 
generalisations of the classical modular forms, and in several cases, it has been observed that such 
mirror maps and $q$-coordinates have integral Taylor coefficients at the origin, like 
the $q$-expansion of Eisenstein series for instance. The arithmetical study of mirror maps 
began with the famous example of a family of mirror for quintic threefolds in $\mathbb{P}^4$ 
given by Candelas et al. \cite{Candelas} and associated with the Picard-Fuchs equation
$$
\theta^4\omega-5z(5\theta+1)(5\theta+2)(5\theta+3)(5\theta+4)\omega=0,\quad\theta=z\frac{d}{dz}.
$$
This equation is (a rescaling of) a generalized hypergeometric differential equation with  two 
linearly independent local solutions at $z=0$ given by
$$
\omega_1(z)=\sum_{n=0}^{\infty}\frac{(5n)!}{(n!)^5}z^n\quad\textup{and}\quad \omega_2(z)=G(z)+\log(z)\omega_1(z),
$$ 
where 
$$
G(z)=\sum_{n=1}^\infty\frac{(5n)!}{(n!)^5}\big(5H_{5n}-5H_n\big)z^n\quad\textup{and}\quad H_n:=\sum_{k=1}^{n}\frac{1}{k}.
$$
The $q$-coordinate $\exp\big(\omega_2(z)/\omega_1(z)\big)$ occurs in enumerative geometry and 
in the Mirror Conjecture associated with quintics threefolds in $\mathbb{P}^4$ (see \cite{LianYau2}). 
The integrality of its Taylor coefficients at the origin has been proved by Lian and Yau in \cite{LianYau}. 

In a more general context, Batyrev and van Straten conjectured the integrality of the Taylor 
coefficients at the origin of a large class of $q$-coordinates \cite[Conjecture 6.3.4]{Batyrev} 
built on $A$-hypergeometric series. (See \cite{Stienstra} for an introduction to these series, 
which generalize the classical hypergeometric series to the multivariate case). Furthermore they 
provided a lot of examples of univariate $q$-coordinates whose Taylor coefficients were subsequently 
proved to be integers in many cases by Zudilin  \cite{Zudilin0} and Krattenthaler-Rivoal \cite{Tanguy1}.

Motivated by the search for differential operators $\mathcal{L}$ associated with particular 
families of Calabi-Yau varieties, Almkvist et al. \cite{Tables} and Bogner and Reiter \cite{Bogner} 
introduced the notion of ``Calabi-Yau operators''. Even if both notions slightly differ, both 

require that an irreducible differential operator $\mathcal{L}\in\mathbb{Q}(z)[d/dz]$ of Calabi-Yau type satisfies
\begin{itemize} 
\item[$(P_1)$] $\mathcal{L}$ has a solution $\omega_1(z)\in 1+z\mathbb{C}[[z]]$ at $z=0$ which is $N$-integral.
\item[$(P_2)$] $\mathcal{L}$ has a linearly independent solution $\omega_2=G(z)+\log(z)\omega_1(z)$ 
at $z=0$ with $G(z)\in z\mathbb{C}[[z]]$ and $\exp\big(\omega_2(z)/\omega_1(z)\big)$ is $N$-integral. 
\end{itemize}
We say that a power series $f(z)\in\mathbb{C}[[z]]$ is $N$-integral if there exists $c\in\mathbb{Q}$ 
such that $f(cz)\in\mathbb{Z}[[z]]$. The constant $c$ might be called the Eisenstein constant of $f$, 
in reference to Eisenstein's theorem that such a constant $c$ 
exists when $f$ is a holomorphic algebraic function over $\mathbb Q(z)$.

\medskip

One of the main results of this article is an effective criterion for an irreducible generalized 
hypergeometric differential operator $\mathcal{L}\in\mathbb{Q}(z)[d/dz]$ to satisfy properties $(P_1)$ and $(P_2)$. 

For all tuples $\boldsymbol{\alpha}:=(\alpha_1,\dots,\alpha_r)$ and $\boldsymbol{\beta}:=(\beta_1,\dots,\beta_s)$ 
of parameters in $\mathbb{Q}\setminus\mathbb{Z}_{\leq 0}$, we write $\mathcal{L}_{\boldsymbol{\alpha},\boldsymbol{\beta}}$ 
for the generalized hypergeometric differential operator associated 
with $(\boldsymbol{\alpha},\boldsymbol{\beta})$ and defined by
$$
\mathcal{L}_{\boldsymbol{\alpha},\boldsymbol{\beta}}:=\prod_{i=1}^s(\theta+\beta_i-1)
-z\prod_{i=1}^r(\theta+\alpha_i),\quad\theta=z\frac{d}{dz}.
$$
We recall that $\mathcal{L}_{\boldsymbol{\alpha},\boldsymbol{\beta}}$ is irreducible if and 
only if, for all $i\in\{1,\dots,r\}$ and all $j\in\{1,\dots,s\}$, we have 
$\alpha_i\not\equiv\beta_j\mod\mathbb{Z}$. The equation $\mathcal{L}_{\boldsymbol{\alpha},\boldsymbol{\beta}}\cdot\omega=0$ 
admits a formal solution $F_{\boldsymbol{\alpha},\boldsymbol{\beta}}(z)\in 1+z\mathbb{C}[[z]]$ if and only 
if there exists $i\in\{1,\dots,s\}$ such that $\beta_i=1$. In this case, 
$F_{\boldsymbol{\alpha},\boldsymbol{\beta}}$ is uniquely determined by $\boldsymbol{\alpha}$ and $\boldsymbol{\beta}$ and is 
\begin{equation}\label{DefHyp}
F_{\boldsymbol{\alpha},\boldsymbol{\beta}}(z):=\sum_{n=0}^\infty\frac{(\alpha_1)_n\cdots(\alpha_r)_n}{(\beta_1)_n\cdots(\beta_s)_n}z^n,
\end{equation}
where $(x)_n$ denotes Pochhammer  symbol $(x)_n=x(x+1)\cdots(x+n-1)$ if $n\geq 1$ and $(x)_0=1$ 
otherwise. $F_{\boldsymbol{\alpha},\boldsymbol{\beta}}$ is a generalized hypergeometric series and 
if one assumes that $\beta_s=1$, then our definition \eqref{DefHyp} agrees with the classical notation
$$
F_{\boldsymbol{\alpha},\boldsymbol{\beta}}(z)=  {_r}F_{s-1}\left[\begin{array}{cc}\alpha_1,\dots,\alpha_r\\\beta_1,\dots,\beta_{s-1}\end{array} ; z\right]:=\sum_{n=0}^\infty\frac{(\alpha_1)_n\cdots(\alpha_r)_n}{(\beta_1)_n\cdots(\beta_{s-1})_n}\frac{z^n}{n!}.
$$ 

An elementary computation of the $p$-adic valuation of Pochhammer symbols leads to the following result. 

\begin{propo}\label{Lemma Mp}
Let $\boldsymbol{\alpha}$ and $\boldsymbol{\beta}$ be tuples of parameters in $\mathbb{Q}\setminus\mathbb{Z}_{\leq 0}$. Then, $F_{\boldsymbol{\alpha},\boldsymbol{\beta}}$ is $N$-integral if and only if, for almost all primes $p$, we have $F_{\boldsymbol{\alpha},\boldsymbol{\beta}}(z)\in\mathbb{Z}_p[[z]]$.
\end{propo}

\begin{Remark}
We say that an assertion $\mathcal{A}_p$ is true for almost all primes $p$ if there is $C\in\mathbb{N}$ such that $\mathcal{A}_p$ is true for all primes $p\geq C$.
\end{Remark}

Using Proposition \ref{Lemma Mp} in combination with a result of Christol (Proposition $1$ in \cite{Christol}), one obtains an effective criterion for $F_{\boldsymbol{\alpha},\boldsymbol{\beta}}$ to be $N$-integral, \textit{i.e.} for $\mathcal{L}_{\boldsymbol{\alpha},\boldsymbol{\beta}}$ to satisfy $(P_1)$. To  state this criterion, we introduce some notations.

For all $x\in\mathbb{Q}$, we write $\underline{x}$ for the unique element in $(0,1]$ such that $x-\underline{x}\in\mathbb{Z}$, \textit{i.e.} $\underline{x}=1-\{1-x\}$, where $\{\cdot\}$ is the fractional part function. In particular, we have $x-\underline{x}=-\lfloor 1-x\rfloor$, where $\lfloor\cdot\rfloor$ denotes the floor function. We write $\preceq$ for the total order on $\mathbb{R}$ defined by
$$
x\preceq y\Longleftrightarrow\Big( \underline{x}<\underline{y}\quad\textup{or}\quad\big(\underline{x}
=\underline{y}\quad\textup{and}\quad x\geq y\big)\Big).
$$
Given tuples $\boldsymbol{\alpha}=(\alpha_1,\dots,\alpha_r)$ and $\boldsymbol{\beta}=(\beta_1,\dots,\beta_s)$ of parameters in $\mathbb{Q}\setminus\mathbb{Z}_{\leq 0}$, we write $d_{\boldsymbol{\alpha},\boldsymbol{\beta}}$ for the least common multiple of the exact denominators of elements of $\boldsymbol{\alpha}$ and $\boldsymbol{\beta}$. For all $a\in\{1,\dots,d_{\boldsymbol{\alpha},\boldsymbol{\beta}}\}$ coprime to $d_{\boldsymbol{\alpha},\boldsymbol{\beta}}$ and all $x\in\mathbb{R}$, we set
$$
\xi_{\boldsymbol{\alpha},\boldsymbol{\beta}}(a,x):=\#\{1\leq i\leq r\,:\,a\alpha_i\preceq x\}-\#\{1\leq j\leq s\,:\,a\beta_j\preceq x\}.
$$
We now state Christol's criterion for the $N$-integrality of $F_{\boldsymbol{\alpha},\boldsymbol{\beta}}$. 

\begin{theoA}[Christol, \cite{Christol}]
Let $\boldsymbol{\alpha}:=(\alpha_1,\dots,\alpha_r)$ and $\boldsymbol{\beta}:=(\beta_1,\dots,\beta_s)$ be tuples of parameters in $\mathbb{Q}\setminus\mathbb{Z}_{\leq 0}$. Then, the following assertions are equivalent: 
\begin{itemize}
\item[$(i)$] $F_{\boldsymbol{\alpha},\boldsymbol{\beta}}$ is $N$-integral;
\item[$(ii)$] For all $a\in\{1,\dots,d_{\boldsymbol{\alpha},\boldsymbol{\beta}}\}$ coprime to $d_{\boldsymbol{\alpha},\boldsymbol{\beta}}$ and all $x\in\mathbb{R}$, we have $\xi_{\boldsymbol{\alpha},\boldsymbol{\beta}}(a,x)\geq~0$.
\end{itemize}
\end{theoA}

\begin{Remark}
Formally, Christol proved Theorem A (Proposition $1$ in \cite{Christol}) under the assumptions that $r=s$, that there is $j\in\{1,\dots,s\}$ such that $\beta_j\in\mathbb{N}$ and that all elements $\alpha\in\mathbb{N}$ of $\boldsymbol{\alpha}$ and $\boldsymbol{\beta}$ satisfies $\alpha\geq \beta_j$. However, his proof does not use these assumptions.
\end{Remark}

Theorem A provides a criterion for an irreducible operator $\mathcal{L}_{\boldsymbol{\alpha},\boldsymbol{\beta}}$ to satisfy $(P_1)$ but do not give any information on the rational numbers $C$ satisfying $F_{\boldsymbol{\alpha},\boldsymbol{\beta}}(Cz)\in\mathbb{Z}[[z]]$. If $F_{\boldsymbol{\alpha},\boldsymbol{\beta}}$ is $N$-integral then it is not hard to see that the set of all $C\in\mathbb{Q}$ satisfying $F_{\boldsymbol{\alpha},\boldsymbol{\beta}}(Cz)\in\mathbb{Z}[[z]]$ is $C_{\boldsymbol{\alpha},\boldsymbol{\beta}}\mathbb{Z}$ for some $C_{\boldsymbol{\alpha},\boldsymbol{\beta}}\in\mathbb{Q}\setminus\{0\}$. Our first result, Theorem \ref{theo Const} below, gives some arithmetical properties of $C_{\boldsymbol{\alpha},\boldsymbol{\beta}}$ and gives a simple formula for $C_{\boldsymbol{\alpha},\boldsymbol{\beta}}$ when $r=s$ and when all elements of $\boldsymbol{\alpha}$ and $\boldsymbol{\beta}$ lie in $(0,1]$. Before stating this result, we introduce some notations. 
\medskip

Let $\mathbb{Z}_p$ denotes the ring of $p$-adic integers. Then, for all primes $p$, we define
$$
\lambda_p(\boldsymbol{\alpha},\boldsymbol{\beta}):=\#\{1\leq i\leq r\,:\,\alpha_i\in\mathbb{Z}_p\}-\#\{1\leq j\leq s\,:\,\beta_j\in\mathbb{Z}_p\}.
$$
Note that if $\alpha\in\mathbb{Q}\setminus\{0\}$, then $\alpha\in\mathbb{Z}_p$ if and only if $p$ does not divide the exact denominator of $\alpha$. Furthermore, we write $\mathcal{P}_{\boldsymbol{\alpha},\boldsymbol{\beta}}$ for the set of all primes $p$ such that $p$ divides $d_{\boldsymbol{\alpha},\boldsymbol{\beta}}$ or $p\leq r-s+1$. In particular, if $r=s$, then $\mathcal{P}_{\boldsymbol{\alpha},\boldsymbol{\beta}}$ is the set of the prime divisors of $d_{\boldsymbol{\alpha},\boldsymbol{\beta}}$. Finally, for all rational numbers $a$, we write $d(a)$ for the exact denominator of $a$. Our result on $C_{\boldsymbol{\alpha},\boldsymbol{\beta}}$ is the following.

\begin{theo}\label{theo Const}
Let $\boldsymbol{\alpha}:=(\alpha_1,\dots,\alpha_r)$ and $\boldsymbol{\beta}:=(\beta_1,\dots,\beta_s)$ be tuples of parameters in $\mathbb{Q}\setminus\mathbb{Z}_{\leq 0}$ such that $F_{\boldsymbol{\alpha},\boldsymbol{\beta}}$ is $N$-integral. Then, there exists $C\in\mathbb{N}$ such that
\begin{equation}\label{facile}
C_{\boldsymbol{\alpha},\boldsymbol{\beta}}=C\frac{\prod_{i=1}^rd(\alpha_i)}{\prod_{j=1}^sd(\beta_j)}\underset{p\in\mathcal{P}_{\boldsymbol{\alpha},\boldsymbol{\beta}}}{\prod}p^{-\left\lfloor\frac{\lambda_p(\boldsymbol{\alpha},\boldsymbol{\beta})}{p-1}\right\rfloor}. 
\end{equation}
Furthermore, if $r=s$ and if all elements of $\boldsymbol{\alpha}$ and $\boldsymbol{\beta}$ lie in $(0,1]$, then we have $C=1$.
\end{theo}

\begin{Remarks}
The following comments are detailed in Section \ref{section Comments 1}.

$\bullet$ If $\boldsymbol{\alpha}$ and $\boldsymbol{\beta}$ are tuples of same length of parameters in $(0,1]$ such that $F_{\boldsymbol{\alpha},\boldsymbol{\beta}}\in\mathbb{Q}[[z]]$ is algebraic over $\mathbb{Q}(z)$, then Theorem \ref{theo Const} gives a simple formula for the Eisenstein constant of $F_{\boldsymbol{\alpha},\boldsymbol{\beta}}$. 

$\bullet$ When all the elements of $\boldsymbol{\alpha}$ and $\boldsymbol{\beta}$ lie in $(0,1]$, Theorem \ref{theo Const} also gives a necessary condition on the numerators of the elements of $\boldsymbol{\alpha}$ and $\boldsymbol{\beta}$ for $F_{\boldsymbol{\alpha},\boldsymbol{\beta}}$ to be $N$-integral.

$\bullet$ Given tuples $\boldsymbol{\alpha}$ and $\boldsymbol{\beta}$, one can follow the proofs of Theorem \ref{theo Const} and Lemma \ref{valeurs Delta} to determine an explicit bound for the prime divisors of $C$.
\end{Remarks}

Let $\boldsymbol{\alpha}:=(\alpha_1,\dots,\alpha_r)$ and $\boldsymbol{\beta}:=(\beta_1,\dots,\beta_s)$ be tuples of parameters in $\mathbb{Q}\setminus\mathbb{Z}_{\leq 0}$ such that $\beta_s=1$ and $F_{\boldsymbol{\alpha},\boldsymbol{\beta}}$ is $N$-integral, so that $\mathcal{L}_{\boldsymbol{\alpha},\boldsymbol{\beta}}$ satisfies property $(P_1)$. Then, a simple computation (see Equation $(7)$ in \cite{Batyrev}) shows that $\mathcal{L}_{\boldsymbol{\alpha},\boldsymbol{\beta}}$ has a formal solution $G_{\boldsymbol{\alpha},\boldsymbol{\beta}}(z)+\log(z)F_{\boldsymbol{\alpha},\boldsymbol{\beta}}(z)$ with $G_{\boldsymbol{\alpha},\boldsymbol{\beta}}(z)\in z\mathbb{C}[[z]]$ if and only if there exists $i\in\{1,\dots,s-1\}$ such that $\beta_i=1$. In this case $G_{\boldsymbol{\alpha},\boldsymbol{\beta}}$ is uniquely determined by $\boldsymbol{\alpha}$ and $\boldsymbol{\beta}$ and it is explicitely given by
$$
G_{\boldsymbol{\alpha},\boldsymbol{\beta}}(z):=\sum_{n=1}^\infty\frac{(\alpha_1)_n\cdots(\alpha_r)_n}{(\beta_1)_n\cdots(\beta_s)_n}\left(\sum_{i=1}^r H_{\alpha_i}(n)-\sum_{j=1}^s H_{\beta_j}(n)\right),
$$
where, for all $n\in\mathbb{N}$ and all $x\in\mathbb{Q}\setminus\mathbb{Z}_{\leq 0}$, $H_x(n):=\sum_{k=0}^{n-1}\frac{1}{x+k}$.
\medskip

For all tuples $\boldsymbol{\alpha}$ and $\boldsymbol{\beta}$ of parameters in $\mathbb{Q}\setminus\mathbb{Z}_{\leq 0}$, we define
\begin{equation}\label{define q}
q_{\boldsymbol{\alpha},\boldsymbol{\beta}}(z):=\exp\left(\frac{G_{\boldsymbol{\alpha},\boldsymbol{\beta}}(z)+\log(z)F_{\boldsymbol{\alpha},\boldsymbol{\beta}}(z)}{F_{\boldsymbol{\alpha},\boldsymbol{\beta}}(z)}\right)=z\exp\left(\frac{G_{\boldsymbol{\alpha},\boldsymbol{\beta}}(z)}{F_{\boldsymbol{\alpha},\boldsymbol{\beta}}(z)}\right),
\end{equation}
so that $\mathcal{L}_{\boldsymbol{\alpha},\boldsymbol{\beta}}$ satisfies $(P_2)$ if and only if there are at least two elements equal to $1$ in $\boldsymbol{\beta}$ and $q_{\boldsymbol{\alpha},\boldsymbol{\beta}}$ is $N$-integral. The mirror map $z_{\boldsymbol{\alpha},\boldsymbol{\beta}}$ associated with $(\boldsymbol{\alpha},\boldsymbol{\beta})$ is the compositional inverse of $q_{\boldsymbol{\alpha},\boldsymbol{\beta}}$. 
For all $C\in\mathbb{Q}$, we have  $z_{\boldsymbol{\alpha},\boldsymbol{\beta}}(Cq)\in\mathbb{Z}[[q]]$ if and only if $q_{\boldsymbol{\alpha},\boldsymbol{\beta}}(Cz)\in\mathbb{Z}[[z]]$, thus our results and comments on $N$-integrality of $q$-coordinates also hold for the corresponding mirror maps.
\medskip

For a detailed description of the known results on $N$-integrality of $q_{\boldsymbol{\alpha},\boldsymbol{\beta}}$ and $z_{\boldsymbol{\alpha},\boldsymbol{\beta}}$, we refer the reader to Section \ref{section comparison} below. The tuples $\boldsymbol{\alpha}$ and $\boldsymbol{\beta}$ such that $F_{\boldsymbol{\alpha},\boldsymbol{\beta}}$ and $q_{\boldsymbol{\alpha},\boldsymbol{\beta}}$ are $N$-integral are entirely determined:

\medskip

$\bullet$ when $\boldsymbol{\alpha}$ and $\boldsymbol{\beta}$ are $R$-\textit{partitioned}.  Throughout this article, we say that $\boldsymbol{\alpha}$ is $R$-partitioned if $\boldsymbol{\alpha}$ is the concatenation of tuples of the form $(b/m)_{b\in\{1,\dots,m\}, \gcd(b,m)=1}$ for $m\in\mathbb{N}$, $m\geq 1$, up to permutation (\footnote{We say that, up to permutation, $(\alpha_1,\dots,\alpha_n)=(\alpha_1',\dots,\alpha_n')$ if there exists a permutation $\sigma$ of $\{1,\dots,n\}$ such that, for all $i\in\{1,\dots,n\}$, we have $\alpha_i=\alpha_{\sigma(i)}'$.}). The characterization for this class  of parameters is a consequence of Delaygue's Theorems $1$ and $3$ in \cite{Delaygue1};

$\bullet$ when all parameters of $\boldsymbol{\alpha}$ lie in $(0,1]$, $\boldsymbol{\beta}=(1,\dots,1)$ and $\boldsymbol{\alpha}$ and $\boldsymbol{\beta}$ have the same number of parameters $r\geq 2$. The characterization for this class of parameters is a consequence of Roques' articles \cite{Roques1} and \cite{Roques2}.

The starting point of the proofs of these characterizations is to reduce the problem to a  $p$-adic statement for any prime $p$, according to the following simple principle:
\begin{quote}
If $x\in\mathbb{Q}$, then $x\in\mathbb{Z}$ if and only if, for all prime $p$, $x$ belongs to the ring of $p$-adic integers $\mathbb{Z}_p$.
\end{quote}

Then, given a fixed prime $p$, one can apply the following lemma of Dieudonn\'e and Dwork \cite[Chap. IV, Sec. $2$, Lemma $3$]{Koblitz} and its corollary (see  \cite[Lemma~$5$]{Zudilin0})  to get rid of the exponential map in \eqref{define q}. 

\begin{Lemma}[Dieudonn\'e, Dwork]\label{Dieudonn\'e-Dwork}
Given a prime $p$ and $F(z)\in 1+z\mathbb{Q}_p[[z]]$, we have $F(z)\in 1+z\mathbb{Z}_p[[z]]$ if and only if $F(z^p)/F(z)^p\in 1+pz\mathbb{Z}_p[[z]]$.
\end{Lemma}

\begin{cor}\label{cor exp}
Given a prime $p$ and $f(z)\in z\mathbb{Q}[[z]]$, we have $\exp\big(f(z)\big)\in 1+z\mathbb{Z}_p[[z]]$ if and only if $f(z^p)-pf(z)\in pz\mathbb{Z}_p[[z]]$.
\end{cor}
By Corollary \ref{cor exp}, $q_{\boldsymbol{\alpha},\boldsymbol{\beta}}$ is $N$-integral if and only if there exists $C\in\mathbb{Q}$ such that, for all primes $p$, we have
\begin{equation}\label{Cond reduced}
\frac{G_{\boldsymbol{\alpha},\boldsymbol{\beta}}(Cz^p)}{F_{\boldsymbol{\alpha},\boldsymbol{\beta}}(Cz^p)}-p\frac{G_{\boldsymbol{\alpha},\boldsymbol{\beta}}(Cz)}{F_{\boldsymbol{\alpha},\boldsymbol{\beta}}(Cz)}\in pz\mathbb{Z}_p[[z]].
\end{equation}

One of the main results of this article (Theorem \ref{theo expand} below) provides an analogous version of \eqref{Cond reduced} for a large class of tuples $\boldsymbol{\alpha}$ and $\boldsymbol{\beta}$ and whith $\mathbb{Z}_p$ replaced by certain algebras of $\mathbb{Z}_p$-valued functions. This result enables us to prove a complete characterization (Theorem \ref{Criterion} below) of tuples $\boldsymbol{\alpha}$ and $\boldsymbol{\beta}$ such that $F_{\boldsymbol{\alpha},\boldsymbol{\beta}}$ and $q_{\boldsymbol{\alpha},\boldsymbol{\beta}}$ are $N$-integral, without any restriction on the shape of $\boldsymbol{\alpha}$ nor $\boldsymbol{\beta}$.

\subsection{Additional notations}

To  state Theorem \ref{theo expand}, we first define some algebras of $\mathbb{Z}_p$-valued functions and a constant associated with $(\boldsymbol{\alpha},\boldsymbol{\beta})$.

$\bullet$ For all primes $p$ and all positive integers $n$, we write $\mathfrak{A}_{p,n}$, respectively $\mathfrak{A}_{p,n}^\ast$, for the $\mathbb{Z}_p$-algebra of the functions $f:(\mathbb{Z}_p^{\times})^n\rightarrow\mathbb{Z}_p$ such that, for all positive integers $m$, all $\mathbf{x}\in(\mathbb{Z}_p^{\times})^n$ and all $\mathbf{a}\in\mathbb{Z}_p^n$, we have 
$$
f(\mathbf{x}+\mathbf{a}p^m)\equiv f(\mathbf{x})\mod p^{m}\mathbb{Z}_p,
$$
respectively $f(\mathbf{x}+\mathbf{a}p^m)\equiv f(\mathbf{x})\mod p^{m-1}\mathbb{Z}_p$.

$\bullet$ If $D$ is a positive integer coprime to $p$, then, for all $\nu\in\mathbb{N}$, and all $b\in\{1,\dots,D\}$ coprime to $D$, we write $\Omega_{b}(p^\nu,D)$ for the set of all $t\in\{1,\dots,p^\nu D\}$ coprime to $p^\nu D$ satisfying $t\equiv b\mod D$. 

$\bullet$ We write $\mathcal{A}_{b}(p^\nu,D)$, respectively $\mathcal{A}_{b}(p^\nu,D)^\ast$, for the $\mathbb{Z}_p$-algebra of the functions $f:\Omega_{b}(p^\nu,D)\rightarrow\mathbb{Z}_p$ such that, for all positive integers $m$ and all $t_1,t_2\in\Omega_{b}(p^\nu,D)$ satisfying $t_1\equiv t_2\mod p^m$, we have $f(t_1)\equiv f(t_2)\mod p^m\mathbb{Z}_p$, respectively $f(t_1)\equiv f(t_2)\mod p^{m-1}\mathbb{Z}_p$. 

$\bullet$ For all $t\in\Omega_b(p^\nu,D)$ and all $r\in\mathbb{N}$, we write $t^{(r)}$ for the unique element of $\{1,\dots,p^\nu D\}$ satisfying
$t^{(r)}\equiv t\mod p^\nu$ and $p^rt^{(r)}\equiv t\mod D$.

$\bullet$ Furthermore, if $\boldsymbol{\beta}\notin\mathbb{Z}^s$, then we write $\mathfrak{m}_{\boldsymbol{\alpha},\boldsymbol{\beta}}$ for the number of elements of $\boldsymbol{\alpha}$ and $\boldsymbol{\beta}$ with exact denominator divisible by $4$. We write $d_{\boldsymbol{\alpha},\boldsymbol{\beta}}^\ast$ for the integer obtained by dividing $d_{\boldsymbol{\alpha},\boldsymbol{\beta}}$ by the product of its prime divisors. We set $C_{\boldsymbol{\alpha},\boldsymbol{\beta}}'=2C_{\langle\boldsymbol{\alpha}\rangle,\langle\boldsymbol{\beta}\rangle}$ and $d_{\boldsymbol{\alpha},\boldsymbol{\beta}}'=2d_{\boldsymbol{\alpha},\boldsymbol{\beta}}^\ast$ if $\boldsymbol{\beta}\notin\mathbb{Z}^s$ and if $\mathfrak{m}_{\boldsymbol{\alpha},\boldsymbol{\beta}}$ is odd, and we set $C_{\boldsymbol{\alpha},\boldsymbol{\beta}}'=C_{\langle\boldsymbol{\alpha}\rangle,\langle\boldsymbol{\beta}\rangle}$ and $d_{\boldsymbol{\alpha},\boldsymbol{\beta}}'=d_{\boldsymbol{\alpha},\boldsymbol{\beta}}^\ast$ otherwise.

\subsection{Statements of the main results}

By Theorem A, the $N$-integrality of $F_{\boldsymbol{\alpha},\boldsymbol{\beta}}$ depends on the graphs of Christol's functions $\xi_{\boldsymbol{\alpha},\boldsymbol{\beta}}(a,\cdot)$. The $N$-integrality of $q_{\boldsymbol{\alpha},\boldsymbol{\beta}}$ also strongly depends on the graphs of these functions. More precisely, let $m_{\boldsymbol{\alpha},\boldsymbol{\beta}}(a)$ denote the smallest element in the ordered set $\big(\{a\alpha_1,\dots,a\alpha_r,a\beta_1,\dots,a\beta_s\},\preceq\big)$. Let $H_{\boldsymbol{\alpha},\boldsymbol{\beta}}$ denotes the assertion
\begin{quote}
$H_{\boldsymbol{\alpha},\boldsymbol{\beta}}$: ``For all $a\in\{1,\dots,d_{\boldsymbol{\alpha},\boldsymbol{\beta}}\}$ coprime to $d_{\boldsymbol{\alpha},\boldsymbol{\beta}}$ and all $x\in\mathbb{R}$ satisfying $m_{\boldsymbol{\alpha},\boldsymbol{\beta}}(a)\preceq x\prec a$, we have $\xi_{\boldsymbol{\alpha},\boldsymbol{\beta}}(a,x)\geq 1$.''
\end{quote}

One of our main results is the following.

\begin{theo}\label{theo expand}
Let $\boldsymbol{\alpha}$ and $\boldsymbol{\beta}$ be tuples of parameters in $\mathbb{Q}\setminus\mathbb{Z}_{\leq 0}$ with the same number of elements such that $\langle\boldsymbol{\alpha}\rangle$ and $\langle\boldsymbol{\beta}\rangle$ are disjoint (this is equivalent to the irreducibility of $\mathcal{L}_{\boldsymbol{\alpha},\boldsymbol{\beta}}$) and such that $H_{\boldsymbol{\alpha},\boldsymbol{\beta}}$ holds.

Let $p$ be a fixed prime  and write $d_{\boldsymbol{\alpha},\boldsymbol{\beta}}=p^\nu D$ with $\nu,D\in\mathbb{N}$ and $D$ coprime to $p$. Let $b\in\{1,\dots,D\}$ be coprime to $D$. Then, there exists a sequence $(R_{k,b})_{k\geq 0}$ of elements in $\mathcal{A}_b(p^\nu,D)^\ast$ such that, for all $t\in\Omega_b(p^\nu,D)$, we have
$$
\frac{G_{\underline{t^{(1)}\boldsymbol{\alpha}},\underline{t^{(1)}\boldsymbol{\beta}}}}{F_{\underline{t^{(1)}\boldsymbol{\alpha}},\underline{t^{(1)}\boldsymbol{\beta}}}}(C'_{\boldsymbol{\alpha},\boldsymbol{\beta}}z^p)-p\frac{G_{\underline{t\boldsymbol{\alpha}},\underline{t\boldsymbol{\beta}}}}{F_{\underline{t\boldsymbol{\alpha}},\underline{t\boldsymbol{\beta}}}}(C'_{\boldsymbol{\alpha},\boldsymbol{\beta}}z)= p\sum_{k=0}^{\infty}R_{k,b}(t)z^k.
$$
Furthermore, if $p$ is a prime divisor of $d_{\boldsymbol{\alpha},\boldsymbol{\beta}}$, then, for all $k\in\mathbb{N}$,
\begin{itemize}
\item if $\boldsymbol{\beta}\in\mathbb{Z}^r$, then we have $R_{k,b}\in p^{-1-\lfloor\lambda_p/(p-1)\rfloor}\mathcal{A}_b(p^\nu,D)$;
\item if $\boldsymbol{\beta}\notin\mathbb{Z}^r$ and $p-1\nmid\lambda_p$, then we have $R_{k,b}\in\mathcal{A}_b(p^\nu,D)$;
\item if $\boldsymbol{\beta}\notin\mathbb{Z}^r$, $\mathfrak{m}_{\boldsymbol{\alpha},\boldsymbol{\beta}}$ is odd and $p=2$, then we have $R_{k,b}\in \mathcal{A}_b(p^\nu,D)$.
\end{itemize}
\end{theo}

In Theorem \ref{theo expand} and throughout this article, when $\boldsymbol{\alpha}=(\alpha_1,\dots,\alpha_r)$ and $f$ is a map defined on $\{\alpha_1,\dots,\alpha_r\}$, we write $f(\boldsymbol{\alpha})$ for $\big(f(\alpha_1),\dots,f(\alpha_r)\big)$. 

\begin{Remarks}

$\bullet$ If $\boldsymbol{\alpha}$ and $\boldsymbol{\beta}$ satisfy hypothesis of Theorem \ref{theo expand}, then using Theorem A, we obtain that $F_{\boldsymbol{\alpha},\boldsymbol{\beta}}$ is $N$-integral.

$\bullet$  If $\boldsymbol{\beta}\in\mathbb{Z}^r$, then $\lambda_p\leq -1$ and $-1-\lfloor\lambda_p/(p-1)\rfloor\geq 0$ so that $p^{-1-\lfloor\lambda_p/(p-1)\rfloor}\mathcal{A}_b(p^\nu,D)
\subset\mathcal{A}_b(p^\nu,D)\subset\mathcal{A}_b(p^\nu,D)^\ast$.

\end{Remarks}

The $N$-integrality of $q_{\boldsymbol{\alpha},\boldsymbol{\beta}}$ is closely related to the  $N$-integrality of a product $\exp\big(S_{\boldsymbol{\alpha},\boldsymbol{\beta}}(z)\big)$ of $q$-coordinates associated with $(\boldsymbol{\alpha},\boldsymbol{\beta})$, that we now define. We set 
$$
S_{\boldsymbol{\alpha},\boldsymbol{\beta}}(z):=\underset{\gcd(a,d)=1}{\sum_{a=1}^d}\frac{G_{\langle a\boldsymbol{\alpha}\rangle,\langle a\boldsymbol{\beta}\rangle}(z)}{F_{\langle a\boldsymbol{\alpha}\rangle,\langle a\boldsymbol{\beta}\rangle}(z)},
$$
with $d=d_{\boldsymbol{\alpha},\boldsymbol{\beta}}$, so that
$$
\exp\big(S_{\boldsymbol{\alpha},\boldsymbol{\beta}}(z)\big)=\frac{1}{z^{\varphi(d)}}\underset{\gcd(a,d)=1}{\prod_{a=1}^d}q_{\langle a\boldsymbol{\alpha}\rangle,\langle a\boldsymbol{\beta}\rangle}(z),
$$
where $\varphi$ denotes Euler's totient function. Our criterion for the $N$-integrality of $q_{\boldsymbol{\alpha},\boldsymbol{\beta}}(z)$ and $\exp\big(S_{\boldsymbol{\alpha},\boldsymbol{\beta}}(z)\big)$ is the following.

\begin{theo}\label{Criterion}
Let $\boldsymbol{\alpha}:=(\alpha_1,\dots,\alpha_r)$ and $\boldsymbol{\beta}:=(\beta_1,\dots,\beta_s)$ be tuples of parameters in $\mathbb{Q}\setminus\mathbb{Z}_{\leq 0}$ such that $\langle\boldsymbol{\alpha}\rangle$ and $\langle\boldsymbol{\beta}\rangle$ are disjoint (this is equivalent to the irreducibility of $\mathcal{L}_{\boldsymbol{\alpha},\boldsymbol{\beta}}$) and such that $F_{\boldsymbol{\alpha},\boldsymbol{\beta}}$ is $N$-integral. Then,
\begin{enumerate}
\item For all $a\in\{1,\dots,d_{\boldsymbol{\alpha},\boldsymbol{\beta}}\}$ coprime to $d_{\boldsymbol{\alpha},\boldsymbol{\beta}}$, all Taylor coefficients at the origin of $q_{\langle a\boldsymbol{\alpha}\rangle,\langle a\boldsymbol{\beta}\rangle}(z)$ are positive, but its constant term which is $0$;

\item The following assertions are equivalent.
\begin{itemize}
\item[$(i)$] $q_{\boldsymbol{\alpha},\boldsymbol{\beta}}(z)$ is $N$-integral;

\item[$(ii)$] $q_{\boldsymbol{\alpha},\boldsymbol{\beta}}(C_{\boldsymbol{\alpha},\boldsymbol{\beta}}'z)\in\mathbb{Z}[[z]]$;

\item[$(iii)$] Assertion $H_{\boldsymbol{\alpha},\boldsymbol{\beta}}$ holds, we have $r=s$ and, for all $a\in\{1,\dots,d_{\boldsymbol{\alpha},\boldsymbol{\beta}}\}$ coprime to $d_{\boldsymbol{\alpha},\boldsymbol{\beta}}$, we have $q_{\boldsymbol{\alpha},\boldsymbol{\beta}}(z)=q_{\langle a\boldsymbol{\alpha}\rangle,\langle a\boldsymbol{\beta}\rangle}(z)$.
\end{itemize}
\medskip
\noindent Furthermore, if $(i)$ holds, then we have either $\boldsymbol{\alpha}=(1/2)$ and $\boldsymbol{\beta}=(1)$, or $s\geq 2$ and there are at least two $1$'s in $\langle\boldsymbol{\beta}\rangle$.

\item If $r=s$ and if $H_{\boldsymbol{\alpha},\boldsymbol{\beta}}$ holds, then $\exp\big(S_{\boldsymbol{\alpha},\boldsymbol{\beta}}(z)\big)$ is $N$-integral and we have
$$
\exp\left(\frac{S_{\boldsymbol{\alpha},\boldsymbol{\beta}}(C_{\boldsymbol{\alpha},\boldsymbol{\beta}}'z)}{\mathfrak{n}_{\boldsymbol{\alpha},\boldsymbol{\beta}}}\right)\in\mathbb{Z}[[z]],
$$
where
$$
\mathfrak{n}_{\boldsymbol{\alpha},\boldsymbol{\beta}}:=d_{\boldsymbol{\alpha},\boldsymbol{\beta}}\prod_{p\mid d_{\boldsymbol{\alpha},\boldsymbol{\beta}}}p^{-2-\left\lfloor\frac{\lambda_p}{p-1}\right\rfloor}\textup{ if $\boldsymbol{\beta}\in\mathbb{Z}^s$,}\quad\textup{and}\quad\mathfrak{n}_{\boldsymbol{\alpha},\boldsymbol{\beta}}:=d_{\boldsymbol{\alpha},\boldsymbol{\beta}}'\underset{p-1\mid\lambda_p}{\prod_{p\mid d_{\boldsymbol{\alpha},\boldsymbol{\beta}}'}}p^{-1}\textup{ otherwise.}
$$
\end{enumerate}
\end{theo}

As a consequence, we obtain the following relation between $N$-integrality of $q_{\boldsymbol{\alpha},\boldsymbol{\beta}}(z)$ and $\exp\big(S_{\boldsymbol{\alpha},\boldsymbol{\beta}}(z)\big)$.

\begin{cor}
Let $\boldsymbol{\alpha}:=(\alpha_1,\dots,\alpha_r)$ and $\boldsymbol{\beta}:=(\beta_1,\dots,\beta_s)$ be tuples of parameters in $\mathbb{Q}\setminus\mathbb{Z}_{\leq 0}$ such that $\langle\boldsymbol{\alpha}\rangle$ and $\langle\boldsymbol{\beta}\rangle$ are disjoint (this is equivalent to the irreducibility of $\mathcal{L}_{\boldsymbol{\alpha},\boldsymbol{\beta}}$) and such that $F_{\boldsymbol{\alpha},\boldsymbol{\beta}}$ is $N$-integral. Then the following assertions are equivalent: 
\begin{enumerate}
\item $q_{\boldsymbol{\alpha},\boldsymbol{\beta}}(z)$ is $N$-integral;
\item $\exp\big(S_{\boldsymbol{\alpha},\boldsymbol{\beta}}(z)\big)$ is $N$-integral and $\exp\big(S_{\boldsymbol{\alpha},\boldsymbol{\beta}}(z)\big)=\big(q_{\boldsymbol{\alpha},\boldsymbol{\beta}}(z)/z\big)^{\varphi(d_{\boldsymbol{\alpha},\boldsymbol{\beta}})}$.
\end{enumerate}
\end{cor}

Under the assumptions of Theorem \ref{Criterion} for the tuples $\boldsymbol{\alpha}$ and $\boldsymbol{\beta}$ and if $q_{\boldsymbol{\alpha},\boldsymbol{\beta}}(z)$ is $N$-integral,  then Assertion $(3)$ of Theorem \ref{Criterion} leads to
\begin{equation}\label{RootCons}
\left(\frac{1}{C_{\boldsymbol{\alpha},\boldsymbol{\beta}}'z}q_{\boldsymbol{\alpha},\boldsymbol{\beta}}(C_{\boldsymbol{\alpha},\boldsymbol{\beta}}'z)\right)^{\varphi(d_{\boldsymbol{\alpha},\boldsymbol{\beta}})/\mathfrak{n}_{\boldsymbol{\alpha},\boldsymbol{\beta}}}\in\mathbb{Z}[[z]].
\end{equation}
It follows that in some cases, we obtain the integrality of the Taylor coefficients of a non-trivial root of the $q$-coordinate. But when all the elements of $\boldsymbol{\beta}$ are integers, Theorem \ref{theo expand} in combination with Corollary \ref{cor exp} and Assertion $(2)$ of Theorem  \ref{Criterion} provides a better result than \eqref{RootCons}. Indeed, we obtain the following result.

\begin{cor}\label{cor Root}
Let $\boldsymbol{\alpha}$, respectively $\boldsymbol{\beta}$, be a tuple of parameters in $\mathbb{Q}\setminus\mathbb{Z}_{\leq 0}$, respectively of positive integers, such that $\langle\boldsymbol{\alpha}\rangle$ and $\langle\boldsymbol{\beta}\rangle$ are disjoint. If $F_{\boldsymbol{\alpha},\boldsymbol{\beta}}(z)$ and $q_{\boldsymbol{\alpha},\boldsymbol{\beta}}(z)$ are $N$-integral, then we have
$$
\left(\frac{1}{C_{\boldsymbol{\alpha},\boldsymbol{\beta}}'z}q_{\boldsymbol{\alpha},\boldsymbol{\beta}}(C_{\boldsymbol{\alpha},\boldsymbol{\beta}}'z)\right)^{1/\mathfrak{n}_{\boldsymbol{\alpha},\boldsymbol{\beta}}'}\in\mathbb{Z}[[z]],
$$
with
$$
\mathfrak{n}_{\boldsymbol{\alpha},\boldsymbol{\beta}}'=\prod_{p\mid d_{\boldsymbol{\alpha},\boldsymbol{\beta}}}p^{-1-\lfloor\frac{\lambda_p}{p-1}\rfloor}.
$$
\end{cor}

Corollary \ref{cor Root} is stronger than \eqref{RootCons} because $\mathfrak{n}_{\boldsymbol{\alpha},\boldsymbol{\beta}}/\mathfrak{n}_{\boldsymbol{\alpha},\boldsymbol{\beta}}'=d_{\boldsymbol{\alpha},\boldsymbol{\beta}}^\ast$ divides $\varphi(d_{\boldsymbol{\alpha},\boldsymbol{\beta}})$. Let us now make some remarks on Theorems \ref{theo expand} and \ref{Criterion} and their corollaries.

$\bullet$  Note that, by \cite[Lemma $5$]{Heninger}, if $f(z)\in\mathbb{Z}[[z]]$ and if $V$ is the greatest positive integer satisfying $f(z)^{1/V}\in\mathbb{Z}[[z]]$, then the positive integers $U$ satisfying $f(z)^{1/U}\in\mathbb{Z}[[z]]$ are exactely the positive divisors of $V$. Furthermore, by \cite[Introduction]{LianYauRoots}, for all positive integers $v$ and all $C\in\mathbb{Q}$, we have $\big((Cq)^{-1}z_{\boldsymbol{\alpha},\boldsymbol{\beta}}(Cq)\big)^{1/v}\in\mathbb{Z}[[z]]$ if and only if $\big((Cz)^{-1}q_{\boldsymbol{\alpha},\boldsymbol{\beta}}(Cz)\big)^{1/v}\in\mathbb{Z}[[z]]$. We deduce that Corollary \ref{cor Root} also gives the integrality of the Taylor coefficients at the origin of roots of mirror maps.

$\bullet$  In Section \ref{section proof pos}, we prove Proposition \ref{propo positive} which generalizes Assertion $(1)$ of Theorem  \ref{Criterion}. Furthermore, if $q_{\boldsymbol{\alpha},\boldsymbol{\beta}}(z)$ is $N$-integral, then, according to Assertions $(1)$ and $(2)$ of Theorem~\ref{Criterion}, all Taylor coefficients at $z=0$ of $q_{\boldsymbol{\alpha},\boldsymbol{\beta}}(C_{\boldsymbol{\alpha},\boldsymbol{\beta}}'z)$ are positive integers, but its constant term $=0$. This leads to a  natural question: do these coefficients count any object related to the geometric origin of $q_{\boldsymbol{\alpha},\boldsymbol{\beta}}(z)$?

$\bullet$  In all cases, $\mathfrak{n}_{\boldsymbol{\alpha},\boldsymbol{\beta}}$ and $\mathfrak{n}_{\boldsymbol{\alpha},\boldsymbol{\beta}}'$ are positive integers.

$\bullet$  Suppose that $\mathcal{L}_{\boldsymbol{\alpha},\boldsymbol{\beta}}$ is an irreducible operator satisfying $(P_1)$. We can formally consider $q_{\boldsymbol{\alpha},\boldsymbol{\beta}}$ without assuming that there are at least two $1$'s in $\boldsymbol{\beta}$. But if $q_{\boldsymbol{\alpha},\boldsymbol{\beta}}$ is $N$-integral, then $q_{\boldsymbol{\alpha},\boldsymbol{\beta}}=q_{\langle\boldsymbol{\alpha}\rangle,\langle\boldsymbol{\beta}\rangle}$ and if furthermore $s\geq 2$, then there are at least two $1$'s in $\langle\boldsymbol{\beta}\rangle$ so that $q_{\boldsymbol{\alpha},\boldsymbol{\beta}}$ is the exponential of a ratio of power series canceled by $\mathcal{L}_{\langle\boldsymbol{\alpha}\rangle,\langle\boldsymbol{\beta}\rangle}$. The operator $\mathcal{L}_{\boldsymbol{\alpha},\boldsymbol{\beta}}$ may not satisfy $(P_2)$, but (see section \ref{section reduction1}) $\mathcal{L}_{\langle\boldsymbol{\alpha}\rangle,\langle\boldsymbol{\beta}\rangle}$ is an irreducible operator satisfying $(P_1)$ and $(P_2)$. Furthermore, if $q_{\boldsymbol{\alpha},\boldsymbol{\beta}}$ is $N$-integral, then we have $r=s$ so that $\mathcal{L}_{\boldsymbol{\alpha},\boldsymbol{\beta}}$ and $\mathcal{L}_{\langle\boldsymbol{\alpha}\rangle,\langle\boldsymbol{\beta}\rangle}$ are Fuchsian.

$\bullet$  As explained in more details in Section \ref{section comparison}, Theorem \ref{Criterion} generalizes previous results on the integrality of the Taylor coefficients at the origin of $q$-coordinates associated with generalized hypergeometric functions. 

$\bullet$  Let us explain the reason why Theorem \ref{Criterion} provides an effective criterion for the $N$-integrality of $q$-coordinates $q_{\boldsymbol{\alpha},\boldsymbol{\beta}}$. Given tuples $\boldsymbol{\alpha}$ and $\boldsymbol{\beta}$, Assertion $H_{\boldsymbol{\alpha},\boldsymbol{\beta}}$ can easily be checked and the identity $q_{\boldsymbol{\alpha},\boldsymbol{\beta}}(z)=q_{\langle a\boldsymbol{\alpha}\rangle,\langle a\boldsymbol{\beta}\rangle}(z)$ is equivalent to $F_{\boldsymbol{\alpha},\boldsymbol{\beta}}(z)G_{\langle a\boldsymbol{\alpha}\rangle,\langle a\boldsymbol{\beta}\rangle}(z)=F_{\langle a\boldsymbol{\alpha}\rangle,\langle a\boldsymbol{\beta}\rangle}(z)G_{\boldsymbol{\alpha},\boldsymbol{\beta}}(z)$. Let us assume that there are at least two $1$'s in $\boldsymbol{\beta}$ so that $G_{\boldsymbol{\alpha},\boldsymbol{\beta}}(z)+\log(z)F_{\boldsymbol{\alpha},\boldsymbol{\beta}}(z)$ is canceled by $\mathcal{L}_{\boldsymbol{\alpha},\boldsymbol{\beta}}$. 

On the one hand, if $r\neq s$, then $q_{\boldsymbol{\alpha},\boldsymbol{\beta}}$ is not $N$-integral. On the other hand, if $r=s$, then $F_{\boldsymbol{\alpha},\boldsymbol{\beta}}(z)G_{\langle a\boldsymbol{\alpha}\rangle,\langle a\boldsymbol{\beta}\rangle}(z)$ and $F_{\langle a\boldsymbol{\alpha}\rangle,\langle a\boldsymbol{\beta}\rangle}(z)G_{\boldsymbol{\alpha},\boldsymbol{\beta}}(z)$ are analytic functions at $z=0$ canceled by the tensor product $\mathcal{L}'$ of the differential operators $\mathcal{L}_{\boldsymbol{\alpha},\boldsymbol{\beta}}$ and $\mathcal{L}_{\langle a\boldsymbol{\alpha}\rangle,\langle a\boldsymbol{\beta}\rangle}$. Since the order of $\mathcal{L}'$ is less than or equal to $r^2$ then we have $q_{\boldsymbol{\alpha},\boldsymbol{\beta}}=q_{\langle a\boldsymbol{\alpha}\rangle,\langle a\boldsymbol{\beta}\rangle}$ if and only if the first $r^2$ Taylor coefficients at the origin of $F_{\boldsymbol{\alpha},\boldsymbol{\beta}}(z)G_{\langle a\boldsymbol{\alpha}\rangle,\langle a\boldsymbol{\beta}\rangle}(z)$ and $F_{\langle a\boldsymbol{\alpha}\rangle,\langle a\boldsymbol{\beta}\rangle}(z)G_{\boldsymbol{\alpha},\boldsymbol{\beta}}(z)$ are equal, which can be checked in a finite number of elementary algebraic operations.

$\bullet$  If $q_{\boldsymbol{\alpha},\boldsymbol{\beta}}(z)$ is $N$-integral, then the power series $q_{\boldsymbol{\alpha},\boldsymbol{\beta}}(C_{\boldsymbol{\alpha},\boldsymbol{\beta}}'z)/(C_{\boldsymbol{\alpha},\boldsymbol{\beta}}'z)$ to the power $\varphi(d_{\boldsymbol{\alpha},\boldsymbol{\beta}})/\mathfrak{n}_{\boldsymbol{\alpha},\boldsymbol{\beta}}$ lies in $\mathbb{Z}[[z]]$ so that, in some cases, a non-trivial root of $q_{\boldsymbol{\alpha},\boldsymbol{\beta}}(z)$ is $N$-integral. This suggests that one might be able to improve Assertion $(3)$ of Theorem  \ref{Criterion} by replacing $\mathfrak{n}_{\boldsymbol{\alpha},\boldsymbol{\beta}}$ by $\varphi(d_{\boldsymbol{\alpha},\boldsymbol{\beta}})$ or $d_{\boldsymbol{\alpha},\boldsymbol{\beta}}$ (\footnote{Note that $\exp\big(S_{\boldsymbol{\alpha},\boldsymbol{\beta}}(z)/\varphi(d_{\boldsymbol{\alpha},\boldsymbol{\beta}})\big)$ is the geometric mean of the $\frac{1}{z}q_{\langle a\boldsymbol{\alpha}\rangle,\langle a\boldsymbol{\beta}\rangle}(z)$.}). But this statement is not always true. Indeed, a counterexample is given by $\boldsymbol{\alpha}=(1/7,1/4,3/7,6/7)$ and $\boldsymbol{\beta}=(1,1,1,1)$, where we have $d_{\boldsymbol{\alpha},\boldsymbol{\beta}}=28$, $C_{\boldsymbol{\alpha},\boldsymbol{\beta}}'=C_{\boldsymbol{\alpha},\boldsymbol{\beta}}=2^37^2$, $\varphi(28)=12$, $\mathfrak{n}_{\boldsymbol{\alpha},\boldsymbol{\beta}}=2$,
$$
\exp\left(\frac{S_{\boldsymbol{\alpha},\boldsymbol{\beta}}(2^37^2z)}{12}\right)\in 1+4802z+\frac{81541341}{2}z^2+\frac{1328534273395}{3}z^3+z^4\mathbb{Q}[[z]]
$$
and
$$
\exp\left(\frac{S_{\boldsymbol{\alpha},\boldsymbol{\beta}}(2^37^2z)}{28}\right)\in 1+2058z+\frac{29299137}{2}z^2+z^3\mathbb{Q}[[z]].
$$

Before ending this introduction, we would like to mention that this article also contains two useful results, that play a central role in the proof of Theorem \ref{Criterion}, but we  need too many definitions to state them here.  
The first one is Proposition \ref{propo magie} stated in Section \ref{section analogues} which gives an useful formula for the $p$-adic valuation of the Taylor coefficients at the origin of $F_{\boldsymbol{\alpha},\boldsymbol{\beta}}(C_{\boldsymbol{\alpha},\boldsymbol{\beta}}z)$ when $p$ is a prime divisor of $d_{\boldsymbol{\alpha},\boldsymbol{\beta}}$.  
The second one is Theorem \ref{congruences formelles} stated in Section \ref{section cong form} which generalizes Dwork's theorem on formal congruences \cite[Theorem $1.1$]{Dwork} also used by Dwork in \cite{cycles} to obtain the  analytic continuation of certain $p$-adic functions.

While working on this article, we found an error in a lemma in Lang's book \cite[Lemma~1.1, Section 1, Chapter 14]{Lang} about the arithmetic properties of Mojita's $p$-adic Gamma function. This lemma has been used in several articles on the integrality of the Taylor coefficients of mirror maps including papers of the authors. Even if we do not use this lemma in this article, we give in Section \ref{Cor Lang} a corrected version and we explain why the initial error does not change the validity of our previous results.

\subsection{Structure of the paper}

In Section \ref{section Comments}, we make comments on results stated in introduction and we compare this results with previous ones on $N$-integrality of mirror maps associated with generalized hypergeometric functions. Furthermore, we give a corrected version of a lemma of Lang on Mojita's $p$-adic Gamma function at the end of this section.

Section \ref{section valuation} is devoted to a detailed study of the $p$-adic valuation of Pochhammer  symbol. In particular, we prove Proposition \ref{Lemma Mp} and we define and study step functions $\Delta_{\boldsymbol{\alpha},\boldsymbol{\beta}}$ associated with tuples $\boldsymbol{\alpha}$ and $\boldsymbol{\beta}$ which play a central role in proofs of Theorems \ref{theo Const}-\ref{Criterion}. 

We prove Theorem \ref{theo Const} in Section \ref{section demo 1}. 

Section \ref{section cong form} is devoted to the statement and the proof of Theorem \ref{congruences formelles} on formal congruences between formal power series. We also compare Theorem \ref{congruences formelles} with previous generalizations of Dwork's theorem on formal congruences \cite[Theorem $1.1$]{Dwork}. Theorem \ref{congruences formelles} is the most important tool in the proof of Theorem \ref{theo expand}. 

We prove Theorem \ref{theo expand} in Section \ref{section proof theo expand}, which is by far the longest and the most technical part of this article.

Sections \ref{section proof pos}, \ref{section reformulation} and \ref{section last} are respectively dedicated to the proofs of Assertions $(1)$, $(3)$ and $(2)$ of Theorem \ref{Criterion}.

\section{Comments on the main results and comparison with previous ones}\label{section Comments}

This section is devoted to a detailed study of certain consequences of Theorems \ref{theo Const}-\ref{Criterion}. In particular, we compare these theorems with previous results on the integrality of the Taylor coefficients at the origin of generalized hypergeometric series and their associated (roots of) mirror maps. This section also contains some results that we use throughout this article.

\subsection{Comments on the main results}\label{section Comments 1}

We provide precisions on Theorems A, \ref{theo Const} and \ref{Criterion}.

\subsubsection{An example of application of Theorem \ref{theo Const}}

We illustrate Theorem A and Theorem \ref{theo Const} with an example. Let $\boldsymbol{\alpha}:=(1/6,1/2,2/3)$ and $\boldsymbol{\beta}:=(1/3,1,1)$ so that we have $d_{\boldsymbol{\alpha},\boldsymbol{\beta}}=6$. According to Theorem~A, $F_{\boldsymbol{\alpha},\boldsymbol{\beta}}$ is $N$-integral if and only if, for all $a\in\{1,5\}$ and all $x\in\mathbb{R}$, we have $\xi_{\boldsymbol{\alpha},\boldsymbol{\beta}}(a,x)\geq 0$.

We have $1/6\prec 1/3\prec 1/2\prec 2/3\prec 1$ thus, for all $x\in\mathbb{R}$, we get $\xi_{\boldsymbol{\alpha},\boldsymbol{\beta}}(1,x)\geq 0$. Furthermore, we have $1/3+3=10/3\prec 5/2\prec 5/3\prec 5/6\prec 5$ and thus, for all $x\in\mathbb{R}$, we get $\xi_{\boldsymbol{\alpha},\boldsymbol{\beta}}(5,x)\geq 0$. This shows that $F_{\boldsymbol{\alpha},\boldsymbol{\beta}}$ is $N$-integral.

Moreover, we have $r=s$, all elements of $\boldsymbol{\alpha}$ and $\boldsymbol{\beta}$ lie in $(0,1]$, $\lambda_2(\boldsymbol{\alpha},\boldsymbol{\beta})=1-3=-2$ and $\lambda_3(\boldsymbol{\alpha},\boldsymbol{\beta})=1-2=-1$ thus, according to Theorem \ref{theo Const}, we get
$$
C_{\boldsymbol{\alpha},\boldsymbol{\beta}}=\frac{6\cdot 2\cdot 3}{3}2^{-\lfloor -2\rfloor}3^{-\lfloor -1/2\rfloor}=2^43^2.
$$

\subsubsection{$N$-integrality of $F_{\langle\boldsymbol{\alpha}\rangle,\langle\boldsymbol{\beta}\rangle}$}\label{section reduction1}

We show that if $F_{\boldsymbol{\alpha},\boldsymbol{\beta}}$ is $N$-integral then $F_{\langle\boldsymbol{\alpha}\rangle,\langle\boldsymbol{\beta}\rangle}$ is also $N$-integral. The converse is false in general, a counterexample being given by $\boldsymbol{\alpha}=(1/2,1/2)$ and $\boldsymbol{\beta}=(3/2,1)$ since we have
$3/2\prec 1/2\prec 1$ and $\langle\boldsymbol{\alpha}\rangle=(1/2,1/2)$, $\langle\boldsymbol{\beta}\rangle=(1/2,1)$. But, if $\langle\boldsymbol{\alpha}\rangle$ and $\langle\boldsymbol{\beta}\rangle$ are disjoint, then, for all $a\in\{1,\dots,d_{\boldsymbol{\alpha},\boldsymbol{\beta}}\}$ coprime to $d_{\boldsymbol{\alpha},\boldsymbol{\beta}}$, $\langle a\boldsymbol{\alpha}\rangle$ and $\langle a\boldsymbol{\beta}\rangle$ are disjoint. Hence, applying Theorem A, we obtain that ($F_{\langle\boldsymbol{\alpha}\rangle,\langle\boldsymbol{\beta}\rangle}$ is $N$-integral)$\Rightarrow$($F_{\boldsymbol{\alpha},\boldsymbol{\beta}}$ is $N$-integral). More precisely, we shall prove the following proposition that we use several times in this article.

\begin{propo}\label{propo reduction}
Let $\boldsymbol{\alpha}$ and $\boldsymbol{\beta}$ be tuples of parameters in $\mathbb{Q}\setminus\mathbb{Z}_{\leq 0}$ and $a\in\{1,\dots,d_{\boldsymbol{\alpha},\boldsymbol{\beta}}\}$ coprime to $d_{\boldsymbol{\alpha},\boldsymbol{\beta}}$. Then we have $d_{\langle a\boldsymbol{\alpha}\rangle,\langle a\boldsymbol{\beta}\rangle}=d_{\boldsymbol{\alpha},\boldsymbol{\beta}}$. Let $c\in\{1,\dots,d_{\boldsymbol{\alpha},\boldsymbol{\beta}}\}$ coprime to $d_{\boldsymbol{\alpha},\boldsymbol{\beta}}$ and $x\in\mathbb{R}$ be fixed and let $b\in\{1,\dots,d_{\boldsymbol{\alpha},\boldsymbol{\beta}}\}$ be such that $b\equiv ca\mod d_{\boldsymbol{\alpha},\boldsymbol{\beta}}$. Then we have
$$
\xi_{\langle a\boldsymbol{\alpha}\rangle,\langle a\boldsymbol{\beta}\rangle}(c,x)=\begin{cases}
\xi_{\boldsymbol{\alpha},\boldsymbol{\beta}}(b,\langle x\rangle-)\textup{ if $x>c$};\\
r-s\textup{ if $x\leq c$ and $x\in\mathbb{Z}$};\\
\xi_{\boldsymbol{\alpha},\boldsymbol{\beta}}(b,\langle x\rangle-)\textup{ or }\xi_{\boldsymbol{\alpha},\boldsymbol{\beta}}(b,\langle x\rangle+)\textup{ otherwise},
\end{cases}.
$$
where $r$, respectively $s$, is the number of elements of $\boldsymbol{\alpha}$, respectively of $\boldsymbol{\beta}$.
\end{propo}

\begin{Remark}
For all $a\in\{1,\dots,d_{\boldsymbol{\alpha},\boldsymbol{\beta}}\}$ coprime to $d_{\boldsymbol{\alpha},\boldsymbol{\beta}}$, $r-s$ is the limit of $\xi_{\boldsymbol{\alpha},\boldsymbol{\beta}}(a,n)$ when $n\in\mathbb{Z}$ tends to $-\infty$.
\end{Remark}

In Proposition \ref{propo reduction} and throughout this article, if $f$ is a function defined on $\mathcal{D}\subset\mathbb{R}$ and $x\in\mathcal{D}$, then we adopt the notations
$$
f(x+):=\underset{y\in\mathcal{D},y>x}{\underset{y\rightarrow x}{\lim}}f(y)\quad\textup{and}\quad f(x-):=\underset{y\in\mathcal{D},y<x}{\underset{y\rightarrow x}{\lim}}f(y).
$$

\begin{proof}
For all elements $\alpha$ and $\beta$ of $\boldsymbol{\alpha}$ or $\boldsymbol{\beta}$, we have $\big\langle c\langle a\alpha\rangle\big\rangle=\langle ca\alpha\rangle=\langle b\alpha\rangle$ and $\langle b\alpha\rangle=\langle b\beta\rangle$ if and only if $\langle\alpha\rangle=\langle\beta\rangle$. If $\langle b\alpha\rangle=\langle x\rangle$, then we have $c\langle a\alpha\rangle\preceq x\Leftrightarrow c\langle a\alpha\rangle\geq x$. It follows that if $x>c$, then we have 
\begin{align*}
\xi_{\langle a\boldsymbol{\alpha}\rangle,\langle a\boldsymbol{\beta}\rangle}(c,x)&=\#\big\{1\leq i\leq r\,:\,\langle b\alpha_i\rangle<\langle x\rangle\big\}-\#\big\{1\leq j\leq s\,:\,\langle b\beta_j\rangle<\langle x\rangle\big\}\\
&=\xi_{\boldsymbol{\alpha},\boldsymbol{\beta}}(b,\langle x\rangle-).
\end{align*} 

If $x\in\mathbb{Z}$ and $x\leq c$, then we have $\langle x\rangle=1$ and $\xi_{\langle a\boldsymbol{\alpha}\rangle,\langle a\boldsymbol{\beta}\rangle}(c,x)=r-s$. Now we assume that $x\leq c$ and $x\notin\mathbb{Z}$. If $\alpha$ and $\beta$ are elements of $\boldsymbol{\alpha}$ or $\boldsymbol{\beta}$ satisfying $\langle x\rangle=\langle b\alpha\rangle=\langle b\beta\rangle$, then $\langle\alpha\rangle=\langle\beta\rangle$ so $\langle a\alpha\rangle=\langle a\beta\rangle$ and we obtain that $c\langle a\alpha\rangle\preceq x\Leftrightarrow c\langle a\beta\rangle\preceq x$. Thus we have
\begin{align*}
\xi_{\langle a\boldsymbol{\alpha}\rangle,\langle a\boldsymbol{\beta}\rangle}(c,x)&=\begin{cases}
\#\big\{1\leq i\leq r\,:\,\langle b\alpha_i\rangle<\langle x\rangle\big\}-\#\big\{1\leq j\leq s\,:\,\langle b\beta_j\rangle<\langle x\rangle\big\}\\ \textup{or}\\ \#\big\{1\leq i\leq r\,:\,\langle b\alpha_i\rangle\leq\langle x\rangle\big\}-\#\big\{1\leq j\leq s\,:\,\langle b\beta_j\rangle\leq\langle x\rangle\big\}\end{cases}
\\
&=\begin{cases}\xi_{\boldsymbol{\alpha},\boldsymbol{\beta}}(b,\langle x\rangle -)\\ \textup{or}\\ \xi_{\boldsymbol{\alpha},\boldsymbol{\beta}}(b,\langle x\rangle +)\end{cases}
\end{align*} 
because $\langle x\rangle<1$.
\end{proof}

By Proposition \ref{propo reduction} with $a=1$ together with Theorem A, we obtain that, if $F_{\boldsymbol{\alpha},\boldsymbol{\beta}}$ is $N$-integral, then $F_{\langle\boldsymbol{\alpha}\rangle,\langle\boldsymbol{\beta}\rangle}$ is also $N$-integral. Similarly, if $H_{\boldsymbol{\alpha},\boldsymbol{\beta}}$ holds then $H_{\langle\boldsymbol{\alpha}\rangle,\langle\boldsymbol{\beta}\rangle}$ also holds. More precisely, we have the following result, used several times in the proof of Theorem \ref{Criterion}.

\begin{Lemma}\label{H transfert}
Let $\boldsymbol{\alpha}$ and $\boldsymbol{\beta}$ be two disjoint tuples of parameters in $\mathbb{Q}\setminus\mathbb{Z}_{\leq 0}$ with the same number of elements and such that $H_{\boldsymbol{\alpha},\boldsymbol{\beta}}$ holds. Then, for all $a\in\{1,\dots,d_{\boldsymbol{\alpha},\boldsymbol{\beta}}\}$ coprime to $d_{\boldsymbol{\alpha},\boldsymbol{\beta}}$, Assertion $H_{\langle a\boldsymbol{\alpha}\rangle,\langle a\boldsymbol{\beta}\rangle}$ holds.
\end{Lemma}

\begin{proof}
Let $c\in\{1,\dots,d_{\boldsymbol{\alpha},\boldsymbol{\beta}}\}$ be coprime to $d_{\boldsymbol{\alpha},\boldsymbol{\beta}}$ and $x\in\mathbb{R}$ be such that $m_{\langle a\boldsymbol{\alpha}\rangle,\langle a\boldsymbol{\beta}\rangle}(c)\preceq x\prec c$. We shall prove that $\xi_{\langle a\boldsymbol{\alpha}\rangle,\langle a\boldsymbol{\beta}\rangle}(c,x)\geq 1$ by applying Proposition~\ref{propo reduction}. 

Let $b\in\{1,\dots,d_{\boldsymbol{\alpha},\boldsymbol{\beta}}\}$ be such that $b\equiv ac\mod d_{\boldsymbol{\alpha},\boldsymbol{\beta}}$. First, note that there exists an element $\alpha$ of $\boldsymbol{\alpha}$ or $\boldsymbol{\beta}$ such that $c\langle a\alpha\rangle\preceq x$, that is $\langle x\rangle>\langle b\alpha\rangle$ or $\big(\langle x\rangle=\langle b\alpha\rangle$ and $c\langle a\alpha\rangle\geq x\big)$. We distinguish three cases.
\medskip

$\bullet$ If $x>c$ then we have $\langle x\rangle >\langle b\alpha\rangle$ and $\xi_{\langle a\boldsymbol{\alpha}\rangle,\langle a\boldsymbol{\beta}\rangle}(c,x)=\xi_{\boldsymbol{\alpha},\boldsymbol{\beta}}(b,\langle x\rangle-)$. Thus there exists $y\in\mathbb{R}$, $m_{\boldsymbol{\alpha},\boldsymbol{\beta}}(b)\preceq y\prec b$ such that $\xi_{\langle a\boldsymbol{\alpha}\rangle,\langle a\boldsymbol{\beta}\rangle}(c,x)=\xi_{\boldsymbol{\alpha},\boldsymbol{\beta}}(b,y)\geq 1$. 
\medskip

$\bullet$ If $x\leq c$ and $x\notin\mathbb{Z}$, then we have $\langle x\rangle<1$ and $\xi_{\langle a\boldsymbol{\alpha}\rangle,\langle a\boldsymbol{\beta}\rangle}(c,x)=\xi_{\boldsymbol{\alpha},\boldsymbol{\beta}}(b,\langle x\rangle-)$ or $\xi_{\boldsymbol{\alpha},\boldsymbol{\beta}}(b,\langle x\rangle+)$. Since $\langle x\rangle\geq\langle b\alpha\rangle$, there exists $y\in\mathbb{R}$, $m_{\boldsymbol{\alpha},\boldsymbol{\beta}}(b)\preceq y\prec b$ such that $\xi_{\boldsymbol{\alpha},\boldsymbol{\beta}}(b,\langle x\rangle+)=\xi_{\boldsymbol{\alpha},\boldsymbol{\beta}}(b,y)\geq 1$. Furthermore, if $\langle x\rangle>\langle b\alpha\rangle$ then we have $\xi_{\boldsymbol{\alpha},\boldsymbol{\beta}}(b,\langle x\rangle-)\geq~1$ as in the case $x>c$. Now we assume that, for all elements $\beta$ of $\boldsymbol{\alpha}$ or $\boldsymbol{\beta}$, we have $\langle x\rangle\leq\langle b\beta\rangle$. Hence we have $\langle x\rangle=\langle b\alpha\rangle$ and, as explained in the proof of Proposition \ref{propo reduction}, we have
\begin{align*}
\xi_{\langle a\boldsymbol{\alpha}\rangle,\langle a\boldsymbol{\beta}\rangle}(c,x)&=\#\big\{1\leq i\leq r\,:\,\langle b\alpha_i\rangle\leq\langle x\rangle\big\}-\#\big\{1\leq j\leq s\,:\,\langle b\beta_j\rangle\leq\langle x\rangle\big\}\\
&=\xi_{\boldsymbol{\alpha},\boldsymbol{\beta}}(b,\langle x\rangle+)\geq 1.
\end{align*}
\medskip

$\bullet$ It remains to consider the case $x\leq c$ and $x\in\mathbb{Z}$. But in this case we do not have $x\prec c$ thus $H_{\langle a\boldsymbol{\alpha}\rangle,\langle a\boldsymbol{\beta}\rangle}$ is proved.
\end{proof}

\subsubsection{Numerators of the elements of $\boldsymbol{\alpha}$ and $\boldsymbol{\beta}$}\label{section numerators}

Let $\boldsymbol{\alpha}=(\alpha_1,\dots,\alpha_r)$ and $\boldsymbol{\beta}=(\beta_1,\dots,\beta_r)$ be tuples of parameters in $\mathbb{Q}\setminus\mathbb{Z}_{\leq 0}$. Then Theorem \ref{theo Const} gives a necessary condition on the numerators of elements of $\boldsymbol{\alpha}$ and $\boldsymbol{\beta}$ for $F_{\boldsymbol{\alpha},\boldsymbol{\beta}}$ to be $N$-integral. Indeed, let us assume that $F_{\boldsymbol{\alpha},\boldsymbol{\beta}}$ is $N$-integral. Then, according to Section \ref{section reduction1}, $F_{\langle\boldsymbol{\alpha}\rangle,\langle\boldsymbol{\beta}\rangle}$ is also $N$-integral. We write $n_i$, respectively $n_j'$, for the exact numerator of $\langle\alpha_i\rangle$, respectively of $\langle\beta_j\rangle$. Then by Theorem~\ref{theo Const}, the first-order Taylor coefficient at the origin of $F_{\langle\boldsymbol{\alpha}\rangle,\langle\boldsymbol{\beta}\rangle}\big(C_{\langle\boldsymbol{\alpha}\rangle,\langle\boldsymbol{\beta}\rangle}z\big)$ is (\footnote{Note that, for all primes $p$, we have $\lambda_p(\boldsymbol{\alpha},\boldsymbol{\beta})=\lambda_p(\langle\boldsymbol{\alpha}\rangle,\langle\boldsymbol{\beta}\rangle)$.})
$$
\frac{\prod_{i=1}^rn_i}{\prod_{j=1}^rn_j'}\underset{p\mid d_{\boldsymbol{\alpha},\boldsymbol{\beta}}}{\prod}p^{-\left\lfloor\frac{\lambda_p(\boldsymbol{\alpha},\boldsymbol{\beta})}{p-1}\right\rfloor}\in\mathbb{Z},
$$
so that, for all primes $p$, we have 
$$
v_p\left(\frac{\prod_{i=1}^rn_i}{\prod_{j=1}^rn_j'}\right)\geq \left\lfloor\frac{\lambda_p(\boldsymbol{\alpha},\boldsymbol{\beta})}{p-1}\right\rfloor.
$$
For instance, the last inequality is not satisfied with $p=2$, $\boldsymbol{\alpha}=(1/5,1/3,3/5)$ and $\boldsymbol{\beta}=(1/2,1,1)$, or with $p=3$, $\boldsymbol{\alpha}=(1/7,2/7,4/7,5/7)$ and $\boldsymbol{\beta}=(3/4,1,1,1)$. Thus in both cases the associated generalized hypergeometric series $F_{\boldsymbol{\alpha},\boldsymbol{\beta}}$ is not $N$-integral.

\subsubsection{The Eisenstein constant of algebraic generalized hypergeometric series}

Let $\boldsymbol{\alpha}=(\alpha_1,\dots,\alpha_r)$ and $\boldsymbol{\beta}=(\beta_1,\dots,\beta_r)$ be tuples of parameters in $\mathbb{Q}\setminus\mathbb{Z}_{\leq 0}$. If $F_{\boldsymbol{\alpha},\boldsymbol{\beta}}(z)$ is algebraic over $\mathbb{Q}(z)$ then $F_{\boldsymbol{\alpha},\boldsymbol{\beta}}$ is $N$-integral (Eisenstein's theorem) and one can apply Theorem \ref{theo Const} to get arithmetical properties of the Eisenstein constant of $F_{\boldsymbol{\alpha},\boldsymbol{\beta}}$. For the sake of completeness, let us remind the reader of a result of Beukers and Heckman \cite[Theorem 1.5]{Beukers} proved in \cite{Beukers Heckman} on algebraic hypergeometric functions: 
\begin{quote}
``Assume that $\beta_r=1$ and that $\mathcal{L}_{\boldsymbol{\alpha},\boldsymbol{\beta}}$ is irreducible. Then the set of solutions of the hypergeometric equation associated with $\mathcal{L}_{\boldsymbol{\alpha},\boldsymbol{\beta}}$ consists of algebraic functions (over $\mathbb{C}(z)$) if and only if the sets $\{a\alpha_i\,:\,1\leq 1\leq r\}$ and $\{a\beta_i\,:\,1\leq i\leq r\}$ interlace modulo $1$ for every integer $a$ with $1\leq a\leq d_{\boldsymbol{\alpha},\boldsymbol{\beta}}$ and $\gcd(a,d_{\boldsymbol{\alpha},\boldsymbol{\beta}})=1$.''
\end{quote}
The sets $\{\alpha_i\,:\,1\leq i\leq r\}$ and $\{\beta_i\,:\,1\leq i\leq r\}$ interlace modulo $1$ if the points of the sets $\{e^{2\pi i\alpha_j}\,:\,1\leq j\leq r\}$ and $\{e^{2\pi i\beta_j}\,:\,1\leq j\leq r\}$ occur alternatively when running along the unit circle.

Beukers-Heckman criterion can be reformulated in terms of Christol's functions as follows. 
\begin{quote}
``Assume that $\beta_r=1$ and that $\mathcal{L}_{\boldsymbol{\alpha},\boldsymbol{\beta}}$ is irreducible. Then the solution set of the hypergeometric equation associated with $\mathcal{L}_{\boldsymbol{\alpha},\boldsymbol{\beta}}$ consists of algebraic functions (over $\mathbb{C}(z)$) if and only if, for every integer $a$ with $1\leq a\leq d_{\boldsymbol{\alpha},\boldsymbol{\beta}}$ and $\gcd(a,d_{\boldsymbol{\alpha},\boldsymbol{\beta}})=1$, we have $\xi_{\boldsymbol{\alpha},\boldsymbol{\beta}}(a,\mathbb{R})=\{0,1\}$.''
\end{quote}

\subsection{Comparison with previous results}\label{section comparison}

\subsubsection{Theorem \ref{theo expand} and previous results}\label{section comparison 1}

The first result on $p$-adic integrality of $q_{\boldsymbol{\alpha},\boldsymbol{\beta}}$ is due to Dwork \cite[Theorem $4.1$]{Dwork}. This result enables us to prove that, for particular tuples $\boldsymbol{\alpha}$ and $\boldsymbol{\beta}$, we have $q_{\boldsymbol{\alpha},\boldsymbol{\beta}}(z)\in\mathbb{Z}_p[[z]]$ for almost all primes $p$. It follows without much trouble that   $q_{\boldsymbol{\alpha},\boldsymbol{\beta}}$ is $N$-integral. Thus we know that there exists $C\in\mathbb{N}$, $C\geq 1$, such that $q_{\boldsymbol{\alpha},\boldsymbol{\beta}}(Cz)\in\mathbb{Z}[[z]]$ but the only information on $C$ given by Dwork's result is that we can choose $C$ with prime divisors in an explicit finite set associated with $(\boldsymbol{\alpha},\boldsymbol{\beta})$. Hence, improvements of Dwork's method consist in finding explicit formulas for $C$ and we discuss such previous improvements in the next section. But Theorem \ref{theo expand} is more general and, in order to compare this theorem with Dwork's result \cite[Theorem~$4.1$]{Dwork}, we introduce some notations that we use throughout this article. Until the end of this section, we restrict ourself to the case where $\boldsymbol{\alpha}$ and $\boldsymbol{\beta}$ have the same numbers of elements. 
\medskip

$\bullet$ For all primes $p$ and all $p$-adic integers $\alpha$ in $\mathbb{Q}$, we write $\mathfrak{D}_p(\alpha)$ for the unique $p$-adic integer in $\mathbb{Q}$ satisfying $p\mathfrak{D}_p(\alpha)-\alpha\in\{0,\dots,p-1\}$. The operator $\alpha\mapsto\mathfrak{D}_p(\alpha)$ has been used by Dwork in \cite{Dwork} and denoted by $\alpha\mapsto\alpha'$ (\footnote{See Section \ref{section valuation} for a detailed study of Dwork's map $\mathfrak{D}_p$.}). 

$\bullet$ For all primes $p$, all $x\in\mathbb{Q}\cap\mathbb{Z}_p$ and all $a\in[0,p)$ 
we define
$$
\rho_p(a,x):=\begin{cases}
\textup{$0$ if $a\leq p\mathfrak{D}_p(x)-x$;}\\
\textup{$1$ if $a>p\mathfrak{D}_p(x)-x$.}
\end{cases}.
$$
$\bullet$ We write $\boldsymbol{\alpha}=(\alpha_1,\dots,\alpha_r)$ and $\boldsymbol{\beta}=(\beta_1,\dots,\beta_r)$. Let $r'$ be the number of elements $\beta_i$ of $\boldsymbol{\beta}$ such that $\beta_i\neq 1$. We rearrange the subscripts so that $\beta_i\neq 1$ for $i\leq r'$. For all $a\in[0,p)$ and all $k\in\mathbb{N}$, we set
$$
N^k_{p,\boldsymbol{\alpha}}(a)=\sum_{i=1}^r\rho_p\big(a,\mathfrak{D}_p^k(\alpha_i)\big)\quad\textup{and}\quad N^k_{p,\boldsymbol{\beta}}(a)=\sum_{i=1}^{r'}\rho_p\big(a,\mathfrak{D}_p^k(\beta_i)\big).
$$
$\bullet$ For a given prime $p$ not dividing $d_{\boldsymbol{\alpha},\boldsymbol{\beta}}$, we define two assertions: 
\begin{itemize}
\item[$(v)_p$] for all $i\in\{1,\dots,r'\}$ and all $k\in\mathbb{N}$, we have $\mathfrak{D}_p^k(\beta_i)\in\mathbb{Z}_p^\times$;
\item[$(vi)_p$] for all $a\in[0,p)$ and all $k\in\mathbb{N}$, we have either $N^k_{p,\boldsymbol{\alpha}}(a)=N^k_{p,\boldsymbol{\beta}}(a+)=0$ or $N^k_{p,\boldsymbol{\alpha}}(a)-N^k_{p,\boldsymbol{\beta}}(a+)\geq 1$.
\end{itemize}
\medskip
Dwork's result \cite[Theorem $4.1$]{Dwork} restricted to the case where $\boldsymbol{\alpha}$ and $\boldsymbol{\beta}$ have the same number of elements is the following.
\begin{theoB}[Dwork]
Let $\boldsymbol{\alpha}$ and $\boldsymbol{\beta}$ be two tuples of parameters in $\mathbb{Q}\setminus\mathbb{Z}_{\leq 0}$ with the same number of  elements. Let $p$ be a prime not dividing $d_{\boldsymbol{\alpha},\boldsymbol{\beta}}$ such that $\boldsymbol{\alpha}$ and $\boldsymbol{\beta}$ satisfy $(v)_p$ and $(vi)_p$. Then we have
$$
\frac{G_{\mathfrak{D}_p(\boldsymbol{\alpha}),\mathfrak{D}_p(\boldsymbol{\beta})}}{F_{\mathfrak{D}_p(\boldsymbol{\alpha}),\mathfrak{D}_p(\boldsymbol{\beta})}}(z^p)-p\frac{G_{\boldsymbol{\alpha},\boldsymbol{\beta}}(z)}{F_{\boldsymbol{\alpha},\boldsymbol{\beta}}}(z)\in pz\mathbb{Z}_p[[z]].
$$
\end{theoB}

Now let us assume that $\boldsymbol{\alpha}$ and $\boldsymbol{\beta}$ are disjoint with elements in $(0,1]$ and that $H_{\boldsymbol{\alpha},\boldsymbol{\beta}}$ holds. For all primes $p$ not dividing $d_{\boldsymbol{\alpha},\boldsymbol{\beta}}$, we have $\mathfrak{D}_p(\boldsymbol{\alpha})=\langle\omega\boldsymbol{\alpha}\rangle$ and $\mathfrak{D}_p(\boldsymbol{\beta})=\langle\omega\boldsymbol{\beta}\rangle$ where $\omega\in\mathbb{Z}$ satisfies $\omega p\equiv 1\mod d_{\boldsymbol{\alpha},\boldsymbol{\beta}}$ (see Section \ref{section Dwork map} below).  Then, by Theorem \ref{theo expand} for a fixed prime $p$ and $b=t=1$, we obtain that 
\begin{equation}\label{refer 0}
\frac{G_{\mathfrak{D}_p(\boldsymbol{\alpha}),\mathfrak{D}_p(\boldsymbol{\beta})}}{F_{\mathfrak{D}_p(\boldsymbol{\alpha}),\mathfrak{D}_p(\boldsymbol{\beta})}}(C_{\boldsymbol{\alpha},\boldsymbol{\beta}}'z^p)-p\frac{G_{\boldsymbol{\alpha},\boldsymbol{\beta}}}{F_{\boldsymbol{\alpha},\boldsymbol{\beta}}}(C_{\boldsymbol{\alpha},\boldsymbol{\beta}}'z)\in pz\mathbb{Z}_p[[z]].
\end{equation}

Thus, contrary to Theorem B, there is no restriction on the primes $p$ because of the constant $C_{\boldsymbol{\alpha},\boldsymbol{\beta}}'$. Furthermore, in the proof of Lemma \ref{Lemma Dwork} in Section \ref{section Lemma Dwork}, we show that if $H_{\boldsymbol{\alpha},\boldsymbol{\beta}}$ holds then $\boldsymbol{\alpha}$ and $\boldsymbol{\beta}$ satisfy Assertions $(v)_p$ and $(vi)_p$ for almost all primes $p$. By  Theorems \ref{Criterion} and B, the converse holds when, for all $a\in\{1,\dots,d_{\boldsymbol{\alpha},\boldsymbol{\beta}}\}$ coprime to $d_{\boldsymbol{\alpha},\boldsymbol{\beta}}$, we have 
$$
\frac{G_{\langle a\boldsymbol{\alpha}\rangle,\langle a\boldsymbol{\beta}\rangle}}{F_{\langle a\boldsymbol{\alpha}\rangle,\langle a\boldsymbol{\beta}\rangle}}(z)=\frac{G_{\boldsymbol{\alpha},\boldsymbol{\beta}}}{F_{\boldsymbol{\alpha},\boldsymbol{\beta}}}(z).
$$
Indeed, in this case, Theorem B in combination with Corollary \ref{cor exp} implies that, for almost all primes $p$, we have $q_{\boldsymbol{\alpha},\boldsymbol{\beta}}(z)\in\mathbb{Z}_p[[z]]$. Then it is a simple exercise to show that $q_{\boldsymbol{\alpha},\boldsymbol{\beta}}$ is $N$-integral and, by Theorem \ref{Criterion}, we obtain that $H_{\boldsymbol{\alpha},\boldsymbol{\beta}}$ holds.
\medskip

The main improvement in Theorem  \ref{theo expand} is the use of algebras of $\mathbb{Z}_p$-valued functions instead of $\mathbb{Z}_p$. This is precisely this generalization which enables us to prove the integrality of the Taylor coefficients of certain roots of $S_{\boldsymbol{\alpha},\boldsymbol{\beta}}(C_{\boldsymbol{\alpha},\boldsymbol{\beta}}'z)$.

\subsubsection{Theorem \ref{Criterion} and previous results}\label{section comparison 2}

The constants $C\in\mathbb{Q}$ such that an $N$-integral canonical coordinate $q_{\boldsymbol{\alpha},\boldsymbol{\beta}}$ satisfies $q_{\boldsymbol{\alpha},\boldsymbol{\beta}}(Cz)\in\mathbb{Z}[[z]]$ was  first studied when there exist some disjoint tuples of positive integers $\mathbf{e}=(e_1,\dots,e_u)$, $\mathbf{f}=(f_1,\dots,f_v)$ and a constant $C_0\in\mathbb{Q}$ such that
\begin{equation}\label{HyperFact}
F_{\boldsymbol{\alpha},\boldsymbol{\beta}}(C_0z)=\sum_{n=0}^{\infty}\frac{(e_1n)!\cdots,(e_un)!}{(f_1n)!\cdots(f_vn)!}z^n\in\mathbb{Z}[[z]].
\end{equation}
The results obtained by Lian and Yau \cite{LianYau}, Zudilin \cite{Zudilin0}, Krattenthaler-Rivoal \cite{Tanguy1} and Delaygue \cite{Delaygue1} led to an effective criterion \cite[Theorem $1$]{Delaygue1} based on simple analytical properties of Landau's function
$$
\boldsymbol{\Delta}_{\mathbf{e},\mathbf{f}}(x):=\sum_{i=1}^u\lfloor e_ix\rfloor-\sum_{j=1}^v\lfloor f_jx\rfloor.
$$ 
By combining and reformulating this criterion and \cite[Theorem $3$]{Delaygue1}, we obtain the following result.
\begin{theoC}[Delaygue]
If \eqref{HyperFact} holds, then the following assertions are equivalent:
\begin{enumerate}
\item $q_{\boldsymbol{\alpha},\boldsymbol{\beta}}(z)$ is $N$-integral;
\item $q_{\boldsymbol{\alpha},\boldsymbol{\beta}}(C_0z)\in\mathbb{Z}[[z]]$;
\item we have $\sum_{i=1}^ue_i=\sum_{j=1}^vf_j$ and, for all $x\in[1/M_{\mathbf{e},\mathbf{f}},1[$, we have $\boldsymbol{\Delta}_{\mathbf{e},\mathbf{f}}(x)\geq 1$, where $M_{\mathbf{e},\mathbf{f}}$ is the largest element of $\mathbf{e}$ and $\mathbf{f}$.
\end{enumerate}
\end{theoC}

According to \cite[Proposition $2$]{Delaygue1}, one can write (a rescaling of) $F_{\boldsymbol{\alpha},\boldsymbol{\beta}}$ as the generating function of a sequence of factorial ratios if and only if $\boldsymbol{\alpha}$ and $\boldsymbol{\beta}$ are $R$-partitioned. In this case, Landau's criterion \cite{Landau} asserts that the additional condition of integrality in \eqref{HyperFact} is equivalent to the nonnegativity of $\boldsymbol{\Delta}_{\mathbf{e},\mathbf{f}}$ on $[0,1]$, which can be checked easily because, by \cite[Proposition $3$]{Delaygue1}, for all $x\in[0,1]$, we have
\begin{equation}\label{BoldDelta}
\boldsymbol{\Delta}_{\mathbf{e},\mathbf{f}}(x)=\#\{i\,:\,x\geq\alpha_i\}-\#\{j\,:\,x\geq\beta_j\}.
\end{equation}

Furthermore, by \cite[Proposition $2$]{Delaygue1}, if $\boldsymbol{\alpha}=(\alpha_1,\dots,\alpha_r)$, respectively $\boldsymbol{\beta}=(\beta_1,\dots,\beta_s)$, is the concatenation of tuples $(b/N_i)_{b\in\{1,\dots,N_i\},\gcd(b,N_i)=1}$, $1\leq i\leq r'$, respectively of tuples $(b/N_j')_{b\in\{1,\dots,N_j'\},\gcd(b,N_j')=1}$, $1\leq j\leq s'$, then we have
\begin{equation}\label{formule C_0}
C_0=\frac{\prod_{i=1}^{r'}N_i^{\varphi(N_i)}\prod_{p\mid N_i}p^{\frac{\varphi(N_i)}{p-1}}}{\prod_{j=1}^{s'}N_j'^{\varphi(N_j')}\prod_{p\mid N_j'}p^{\frac{\varphi(N_j')}{p-1}}}\quad\textup{and}\quad\sum_{i=1}^ue_i-\sum_{j=1}^vf_j=r-s.
\end{equation}

Let us show that Theorem \ref{Criterion} implies Theorem C. Let $\boldsymbol{\alpha}$ and $\boldsymbol{\beta}$ be disjoint tuples of parameters in $\mathbb{Q}\setminus\mathbb{Z}_{\leq 0}$ such that \eqref{HyperFact} holds. Then $\boldsymbol{\alpha}$ and $\boldsymbol{\beta}$ are $R$-partitioned and their elements lie in $(0,1]$ so that $\langle\boldsymbol{\alpha}\rangle$ and $\langle\boldsymbol{\beta}\rangle$ are disjoint and $F_{\boldsymbol{\alpha},\boldsymbol{\beta}}$ is $N$-integral. First we prove that if $r=s$, then we have $C_{\boldsymbol{\alpha},\boldsymbol{\beta}}'=C_{\boldsymbol{\alpha},\boldsymbol{\beta}}=C_0$. We write $\lambda_p$ for $\lambda_p(\boldsymbol{\alpha},\boldsymbol{\beta})$. Since $\boldsymbol{\alpha}$ and $\boldsymbol{\beta}$ are $R$-partitioned, the number of elements of $\boldsymbol{\alpha}$ and $\boldsymbol{\beta}$ with exact denominator divisible by $4$ is a sum of multiple of integers of the form $\varphi(2^k)$ with $k\in\mathbb{N}$, $k\geq 2$, so this number is even. Thus, we have $C_{\boldsymbol{\alpha},\boldsymbol{\beta}}'=C_{\boldsymbol{\alpha},\boldsymbol{\beta}}$. Furthermore, for all primes $p$, we have 
$$
\lambda_p=r-\underset{p\mid N_i}{\sum_{i=1}^{r'}}\varphi(N_i)-s+\underset{p\mid N_j'}{\sum_{j=1}^{s'}}\varphi(N_j')=-\underset{p\mid N_i}{\sum_{i=1}^{r'}}\varphi(N_i)+\underset{p\mid N_j'}{\sum_{j=1}^{s'}}\varphi(N_j').
$$
If $p$ divides $N_i$ then $p-1$ divides $\varphi(N_i)$ so that $-\big\lfloor\lambda_p/(p-1)\big\rfloor=-\lambda_p/(p-1)$ and $C_{\boldsymbol{\alpha},\boldsymbol{\beta}}=C_0$ as expected. Now we assume that \eqref{HyperFact} and Theorem \ref{Criterion} hold and we prove that Assertions $(1)$, $(2)$ and $(3)$ of Theorem  C are equivalent.
\medskip

$\bullet$ $(1)\Rightarrow(2)$: If $q_{\boldsymbol{\alpha},\boldsymbol{\beta}}(z)$ is $N$-integral, then we obtain that $q_{\boldsymbol{\alpha},\boldsymbol{\beta}}(C_{\boldsymbol{\alpha},\boldsymbol{\beta}}'z)\in\mathbb{Z}[[z]]$ and $r=s$ so that $C_{\boldsymbol{\alpha},\boldsymbol{\beta}}'=C_0$ and Assertion $(2)$ of Theorem  C holds.
\medskip

$\bullet$ $(2)\Rightarrow(3)$: If $q_{\boldsymbol{\alpha},\boldsymbol{\beta}}(C_0z)\in\mathbb{Z}[[z]]$ then $q_{\boldsymbol{\alpha},\boldsymbol{\beta}}(z)$ is $N$-integral and, according to Theorem~\ref{Criterion}, we have $r=s$ and $H_{\boldsymbol{\alpha},\boldsymbol{\beta}}$ is true. We deduce that we have $\sum_{i=1}^ue_i=\sum_{j=1}^vf_j$.  Now, since $\boldsymbol{\alpha}$ and $\boldsymbol{\beta}$ are disjoint tuples with elements in $(0,1]$, equation \eqref{BoldDelta} ensures that the assertions ``for all $x\in[1/M_{\mathbf{e},\mathbf{f}},1[$, we have $\boldsymbol{\Delta}_{\mathbf{e},\mathbf{f}}(x)\geq 1$'' and ``for all $x\in\mathbb{R}$, $m_{\boldsymbol{\alpha},\boldsymbol{\beta}}(1)\preceq x\prec 1$, we have $\xi_{\boldsymbol{\alpha},\boldsymbol{\beta}}(1,x)\geq 1$'' are equivalent. Thus Assertion $(3)$ of Theorem  C holds.
\medskip

$\bullet$ $(3)\Rightarrow(1)$: We assume that $\sum_{i=1}^ue_i=\sum_{j=1}^vf_j$, that is $r=s$, and that, for all $x\in[1/M_{\mathbf{e},\mathbf{f}},1[$, we have $\boldsymbol{\Delta}_{\mathbf{e},\mathbf{f}}(x)\geq 1$. Since $\boldsymbol{\alpha}$ and $\boldsymbol{\beta}$ are $R$-partitioned, for all $a\in\{1,\dots,d_{\boldsymbol{\alpha},\boldsymbol{\beta}}\}$ coprime to $d_{\boldsymbol{\alpha},\boldsymbol{\beta}}$ we have $\langle a\boldsymbol{\alpha}\rangle=\boldsymbol{\alpha}$ and $\langle a\boldsymbol{\beta}\rangle=\boldsymbol{\beta}$, and these tuples are disjoint. We deduce that for all $a\in\{1,\dots,d_{\boldsymbol{\alpha},\boldsymbol{\beta}}\}$ coprime to $d_{\boldsymbol{\alpha},\boldsymbol{\beta}}$ and all $x\in\mathbb{R}$, $m_{\boldsymbol{\alpha},\boldsymbol{\beta}}(a)\preceq x\prec a$, Equation~\eqref{BoldDelta} gives us that $\xi_{\boldsymbol{\alpha},\boldsymbol{\beta}}(a,x)\geq 1$, so that $H_{\boldsymbol{\alpha},\boldsymbol{\beta}}$ holds. Thus Assertion $(iii)$ of Theorem \ref{Criterion} holds and $q_{\boldsymbol{\alpha},\boldsymbol{\beta}}(z)$ is $N$-integral as expected. This finishes the proof that Theorem \ref{Criterion} implies Theorem C.
\medskip

Furthermore, when \eqref{HyperFact} holds, Delaygue \cite[Theorem $8$]{Delaygue0} generalized some of the results of Krattenthaler-Rivoal \cite{TanguyPos} and proved that all Taylor coefficients at the origin of $q_{\boldsymbol{\alpha},\boldsymbol{\beta}}(C_0z)$ are positive, but its constant term which is $0$. Assertion $(1)$ of Theorem  \ref{Criterion} generalizes this result since it does not use the assumption that $\boldsymbol{\alpha}$ and $\boldsymbol{\beta}$ are $R$-partitioned.
\medskip

Later, Roques studied (see \cite{Roques1} and \cite{Roques2}) the integrality of the Taylor coefficients of canonical coordinates $q_{\boldsymbol{\alpha},\boldsymbol{\beta}}$ without assuming that \eqref{HyperFact} holds, in the case $\boldsymbol{\alpha}$ and $\boldsymbol{\beta}$ have the same number of elements $r\geq 2$, all the elements of $\boldsymbol{\beta}$ are equal to $1$ and all the elements of $\boldsymbol{\alpha}$ lie in $(0,1]\cap\mathbb{Q}$. In this case, we have $r=s$ and it is easy to prove that $H_{\boldsymbol{\alpha},\boldsymbol{\beta}}$ holds but $\boldsymbol{\alpha}$ is not necessarily $R$-partitioned. Roques proved that $q_{\boldsymbol{\alpha},\boldsymbol{\beta}}(z)$ is $N$-integral if and only if, for all $a\in\{1,\dots,d_{\boldsymbol{\alpha},\boldsymbol{\beta}}\}$ coprime to $d_{\boldsymbol{\alpha},\boldsymbol{\beta}}$, we have $q_{\langle a\boldsymbol{\alpha}\rangle,\langle a\boldsymbol{\beta}\rangle}(z)=q_{\boldsymbol{\alpha},\boldsymbol{\beta}}(z)$ in accordance with Theorem \ref{Criterion}. Furthermore, when $r=2$, he found the exact finite set (\footnote{This sets contains $28$ elements amongst which $4$ are $R$-partitioned.}) of tuples $\boldsymbol{\alpha}$ such that $q_{\boldsymbol{\alpha},\boldsymbol{\beta}}(z)$ is $N$-integral (see \cite[Theorem~$3$]{Roques1}) and, when $r\geq 3$, he proved (see \cite{Roques2}) that $q_{\boldsymbol{\alpha},\boldsymbol{\beta}}(z)$ is $N$-integral if and only if $\boldsymbol{\alpha}$ is $R$-partitioned (the ``if part''  is proved by Krattenthaler-Rivoal in \cite{Tanguy1}). Note that if $\boldsymbol{\beta}=(1,\dots,1)$, then it is easy to prove that $F_{\boldsymbol{\alpha},\boldsymbol{\beta}}(z)$ is $N$-integral.

In the case $r=2$, Roques gave constants $C_1$ such that $q_{\boldsymbol{\alpha},\boldsymbol{\beta}}(C_1z)\in\mathbb{Z}[[z]]$ which equal $C_{\boldsymbol{\alpha},\boldsymbol{\beta}}'$ unless for $\boldsymbol{\alpha}=(1/2,1/4)$ or $(1/2,3/4)$ where $C_1=256$ and Theorem \ref{Criterion} improves this constant since $C_{\boldsymbol{\alpha},\boldsymbol{\beta}}'=32$.
\medskip

The integrality of Taylor coefficients of roots of $z^{-1}q_{\boldsymbol{\alpha},\boldsymbol{\beta}}(z)$ has been studied in case \eqref{HyperFact} holds by Lian-Yau \cite{LianYauRoots}, Krattenthaler-Rivoal \cite{Tanguy2}, and by Delaygue \cite{Delaygue2}. For a detailed survey of these results, we refer the reader to \cite[Section~$1.2$]{Delaygue2}. 

$\bullet$ In \cite{LianYauRoots}, Lian-Yau studied the case $\mathbf{e}=(p)$ and $\mathbf{f}=(1,\dots,1)$ with $p$ $1$'s in $\mathbf{f}$ and where $p$ is a prime. In this case, we have $\boldsymbol{\beta}=(1,\dots,1)$ and $\mathfrak{n}_{\boldsymbol{\alpha},\boldsymbol{\beta}}'=1$, thus we do not obtain a root with Corollary \ref{cor Root}.

$\bullet$ In \cite{Tanguy2}, Krattenthaler-Rivoal studied the case $\mathbf{e}=(N,\dots,N)$ with $k$ $N$'s in $\mathbf{e}$ and $\mathbf{f}=(1,\dots,1)$ with $kN$ $1$'s in $\mathbf{f}$. In this case, we also have $\boldsymbol{\beta}=(1,\dots,1)$. For all prime divisors $p$ of $N$, we write $N=p^{\eta_p}N_p$ with $\eta_p,N_p\in\mathbb{N}$ and $N_p$ not divisible by $p$. A simple computation of the associated tuples $\boldsymbol{\alpha}$ and $\boldsymbol{\beta}$ shows that $d_{\boldsymbol{\alpha},\boldsymbol{\beta}}=N$ and $\lambda_p=k(N_p-N)$. Thus, for all prime divisors $p$ of $N$, $p-1$ divides $\lambda_p$ and we have 
$$
\mathfrak{n}_{\boldsymbol{\alpha},\boldsymbol{\beta}}'=\prod_{p\mid N}p^{-1+k\frac{N-N_p}{p-1}},
$$ 
but it seems that the roots found by Krattenthaler-Rivoal are always better than $\mathfrak{n}_{\boldsymbol{\alpha},\boldsymbol{\beta}}'$ in these cases.

$\bullet$ However, in a lot of cases, our root $\mathfrak{n}_{\boldsymbol{\alpha},\boldsymbol{\beta}}'$ improves the one found by Delaygue in \cite{Delaygue2}. For example, if $\mathbf{e}=(4,2)$ and $\mathbf{f}=(1,1,1,1,1,1)$, then \cite[Corollary $1.1$]{Delaygue2} gives us the root $4$ while $\boldsymbol{\beta}=(1,\dots,1)$ and $\mathfrak{n}_{\boldsymbol{\alpha},\boldsymbol{\beta}}'=32$.

\subsection{Open questions}

We formulate some open questions about $N$-integrality of mirror maps and $q$-coordinates.

$\bullet$ Does the equivalence in Assertion $(2)$ of Theorem  \ref{Criterion} still hold if we do not assume that $F_{\boldsymbol{\alpha},\boldsymbol{\beta}}(z)$ is $N$-integral?

$\bullet$ One of the conditions for $q_{\boldsymbol{\alpha},\boldsymbol{\beta}}(z)$ to be $N$-integral is that, for all $a\in\{1,\dots,d_{\boldsymbol{\alpha},\boldsymbol{\beta}}\}$ coprime to $d_{\boldsymbol{\alpha},\boldsymbol{\beta}}$, we have $q_{\boldsymbol{\alpha},\boldsymbol{\beta}}(z)=q_{\langle a\boldsymbol{\alpha}\rangle,\langle a\boldsymbol{\beta}\rangle}(z)$. According to \cite{Roques1} and \cite{Roques2}, we know that, when $\boldsymbol{\beta}=(1,\dots,1)$ and all elements of $\boldsymbol{\alpha}$ belong to $(0,1]$, this condition implies a stronger characterization related to the exact forms of $\boldsymbol{\alpha}$ and $\boldsymbol{\beta}$. Is it possible to deduce a similar characterization in the general case?

\subsection{A corrected version of a lemma of Lang}\label{Cor Lang}

As mentioned in the Introduction, while working on this article, we noticed an error in a lemma stated by Lang \cite[Lemma~$1.1$, Section $1$, Chapter $14$]{Lang} about arithmetic properties of Mojita's $p$-adic Gamma function. This lemma has been used in several articles on the integrality of the Taylor coefficients of mirror maps including papers of the authors. First we give a corrected version of Lang's lemma, then we explain why this error does not change the validity of our previous results.

Let $p$ be a fixed prime. For all $n\in\mathbb{N}$, we define the $p$-adic Gamma function $\Gamma_p$ by
$$
\Gamma_p(n):=(-1)^n\underset{\mathrm{gcd}(k,p)=1}{\prod_{k=1}^{n-1}}k.
$$
In particular, $\Gamma_p(0)=1$, $\Gamma_p(1)=-1$ and  $\Gamma_p$ can be extended to $\mathbb Z_p$.

\begin{propo}\label{propo Lang +}
For all $k,m,s\in\mathbb{N}$, we have 
$$
\Gamma_p(k+mp^s)\equiv \begin{cases}
\textup{$\Gamma_p(k)\mod p^s$ if $p^s\neq 4$;}\\
\textup{$(-1)^m\Gamma_p(k)\mod p^s$ if $p^s=4$.}
\end{cases}.
$$
\end{propo}

The case $p^s\neq 4$ in Proposition~\ref{propo Lang +} is proved by Morita in  \cite{Morita}. We provide a complete proof of the proposition.

\begin{proof}
If $s=0$ or if  $m=0$ this is trivial. We assume in the sequel that $s\geq 1$ and $m\geq 1$. Then 
\begin{align}
\frac{\Gamma_p(k+mp^s)}{\Gamma_p(k)}&=(-1)^{mp^s}\underset{\mathrm{gcd}(i,p)=1}{\prod_{i=k}^{k+mp^s-1}}i=(-1)^{mp^s}\underset{\mathrm{gcd}(k+i,p)=1}{\prod_{i=0}^{p^s-1}}\prod_{j=0}^{m-1}(k+i+jp^s)\notag\\
&\equiv(-1)^{mp^s}\underset{\mathrm{gcd}(k+i,p)=1}{\prod_{i=0}^{p^s-1}}(k+i)^m\mod p^s\notag\\
&\equiv(-1)^{mp^s}\underset{\mathrm{gcd}(j,p)=1}{\prod_{j=0}^{p^s-1}}j^m\mod p^s\label{laref},
\end{align}
because, for all $j\in\{0,\dots,p^s-1\}$, there exists a unique $i\in\{0,\dots,p^s-1\}$ such that $k+i\equiv j\mod p^s$.
\medskip

We first assume that $p\geq 3$. In this case, the group $(\mathbb{Z}/p^s\mathbb{Z})^\times$ is cyclic and contains just one element of order $2$. Collecting each element of $(\mathbb{Z}/p^s\mathbb{Z})^\times$ of order $\ge 3$ with its inverse, we obtain 
$$
\underset{\mathrm{gcd}(j,p)=1}{\prod_{j=0}^{p^s-1}}j\equiv -1\mod p^s.
$$
Together, with \eqref{laref}, we get
$$
\frac{\Gamma_p(k+mp^s)}{\Gamma_p(k)}\equiv 1\mod p^s,
$$
because $p$ is odd.
\medskip

Let us now assume that $p=2$. If $s=1$, then
$$
\underset{\mathrm{gcd}(j,p)=1}{\prod_{j=0}^{p^s-1}}j=1
$$
and by \eqref{laref} this yields $\Gamma_p(k+mp^s)\equiv\Gamma_p(k)\mod p^s$. If  $s=2$, then
$$
\underset{\mathrm{gcd}(j,p)=1}{\prod_{j=0}^{p^s-1}}j=3\equiv -1\mod p^s,
$$
and by \eqref{laref}, this yields $\Gamma_p(k+mp^s)\equiv (-1)^m\Gamma_p(k)\mod p^s$. It remains to deal with the case $s\geq 3$. The group $(\mathbb{Z}/2^s\mathbb{Z})^\times$ is isomorphic to  $\mathbb{Z}/2^{s-2}\mathbb{Z}\times\mathbb{Z}/2\mathbb{Z}$. Moreover, 
$$
\sum_{k=0}^{2^{s-2}-1}\sum_{j=0}^1(k,j)=\left(2\sum_{k=0}^{2^{s-2}-1}k,2^{s-2}\right)=\left(2^{s-2}(2^{s-2}-1),2^{s-2}\right)\in 2^{s-2}\mathbb{Z}\times 2\mathbb{Z},
$$
because $s\geq 3$. Hence, 
$$
\underset{\mathrm{gcd}(j,p)=1}{\prod_{j=0}^{p^s-1}}j\equiv 1 \mod p^s
$$
and by \eqref{laref}, this yields
$$
\frac{\Gamma_p(k+mp^s)}{\Gamma_p(k)}\equiv 1\mod p^s,
$$
which completes the proof of the proposition.
\end{proof}

The error in Lang's version is that he wrote  $\Gamma_p(k+mp^s)\equiv\Gamma_p(k)\mod p^s$ if $p^s=4$, forgetting the factor $(-1)^m$. He gives the proof only for $p\ge 3$ and he claims that the proof goes through similarly when $p=2$, overlooking the subtility. Delaygue and Krattenthaler-Rivoal used Lang's version in  \cite[Lemma $11$]{Delaygue1}, \cite[Lemma $8$]{Delaygue3} and \cite{Tanguy1}. Fortunately, the resulting mistakes in these papers are purely local and can be fixed. Indeed, the factor $(-1)^\ell$ (that should have been added when $p^s=4$) would have occurred for an {\em even} value of $\ell$ and thus would have immediately disappeared without changing the rest of the proof.

\section{The $p$-adic valuation of Pochhammer symbols}\label{section valuation} 

We introduce certain step functions, defined  
over  $\mathbb{R}$, that enable us to compute the $p$-adic valuation of Pochhammer symbols. We will then provide a connection between the values of these functions and the functions $\xi_{\boldsymbol{\alpha},\boldsymbol{\beta}}(a,\cdot)$. This construction is  inspired by various works of Christol \cite{Christol}, Dwork \cite{Dwork} and Katz \cite{Katz}. We first  prove Proposition \ref{Lemma Mp}.

\subsection{Proof of Proposition \ref{Lemma Mp}}\label{section demo propo}

Let  $\boldsymbol{\alpha}$ and $\boldsymbol{\beta}$ be two sequences taking their values in $\mathbb{Q}\setminus\mathbb{Z}_{\leq 0}$. If there exists  $C\in\mathbb{Q}^{\ast}$ such that $F_{\boldsymbol{\alpha},\boldsymbol{\beta}}(C z)\in\mathbb{Z}[[z]]$, then for all primes $p$ such that $v_p(C)\leq 0$, we have $F_{\boldsymbol{\alpha},\boldsymbol{\beta}}(z)\in\mathbb{Z}_p[[z]]$. Hence, there exists only a finite number of primes $p$ such that $F_{\boldsymbol{\alpha},  \boldsymbol{\beta}}(z)\notin\mathbb{Z}_p[[z]]$.
\medskip

Conversely, let us assume there exists only a finite number of primes $p$ such that $F_{\boldsymbol{\alpha},\boldsymbol{\beta}}(z)\notin\mathbb{Z}_p[[z]]$. To prove  Proposition \ref{Lemma Mp}, it is enough to prove that, for all prime $p$, there exists $m\in\mathbb{Z}_{\leq 0}$ such that for all $n\in\mathbb{N}$, we have 
\begin{equation}\label{inter suff}
v_p\left(\frac{(\alpha_1)_n\cdots(\alpha_r)_n}{(\beta_1)_n\cdots(\beta_s)_n}\right)\geq mn.
\end{equation}
Let $x\in\mathbb{Q}$, $x=a/b$ with $a,b\in\mathbb{Z}$, $b\geq 1$, and $a$ and $b$ coprime. If $b$ is not divisible by $p$, then for all $n\in\mathbb{N}$, we have  $v_p\big((x)_n\big)\geq 0$. On the other hand, if $p$ divides $b$, then $v_p\big((x)_n\big)=v_p(x)n$. 

Let us now assume that $x\notin\mathbb{Z}_{\leq 0}$. Then, for all $n\in\mathbb{N}$, $n\geq 1$, 
\begin{align*}
v_p\left(\frac{1}{(x)_n}\right)=v_p\left(\frac{b^n}{a(a+b)\cdots\big(a+b(n-1)\big)}\right)&\geq v_p\left(\frac{b^n}{|a|!(|a|+bn)!}\right)\\
&\geq\left(v_p(b)-\frac{b}{p-1}\right)n-2\frac{|a|}{p-1},
\end{align*}
because
$$
v_p\big((|a|+bn)!\big)=\sum_{\ell=1}^{\infty}\left\lfloor\frac{|a|+bn}{p^{\ell}}\right\rfloor<\sum_{\ell=1}^{\infty}\frac{|a|+bn}{p^{\ell}}=\frac{|a|}{p-1}+\frac{b}{p-1}n.
$$
Hence,  \eqref{inter suff} holds and Proposition \ref{Lemma Mp} is proved.
\hfill$\square$

\subsection{Dwork's map $\mathfrak{D}_p$}\label{section Dwork map}

Given a prime  $p$ and some $\alpha \in \mathbb Z_p\cap \mathbb{Q}$, we  recall that $\mathfrak{D}_p(\alpha)$ denotes the unique element in $\mathbb Z_p\cap \mathbb{Q}$ such that
$$
p\mathfrak{D}_p(\alpha)-\alpha\in\{0,\dots,p-1\}.
$$
The map $\alpha\mapsto\mathfrak{D}_p(\alpha)$ was used by Dwork in \cite{Dwork} (denoted there as $\alpha\mapsto\alpha'$). We observe that the unique element $k\in\{0,\dots,p-1\}$ such that $k+\alpha\in p\mathbb{Z}_p$ is  $k=p\mathfrak{D}_p(\alpha)-\alpha$. More precisely, the $p$-adic expansion of $-\alpha$ in $\mathbb{Z}_p$ is
$$
-\alpha=\sum_{\ell=0}^{\infty}\big(p\mathfrak{D}_p^{\ell+1}(\alpha)-\mathfrak{D}_p^{\ell}(\alpha)\big)p^{\ell},
$$
where $\mathfrak{D}_p^{\ell}$ is the $\ell$-th iteration of  $\mathfrak{D}_p$. In  particular, for all $\ell\in\mathbb{N}$, $\ell\geq 1$, $\mathfrak{D}_p^{\ell}(\alpha)$ is the unique element in $\mathbb{Z}_p\cap \mathbb Q$ such that $p^{\ell}\mathfrak{D}_p^{\ell}(\alpha)-\alpha\in\{0,\dots,p^{\ell}-1\}$.
\medskip

For all primes  $p$, we have $\mathfrak{D}_p(1)=1$. Let us now assume that  $\alpha$ is in $\mathbb Z_p \cap \mathbb{Q}\cap(0,1)$. Set  $N\in\mathbb{N}$, $N\geq 2$ and $r\in\{1,\dots,N-1\}$, $\gcd(r,N)=1$, such that $\alpha=r/N$. Let $s_N$  be the unique right inverse of the canonical morphism $\pi_N:\mathbb{Z}\rightarrow\mathbb{Z}/N\mathbb{Z}$ with values in $\{0,\dots,N-1\}$. Then  (see \cite{Roques1} for details)
$$
\mathfrak{D}_p(\alpha)=\frac{s_N\big(\pi_N(p)^{-1}\pi_N(r)\big)}{N}.
$$
Hence, for all $\ell\in\mathbb{N}$, $\ell\geq 1$, we obtain 
\begin{equation}\label{Prime 0 1}
\mathfrak{D}_p^{\ell}(\alpha)=\frac{s_N\big(\pi_N(p)^{-\ell}\pi_N(r)\big)}{N}.
\end{equation}
In particular, if $\alpha\in(0,1)$, then $\mathfrak{D}_p(\alpha)$ depends only on the congruence class  of $p$ modulo $N$. If $a\in\mathbb{Z}$ satisfies $ap\equiv 1\mod N$, then $\mathfrak{D}_p^\ell(\alpha)=\{a^\ell\alpha\}=\langle a^\ell\alpha\rangle$ because $a$ is coprime to $N$, hence $a^\ell\alpha\notin\mathbb{Z}$. This formula  is still valid when $\alpha=1$ and $a$ is any integer.

\begin{Lemma}\label{Lemma infty}
Let $\alpha\in\mathbb{Q}\setminus\mathbb{Z}_{\geq 0}$. Then for any prime $p$ such that $\alpha\in\mathbb{Z}_p$ and all $\ell\in\mathbb{N}$, $\ell\geq 1$, such that $p^{\ell}\geq d(\alpha)\big(|\lfloor 1-\alpha\rfloor|+\underline{\alpha}\big)$, we have $\mathfrak{D}_p^{\ell}(\alpha)=\mathfrak{D}_p^{\ell}(\underline{\alpha})=\langle\omega\alpha\rangle$, where $\omega\in\mathbb{Z}$ satisfies $\omega p^\ell\equiv 1\mod d(\alpha)$.
\end{Lemma}

\begin{proof} Let $\alpha\in\mathbb{Q}\setminus\mathbb{Z}_{\leq 0}$ and  $p$ be such that $\alpha\in\mathbb{Z}_p$ and $\ell\in\mathbb{N}$, $\ell\geq 1$ be such that $p^{\ell}\geq d(\alpha)(|\lfloor 1-\alpha\rfloor|+\underline{\alpha})$. By definition, $\mathfrak{D}_p^{\ell}(\alpha)$ is the unique rational number in  $\mathbb{Z}_p$ such that $p^{\ell}\mathfrak{D}_p^{\ell}(\alpha)-\alpha\in\{0,\dots,p^{\ell}-1\}$. We set $\alpha=\underline{\alpha}+k$, $k\in\mathbb{Z}$ and $r:=\mathfrak{D}_p^{\ell}(\underline{\alpha})+\lfloor k/p^{\ell}\rfloor+a$, with $a=0$ if  $k-p^{\ell}\lfloor k/p^{\ell}\rfloor\leq p^{\ell}\mathfrak{D}_p^{\ell}(\underline{\alpha})-\underline{\alpha}$ and $a=1$ otherwise. We obtain
$$
p^{\ell}r-\alpha=p^{\ell}\mathfrak{D}_p^{\ell}(\underline{\alpha})-\underline{\alpha}+p^{\ell}\left\lfloor\frac{k}{p^{\ell}}\right\rfloor-k+p^{\ell}a\in\{0,\dots,p^{\ell}-1\},
$$
because $p^{\ell}\mathfrak{D}_p^\ell(\underline{\alpha})-\underline{\alpha}$ and $k-p^{\ell}\lfloor k/p^{\ell}\rfloor$ are in $\{0,\dots,p^{\ell}-1\}$. Since $r\in\mathbb{Z}_p$, we get $\mathfrak{D}_p^{\ell}(\alpha)=r$. We have $d(\alpha)(|k|+\underline{\alpha})>|k|$ thus $\lfloor k/p^{\ell}\rfloor\in\{-1,0\}$.
\medskip

If $\lfloor k/p^{\ell}\rfloor=0$, then since  $\mathfrak{D}_p^{\ell}(\underline{\alpha})\geq 1/d(\alpha)$, we get $p^{\ell}\mathfrak{D}_p^{\ell}(\underline{\alpha})-\underline{\alpha}\geq |k|$ and thus $a=0$. In this case, we have $\mathfrak{D}_p^{\ell}(\alpha)=\mathfrak{D}_p^{\ell}(\underline{\alpha})$.
\medskip

Let us now assume that $\lfloor k/p^{\ell}\rfloor=-1$, \textit{i.e.} $k\leq -1$. We have $\underline{\alpha}<1$ because $\alpha\notin\mathbb{Z}_{\leq 0}$, hence $d(\alpha)\geq 2$. We have
\begin{align*}
p^{\ell}\mathfrak{D}_p^{\ell}(\underline{\alpha})-\underline{\alpha}-(k+p^{\ell})&\leq p^{\ell}\left(\frac{d(\alpha)-1}{d(\alpha)}-1\right)-\underline{\alpha}-k\leq -\frac{p^{\ell}}{d(\alpha)}-\underline{\alpha}-k\\
&\leq -|k|-2\underline{\alpha}-k\leq -2\underline{\alpha}<0,
\end{align*}
thus $a=1$ and $\mathfrak{D}_p^{\ell}(\alpha)=\mathfrak{D}_p^{\ell}(\underline{\alpha})$. 
\end{proof}

\subsection{Analogues of Landau functions}\label{section analogues}

We now define the step functions that will enable us to compute the $p$-adic valuation of the Taylor coefficients at $z=0$ of $F_{\boldsymbol{\alpha},\boldsymbol{\beta}}(z)$. For all primes $p$, all $\alpha\in\mathbb{Q}\cap\mathbb{Z}_p$ and all $\ell\in\mathbb{N}$, $\ell\geq 1$, we denote by $\delta_{p,\ell}(\alpha,\cdot)$ the step function defined, for all $x\in\mathbb{R}$, by
$$
\left(\delta_{p,\ell}(\alpha,x)=k\Longleftrightarrow x-\mathfrak{D}_p^{\ell}(\alpha)-\frac{\lfloor 1-\alpha\rfloor}{p^{\ell}}\in[k-1,k)\right),\,k\in\mathbb{Z}.
$$
In particular, if $\alpha\in(0,1]$, then for all $k\in\mathbb{Z}$, we have 
$$
\delta_{p,\ell}(\alpha,x)=k\Longleftrightarrow x-\mathfrak{D}_p^{\ell}(\alpha)\in[k-1,k).
$$
\medskip

Let $\boldsymbol{\alpha}:=(\alpha_1,\dots,\alpha_r)$ and $\boldsymbol{\beta}:=(\beta_1,\dots,\beta_s)$ be two sequences taking their values in $\mathbb{Q}\setminus\mathbb{Z}_{\leq 0}$. For any $p$ that does not divide  $d_{\boldsymbol{\alpha},\boldsymbol{\beta}}$, and all $\ell\in\mathbb{N}$, $\ell\geq 1$, we denote by $\Delta_{\boldsymbol{\alpha},\boldsymbol{\beta}}^{p,\ell}$ the step function defined, for all $x\in\mathbb{R}$, by 
$$
\Delta_{\boldsymbol{\alpha},\boldsymbol{\beta}}^{p,\ell}(x):=\sum_{i=1}^r\delta_{p,\ell}(\alpha_i,x)-\sum_{j=1}^s\delta_{p,\ell}(\beta_j,x).
$$
The motivation behind the functions $\Delta_{\boldsymbol{\alpha},\boldsymbol{\beta}}^{p,\ell}$ is given by the following result. 

\begin{propo}\label{valuation}
Let $\boldsymbol{\alpha}:=(\alpha_1,\dots,\alpha_r)$ and $\boldsymbol{\beta}:=(\beta_1,\dots,\beta_s)$ be two sequences taking their values in $\mathbb{Q}\setminus\mathbb{Z}_{\leq 0}$. Let $p$ be such that  $\boldsymbol{\alpha}$ and $\boldsymbol{\beta}$ are in $\mathbb{Z}_p$. Then, for all $n\in\mathbb{N}$, we have
$$
v_p\left(\frac{(\alpha_1)_n\cdots(\alpha_r)_n}{(\beta_1)_n\cdots(\beta_s)_n}\right)=\sum_{\ell=1}^{\infty}
\Delta_{\boldsymbol{\alpha},\boldsymbol{\beta}}^{p,\ell}\left(\frac{n}{p^{\ell}}\right)=\sum_{\ell=1}^{\infty}\Delta_{\boldsymbol{\alpha},\boldsymbol{\beta}}^{p,\ell}\left(\left\{\frac{n}{p^{\ell}}\right\}\right)+(r-s)v_p(n!).
$$
\end{propo}

\begin{Remark}
This proposition is a reformulation of results in Section III of \cite{Christol}, proved  by Christol in order to compute the $p$-adic valuation of the Pochhammer symbol $(x)_n$ for  $x\in\mathbb{Z}_p$.
\end{Remark}

\begin{proof}
For any $p$, any $n:=\sum_{k=0}^{\infty}n_kp^k\in\mathbb{Z}_p$ with $n_k\in\{0,\dots,p-1\}$, and any $\ell\in\mathbb{N}$, $\ell\geq 1$, we set $T_p(n,\ell):=\sum_{k=0}^{\ell-1}n_kp^k$. For all $\ell\in\mathbb{N}$, $\ell\geq 1$, we have 
$$
T_p(-\alpha,\ell)=p^{\ell}\mathfrak{D}_p^{\ell}(\alpha)-\alpha.
$$
We fix a $p$-adic integer $\alpha\in\mathbb{Q}\setminus\mathbb{Z}_{\leq 0}$. For all $k\in\mathbb{Z}$ and all $\ell\in\mathbb{N}$, $\ell\geq 1$, we have
\begin{align}
\delta_{p,\ell}\left(\alpha,\frac{n}{p^{\ell}}\right)=k
&\Longleftrightarrow\mathfrak{D}_p^{\ell}(\alpha)+\frac{\lfloor 1-\alpha\rfloor}{p^{\ell}}+k-1\leq\frac{n}{p^{\ell}}<\mathfrak{D}_p^{\ell}(\alpha)+\frac{\lfloor 1-\alpha\rfloor}{p^{\ell}}+k\notag\\
&\Longleftrightarrow p^{\ell}\mathfrak{D}_p^{\ell}(\alpha)+\lfloor 1-\alpha\rfloor+(k-1)p^{\ell}\leq n<p^{\ell}\mathfrak{D}_p^{\ell}(\alpha)+\lfloor 1-\alpha\rfloor+kp^{\ell}\notag\\
&\Longleftrightarrow p^{\ell}\mathfrak{D}_p^{\ell}(\alpha)-\alpha+(k-1)p^{\ell}<n\leq p^{\ell}\mathfrak{D}_p^{\ell}(\alpha)-\alpha+kp^{\ell}\label{equ expli}\\
&\Longleftrightarrow \left\lceil\frac{n-T_p(-\alpha,\ell)}{p^{\ell}}\right\rceil=k,\notag
\end{align}
where, for all $x\in\mathbb{R}$,  $\lceil x\rceil$ is the smallest integer larger than $x$. We have used in \eqref{equ expli} the fact that $-\alpha=-\underline{\alpha}+\lfloor 1-\alpha\rfloor$, $-1\leq-\underline{\alpha}<0$ and $p^{\ell}\mathfrak{D}_p^{\ell}(\alpha)-\alpha\in\mathbb{N}$. We then obtain 
\begin{equation}\label{formu sup}
\delta_{p,\ell}\left(\alpha,\frac{n}{p^{\ell}}\right)=\left\lceil\frac{n-T_p(-\alpha,\ell)}{p^{\ell}}\right\rceil.
\end{equation}

\medskip
 
Christol proved in \cite{Christol} that for all  $\alpha\in\mathbb{Z}_p\setminus\mathbb{Z}_{\leq 0}$ and all $n\in\mathbb{N}$, we have  
\begin{equation}\label{Christol val}
v_p\big((\alpha)_n\big)=\sum_{\ell=1}^{\infty}\left\lfloor\frac{n+p^{\ell}-1-T_p(-\alpha,\ell)}{p^{\ell}}\right\rfloor.
\end{equation}
For all $\ell\in\mathbb{N}$, $\ell\geq 1$, we have 
$$
\frac{n+p^{\ell}-1-T_p(-\alpha,\ell)}{p^{\ell}}\in\frac{1}{p^{\ell}}\mathbb{Z},
$$
so that if $k\in\mathbb{Z}$ is such that
$$
k\leq\frac{n+p^{\ell}-1-T_p(-\alpha,\ell)}{p^{\ell}}<k+1,
$$
then 
$$
k-1<\frac{n-T_p(-\alpha,\ell)}{p^{\ell}}\leq k.
$$
Hence, we get 
$$
\left\lfloor\frac{n+p^{\ell}-1-T_p(-\alpha,\ell)}{p^{\ell}}\right\rfloor=\left\lceil\frac{n-T_p(-\alpha,\ell)}{p^{\ell}}\right\rceil.
$$
By \eqref{formu sup} and \eqref{Christol val}, it follows that
$$
v_p\big((\alpha)_n\big)=\sum_{\ell=1}^{\infty}\delta_{p,\ell}\left(\alpha,\frac{n}{p^{\ell}}\right)=\sum_{\ell=1}^{\infty}\delta_{p,\ell}\left(\alpha,\left\{\frac{n}{p^{\ell}}\right\}\right)+v_p(n!),
$$
because $\delta_{p,\ell}(\alpha,n/p^{\ell})=\delta_{p,\ell}(\alpha,\{n/p^{\ell}\})+\lfloor n/p^{\ell}\rfloor$ and  $v_p(n!)=\sum_{\ell=1}^{\infty}\lfloor n/p^{\ell}\rfloor$.
\end{proof}

The following lemma provides an upper for the abscissae of the jumps of the functions  $\Delta_{\boldsymbol{\alpha},\boldsymbol{\beta}}^{p,\ell}$.

\begin{Lemma}\label{Lemma structure}
Let $\alpha\in\mathbb{Q}\setminus\mathbb{Z}_{\leq 0}$. There exists a constant $M(\alpha)>0$ such that, for all $p$ such that $\alpha\in\mathbb{Z}_p$, and all $\ell\in\mathbb{N}$, $\ell\geq 1$, we have
$$
\frac{1}{M(\alpha)}\leq\mathfrak{D}_p^{\ell}(\alpha)+\frac{\lfloor 1-\alpha\rfloor}{p^{\ell}}\leq 1.
$$
\end{Lemma}

\begin{Remark}
In particular, if $\boldsymbol{\alpha}$ and $\boldsymbol{\beta}$ are two sequences taking their values in $\mathbb{Q}\setminus\mathbb{Z}_{\leq 0}$, there exists a  constant $M(\boldsymbol{\alpha},\boldsymbol{\beta})>0$ such that for  all $p$ that does not divide $d_{\boldsymbol{\alpha},\boldsymbol{\beta}}$, all $\ell\in\mathbb{N}$, $\ell\geq 1$, and all $x\in[0,1/M(\boldsymbol{\alpha},\boldsymbol{\beta}))$, we have  $\Delta_{\boldsymbol{\alpha},\boldsymbol{\beta}}^{p,\ell}(x)=0$.
\end{Remark}

\begin{proof}
Set $a:=p^{\ell}\mathfrak{D}_p^{\ell}(\alpha)-\alpha\in\{0,\dots,p^{\ell}-1\}$. We have 
$$
\mathfrak{D}_p^{\ell}(\alpha)+\frac{\lfloor 1-\alpha\rfloor}{p^{\ell}}=\frac{a}{p^{\ell}}+\frac{\underline{\alpha}}{p^{\ell}}\in(0,1],
$$
because  $0<\underline{\alpha}\leq 1$. By Lemma \ref{Lemma infty}, if $p^{\ell}\geq d(\alpha)\big(|\lfloor 1-\alpha\rfloor |+\underline{\alpha}\big)$, then  $\mathfrak{D}_p^{\ell}(\alpha)=\mathfrak{D}_p^{\ell}(\underline{\alpha})\geq 1/d(\underline{\alpha})$ and hence 
$$
\mathfrak{D}_p^{\ell}(\alpha)+\frac{\lfloor 1-\alpha\rfloor}{p^{\ell}}\geq\frac{1}{d(\alpha)}\left(\frac{\underline{\alpha}}{|\lfloor 1-\alpha\rfloor |+\underline{\alpha}}\right).
$$
This completes the proof of Lemma \ref{Lemma structure} because there exists only a finite number of couples $(p,\ell)$ such that $p^{\ell}<d(\alpha)\big(|\lfloor 1-\alpha\rfloor |+\underline{\alpha}\big)$.
\end{proof}

Finally, our next lemma enables us to connect the fonctions $\Delta_{\boldsymbol{\alpha},\boldsymbol{\beta}}^{p,\ell}$ to the values of the functions $\xi_{\boldsymbol{\alpha},\boldsymbol{\beta}}(a,\cdot)$. This is useful to decide if $F_{\boldsymbol{\alpha},\boldsymbol{\beta}}$ is $N$-integral.

\begin{Lemma}\label{valeurs Delta}
Let $\boldsymbol{\alpha}$ and $\boldsymbol{\beta}$ be two sequences taking their values in  $\mathbb{Q}\setminus\mathbb{Z}_{\leq 0}$. There exists a constant  $\mathcal{N}_{\boldsymbol{\alpha},\boldsymbol{\beta}}$ such that for all elements  $\alpha$ and $\beta$ of the sequence $\boldsymbol{\alpha}$ or $\boldsymbol{\beta}$, for all $p$ that does not divide $d_{\boldsymbol{\alpha},\boldsymbol{\beta}}$ and all  $\ell\in\mathbb{N}$, $\ell\geq 1$ such that $p^{\ell}\geq\mathcal{N}_{\boldsymbol{\alpha},\boldsymbol{\beta}}$, we have 
$$
a\alpha\preceq a\beta\Longleftrightarrow\mathfrak{D}_p^\ell(\alpha)+\frac{\lfloor 1-\alpha\rfloor}{p^\ell}\leq\mathfrak{D}_p^\ell(\beta)+\frac{\lfloor 1-\beta\rfloor}{p^\ell},
$$
where $a\in\{1,\dots,d_{\boldsymbol{\alpha},\boldsymbol{\beta}}\}$ satisfies $p^{\ell}a\equiv 1\mod d_{\boldsymbol{\alpha},\boldsymbol{\beta}}$. Moreover, if the sequence $\boldsymbol{\alpha}$ and $\boldsymbol{\beta}$ take their values in $(0,1]$, then we can take $\mathcal{N}_{\boldsymbol{\alpha},\boldsymbol{\beta}}=1$.
\end{Lemma}

\begin{proof}
Let $p$ be such that the sequences  $\boldsymbol{\alpha}$ and $\boldsymbol{\beta}$ take their values in $\mathbb{Z}_p$. By Lemma \ref{Lemma infty}, there exists a  constant $\mathcal{N}_1$ such that, for all  $\ell\in\mathbb{N}$, $\ell\geq 1$ such that $p^{\ell}\geq\mathcal{N}_1$, and all elements $\alpha$ of $\boldsymbol{\alpha}$ or  $\boldsymbol{\beta}$, we have $\mathfrak{D}_p^{\ell}(\alpha)=\mathfrak{D}_p^{\ell}(\underline{\alpha})$. Moreover, if $\boldsymbol{\alpha}$ and $\boldsymbol{\beta}$ take their values in $(0,1]$, we can take $\mathcal{N}_1=1$ because $\alpha=\underline{\alpha}$. We set 
$$
\mathcal{N}_2:=\max\big\{d_{\boldsymbol{\alpha},\boldsymbol{\beta}}|\lfloor 1-\alpha_i\rfloor-\lfloor 1-\beta_j\rfloor|\,:\,1\leq i\leq r,\, 1\leq j\leq s\big\}+1
$$ 
and $\mathcal{N}_{\boldsymbol{\alpha},\boldsymbol{\beta}}:=\max(\mathcal{N}_1,\mathcal{N}_2)$. In particular, if $\boldsymbol{\alpha}$ and $\boldsymbol{\beta}$ take their values in $(0,1]$, then $\mathcal{N}_{\boldsymbol{\alpha},\boldsymbol{\beta}}=1$. Let $\ell\in\mathbb{N}$, $\ell\geq 1$ be such that $p^{\ell}\geq\mathcal{N}_{\boldsymbol{\alpha},\boldsymbol{\beta}}$ and $a\in\{1,\dots,d_{\boldsymbol{\alpha},\boldsymbol{\beta}}\}$ coprime to $d_{\boldsymbol{\alpha},\boldsymbol{\beta}}$ such that $p^{\ell}a\equiv 1\mod d_{\boldsymbol{\alpha},\boldsymbol{\beta}}$.

Let $\alpha$ and $\beta$ be elements of $\boldsymbol{\alpha}$ or  $\boldsymbol{\beta}$.
We set $k_1:=\lfloor 1-\alpha\rfloor$ and $k_2:=\lfloor 1-\beta\rfloor$. By \eqref{Prime 0 1}, we have  $a\underline{\alpha}-\mathfrak{D}_p^{\ell}(\underline{\alpha})\in\mathbb{Z}$. Hence, 
$$
a\alpha=a\underline{\alpha}-ak_1=\mathfrak{D}_p^{\ell}(\underline{\alpha})+a\underline{\alpha}-\mathfrak{D}_p^{\ell}(\underline{\alpha})-ak_1,
$$
with $\mathfrak{D}_p^{\ell}(\underline{\alpha})\in(0,1]$ and $a\underline{\alpha}-\mathfrak{D}_p^{\ell}(\underline{\alpha})-ak_1\in\mathbb{Z}$. Moreover, if  $\mathfrak{D}_p^{\ell}(\underline{\alpha})=\mathfrak{D}_p^{\ell}(\underline{\beta})$, then still by \eqref{Prime 0 1}, we have  $\underline{\alpha}=\underline{\beta}$. By definition of the total order $\prec$, we obtain 
\begin{align}
a\alpha\preceq a\beta&\Longleftrightarrow\mathfrak{D}_p^{\ell}(\underline{\alpha})<\mathfrak{D}_p^{\ell}(\underline{\beta})\quad\textup{or}\quad\left(\mathfrak{D}_p^{\ell}(\underline{\alpha})=\mathfrak{D}_p^{\ell}(\underline{\beta})\quad\textup{and}\quad a\alpha\geq a\beta\right)\notag\\
&\Longleftrightarrow\mathfrak{D}_p^{\ell}(\underline{\alpha})<\mathfrak{D}_p^{\ell}(\underline{\beta})\quad\textup{or}\quad\left(\mathfrak{D}_p^{\ell}(\underline{\alpha})=\mathfrak{D}_p^{\ell}(\underline{\beta})\quad\textup{and}\quad k_2\geq k_1\right)\notag\\
&\Longleftrightarrow\mathfrak{D}_p^{\ell}(\underline{\alpha})-\mathfrak{D}_p^{\ell}(\underline{\beta})\leq\frac{ k_2- k_1}{p^{\ell}}\label{explain 1}\\
&\Longleftrightarrow\mathfrak{D}_p^{\ell}(\underline{\alpha})+\frac{ k_1}{p^{\ell}}\leq\mathfrak{D}_p^{\ell}(\underline{\beta})+\frac{ k_2}{p^{\ell}}\notag\\
&\Longleftrightarrow\mathfrak{D}_p^{\ell}(\alpha)+\frac{ k_1}{p^{\ell}}\leq\mathfrak{D}_p^{\ell}(\beta)+\frac{ k_2}{p^{\ell}}\label{Bilan},
\end{align}
where in \eqref{explain 1} we have used the fact that if $\mathfrak{D}_p^{\ell}(\underline{\alpha})\neq\mathfrak{D}_p^{\ell}(\underline{\beta})$, then 
$|\mathfrak{D}_p^{\ell}(\underline{\alpha})-\mathfrak{D}_p^{\ell}(\underline{\beta})|\geq 1/d_{\boldsymbol{\alpha},\boldsymbol{\beta}}$. The equivalence \eqref{Bilan} finishes the proof of Lemma \ref{valeurs Delta}.
\end{proof}

Proposition \ref{valuation} shows that the functions $\Delta_{\boldsymbol{\alpha},\boldsymbol{\beta}}^{p,\ell}$ allow to compute the $p$-adic valuation of $(\boldsymbol{\alpha})_n/(\boldsymbol{\beta})_n$ when $p$ does not  divide  $d_{\boldsymbol{\alpha},\boldsymbol{\beta}}$. If $\boldsymbol{\alpha}$ and  $\boldsymbol{\beta}$ have the same number of parameters and if these parameters are in $(0,1]$, the constant $C_{\boldsymbol{\alpha},\boldsymbol{\beta}}$ enables us to get a very convenient formula for the computation of the $p$-adic valuation of $C_{\boldsymbol{\alpha},\boldsymbol{\beta}}^n(\boldsymbol{\alpha})_n/(\boldsymbol{\beta})_n$ when $p$ divides  $d_{\boldsymbol{\alpha},\boldsymbol{\beta}}$. This formula, stated in the next proposition, is key to the proof of Theorem \ref{theo Const} and is also used many times in the proof of Theorem \ref{theo expand}.

\begin{propo}\label{propo magie} 
Let $\boldsymbol{\alpha}$ and $\boldsymbol{\beta}$ be two tuples of $r$ parameters in $\mathbb{Q}\cap(0,1]$ such that  $F_{\boldsymbol{\alpha},\boldsymbol{\beta}}$ is $N$-integral. Let $p$ be a prime divisor of  $d_{\boldsymbol{\alpha},\boldsymbol{\beta}}$. We set $d_{\boldsymbol{\alpha},\boldsymbol{\beta}}=p^fD$, $f\geq 1$, with  $D\in\mathbb{N}$, $D$ not divisible by $p$. For all $a\in\{1,\dots,p^f\}$ not divisible by $p$, and all  $\ell\in\mathbb{N}$, $\ell\geq 1$, we choose a prime   $p_{a,\ell}$ such that
\begin{equation}\label{ref cong}
p_{a,\ell}\equiv p^\ell\mod D\quad\textup{and}\quad p_{a,\ell}\equiv a\mod p^f.
\end{equation}
Then,  for all $n\in\mathbb{N}$, we have 
\begin{equation}\label{magie}
v_p\left(C_0^n\frac{(\alpha_1)_n\cdots(\alpha_r)_n}{(\beta_1)_n\cdots(\beta_s)_n}\right)=\frac{1}{\varphi\big(p^f\big)}\underset{\gcd(a,p)=1}{\sum_{a=1}^{p^f}}\sum_{\ell=1}^{\infty}\Delta_{\boldsymbol{\alpha},\boldsymbol{\beta}}^{p_{a,\ell},1}\left(\left\{\frac{n}{p^{\ell}}\right\}\right)+n\left\{\frac{\lambda_p(\boldsymbol{\alpha},\boldsymbol{\beta})}{p-1}\right\},
\end{equation}
where
$$
C_0=\frac{\prod_{i=1}^rd(\alpha_i)}{\prod_{j=1}^rd(\beta_j)}\prod_{p\mid d_{\boldsymbol{\alpha},\boldsymbol{\beta}}}p^{-\left\lfloor\frac{\lambda_p(\boldsymbol{\alpha},\boldsymbol{\beta})}{p-1}\right\rfloor}.
$$
\end{propo}

\begin{proof}
We denote by $\widetilde{\boldsymbol{\alpha}}$, respectively  $\widetilde{\boldsymbol{\beta}}$, the (possibly empty) sequence  of elements of   $\boldsymbol{\alpha}$, respectively of $\boldsymbol{\beta}$, whose denominator is not  divisible by $p$. We also set $\lambda_p:=\lambda_p(\boldsymbol{\alpha},\boldsymbol{\beta})$. For all $n\in\mathbb{N}$, we have 
\begin{equation}\label{Terme brut}
v_p\left(C_0^n\frac{(\alpha_1)_n\cdots(\alpha_r)_n}{(\beta_1)_n\cdots(\beta_s)_n}\right)=\sum_{\ell=1}^{\infty}\Delta_{\widetilde{\boldsymbol{\alpha}},\widetilde{\boldsymbol{\beta}}}^{p,\ell}\left(\left\{\frac{n}{p^{\ell}}\right\}\right)+\lambda_pv_p(n!)-n\left\lfloor\frac{\lambda_p}{p-1}\right\rfloor.
\end{equation}
\medskip

Let  $\alpha$ be  an element of $\boldsymbol{\alpha}$ or $\boldsymbol{\beta}$. Let $N$ be the denominator of $\alpha$. If $p$ does not divide $N$, then $N$ divides $D$ and, for all $a\in\{1,\dots,p^f\}$, $\gcd(a,p)=1$, and all $\ell\in\mathbb{N}$, $\ell\geq 1$, we have $p_{a,\ell}\equiv p^{\ell}\mod N$. Hence, $\mathfrak{D}_p^{\ell}(\alpha)=\mathfrak{D}_{p_{a,\ell}}(\alpha)$ because $\alpha\in(0,1]$. 

On the other hand, if $p$ divides $N$, then for all $n,\ell\in\mathbb{N}$, $\ell\geq 1$, we define $\omega_{\ell}(\alpha,n)$ as the number of elements $a\in\{1,\dots,p^f\}$, $\gcd(a,p)=1$, such that $\{n/p^{\ell}\}\geq\mathfrak{D}_{p_{a,\ell}}(\alpha)$.
Thus for all $n,\ell\in\mathbb{N}$, $\ell\geq 1$, we get
\begin{equation}\label{ref omega}
\underset{\gcd(a,p)=1}{\sum_{a=1}^{p^f}}\Delta_{\boldsymbol{\alpha},\boldsymbol{\beta}}^{p_{a,\ell},1}\left(\left\{\frac{n}{p^{\ell}}\right\}\right)=\varphi\big(p^f\big)\Delta_{\widetilde{\boldsymbol{\alpha}},\widetilde{\boldsymbol{\beta}}}^{p,\ell}\left(\left\{\frac{n}{p^{\ell}}\right\}\right)+\underset{\alpha_i\notin\mathbb{Z}_p}{\sum_{i=1}^r}\omega_{\ell}(\alpha_i,n)-\underset{\beta_j\notin\mathbb{Z}_p}{\sum_{j=1}^r}\omega_{\ell}(\beta_j,n).
\end{equation}
\medskip

Let  $\alpha$ be an element of $\boldsymbol{\alpha}$ or $\boldsymbol{\beta}$ such that $p$ divides $d(\alpha)$. We now compute $\sum_{\ell=1}^{\infty}\omega_{\ell}(\alpha,n)$. Let $\alpha=r/(p^eN)$ where $1\leq e\leq f$, $N$ divides $D$, $1\leq r\leq p^eN$ and $r$ is coprime to $p^eN$. Given $\ell\in\mathbb{N}$, $\ell\geq 1$, there exists $r_{a,\ell}\in\{1,\dots,p^eN\}$ coprime to  $p^eN$ such that $\mathfrak{D}_{p_{a,\ell}}(\alpha)=r_{a,\ell}/(p^eN)$ and $p_{a,\ell}r_{a,\ell}-r\equiv 0\mod p^eN$. In particular, by \eqref{ref cong}, we have 
$$
p^{\ell}r_{a,\ell}-r\equiv 0\mod N\quad\textup{and}\quad ar_{a,\ell}-r\equiv 0\mod p^e,
$$
\textit{i. e.} 
$$
r_{a,\ell}\equiv s_N\left(\frac{\pi_N(r)}{\pi_N(p^{\ell+e})}\right)p^e+s_{p^e}\left(\frac{\pi_{p^e}(r)}{\pi_{p^e}(aN)}\right)N\mod p^eN.
$$
In the rest of the proof, if $a/b$ is a rational number written in irreducible form and the integer $c\ge 1$ is coprime to $b$, we set 
$$
\varpi_c\left(\frac{a}{b}\right):=s_c\left(\frac{\pi_c(a)}{\pi_c(b)}\right).
$$
Then, 
\begin{equation}\label{part 0}
\frac{r_{a,\ell}}{p^eN}\equiv\frac{\varpi_N(r/p^{\ell+e})}{N}+\frac{\varpi_{p^e}\big(r/(aN)\big)}{p^e}\mod 1.
\end{equation}

For all $\ell\in\mathbb{N}$, we have $p^{\ell+1}\varpi_N(r/p^{\ell+1})-p^{\ell}\varpi_N(r/p^{\ell})\equiv 0\mod N$, hence, since $p$ and $N$ are coprime, we obtain $p\varpi_N(r/p^{\ell+1})-\varpi_N(r/p^{\ell})\equiv 0\mod N$, \textit{i. e.} 
$$
\mathfrak{D}_p\left(\frac{\varpi_N(r/p^{\ell})}{N}\right)=\frac{\varpi_N(r/p^{\ell+1})}{N},
$$
yielding
$$
\frac{\varpi_N(r/p^{\ell+1})}{N}=\mathfrak{D}_p^{\ell+1}\left(\frac{r}{N}\right).
$$
Let $-r/N=\sum_{k=0}^{\infty}a_kp^k$ be the $p$-adic expansion of $-r/N$. For all  $\ell\in\mathbb{N}$, we have 
$$
p^{\ell+1}\mathfrak{D}_p^{\ell+1}\left(\frac{r}{N}\right)-\frac{r}{N}=\sum_{k=0}^{\ell}a_kp^k
$$
and thus 
\begin{equation}\label{part 1}
\frac{\varpi_N(r/p^{\ell+e})}{N}=\frac{r}{p^{\ell+e}N}+\frac{\sum_{k=0}^{\ell+e-1}a_kp^k}{p^{\ell+e}}=\frac{r}{p^{\ell+e}N}+\frac{\sum_{k=0}^{\ell-1}a_kp^k}{p^{\ell+e}}+\frac{\sum_{k=0}^{e-1}a_{\ell+k}p^{k}}{p^e}.
\end{equation}
Moreover,  $p\varpi_N(r/p)\equiv r\mod N$ but $p\varpi_N(r/p)\neq r$ because $r$ is  not divisible by $p$. Hence,  $p\varpi_N(r/p)-r\geq N$ and $a_0\geq 1$.

The elements of the multiset ({\em i.e.}, a set where repetition of elements is permitted) 
$$
\left\{\left\{\varpi_{p^e}\left(\frac{r}{aN}\right)\,:\,1\leq a\leq p^f,\,\gcd(a,p)=1\right\}\right\}
$$
are those $b\in\{1,\dots,p^e\}$ not divisible by $p$, where each $b$ is repeated exactly  $p^{f-e}$ times.  We fix $\ell\in\mathbb{N}$, $\ell\geq 1$. We have 
$$
0<\frac{r}{p^{\ell+e}N}+\frac{\sum_{k=0}^{\ell-1}a_kp^k}{p^{\ell+e}}\leq \frac{1}{p^{\ell+e}}+\frac{p^{\ell}-1}{p^{\ell+e}}\leq\frac{1}{p^e}\quad\textup{and}\quad\frac{r_{a,\ell}}{p^eN}\in(0,1].
$$
By \eqref{part 0} et \eqref{part 1}, the multiset
$$
\Phi_{\ell}(\alpha):=\left\{\left\{\frac{r_{a,\ell}}{p^eN}\,:\,1\leq a\leq p^f,\,\gcd(a,p)=1\right\}\right\}
$$
has the elements
$$
\eta_{\ell,b}:=\frac{r}{p^{\ell+e}N}+\frac{\sum_{k=0}^{\ell-1}a_kp^k}{p^{\ell+e}}+\frac{b}{p^e},
$$
where $b=\sum_{k=0}^{e-1}b_kp^k$, $b_k\in\{0,\dots,p-1\}$, $b_0\neq a_{\ell}$ and  each $\eta_{\ell,b}$ is repeated exactly $p^{f-e}$ times. In the sequel, we fix  $n=\sum_{k=0}^{\infty}n_kp^k$ with $n_k\in\{0,\dots,p-1\}$ and, for all $k\geq K$, $n_k=0$, where $K\in\mathbb{N}$. For all $\ell\in\mathbb{N}$, we let  $\Lambda_{\ell}(\alpha,n)=1$ if
$$
\sum_{k=0}^{\ell-e-1}n_kp^k>\sum_{k=e}^{\ell-1}a_kp^{k-e},
$$
and $\Lambda_{\ell}(\alpha,n)=0$ otherwise. Let us compute the number  $\omega_{\ell}(\alpha,n)$ of elements in $\Phi_{\ell}(\alpha)$ which are $\le \{n/p^{\ell}\}$.
\medskip

If $\ell\leq e-1$, then
\begin{align*}
\left\{\frac{n}{p^{\ell}}\right\}\geq\eta_{\ell,b}&\Longleftrightarrow\frac{\sum_{k=0}^{\ell-1}n_kp^k}{p^{\ell}}\geq\frac{r}{p^{\ell+e}N}+\frac{\sum_{k=0}^{\ell-1}a_kp^k}{p^{\ell+e}}+\frac{\sum_{k=0}^{e-1}b_kp^k}{p^e}\\
&\Longleftrightarrow\sum_{k=0}^{\ell-1}n_kp^{k}\geq\frac{r}{p^{e}N}+\frac{\sum_{k=0}^{\ell-1}a_kp^k}{p^{e}}+\frac{\sum_{k=0}^{e-1}b_kp^{k+\ell}}{p^e}\\
&\Longleftrightarrow\sum_{k=0}^{\ell-1}n_kp^k>\sum_{k=0}^{\ell-1}b_{e-\ell+k}p^{k},
\end{align*}
because
$$
0<\frac{r}{p^eN}+\frac{\sum_{k=0}^{\ell-1}a_kp^k}{p^e}+\frac{\sum_{k=0}^{e-\ell-1}b_kp^{k+\ell}}{p^e}\leq\frac{1}{p^e}+\frac{p^{\ell}-1}{p^e}+\frac{p^{\ell}(p^{e-\ell}-1)}{p^e}\leq 1.
$$
Thus  
$$
\omega_{\ell}(\alpha,n)=\left((p-1)p^{e-\ell-1}\sum_{k=0}^{\ell-1}n_kp^k\right)p^{f-e}.
$$

If $\ell\geq e$, then 
\begin{align}
\left\{\frac{n}{p^{\ell}}\right\}\geq\eta_{\ell,b}&\Longleftrightarrow\frac{\sum_{k=0}^{\ell-1}n_kp^k}{p^{\ell}}\geq\frac{r}{p^{\ell+e}N}+\frac{\sum_{k=0}^{\ell-1}a_kp^k}{p^{\ell+e}}+\frac{\sum_{k=0}^{e-1}b_kp^k}{p^e}\notag\\
&\Longleftrightarrow\sum_{k=0}^{\ell-1}n_kp^{k}\geq\frac{r}{p^{e}N}+\frac{\sum_{k=0}^{\ell-1}a_kp^k}{p^{e}}+\sum_{k=0}^{e-1}b_kp^{k+\ell-e}\notag\\
&\Longleftrightarrow\sum_{k=0}^{\ell-1}n_kp^k>\sum_{k=e}^{\ell-1}a_kp^{k-e}+\sum_{k=0}^{e-1}b_kp^{k+\ell-e}\label{equiv245},
\end{align}
because
$$
0<\frac{r}{p^eN}+\frac{\sum_{k=0}^{e-1}a_kp^k}{p^e}\leq \frac{1}{p^e}+\frac{p^e-1}{p^e}\leq 1.
$$

If we have
$$
\sum_{k=\ell-e+1}^{\ell-1}n_kp^k>\sum_{k=1}^{e-1}b_kp^{k+\ell-e},
$$
then \eqref{equiv245} holds and we obtain 
$$
(p-1)\frac{\sum_{k=\ell-e+1}^{\ell-1}n_kp^k}{p^{\ell-e+1}}
$$ 
numbers $b$ satisfying the above inequality. Let us now assume that  
$$
\sum_{k=\ell-e+1}^{\ell-1}n_{k}p^{k}=\sum_{k=1}^{e-1}b_kp^{k+\ell-e}.
$$
Then \eqref{equiv245} is the same thing as 
\begin{equation}\label{equiv246}
\sum_{k=0}^{\ell-e}n_kp^k>\sum_{k=e}^{\ell-1}a_kp^{k-e}+b_0p^{\ell-e}.
\end{equation}
\medskip

If $n_{\ell-e}\geq a_{\ell}+1$, then there are $n_{\ell-e}-1$ elements $b_0\in\{0,\dots,p-1\}\setminus\{a_{\ell}\}$ such that $n_{\ell-e}>b_0$, and, for $b_0=n_{\ell-e}$, we have \eqref{equiv246} if and only if $\Lambda_{\ell}(\alpha,n)=1$. Moreover, when $n_{\ell-e}\geq a_{\ell}+1$, we have $\Lambda_{\ell+1}(\alpha,n)=1$. Hence, if $n_{\ell-e}\geq a_{\ell}+1$, we have $n_{\ell-e}+\Lambda_{\ell}(\alpha,n)-\Lambda_{\ell+1}(\alpha,n)$ numbers $b_0$ such that \eqref{equiv246} holds.
\medskip

If $n_{\ell-e}=a_{\ell}$, then there are $n_{\ell-e}$ numbers $b_0$ such that  \eqref{equiv246} holds. Furthermore, we have $\Lambda_{\ell}(\alpha,n)=\Lambda_{\ell+1}(\alpha,n)$ and in this case we also have  $n_{\ell-e}+\Lambda_{\ell}(\alpha,n)-\Lambda_{\ell+1}(\alpha,n)$ numbers $b_0$ such that \eqref{equiv246} holds.
\medskip

If $n_{\ell-e}\leq a_{\ell}-1$, then there are $n_{\ell-e}$ numbers $b_0$ such that  $b_0<n_{\ell-e}$, and for $b_0=n_{\ell-e}$, we have \eqref{equiv246} if and only if $\Lambda_{\ell}(\alpha,n)=1$. Moreover, if $n_{\ell-e}\leq a_{\ell}-1$, then  $\Lambda_{\ell+1}(\alpha,n)=0$ and again there are $n_{\ell-e}+\Lambda_{\ell}(\alpha,n)-\Lambda_{\ell+1}(\alpha,n)$ numbers $b_0$ satisfying \eqref{equiv246}.
\medskip

It follows that if $\ell\geq e$, then,
$$
\omega_{\ell}(\alpha,n)=\left(n_{\ell-e}+\Lambda_{\ell}(\alpha,n)-\Lambda_{\ell+1}(\alpha,n)+(p-1)\sum_{k=\ell-e+1}^{\ell-1}n_{k}p^{k-\ell+e-1}\right)p^{f-e}.
$$
Hence, for all  $m\in\mathbb{N}$, $m\geq K+e$, we get 
\begin{multline}\label{formu fin}
p^{e-f}\sum_{\ell=1}^{m}\omega_{\ell}(\alpha,n)=(p-1)\sum_{\ell=1}^{e-1}p^{e-\ell-1}\sum_{k=0}^{\ell-1}n_kp^k\\
+\sum_{\ell=e}^m\big(n_{\ell-e}+\Lambda_{\ell}(\alpha,n)-\Lambda_{\ell+1}(\alpha,n)\big)
+(p-1)\sum_{\ell=e}^m\sum_{k=\ell-e+1}^{\ell-1}n_{k}p^{k-\ell+e-1}.
\end{multline}

Let us  compute the coefficients $h_k$ of $n_k$, $0\leq k\leq K$, on the right hand side of \eqref{formu fin}, so that 
\begin{equation}\label{formu fin 2}
p^{e-f}\sum_{\ell=1}^m\omega_{\ell}(\alpha,n)=\Lambda_e(\alpha,n)-\Lambda_{m+1}(\alpha,n)+\sum_{k=0}^Kh_kn_k.
\end{equation}
If $e=1$, then for all $k\in\{0,\dots,K\}$, we have $h_k=1=p^{e-1}$. Let us assume that  $e\geq 2$. We have 
$$
h_0=(p-1)\sum_{\ell=1}^{e-1}p^{e-\ell-1}+1=p^{e-1}.
$$
If $1\leq k\leq e-2$, then 
$$
h_k=(p-1)\sum_{\ell=k+1}^{e-1}p^{k-\ell+e-1}+1+(p-1)\sum_{\ell=e}^{k+e-1}p^{k-\ell+e-1}=p^{e-1}-p^k+1+p^k-1=p^{e-1}.
$$
Finally, if $k\geq e-1$, then 
$$
h_k=1+(p-1)\sum_{\ell=k+1}^{k+e-1}p^{k-\ell+e-1}=1+p^{e-1}-1=p^{e-1}.
$$
Hence, we obtain 
$$
p^{e-f}\sum_{\ell=1}^m\omega_{\ell}(\alpha,n)=\Lambda_e(\alpha,n)-\Lambda_{m+1}(\alpha,n)+p^{e-1}\mathfrak{s}_p(n),
$$
where $\mathfrak{s}_p(n):=\sum_{k=0}^{\infty}n_k=\sum_{k=0}^Kn_k$.
\medskip

Moreover, we have  $\Lambda_e(\alpha,n)=0$ and there exists $K'\geq K+e$ such that, for all $m\geq K'$, we have $\Lambda_{m+1}(\alpha,n)=0$. Indeed,  $\sum_{k=0}^{\infty}a_kp^k$ is the $p$-adic expansion of  $-r/N\notin\mathbb{N}$. Thus, there exists $K'\geq K+e$ such that $a_{K'}\neq 0$ and hence, for all $m\geq K'$, we have $\Lambda_{m+1}(\alpha,n)=0$. Consequently, for all large enough $\ell$, we have  $\omega_{\ell}(\alpha,n)=0$ and 
\begin{equation}\label{omega fin}
\sum_{\ell=1}^{\infty}\omega_{\ell}(\alpha,n)=\varphi\big(p^f\big)\frac{\mathfrak{s}_p(n)}{p-1}.
\end{equation}
\medskip

By \eqref{ref omega} and  \eqref{omega fin}, we obtain,  for all $n\in\mathbb{N}$,
$$
\sum_{\ell=1}^{\infty}\underset{\gcd(a,p)=1}{\sum_{a=1}^{p^f}}\Delta_{\boldsymbol{\alpha},\boldsymbol{\beta}}^{p_{a,\ell},1}\left(\left\{\frac{n}{p^{\ell}}\right\}\right)=\varphi\big(p^f\big)\sum_{\ell=1}^\infty\Delta_{\widetilde{\boldsymbol{\alpha}},\widetilde{\boldsymbol{\beta}}}^{p,\ell}\left(\left\{\frac{n}{p^{\ell}}\right\}\right)+(r-s-\lambda_p)\varphi\big(p^f\big)\frac{\mathfrak{s}_p(n)}{p-1}.
$$
Together with \eqref{Terme brut}, this implies that 
\begin{multline}\label{ola 2}
v_p\left(C_0^n\frac{(\alpha_1)_n\cdots(\alpha_r)_n}{(\beta_1)_n\cdots(\beta_s)_n}\right)=\frac{1}{\varphi\big(p^f\big)}\sum_{\ell=1}^{\infty}\underset{\gcd(a,p)=1}{\sum_{a=1}^{p^f}}\Delta_{\boldsymbol{\alpha},\boldsymbol{\beta}}^{p_{a,\ell},1}\left(\left\{\frac{n}{p^{\ell}}\right\}\right)\\
+\lambda_p\left(\frac{\mathfrak{s}_p(n)}{p-1}+v_p(n!)\right)-n\left\lfloor\frac{\lambda_p}{p-1}\right\rfloor.
\end{multline}
But  for all $n\in\mathbb{N}$, we have  
$$
v_p(n!)=\frac{n-\mathfrak{s}_p(n)}{p-1},
$$
so that for all $n\in\mathbb{N}$,  
\begin{equation}\label{last}
\lambda_p\left(\frac{\mathfrak{s}_p(n)}{p-1}+v_p(n!)\right)-n\left\lfloor\frac{\lambda_p}{p-1}\right\rfloor=n\left\{\frac{\lambda_p}{p-1}\right\}.
\end{equation}
Hence, using \eqref{last} in \eqref{ola 2}, we get equation \eqref{magie}, which completes the proof of Proposition \ref{propo magie}.
\end{proof}

\section{Proof of Theorem \ref{theo Const}}\label{section demo 1}

Let $\boldsymbol{\alpha}$ and $\boldsymbol{\beta}$ be two sequences taking their values in $\mathbb{Q}\setminus\mathbb{Z}_{\leq 0}$. Let us assume that  $F_{\boldsymbol{\alpha},\boldsymbol{\beta}}$ is $N$-integral. We first prove \eqref{facile}.
\medskip

We fix a prime $p$. We denote by $\widetilde{\boldsymbol{\alpha}}$, respectively $\widetilde{\boldsymbol{\beta}}$, the  (possibly empty) sequence $(\widetilde{\alpha}_1,\dots,\widetilde{\alpha}_u)$, respectively $(\widetilde{\beta}_1,\dots,\widetilde{\beta}_v)$, made from the  elements of $\boldsymbol{\alpha}$, respectively of $\boldsymbol{\beta}$, and whose  denominator is not divisible  by $p$. In particular, we have $\lambda_p(\boldsymbol{\alpha},\boldsymbol{\beta})=u-v$. By Proposition \ref{valuation}, for all $n\in\mathbb{N}$, we thus have
\begin{align}
v_p\left(\frac{(\alpha_1)_n\cdots(\alpha_r)_n}{(\beta_1)_n\cdots(\beta_s)_n}\right)&=-nv_p\left(\frac{\prod_{i=1}^rd(\alpha_i)}{\prod_{j=1}^sd(\beta_j)}\right)+v_p\left(\frac{(\widetilde{\alpha}_1)_n\cdots(\widetilde{\alpha}_u)_n}{(\widetilde{\beta_1})_n\cdots(\widetilde{\beta_v})_n}\right)\notag\\
&=-nv_p\left(\frac{\prod_{i=1}^rd(\alpha_i)}{\prod_{j=1}^sd(\beta_j)}\right)+\sum_{\ell=1}^{\infty}\Delta_{\widetilde{\boldsymbol{\alpha}},\widetilde{\boldsymbol{\beta}}}^{p,\ell}\left(\left\{\frac{n}{p^{\ell}}\right\}\right)+\lambda_p(\boldsymbol{\alpha},\boldsymbol{\beta})v_p(n!).\label{formu 1}
\end{align}

By Lemma \ref{Lemma structure}, there exists a constant $M>0$ such that, for any prime $p$ that does not divide $d_{\widetilde{\boldsymbol{\alpha}},\widetilde{\boldsymbol{\beta}}}$, for any $\ell\in\mathbb{N}$, $\ell\geq 1$, and any $x\in[0,1/M)$, we have $\Delta_{\widetilde{\boldsymbol{\alpha}},\widetilde{\boldsymbol{\beta}}}^{p,\ell}(x)=0$. Hence, for all $n\in\mathbb{N}$, we have
$$ -v\big\lfloor\log_p(nM)\big\rfloor\leq\sum_{\ell=1}^{\infty}
\Delta_{\widetilde{\boldsymbol{\alpha}},\widetilde{\boldsymbol{\beta}}}^{p,\ell}\left(\left\{\frac{n}{p^{\ell}}\right\}\right)\leq u\big\lfloor\log_p(nM)\big\rfloor,
$$
so that
\begin{equation}\label{inter m1}
\frac{1}{n}\sum_{\ell=1}^{\infty}\Delta_{\widetilde{\boldsymbol{\alpha}},\widetilde{\boldsymbol{\beta}}}^{p,\ell}\left(\left\{\frac{n}{p^{\ell}}\right\}\right)\underset{n\rightarrow+\infty}{\longrightarrow}0.
\end{equation}
Moreover, for all $n\in\mathbb{N}$, we have $v_p(n!)=\sum_{\ell=1}^{\infty}\lfloor n/p^{\ell}\rfloor$, hence 
$$
\sum_{\ell=1}^{\lfloor\log_p(n)\rfloor}\frac{n}{p^{\ell}}-\big\lfloor\log_p(n)\big\rfloor\leq v_p(n!)\leq \sum_{\ell=1}^{\lfloor\log_p(n)\rfloor}\frac{n}{p^{\ell}}
$$
and 
\begin{equation}\label{inter m2}
\frac{1}{n}v_p(n!)\underset{n\rightarrow+\infty}{\longrightarrow}\frac{1}{p-1}.
\end{equation}
We now use \eqref{inter m1} and \eqref{inter m2} in \eqref{formu 1}, and we obtain
$$
\frac{1}{n}v_p\left(\frac{(\alpha_1)_n\cdots(\alpha_r)_n}{(\beta_1)_n\cdots(\beta_s)_n}\right)\underset{n\rightarrow+\infty}{\longrightarrow}-v_p\left(\frac{\prod_{i=1}^rd(\alpha_i)}{\prod_{j=1}^sd(\beta_j)}\right)+\frac{\lambda_p(\boldsymbol{\alpha},\boldsymbol{\beta})}{p-1}.
$$
But for all $n\in\mathbb{N}$,
$$
C_{\boldsymbol{\alpha},\boldsymbol{\beta}}^n\frac{(\alpha_1)_n\cdots(\alpha_r)_n}{(\beta_1)_n\cdots(\beta_s)_n}\in\mathbb{Z}_p.
$$
It follows that for all $n\in\mathbb{N}$, $n\geq 1$, 
$$
v_p\big(C_{\boldsymbol{\alpha},\boldsymbol{\beta}}\big)\geq-\frac{1}{n}v_p\left(\frac{(\alpha_1)_n\cdots(\alpha_r)_n}{(\beta_1)_n\cdots(\beta_s)_n}\right)\underset{n\rightarrow+\infty}{\longrightarrow}v_p\left(\frac{\prod_{i=1}^rd(\alpha_i)}{\prod_{j=1}^sd(\beta_j)}\right)-\frac{\lambda_p(\boldsymbol{\alpha},\boldsymbol{\beta})}{p-1}
$$
and thus
$$
v_p\big(C_{\boldsymbol{\alpha},\boldsymbol{\beta}}\big)\geq v_p\left(\frac{\prod_{i=1}^rd(\alpha_i)}{\prod_{j=1}^sd(\beta_j)}\right)-\left\lfloor\frac{\lambda_p(\boldsymbol{\alpha},\boldsymbol{\beta})}{p-1}\right\rfloor,
$$
because $v_p\big(C_{\boldsymbol{\alpha},\boldsymbol{\beta}}\big)\in\mathbb{Z}$. Furthermore, if $p$ does not divide $d_{\boldsymbol{\alpha},\boldsymbol{\beta}}$ and if  $p\geq r-s+2$, then $\lambda_p(\boldsymbol{\alpha},\boldsymbol{\beta})=r-s$ and $\lfloor\lambda_p(\boldsymbol{\alpha},\boldsymbol{\beta})/(p-1)\rfloor=0$. This proves the existence of $C\in\mathbb{N}^{\ast}$ such that
\begin{equation}
C_{\boldsymbol{\alpha},\boldsymbol{\beta}}=C\frac{\prod_{i=1}^rd(\alpha_i)}{\prod_{j=1}^sd(\beta_j)}\underset{p\in\mathcal{P}_{\boldsymbol{\alpha},\boldsymbol{\beta}}}{\prod}p^{-\left\lfloor\frac{\lambda_p(\boldsymbol{\alpha},\boldsymbol{\beta})}{p-1}\right\rfloor}.
\end{equation}
We now define
$$
C_0:=\frac{\prod_{i=1}^rd(\alpha_i)}{\prod_{j=1}^sd(\beta_j)}\underset{p\mid d_{\boldsymbol{\alpha},\boldsymbol{\beta}}}{\prod}p^{-\left\lfloor\frac{\lambda_p(\boldsymbol{\alpha},\boldsymbol{\beta})}{p-1}\right\rfloor}.
$$
In  the sequel, we assume that both sequences  $\boldsymbol{\alpha}$ and  $\boldsymbol{\beta}$ take their values in $(0,1]$ and that $r=s$. We show that in this case $C=1$ and for this it is enough to prove that $F_{\boldsymbol{\alpha},\boldsymbol{\beta}}(C_0z)\in\mathbb{Z}[[z]]$.
\medskip 

Consider a prime $p$ that does not divide $d_{\boldsymbol{\alpha},\boldsymbol{\beta}}$, so that $\lambda_p(\boldsymbol{\alpha},\boldsymbol{\beta})=r-s=0$. Together with \eqref{formu 1}, this yields
$$
v_p\left(C_0^n\frac{(\alpha_1)_n\cdots(\alpha_r)_n}{(\beta_1)_n\cdots(\beta_s)_n}\right)=\sum_{\ell=1}^{\infty}\Delta_{\boldsymbol{\alpha},\boldsymbol{\beta}}^{p,\ell}\left(\left\{\frac{n}{p^{\ell}}\right\}\right).
$$
By Lemma \ref{valeurs Delta} and Theorem A, for all 
$\ell\in\mathbb{N}$, $\ell\geq 1$, we have
$$
\Delta_{\boldsymbol{\alpha},\boldsymbol{\beta}}^{p,\ell}([0,1])=\xi_{\boldsymbol{\alpha},\boldsymbol{\beta}}(a,\mathbb{R})\subset\mathbb{N},
$$
where $a\in\{1,\dots,d_{\boldsymbol{\alpha},\boldsymbol{\beta}}\}$ satisfies $p^{\ell}a\equiv 1\mod d_{\boldsymbol{\alpha},\boldsymbol{\beta}}$. Hence, we obtain  that  $F_{\boldsymbol{\alpha},\boldsymbol{\beta}}(C_0z)\in\mathbb{Z}_p[[z]]$. It remains  to show that for any prime $p$ that divides $d_{\boldsymbol{\alpha},\boldsymbol{\beta}}$, we also have that $F_{\boldsymbol{\alpha},\boldsymbol{\beta}}(C_0z)\in\mathbb{Z}_p[[z]]$.
\medskip 

Consider a prime $p$ that divides $d_{\boldsymbol{\alpha},\boldsymbol{\beta}}$. With the notations of Proposition~\ref{propo magie}, for all $n\in\mathbb{N}$, we have 
$$
v_p\left(C_0^n\frac{(\alpha_1)_n\cdots(\alpha_r)_n}{(\beta_1)_n\cdots(\beta_s)_n}\right)=\frac{1}{\varphi\big(p^f\big)}\underset{\gcd(a,p)=1}{\sum_{a=1}^{p^f}}\sum_{\ell=1}^{\infty}\Delta_{\boldsymbol{\alpha},\boldsymbol{\beta}}^{p_{a,\ell},1}\left(\left\{\frac{n}{p^{\ell}}\right\}\right)+n\left\{\frac{\lambda_p(\boldsymbol{\alpha},\boldsymbol{\beta})}{p-1}\right\}.
$$ 
Since none of the primes $p_{a,\ell}$ divides $d_{\boldsymbol{\alpha},\boldsymbol{\beta}}$, we have $\Delta_{\boldsymbol{\alpha},\boldsymbol{\beta}}^{p_{a,\ell},1}([0,1])\subset\mathbb{N}$ so that  $F_{\boldsymbol{\alpha},\boldsymbol{\beta}}(C_0z)\in\mathbb{Z}_p[[z]]$. This completes the proof of Theorem~\ref{theo Const}.
\hfill $\square$

\section{Formal congruences}\label{section cong form}

To prove Theorem \ref{theo expand}, we need a ``formal congruences'' result, stated in  Theorem \ref{congruences formelles} below that we prove in this section.

We fix a prime $p$ and denote by $\Omega$ the completion of the algebraic closure of $\mathbb{Q}_p$, and by $\mathcal{O}$ the ring of integers of $\Omega$.
\medskip

To state the main result of this section, we introduce some notations. If $\mathcal{N}:=(\mathcal{N}_r)_{r\geq 0}$ is a sequence of subsets of $\bigcup_{t\geq 1}\big(\{0,\dots,p^t-1\}\times\{t\}\big)$, then for all $r\in\mathbb{Z}$, $r\geq -1$ and all  $s\in\mathbb{N}$, we denote by $\Psi_{\mathcal{N}}(r,s)$ the set of the $u\in\{0,\dots,p^s-1\}$ such that, for all $(n,t)\in\mathcal{N}_{r+s-t+1}$, with $t\leq s$, and all  $j\in\{0,\dots,p^{s-t}-1\}$, we have  $u\neq j+p^{s-t}n$. In particular, for all  $r\geq -1$, we have $\Psi_{\mathcal{N}}(r,0)=\{0\}$.
\medskip

For completeness, let us recall some basic notions. Let $\mathcal{A}$ be a commutative algebra (with a unit) over a commutative ring (with a unit) $\mathcal{Z}$. An element $a\in\mathcal{A}$ is \textit{regular} if, for all  $b\in\mathcal{A}$, we have $(ab=0\Rightarrow b=0)$. We define $\mathcal{S}$ as the  set of the regular elements of $\mathcal{A}$. Hence $\mathcal{S}$ is a multiplicative set  of $\mathcal{A}$ and the ring  $\mathcal{S}^{-1}\mathcal{A}$ with the map 
$$
\begin{array}{ccc}
\mathcal{Z}\times\mathcal{S}^{-1}\mathcal{A} & \longrightarrow & \mathcal{S}^{-1}\mathcal{A}\\
(\lambda\,,\,a/s)                       & \mapsto         & (\lambda\cdot a)/s
\end{array}
$$
is a $\mathcal{Z}$-algebra. Moreover, the morphism of algebra  $a\in\mathcal{A}\mapsto a/1\in\mathcal{S}^{-1}\mathcal{A}$ is injective and enables us to identify $\mathcal{A}$ with a sub-algebra of $\mathcal{S}^{-1}\mathcal{A}$. This is what we do in the statement of Theorem~\ref{congruences formelles}.

\begin{theo}\label{congruences formelles}
Let $\mathcal{Z}$ denote a sub-ring of $\mathcal{O}$ and $\mathcal{A}$ a  $\mathcal{Z}$-algebra (commutative with a unit) such that 
$2$ is a regular element of $\mathcal{A}$. We consider a sequence of maps $(\mathbf{A}_r)_{r\geq 0}$ from $\mathbb{N}$ into $\mathcal{S}$, and a sequence  of maps $(\mathbf{g}_r)_{r\geq 0}$ from $\mathbb{N}$ into $\mathcal{Z}\setminus\{0\}$. We assume there exists a sequence $\mathcal{N}:=(\mathcal{N}_r)_{r\geq 0}$ of subsets of $\bigcup_{t\geq 1}\big(\{0,\dots,p^t-1\}\times\{t\}\big)$ such that, for all $r\geq 0$, we have the following properties:  
\begin{itemize} 
\item[$(i)$] $\mathbf{A}_r(0)$ is invertible in $\mathcal{A}$;
\item[$(ii)$] for all $m\in\mathbb{N}$, we have  $\mathbf{A}_r(m)\in\mathbf{g}_r(m)\mathcal{A}$;
\item[$(iii)$] for all $s,m\in\mathbb{N}$, we have:
\begin{itemize}
\item[$(a)$] for all $u\in\Psi_{\mathcal{N}}(r,s)$ and all $v\in\{0,\dots,p-1\}$, we have 
$$
\frac{\mathbf{A}_r(v+up+mp^{s+1})}{\mathbf{A}_r(v+up)}-\frac{\mathbf{A}_{r+1}(u+mp^s)}{\mathbf{A}_{r+1}(u)}\in p^{s+1}\frac{\mathbf{g}_{r+s+1}(m)}{\mathbf{A}_r(v+up)}\mathcal{A} ;
$$
\begin{itemize}
\item[$(a_1$)] moreover, if $v+up\in\Psi_{\mathcal{N}}(r-1,s+1)$, then 
$$
\mathbf{g}_r(v+up)\left(\frac{\mathbf{A}_r(v+up+mp^{s+1})}{\mathbf{A}_r(v+up)}-\frac{\mathbf{A}_{r+1}(u+mp^s)}{\mathbf{A}_{r+1}(u)}\right)\in p^{s+1}\mathbf{g}_{r+s+1}(m)\mathcal{A};
$$
\item[$(a_2)$] however, if $v+up\notin\Psi_{\mathcal{N}}(r-1,s+1)$, then
$$
\mathbf{g}_r(v+up)\frac{\mathbf{A}_{r+1}(u+mp^s)}{\mathbf{A}_{r+1}(u)}\in p^{s+1}\mathbf{g}_{r+s+1}(m)\mathcal{A};
$$
\end{itemize}
\item[$(b)$] for all $(n,t)\in\mathcal{N}_r$, we have $\mathbf{g}_r(n+mp^t)\in p^t\mathbf{g}_{r+t}(m)\mathcal{Z}$.
\end{itemize}
\end{itemize}
Then, for all  $a\in\{0,\dots,p-1\}$ and all $m,s,r,K\in\mathbb{N}$, we have  
\begin{multline}\label{TheBut}
S_r(a,K,s,p,m):=\\
\sum_{j=mp^s}^{(m+1)p^s-1}\Big(\mathbf{A}_r\big(a+(K-j)p\big)\mathbf{A}_{r+1}(j)-\mathbf{A}_{r+1}(K-j)\mathbf{A}_r(a+jp)\Big)\in p^{s+1}\mathbf{g}_{r+s+1}(m)\mathcal{A},
\end{multline} 
where $\mathbf{A}_r(n)=0$ if $n<0$.
\end{theo}

Theorem \ref{congruences formelles} is a generalisation of a result due to Dwork \cite[Theorem $1.1$]{Dwork}, first used (in weaker version \cite{cycles}) to obtain the analytic continuation of certain $p$-adic functions. Dwork then developped in  \cite{Dwork} a method to prove the $p$-adic integrality of the Taylor coefficients of $q$-coordinates. This method is the basis of the proofs of the  $N$-integrality of  $q_{\boldsymbol{\alpha},\boldsymbol{\beta}}(z)$. In the litterature, one finds many generalisations of Dwork's formal congruences used to prove  the integrality of Taylor coefficients of $q$-coordinates with increasing generality (see \cite{Tanguy3}, \cite{Delaygue1} and \cite{Delaygue3}).

If we consider only the univariate case, then Theorem \ref{congruences formelles} encompasses all the analogous results in \cite{Tanguy3} and \cite{Delaygue3}. Its interest is due to the two following improvements.

$\bullet$ Theorem \ref{congruences formelles} can be applied to  $\mathbb{Z}_p$-algebras more ``abstract'' than $\mathcal{O}$. We use this possibility in this paper, where we consider algebras of functions taking values in $\mathbb{Z}_p$. This improvement enables us to consider the integer  $\mathfrak{n}_{\boldsymbol{\alpha},\boldsymbol{\beta}}$ in Assertion $(3)$ of Theorem \ref{Criterion}.

$\bullet$ Beside this difference, Theorem \ref{congruences formelles} is a univariate version of Theorem~$4$ in \cite{Delaygue3} that allows to consider a set 
$\mathcal{N}$ that depends on $r$. This property is crucial when we deal with the case of non $R$-partitioned tuples $\boldsymbol{\alpha}$ and $\boldsymbol{\beta}$.

There also exist in the litterature other types of generalisations of Dwork's formal congruences, such as the truncated version of  Ota \cite{Ota} and the recent version of Mellit and Vlasenko \cite{Vlasenko} (applied to constant terms of powers of Laurent polynomials).  

\subsection{Proof of Theorem \ref{congruences formelles}}

For all $s\in\mathbb{N}$, $s\geq 1$, we denote by $\alpha_s$ the following assertion: ``For all $a\in\{0,\dots,p-1\}$, all $u\in\{0,\dots,s-1\}$, all $m,r\in\mathbb{N}$ and all $K\in\mathbb{Z}$, we have 
$$
{\mathbf S}_r(a,K,u,p,m)\in p^{u+1}{\mathbf g}_{r+u+1}(m)\mathcal{A}. ''
$$

For all $s\in\mathbb{N}$, $s\geq 1$, and all $t\in\{0,\dots,s\}$, we denote by $\beta_{t,s}$ the following assertion: ``For all $a\in\{0,\dots,p-1\}$, all $m,r\in\mathbb{N}$ and all $K\in\mathbb{Z}$, we have
\begin{multline*}
{\mathbf S}_r(a,K+mp^{s},s,p,m)\equiv\\ \sum_{j\in\Psi_{\mathcal{N}}(r+t,s-t)}\frac{{\mathbf A}_{r+t+1}(j+mp^{s-t})}{{\mathbf A}_{r+t+1}(j)}{\mathbf S}_r(a,K,t,p,j)\mod p^{s+1}{\mathbf g}_{r+s+1}(m)\mathcal{A}.''
\end{multline*}
For all $a\in\{0,\dots,p-1\}$, all $K\in\mathbb{Z}$ an all $r,j\in\mathbb{N}$, we define  
$$
{\mathbf U}_r(a,K,p,j):={\mathbf A}_r\big(a+(K-j)p\big){\mathbf A}_{r+1}(j)-{\mathbf A}_{r+1}(K-j){\mathbf A}_r(a+jp).
$$
Then, we have  
$$
{\mathbf S}_r(a,K,s,p,m)=\sum_{j=0}^{p^s-1}{\mathbf U}_r(a,K,p,j+mp^s).
$$ 
We now state four lemmas that will be needed to prove  \eqref{TheBut}.

\begin{Lemma}\label{Assertion 1}
Assertion $\alpha_1$ holds.
\end{Lemma}

\begin{Lemma}\label{Assertion 2}
For all $s,r,m\in\mathbb{N}$, all $a\in\{0,\dots,p-1\}$, all $j\in\Psi_{\mathcal{N}}(r,s)$ and all $K\in\mathbb{Z}$, we have 
$$
{\mathbf U}_r(a,K+mp^{s},p,j+mp^{s})\equiv\frac{{\mathbf A}_{r+1}(j+mp^{s})}{{\mathbf A}_{r+1}(j)}{\mathbf U}_r(a,K,p,j)\mod p^{s+1}{\mathbf g}_{r+s+1}(m)\mathcal{A}.
$$
\end{Lemma}

\begin{Lemma}\label{iii 0}
For all $s\in\mathbb{N}$, $s\geq 1$, if $\alpha_s$ holds, then, for all $a\in\{0,\dots,p-1\}$, all $K\in\mathbb{Z}$ and all $r,m\in\mathbb{N}$, we have 
$$
{\mathbf S}_r(a,K,s,p,m)\equiv\sum_{j\in\Psi_{\mathcal{N}}(r,s)}{\mathbf U}_r(a,K,p,j+mp^s)\mod p^{s+1}{\mathbf g}_{r+s+1}(m)\mathcal{A};
$$
\end{Lemma}

\begin{Lemma}\label{Assertion 3}
For all $s\in\mathbb{N}$, $s\geq 1$, all $t\in\{0,\dots,s-1\}$, Assertions $\alpha_s$ and  $\beta_{t,s}$ imply Assertion $\beta_{t+1,s}$.
\end{Lemma}

Before we prove these lemmas, let us check that they imply  \eqref{TheBut}. We show  that  $\alpha_s$ holds for all $s\geq 1$ by induction on $s$, which gives the conclusion of Theorem \ref{congruences formelles}. By Lemma \ref{Assertion 1}, $\alpha_1$ holds. Let us assume that $\alpha_s$ holds for some $s\geq 1$. We observe that  $\beta_{0,s}$ is the assertion 
\begin{multline*}
\beta_{0,s}:{\mathbf S}_r(a,K+mp^{s},s,p,m)\equiv\\
\sum_{j\in\Psi_{\mathcal{N}}(r,s)}\frac{{\mathbf A}_{r+1}(j+mp^{s})}{{\mathbf A}_{r+1}(j)}{\mathbf S}_r(a,K,0,p,j)\mod p^{s+1}{\mathbf g}_{r+s+1}(m)\mathcal{A}.
\end{multline*}
Since ${\mathbf S}_r(a,K,0,p,j)={\mathbf U}_r(a,K,p,j)$, we have 
$$
\sum_{j\in\Psi_{\mathcal{N}}(r,s)}\frac{{\mathbf A}_{r+1}(j+mp^{s})}{{\mathbf A}_{r+1}(j)}{\mathbf S}_r(a,K,0,p,j)=\sum_{j\in\Psi_{\mathcal{N}}(r,s)}\frac{{\mathbf A}_{r+1}(j+mp^{s})}{{\mathbf A}_{r+1}(j)}{\mathbf U}_r(a,K,p,j)
$$
and, by Lemma \ref{Assertion 2}, we obtain, modulo $p^{s+1}{\mathbf g}_{r+s+1}(m)\mathcal{A}$, that 
\begin{align}
\sum_{j\in\Psi_{\mathcal{N}}(r,s)}\frac{{\mathbf A}_{r+1}(j+mp^{s})}{{\mathbf A}_{r+1}(j)}{\mathbf U}_r(a,K,p,j)
&\equiv\sum_{j\in\Psi_{\mathcal{N}}(r,s)}{\mathbf U}_r(a,K+mp^{s},p,j+mp^{s})\notag\\
&\equiv {\mathbf S}_r(a,K+mp^{s},s,p,m) ,\label{expli cond 0}
\end{align}
where \eqref{expli cond 0} is obtained \textit{via} Lemma \ref{iii 0}.
\medskip

Consequently, Assertion $\beta_{0,s}$ holds. We then obtain the validity of $\beta_{1,s}$ by means of Lemma~\ref{Assertion 3}. Iterating Lemma \ref{Assertion 3}, we finally obtain $\beta_{s,s}$ which, modulo $p^{s+1}\mathbf{g}_{r+s+1}(m)\mathcal{A}$, can be written
\begin{align}
{\mathbf S}_r(a,K+mp^{s},s,p,m)
&\equiv\sum_{j\in\Psi_{\mathcal{N}}(r+s,0)}\frac{{\mathbf A}_{r+s+1}(j+m)}{{\mathbf A}_{r+s+1}(j)}{\mathbf S}_r(a,K,s,p,j)\notag\\
&\equiv\frac{{\mathbf A}_{r+s+1}(m)}{{\mathbf A}_{r+s+1}(0)}{\mathbf S}_r(a,K,s,p,0)\label{assert beta s},
\end{align}
where we have used in \eqref{assert beta s} the fact that $\Psi_{\mathcal{N}}(r+s,0)=\{0\}$.
\medskip

Let us now prove that, for all  $a\in\{0,\dots,p-1\}$, all $r\in\mathbb{N}$ and all  $K\in\mathbb{Z}$, we have ${\mathbf S}_r(a,K,s,p,0)\in p^{s+1}\mathcal{A}$. For all $N\in\mathbb{Z}$, we denote by $\textsl{P}_{N}$ the assertion: ``For all $a\in\{0,\dots,p-1\}$ and all $r\in\mathbb{N}$, we have $\mathbf{S}_r(a,N,s,p,0)\in p^{s+1}\mathcal{A}$''. 

If $N<0$, then for all $j\in\{0,\dots,p^s-1\}$, we have ${\mathbf A}_r\big(a+(N-j)p\big)=0$ and ${\mathbf A}_{r+1}(N-j)=0$, so that $\mathbf{S}_r(a,N,s,p,0)=0\in p^{s+1}\mathcal{A}$. To find a contradiction, let us assume the existence of a  minimal element $N\in\mathbb{N}$ such that $\textsl{P}_{N}$ does not hold. Consider  $m\in\mathbb{N}$, $m\geq 1$, and set $N':=N-mp^s$. Using \eqref{assert beta s} with  $N'$ instead of $K$, we obtain 
$$
{\mathbf S}_r(a,N,s,p,m)\equiv\frac{{\mathbf A}_{r+s+1}(m)}{{\mathbf A}_{r+s+1}(0)}{\mathbf S}_r(a,N',s,p,0)\mod p^{s+1}{\mathbf g}_{r+s+1}(m)\mathcal{A}.
$$
Since $m\geq 1$, we have $N'<N$, which, by definition of $N$, yields that ${\mathbf S}_r(a,N',s,p,0)\in p^{s+1}\mathcal{A}$. By Condition $(i)$, ${\mathbf A}_{r+s+1}(0)$ is an invertible of $\mathcal{A}$ and thus 
$$
{\mathbf S}_r(a,N,s,p,m)\in p^{s+1}\mathcal{A}. 
$$
Hence, for all $m\in\mathbb{N}$, $m\geq 1$, we have  ${\mathbf S}_r(a,N,s,p,m)\in p^{s+1}\mathcal{A}$. Consider $T\in\mathbb{N}$ such that $(T+1)p^s>N$. Then, 
\begin{align}
\sum_{m=0}^{T}{\mathbf S}_r(a&,N,s,p,m)\notag\\
&=\sum_{m=0}^{T}\sum_{j=mp^s}^{(m+1)p^s-1}\Big({\mathbf A}_r\big(a+(N-j)p\big){\mathbf A}_{r+1}(j)-{\mathbf A}_{r+1}(N-j){\mathbf A}_r(a+jp)\Big)\notag\\
&=\sum_{j=0}^N\Big({\mathbf A}_r\big(a+(N-j)p\big){\mathbf A}_{r+1}(j)-{\mathbf A}_{r+1}(N-j){\mathbf A}_r(a+jp)\Big)\label{rappel conv}\\
&=0\label{expli somme 0},
\end{align}
where we have used in \eqref{rappel conv} the fact that ${\mathbf A}_r(n)=0$ if $n<0$. Equation \eqref{expli somme 0} holds because $2$ is a regular element of $\mathcal{A}$ and the term of the sum \eqref{rappel conv} is changed to its opposite when we change the indice $j$ to $N-j$. It follows that we have 
$$
{\mathbf S}_r(a,N,s,p,0)=-\sum_{m=1}^T{\mathbf S}_r(a,N,s,p,m)\in p^{s+1}\mathcal{A}.
$$
This contradicts the definition of $N$. Hence,  for all $N\in\mathbb{Z}$, $\textsl{P}_{N}$ holds.
\medskip

Moreover, Conditions $(i)$ and $(ii)$ respectively imply that ${\mathbf A}_{r+s+1}(0)$ is an  invertible element of $\mathcal{A}$ and that ${\mathbf A}_{r+s+1}(m)\in {\mathbf g}_{r+s+1}(m)\mathcal{A}$. By \eqref{assert beta s}, we deduce that  
$$
{\mathbf S}_r(a,K+mp^s,s,p,m)\in p^{s+1}{\mathbf g}_{r+s+1}(m)\mathcal{A}. 
$$
The latter congruence holds for all $a\in\{0,\dots,p-1\}$, all $K\in\mathbb{Z}$ and all $m,r\in\mathbb{N}$, which proves that Assertion $\alpha_{s+1}$ holds, and finishes the induction on $s$. It remains to prove Lemmas \ref{Assertion 1}, \ref{Assertion 2}, \ref{iii 0} and \ref{Assertion 3}.

\subsubsection{Proof of Lemma \ref{Assertion 1}}

Let $a\in\{0,\dots,p-1\}$, $K\in\mathbb{Z}$ and $m,r\in\mathbb{N}$. We have
\begin{equation}\label{refref}
{\mathbf S}_r(a,K,0,p,m)={\mathbf A}_r\big(a+(K-m)p\big){\mathbf A}_{r+1}(m)-{\mathbf A}_{r+1}(K-m){\mathbf A}_r(a+mp).
\end{equation}
If $K-m\notin\mathbb{N}$, then ${\mathbf A}_r\big(a+(K-m)p\big)=0$ and ${\mathbf A}_{r+1}(K-m)=0$ so that ${\mathbf S}_r(a,K,0,p,m)=0\in p{\mathbf g}_{r+1}(m)\mathcal{A}$, as stated. We can thus assume that $K-m\in\mathbb{N}$. We write \eqref{refref} as follows:  
\begin{multline}
{\mathbf S}_r(a,K,0,p,m)={\mathbf A}_{r}(a)\Bigg({\mathbf A}_{r+1}(m)\bigg(\frac{{\mathbf A}_{r}\big(a+(K-m)p\big)}{{\mathbf A}_{r}(a)}-\frac{{\mathbf A}_{r+1}(K-m)}{{\mathbf A}_{r+1}(0)}\bigg)\label{der0}\\-{\mathbf A}_{r+1}(K-m)\bigg(\frac{{\mathbf A}_{r}(a+mp)}{{\mathbf A}_{r}(a)}-\frac{{\mathbf A}_{r+1}(m)}{{\mathbf A}_{r+1}(0)}\bigg)\Bigg).
\end{multline}
Since $\Psi_{\mathcal{N}}(r,0)=\{0\}$, we can use Hypothesis $(a)$ of Theorem  \ref{congruences formelles} with $0$ instead of $u$, and $a$ instead of $v$. We get this way
$$
\frac{{\mathbf A}_{r}\big(a+(K-m)p\big)}{{\mathbf A}_{r}(a)}-\frac{{\mathbf A}_{r+1}(K-m)}{{\mathbf A}_{r+1}(0)}\in p\frac{{\mathbf g}_{r+1}(K-m)}{{\mathbf A}_r(a)}\mathcal{A}
$$
and
$$
\frac{{\mathbf A}_{r}(a+mp)}{{\mathbf A}_{r}(a)}-\frac{{\mathbf A}_{r+1}(m)}{{\mathbf A}_{r+1}(0)}\in p\frac{{\mathbf g}_{r+1}(m)}{{\mathbf A}_r(a)}\mathcal{A}.
$$
Therefore,
\begin{align}
{\mathbf A}_r(a){\mathbf A}_{r+1}(m)\left(\frac{{\mathbf A}_{r}\big(a+(K-m)p\big)}{{\mathbf A}_{r}(a)}-\frac{{\mathbf A}_{r+1}(K-m)}{{\mathbf A}_{r+1}(0)}\right)
&\in p{\mathbf g}_{r+1}(K-m){\mathbf A}_{r+1}(m)\mathcal{A}\notag\\
&\in p{\mathbf g}_{r+1}(m)\mathcal{A}\label{der1}
\end{align}
and 
\begin{align}\label{der2}
{\mathbf A}_r(a){\mathbf A}_{r+1}(K-m)\left(\frac{{\mathbf A}_{r}(a+mp)}{{\mathbf A}_{r}(a)}-\frac{{\mathbf A}_{r+1}(m)}{{\mathbf A}_{r+1}(0)}\right)
&\in p{\mathbf g}_{r+1}(m){\mathbf A}_{r+1}(K-m)\mathcal{A}\notag\\
&\in p{\mathbf g}_{r+1}(m)\mathcal{A},
\end{align}
where we have used, in \eqref{der1}, Condition $(ii)$ that yields ${\mathbf A}_{r+1}(m)\in{\mathbf g}_{r+1}(m)\mathcal{A}$. Using \eqref{der1} and  \eqref{der2} in \eqref{der0}, we obtain ${\mathbf S}_r(a,K,0,p,m)\in p{\mathbf g}_{r+1}(m)\mathcal{A}$, as expected.

\subsubsection{Proof of Lemma \ref{Assertion 2}}
We have 
\begin{multline}\label{wesh wesh yo}
{\mathbf U}_r(a,K+mp^{s},p,j+mp^{s})-\frac{{\mathbf A}_{r+1}(j+mp^{s})}{{\mathbf A}_{r+1}(j)}{\mathbf U}_r(a,K,p,j)\\
=-{\mathbf A}_{r+1}(K-j){\mathbf A}_r(a+jp)\left(\frac{{\mathbf A}_r(a+jp+mp^{s+1})}{{\mathbf A}_r(a+jp)}-\frac{{\mathbf A}_{r+1}(j+mp^{s})}{{\mathbf A}_{r+1}(j)}\right).
\end{multline}
Since $j\in\Psi_{\mathcal{N}}(r,s)$, Hypothesis $(a)$ implies that the right hand side of  \eqref{wesh wesh yo} is in 
$$
{\mathbf A}_{r+1}(K-j){\mathbf A}_r(a+jp)p^{s+1}\frac{{\mathbf g}_{r+s+1}(m)}{{\mathbf A}_r(a+jp)}\mathcal{A}. 
$$
These estimates show that the left hand side of \eqref{wesh wesh yo} is in $p^{s+1}{\mathbf g}_{r+s+1}(m)\mathcal{A}$, which concludes the proof of the lemma.

\subsubsection{Proof of Lemma \ref{iii 0}}

We consider $s\in\mathbb{N}$, $s\geq 1$, such that $\alpha_s$ holds. We  fix $r\in\mathbb{N}$. If $\Psi_{\mathcal{N}}(r,s)=\{0,\dots,p^s-1\}$, Lemma~\ref{iii 0} is trivial. In the sequel, we assume that $\Psi_{\mathcal{N}}(r,s)\neq\{0,\dots,p^s-1\}$. 
\medskip

We have $u\in\{0,\dots,p^s-1\}\setminus\Psi_{\mathcal{N}}(r,s)$ if and only if there  exist  $(n,t)\in\mathcal{N}_{r+s-t+1}$, $t\leq s$, and $j\in\{0,\dots,p^{s-t}-1\}$ such that $u=j+p^{s-t}n$. We denote by $\mathcal{M}$ the set of the  $(n,t)\in\mathcal{N}_{r+s-t+1}$ with $t\leq s$. We thus have 
$$
\{0,\dots,p^s-1\}\setminus\Psi_{\mathcal{N}}(r,s)=\bigcup_{(n,t)\in\mathcal{M}}\big\{j+p^{s-t}n\,:\,0\leq j\leq p^{s-t}-1\big\}.
$$
In particular, the set $\mathcal{M}$ is non-empty.

We will show that there exist  $k\in\mathbb{N}$, $k\geq 1$, and $(n_1,t_1),\dots,(n_k,t_k)\in\mathcal{M}$ such that the sets
$$
J(n_i,t_i):=\big\{j+p^{s-t_i}n_i\,:\,0\leq j\leq p^{s-t_i}-1\big\}
$$ 
form a partition of $\{0,\dots,p^s-1\}\setminus\Psi_{\mathcal{N}}(r,s)$. We observe that  
$$
\mathcal{M}\subset\bigcup_{t=1}^s\Big(\big\{0,\dots,p^t-1\big\}\times\big\{t\big\}\Big)
$$ 
and thus $\mathcal{M}$ is finite. Hence, it is enough to show that if $(n,t),(n',t')\in\mathcal{M}$, $j\in\{0,\dots,p^{s-t}-1\}$ and $j'\in\{0,\dots,p^{s-t'}-1\}$ satisfy  $j+p^{s-t}n=j'+p^{s-t'}n'$, then we have either $J(n,t)\subset J(n',t')$ or $J(n',t')\subset J(n,t)$. 

Let us assume, for instance, that $t\leq t'$. Then there exists $j_0\in\{0,\dots,p^{t'-t}-1\}$ such that $j=j'+p^{s-t'}j_0$, so that $p^{s-t'}n'=p^{s-t}n+p^{s-t'}j_0$ and thus  $J(n',t')\subset J(n,t)$. Similarly, if $t\geq t'$, then $J(n,t)\subset J(n',t')$. Hence, we obtain
\begin{multline}
{\mathbf S}_r(a,K,s,p,m)=\\
\sum_{j\in\Psi_{\mathcal{N}}(r,s)}{\mathbf U}_r(a,K,p,j+mp^s)+\sum_{j\in\{0,\dots,p^s-1\}\setminus\Psi_{\mathcal{N}}(r,s)}{\mathbf U}_r(a,K,p,j+mp^s),\label{joint2}
\end{multline}
where
\begin{equation}\label{joint1}
\sum_{j\in\{0,\dots,p^s-1\}\setminus\Psi_{\mathcal{N}}(r,s)}{\mathbf U}_r(a,K,p,j+mp^s)=\sum_{i=1}^k\sum_{j=0}^{p^{s-t_i}-1}{\mathbf U}_r(a,K,p,j+p^{s-t_i}n_i+mp^s).
\end{equation}

We will now prove that for all $i\in\{1,\dots,k\}$, we have 
\begin{equation}\label{\'etoile 287}
\sum_{j=0}^{p^{s-t_i}-1}{\mathbf U}_r(a,K,p,j+p^{s-t_i}n_i+mp^s)\in p^{s+1}{\mathbf g}_{r+s+1}(m)\mathcal{A}.
\end{equation}
Let $i\in\{1,\dots,k\}$. By definition of $\mathbf{U}_r$, we have 
$$
\sum_{j=0}^{p^{s-t_i}-1}{\mathbf U}_r(a,K,p,j+p^{s-t_i}n_i+mp^s)={\mathbf S}_r(a,K,s-t_i,p,n_i+mp^{t_i}).
$$
Since $t_i\geq 1$, we get \textit{via} $\alpha_s$ that 
$$
{\mathbf S}_r(a,K,s-t_i,p,n_i+mp^{t_i})\in p^{s-t_i+1}{\mathbf g}_{r+s-t_i+1}(n_i+mp^{t_i})\mathcal{A}.
$$
We have  $(n_i,t_i)\in\mathcal{N}_{r+s-t_i+1}$ and thus we can apply Hypothesis $(b)$ of Theorem \ref{congruences formelles} with $r+s-t_i+1$ instead of $r$: 
$$
p^{s-t_i+1}{\mathbf g}_{r+s-t_i+1}(n_i+mp^{t_i})\in p^{s-t_i+1}p^{t_i}{\mathbf g}_{r+s+1}(m)\mathcal{Z}=p^{s+1}{\mathbf g}_{r+s+1}(m)\mathcal{Z}.
$$
It follows that for all $i\in\{1,\dots,k\}$, we have \eqref{\'etoile 287}. 
\medskip

Congruence \eqref{\'etoile 287}, together with \eqref{joint1} and \eqref{joint2}, shows that 
$$
{\mathbf S}_r(a,K,s,p,m)\equiv\sum_{j\in\Psi_{\mathcal{N}}(r,s)}{\mathbf U}_r(a,K,p,j+mp^s)\mod p^{s+1}{\mathbf g}_{r+s+1}(m)\mathcal{A},
$$
which completes the proof of Lemma \ref{iii 0}.

\subsubsection{Proof of Lemma \ref{Assertion 3}}

In this proof, $i$ is an element of $\{0,\dots,p-1\}$ and $u$ is an element of $\{0,\dots,p^{s-t-1}-1\}$. For $t<s$, we write $\beta_{t,s}$ as 
\begin{multline}
{\mathbf S}_r(a,K+mp^{s},s,p,m)\equiv\\
\sum_{i+up\in\Psi_{\mathcal{N}}(r+t,s-t)}\frac{{\mathbf A}_{r+t+1}(i+u p+mp^{s-t})}{{\mathbf A}_{r+t+1}(i+u p)}{\mathbf S}_r(a,K,t,p,i+u p)\mod p^{s+1}{\mathbf g}_{r+s+1}(m)\mathcal{A}.\label{beta t s}
\end{multline}
We want to prove the congruence $\beta_{t+1,s}$, which can be written
\begin{multline*}
{\mathbf S}_r(a,K+mp^{s},s,p,m)\equiv\\
\sum_{u\in\Psi_{\mathcal{N}}(r+t+1,s-t-1)}\frac{{\mathbf A}_{r+t+2}(u+mp^{s-t-1})}{{\mathbf A}_{r+t+2}(u)}{\mathbf S}_r(a,K,t+1,p,u)\mod p^{s+1}{\mathbf g}_{r+s+1}(m)\mathcal{A}.
\end{multline*}
We see that ${\mathbf S}_r(a,K,t+1,p,u)=\sum_{i=0}^{p-1}{\mathbf S}_r(a,K,t,p,i+u p)$. Hence, with 
\begin{multline*}
X:={\mathbf S}_r(a,K+mp^{s},s,p,m)\\
-\sum_{i=0}^{p-1}\sum_{u\in\Psi_{\mathcal{N}}(r+t+1,s-t-1)}\frac{{\mathbf A}_{r+t+2}(u+mp^{s-t-1})}{{\mathbf A}_{r+t+2}(u)}{\mathbf S}_r(a,K,t,p,i+u p),
\end{multline*}
it remains to show that $X\in p^{s+1}{\mathbf g}_{r+s+1}(m)\mathcal{A}$. We have 
\begin{equation}\label{Psi+1}
i+up\in\Psi_{\mathcal{N}}(r+t,s-t)\Rightarrow u\in\Psi_{\mathcal{N}}(r+t+1,s-t-1).
\end{equation}
Indeed if $u\notin\Psi_{\mathcal{N}}(r+t+1,s-t-1)$, then there exist $(n,k)\in\mathcal{N}_{r+s-k+1}$, $k\leq s-t-1$, and $j\in\{0,\dots,p^{s-t-1-k}-1\}$ such that $u=j+p^{s-t-1-k}n$. Hence, $i+up=i+jp+p^{s-t-k}n$, so that  $i+up\notin\Psi_{\mathcal{N}}(r+t,s-t)$. By $\beta_{t,s}$ in the form  \eqref{beta t s} and  modulo $p^{s+1}{\mathbf g}_{r+s+1}(m)\mathcal{A}$, we obtain 
\begin{multline*}
X\equiv \sum_{i+up\in\Psi_{\mathcal{N}}(r+t,s-t)}{\mathbf S}_r(a,K,t,p,i+up)\left(\frac{{\mathbf A}_{r+t+1}(i+up+mp^{s-t})}{{\mathbf A}_{r+t+1}(i+up)}-\frac{{\mathbf A}_{r+t+2}(u+mp^{s-t-1})}{{\mathbf A}_{r+t+2}(u)}\right)\\
-\underset{i+up\notin\Psi_{\mathcal{N}}(r+t,s-t)}{\sum_{u\in\Psi_{\mathcal{N}}(r+t+1,s-t-1)}}\frac{{\mathbf A}_{r+t+2}(u+mp^{s-t-1})}{{\mathbf A}_{r+t+2}(u)}{\mathbf S}_r(a,K,t,p,i+up).
\end{multline*}
But, by Hypothesis $(a_1)$ of Theorem \ref{congruences formelles} applied with $s-t-1$ for  $s$ and $r+t+1$ for $r$, we have 
$$
{\mathbf g}_{r+t+1}(i+up)\left(\frac{{\mathbf A}_{r+t+1}(i+up+mp^{s-t})}{{\mathbf A}_{r+t+1}(i+u p)}-\frac{{\mathbf A}_{r+t+2}(u+mp^{s-t-1})}{{\mathbf A}_{r+t+2}(u)}\right)\in p^{s-t}{\mathbf g}_{r+s+1}(m)\mathcal{A}.
$$
Moreover, since  $t<s$ and $\alpha_s$ holds, we have  
\begin{equation}\label{varii1}
{\mathbf S}_r(a,K,t,p,i+up)\in p^{t+1}{\mathbf g}_{r+t+1}(i+u p)\mathcal{A}
\end{equation}
and, modulo $p^{s+1}{\mathbf g}_{r+s+1}(m)\mathcal{A}$, we obtain 
\begin{equation}\label{varii2}
X\equiv -\underset{i+up\notin\Psi_{\mathcal{N}}(r+t,s-t)}{\sum_{u\in\Psi_{\mathcal{N}}(r+t+1,s-t-1)}}\frac{{\mathbf A}_{r+t+2}(u+mp^{s-t-1})}{{\mathbf A}_{r+t+2}(u)}{\mathbf S}_r(a,K,t,p,i+up).
\end{equation}

Finally, when $i+up\notin\Psi_{\mathcal{N}}(r+t,s-t)$, we can apply Condition $(a_2)$ of  Theorem~\ref{congruences formelles} with $s-t-1$ for $s$, $i$ for $v$ and $r+t+1$ for $r$, so that 
\begin{equation}\label{varii3}
{\mathbf g}_{r+t+1}(i+up)\frac{{\mathbf A}_{r+t+2}(u+mp^{s-t-1})}{{\mathbf A}_{r+t+2}(u)}\in p^{s-t}{\mathbf g}_{r+s+1}(m)\mathcal{A}.
\end{equation}
Using \eqref{varii1} and \eqref{varii3} in \eqref{varii2}, we thus have  $X\in p^{s+1}{\mathbf g}_{r+s+1}(m)\mathcal{A}$. This completes the proof of Lemma \ref{Assertion 3} and consequently that of Theorem \ref{congruences formelles}. 
\hfill$\square$

\section{Proof of Theorem \ref{theo expand}}\label{section proof theo expand}

The aim of this section if to prove Theorem  \ref{theo expand}. 
We will first prove some elementary properties of the algebras of functions  $\mathcal{A}_b$ and $\mathcal{A}_b^\ast$.

\subsection{Algebras of functions taking values into  $\mathbb{Z}_p$}

We gather in the following lemma a few properties of the algebras $\mathfrak{A}_{p,n}$ and $\mathfrak{A}_{p,n}^\ast$.

\begin{Lemma}\label{Lemma alg\`ebres 1}
We fix a prime  $p$ and $n\in\mathbb{N}$, $n\geq 1$.
\begin{enumerate}
\item An element $f$ of $\mathfrak{A}_{p,n}$, respectively of  $\mathfrak{A}_{p,n}^\ast$, is invertible in $\mathfrak{A}_{p,n}$, respectively in  $\mathfrak{A}_{p,n}^\ast$, if and only if  $f\big((\mathbb{Z}_p^\times)^n\big)\subset\mathbb{Z}_p^\times$;
\item the algebra $\mathfrak{A}_{p,n}$ contains the rational functions  
$$
\begin{array}{cccc}
f:&(\mathbb{Z}_p^{\times})^n & \rightarrow & \mathbb{Z}_p\\
  & (x_1,\dots,x_n) & \mapsto & \frac{P(x_1,\dots,x_n)}{Q(x_1,\dots,x_n)},
\end{array}
$$
where $P,Q\in\mathbb{Z}_p[X_1,\dots,X_n]$ and, for all $x_1,\dots,x_n\in\mathbb{Z}_p^{\times}$, we have $Q(x_1,\dots,x_n)\in\mathbb{Z}_p^{\times}$;
\item if $f\in\mathfrak{A}_{p,n}^{\times}$ and if $\mathfrak{E}_s$, $s\geq 1$, is the  function \textit{Euler quotient} defined by
$$
\begin{array}{cccc}
\mathfrak{E}_s: & \mathbb{Z}_p^{\times} & \rightarrow & \mathbb{Z}_p\\
 & x & \mapsto & \big(x^{\varphi(p^{s})}-1\big)/p^s,
\end{array}
$$
then we have $\mathfrak{E}_s\circ f\in\mathfrak{A}_{p,n}^\ast$.
\end{enumerate}
\end{Lemma}

\begin{proof}
Let $f\in\mathfrak{A}_{p,n}$. For $f$ to be invertible in $\mathfrak{A}_{p,n}$, we clearly need that $f\big((\mathbb{Z}_p^\times)^n\big)\subset\mathbb{Z}_p^\times$ and in this case, for all $\mathbf{x}\in(\mathbb{Z}_p^\times)^n$, all  $\mathbf{a}\in\mathbb{Z}_p^n$ and all $m\in\mathbb{N}$, $m\geq 1$, we have 
$$
\frac{1}{f(\mathbf{x}+\mathbf{a}p^m)}=\frac{1}{f(\mathbf{x})+\eta p^m}=\frac{1}{f(\mathbf{x})}\frac{1}{1+\frac{\eta}{f(\mathbf{x})}p^m}\equiv\frac{1}{f(\mathbf{x})}\mod p^m\mathbb{Z}_p,
$$
because $f(\mathbf{x})\in\mathbb{Z}_p^\times$, $\eta\in\mathbb{Z}_p$, and  $(1+p^m\mathbb{Z}_p,\times)$ is a group. The case $f\in\mathfrak{A}_{p,n}^\ast$ being similar, Assertion $(1)$ is proved.
\medskip

To prove Assertion $(2)$, we apply Assertion $(1)$ because any polynomial function $f:\mathbf{x}\in(\mathbb{Z}_p^\times)^n\mapsto P(\mathbf{x})$, with  $P\in\mathbb{Z}_p[X_1,\dots,X_n]$ is in $\mathfrak{A}_{p,n}$.
\medskip

Let us now prove Assertion $(3)$. For all $s\in\mathbb{N}$, $s\geq 1$, the cardinal of  $(\mathbb{Z}_p/p^s\mathbb{Z}_p)^{\times}$ is $\varphi(p^s)$ because $\mathbb{Z}_p/p^s\mathbb{Z}_p$ is isomorphic to $\mathbb{Z}/p^s\mathbb{Z}$. Hence, for all $x\in\mathbb{Z}_p^{\times}$, we have $x^{\varphi(p^s)}\equiv 1\mod p^s\mathbb{Z}_p$ and the function  $\mathfrak{E}_s$ is well defined. 

We fix $s\in\mathbb{N}$, $s\geq 1$. To prove Assertion $(3)$, it is enough to prove that for all $x\in\mathbb{Z}_p^{\times}$, all $a\in\mathbb{Z}_p$ and all $m\in\mathbb{N}$, $m\geq 1$, we have $\mathfrak{E}_s(x+ap^m)\equiv\mathfrak{E}_s(x)\mod p^{m-1}\mathbb{Z}_p$. We have 
\begin{align*}
(x+ap^m)^{\varphi(p^s)}&=\sum_{k=0}^{\varphi(p^s)}\binom{\varphi(p^s)}{k}\frac{a^k}{x^k}p^{km}x^{\varphi(p^s)}\\
&\equiv x^{\varphi(p^s)}+\sum_{k=1}^{\varphi(p^s)}\binom{\varphi(p^s)}{k}\frac{a^k}{x^k}p^{km}\mod p^{s+m}\mathbb{Z}_p,
\end{align*}
because $x^{\varphi(p^s)}\equiv 1\mod p^s\mathbb{Z}_p$. By a result of Kummer, the $p$-adic valuation of $\binom{\varphi(p^s)}{k}$ is the number of carries in the addition  of $k$ and $\varphi(p^s)-k$ in base $p$. Let us show that this number is equal to $s-1-v_p(k)$. 

Indeed, if $v_p(k)=0$, then this number is $s-1$ because $\varphi(p^s)=(p-1)p^{s-1}$. If $v_p(k)=\alpha\geq 1$, then we write $k=k'p^{\alpha}$ and $\varphi(p^s)-k=p^{\alpha}\big((p-1)p^{s-1-\alpha}-k'\big)$ with $v_p(k')=0$, so that the number of carries of the addition of $k$ and $\varphi(p^s)-k$ in base $p$ is the number of carries in the addition of $k'$ and $\varphi(p^{s-\alpha})-k'$, \textit{i.e.} $s-1-\alpha=s-1-v_p(k)$.

In particular, we obtain that, for all $k\geq 1$,  
$$
v_p\left(\binom{\varphi(p^s)}{k}\frac{a^k}{x^k}p^{km}\right)\geq s+m+(k-1)m-v_p(k)-1\geq s+m-1,
$$
hence $(x+ap^m)^{\varphi(p^s)}\equiv x^{\varphi(p^s)}\mod p^{s+m-1}\mathbb{Z}_p$. Consequently, we have  $\mathfrak{E}_s(x+ap^m)\equiv\mathfrak{E}_s(x)\mod p^{m-1}\mathbb{Z}_p$, and the proof of Lemma \ref{Lemma alg\`ebres 1} is complete.
\end{proof}

\begin{Lemma}\label{Lemma alg\`ebres 2}

Let $\nu,D\in\mathbb{N}$, $D\geq 1$, and $b\in\{1,\dots,D\}$, $\gcd(b,D)=1$.
\begin{enumerate}
\item We have $\mathcal{A}_b(p^\nu,D)\subset\mathcal{A}_b(p^\nu,D)^\ast$ and  $p\mathcal{A}_b(p^\nu,D)^\ast\subset\mathcal{A}_b(p^\nu,D)$;
\item An element $f$ of $\mathcal{A}_b(p^\nu,D)$, respectively of  $\mathcal{A}_b(p^\nu,D)^\ast$, is invertible in $\mathcal{A}_b(p^\nu,D)$, 
respectively in  $\mathcal{A}_b(p^\nu,D)^\ast$, if and only if  $f\big(\Omega_b(p^\nu,D)\big)\subset\mathbb{Z}_p^\times$;
\item Any constant function from $\Omega_b(p^\nu,D)$ into $\mathbb{Z}_p$ is in  $\mathcal{A}_b(p^\nu,D)$ ;
\item If $r\in\mathbb{N}$ and $\alpha\in\mathbb{Q}$ satisfy  $d(\alpha)=p^\mu D'$, with $1\leq\mu\leq \nu$ and $D'\mid D$, then the map $t\in\Omega_b(p^\nu,D)\mapsto d(\alpha)\underline{t^{(r)}\alpha}$ is in  $\mathcal{A}_b(p^\nu,D)^\times$;
\item If $\alpha\in\mathbb{Q}\cap\mathbb{Z}_p$ and $k\in\mathbb{N}$, then the map $t\in\Omega_b(p^\nu,D)\mapsto\varpi_{p^k}(t\alpha)$ is in $\mathcal{A}_b(p^\nu,D)$;
\item If $n\in\mathbb{N}$, $n\geq 1$, $f_1,\dots,f_n\in\mathcal{A}_b(p^\nu,D)^\times$, $g\in\mathfrak{A}_{p,n}$ and  $h\in\mathfrak{A}_{p,n}^\ast$, then  $g':=g\circ(f_1,\dots,f_n)\in\mathcal{A}_b(p^\nu,D)$ and $h':=h\circ(f_1,\dots,f_n)\in\mathcal{A}_b(p^\nu,D)^\ast$. Furthermore if $g$ is invertible in $\mathfrak{A}_{p,n}$, respectively $h$ is invertible in $\mathfrak{A}_{p,n}^\ast$, then $g'$ is invertible in $\mathcal{A}_b(p^\nu,D)$, respectively $h'$ is invertible in  $\mathcal{A}_b(p^\nu,D)^\ast$;
\item If $f\in\mathcal{A}_b(p^\nu,D)$ and $g\in\mathcal{A}_b(p^\nu,D)^\ast$, then 
$$
\sum_{t\in\Omega_b(p^\nu,D)}f(t)\in p^{\nu-1}\mathbb{Z}_p\quad\textup{and}\quad\sum_{t\in\Omega_b(p^\nu,D)}g(t)\in p^{\nu-2}\mathbb{Z}_p.
$$
\end{enumerate}
\end{Lemma}

\begin{proof}
Assertions $(1)$ and $(3)$ are obvious. The proof of Assertion $(2)$ is similar to that of Assertion $(2)$ of Lemma \ref{Lemma alg\`ebres 1}. 
\medskip

Let us prove Assertion $(4)$. For all $t\in\Omega_b(p^\nu,D)$, the number  $d(\alpha)\underline{t^{(r)}\alpha}$ is the numerator of $\langle t^{(r)}\alpha\rangle$ and thus it is in  $\mathbb{Z}_p^\times$ because $p$ divides $d(\alpha)$. 

Let $\alpha=\kappa/d(\alpha)$,  $t_1,t_2\in\Omega_b(p^\nu,D)$ and $m\in\mathbb{N}$, $m\geq 1$ be such that  $t_1\equiv t_2\mod p^m$. Since  $t_1\equiv t_2\equiv b\mod D$, we get $t_1^{(r)}\equiv t_2^{(r)}\mod D$. 

If $m\geq\mu$, then $t_1^{(r)}\equiv t_2^{(r)}\mod p^\mu$ and the chinese remainder theorem gives $t_1^{(r)}\equiv t_2^{(r)}\mod p^\mu D$. Since $D'\mid D$, we obtain  $t_1^{(r)}\kappa\equiv t_2^{(r)}\kappa\mod d(\alpha)$ and thus $d(\alpha)\underline{t_1^{(r)}\alpha}=d(\alpha)\underline{t_2^{(r)}\alpha}$, as expected.

On the other hand, if $m<\mu$, then $t_1^{(r)}\equiv t_2^{(r)}\mod p^m$. Since  $D'\mid D$ and $d(\alpha)\underline{t_i^{(r)}\alpha}\equiv t_i^{(r)}\kappa\mod d(\alpha)$ for $i\in\{1,2\}$, we obtain $d(\alpha)\underline{t_1\alpha}\equiv d(\alpha)\underline{t_2\alpha}\mod p^m$, which proves Assertion $(4)$.
\medskip

Assertion $(5)$ is obvious and Assertion $(6)$ is a direct consequence of the definitions and of Assertion $(2)$.
\medskip

Let us prove Assertion $(7)$ by induction on $\nu$ in the case $f\in\mathcal{A}_{b}(p^\nu,D)$. We denote by $A_\nu$ the assertion
$$
\sum_{t\in\Omega_b(p^\nu,D)}f(t)\in p^{\nu-1}\mathbb{Z}_p.
$$
Assertion $A_1$ trivially holds. Let $\nu\in\mathbb{N}$, $\nu\geq 1$ be such that $A_\nu$ holds. 

The set $\Omega_b(p^{\nu+1},D)$ is the set of the $t_{\ell,\nu+1}\in\{1,\dots,p^{\nu+1}D\}$ such that $t_{\ell,\nu+1}\equiv b\mod D$ and $t_{\ell,\nu+1}\equiv \ell\mod p^{\nu+1}$, with $\ell\in\{1,\dots,p^{\nu+1}\}$, $\gcd(\ell,p)=1$. Let $\ell:=u+vp^\nu$ with $u\in\{1,\dots,p^\nu\}$, $\gcd(u,p)=1$ and  $v\in\{0,\dots,p-1\}$. Then, we have $t_{\ell,\nu+1}\equiv u\mod p^\nu$ and  by the chinese remainder theorem, we obtain $t_{\ell,\nu+1}\equiv t_{u,\nu}\mod p^\nu D$, so that 
\begin{align*}
\sum_{t\in\Omega_b(p^{\nu+1},D)}f(t)&=\underset{\gcd(\ell,p)=1}{\sum_{\ell=1}^{p^{\nu+1}}}f(t_{\ell,\nu+1})=\underset{\gcd(u,p)=1}{\sum_{u=1}^{p^\nu}}\sum_{v=0}^{p-1}f(t_{u+vp^\nu,\nu+1})\\
&\equiv p\underset{\gcd(u,p)=1}{\sum_{u=1}^{p^\nu}}f(t_{u,\nu})\mod p^\nu\mathbb{Z}_p\\
&\equiv p\sum_{t\in\Omega_b(p^\nu,D)}f(t)\mod p^\nu\mathbb{Z}_p\\
&\equiv 0\mod p^\nu\mathbb{Z}_p,
\end{align*}
by Assertion $A_\nu$. Hence, Assertion $A_{\nu+1}$ holds, which completes the proof of  Asssertion $(7)$ when $f\in\mathcal{A}_b(p^\nu,D)$. The case  $f\in\mathcal{A}_b(p^\nu,D)^\ast$ is similar.
\end{proof}

\subsection{Proof of Theorem \ref{theo expand}}\label{subsection demo theo expand}

In this section, we fix two $r$-tuples $\boldsymbol{\alpha}$ and $\boldsymbol{\beta}$ with parameters in $\mathbb{Q}\setminus\mathbb{Z}_{\leq 0}$. We assume that  $\langle\boldsymbol{\alpha}\rangle$ and $\langle\boldsymbol{\beta}\rangle$ are disjoint and that $H_{\boldsymbol{\alpha},\boldsymbol{\beta}}$ holds.

We set $C=C_{\langle\boldsymbol{\alpha}\rangle,\langle\boldsymbol{\beta}\rangle}$, $C'=C_{\boldsymbol{\alpha},\boldsymbol{\beta}}'$, $\mathfrak{n}=\mathfrak{n}_{\boldsymbol{\alpha},\boldsymbol{\beta}}$, $\mathfrak{m}=\mathfrak{m}_{\boldsymbol{\alpha},\boldsymbol{\beta}}$ and  $\lambda_p=\lambda_p(\boldsymbol{\alpha},\boldsymbol{\beta})$. We write $d_{\boldsymbol{\alpha},\boldsymbol{\beta}}=p^\nu D$ with $\nu\geq 0$ and  $\gcd(D,p)=1$. For all $t\in\{1,\dots,d_{\boldsymbol{\alpha},\boldsymbol{\beta}}\}$ coprime to $d_{\boldsymbol{\alpha},\boldsymbol{\beta}}$ and all $r\in\mathbb{N}$, we recall that  $t^{(r)}$ is the unique element in $\{1,\dots,d_{\boldsymbol{\alpha},\boldsymbol{\beta}}\}$ coprime to $d_{\boldsymbol{\alpha},\boldsymbol{\beta}}$ such that $t^{(r)}\equiv t\mod p^\nu$ and $p^rt^{(r)}\equiv t\mod D$. 

We fix $b\in\{1,\dots,D\}$ coprime to $D$ and set $\Omega_{b}:=\Omega_b(p^\nu,D)$, $\mathcal{A}_{b}:=\mathcal{A}_b(p^\nu,D)$,  $\mathcal{A}_{b}^\ast:=\mathcal{A}_b(p^\nu,D)^\ast$. We recall that if $\nu=0$, then $\Omega_b=\{b\}$ and that $\mathcal{A}_b=\mathcal{A}_b^\ast$ is the algebra of  functions from $\{b\}$ into $\mathbb{Z}_p$.

For all $t\in\Omega_b$ and all $r,n\in\mathbb{N}$, we set 
$$
\mathcal{Q}_{r,t}(n):=(C')^n\frac{(\underline{t^{(r)}\boldsymbol{\alpha}})_n}{(\underline{t^{(r)}\boldsymbol{\beta}})_n}\quad\textup{and}\quad\mathcal{Q}_{r,\cdot}(n):=\big(t\in\Omega_b\mapsto\mathcal{Q}_{r,t}(n)\big). 
$$ 
For all $c\in\{1,\dots,p^{\nu}\}$ not divisible by $p$ and all $\ell\in\mathbb{N}$, $\ell\geq 1$, we fix a prime  $p_{c,\ell}$ such that $p_{c,\ell}\equiv p^{\ell}\mod D$ and $p_{c,\ell}\equiv c\mod p^{\nu}$. For all $t\in\Omega_b$ and all  $r\in\mathbb{N}$, we set  
$$
\Delta_{r,t}^{c,\ell}:=\Delta_{\underline{t^{(r)}\boldsymbol{\alpha}},\underline{t^{(r)}\boldsymbol{\beta}}}^{p_{c,\ell},1}.
$$ 
If $\widetilde{\boldsymbol{\alpha}}$, respectively $\widetilde{\boldsymbol{\beta}}$, is the sequence of elements of $\underline{t^{(r)}\boldsymbol{\alpha}}$, respectively of  $\underline{t^{(r)}\boldsymbol{\beta}}$, whose denominator is not divisible by $p$, then we set  $\widetilde{\Delta}_{r,t}^{p,\ell}:=\Delta_{\widetilde{\boldsymbol{\alpha}},\widetilde{\boldsymbol{\beta}}}^{p,\ell}$. We gather in the following lemma a few properties of the sequences $\mathcal{Q}_{r,\cdot}$. We set $\iota=1$ if  $\mathfrak{m}$ is odd and if  $\boldsymbol{\beta}\notin\mathbb{Z}^r$, and $\iota=0$ otherwise.

\begin{Lemma}\label{Lemma decomp Q}
For all $n,r\in\mathbb{N}$, there exists $\Lambda_{b,r}(n)\in\mathbb{Z}_p$ such that $\mathcal{Q}_{r,\cdot}(n)\in 2^{\iota n}\Lambda_{b,r}(n)\mathcal{A}_b^{\times}$, 
where
\begin{align*} v_p\big(\Lambda_{b,r}(n)\big)&=\sum_{\ell=1}^{\infty}\widetilde{\Delta}_{r,t}^{p,\ell}\left(\left\{\frac{n}{p^{\ell}}\right\}\right)-\lambda_p\frac{\mathfrak{s}_p(n)}{p-1}+n\left\{\frac{\lambda_p}{p-1}\right\}\\
&=\frac{1}{\varphi(p^{\nu})}\sum_{\ell=1}^{\infty}\underset{\gcd(c,p)=1}{\sum_{c=1}^{p^{\nu}}}\Delta_{r,t}^{c,\ell}\left(\left\{\frac{n}{p^{\ell}}\right\}\right)+n\left\{\frac{\lambda_p}{p-1}\right\}.
\end{align*}
If $p$ divides $d_{\boldsymbol{\alpha},\boldsymbol{\beta}}$, then for all $n,r\in\mathbb{N}$, $n\geq 1$, we have $v_p\big(\Lambda_{b,r}(n)\big)\geq 1$ and if  $\boldsymbol{\beta}\in\mathbb{Z}^r$ then 
$$
v_p\big(\Lambda_{b,r}(n)\big)\geq -\left\lfloor\frac{\lambda_p}{p-1}\right\rfloor.
$$
\end{Lemma}

\begin{proof}
For all $t\in\Omega_b$, we have $\mathcal{Q}_{r,t}(n)=2^{\iota n}\Lambda_{b,r}(n)\mathfrak{R}_{r}(n,t)$ with
$$
\Lambda_{b,r}(n):=\left(C\frac{\prod_{\beta_i\notin\mathbb{Z}_p}d(\beta_i)}{\prod_{\alpha_i\notin\mathbb{Z}_p}d(\alpha_i)}\right)^n
\frac{\prod_{\alpha_i\in\mathbb{Z}_p}(\underline{t^{(r)}\alpha_i})_n}{\prod_{\beta_i\in\mathbb{Z}_p}(\underline{t^{(r)}\beta_i})_n}
$$
and 
$$
\mathfrak{R}_{r}(n,t):=\frac{\prod_{\alpha_i\notin\mathbb{Z}_p}
d(\alpha_i)^n(\underline{t^{(r)}\alpha_i})_n}{\prod_{\beta_i\notin\mathbb{Z}_p}d(\beta_i)^n
(\underline{t^{(r)}\beta_i})_n}=\frac{\prod_{\alpha_i\notin\mathbb{Z}_p}
\prod_{k=0}^{n-1}\big(d(\alpha_i)\underline{t^{(r)}\alpha_i}+kd(\alpha_i)\big)}{\prod_{\beta_i\notin\mathbb{Z}_p}\prod_{k=0}^{n-1}\big(d(\beta_i)
\underline{t^{(r)}\beta_i}+kd(\beta_i)\big)}.
$$

By Assertions $(2)$ and $(4)$ of Lemma \ref{Lemma alg\`ebres 2}, we have  $\mathfrak{R}_r(n,\cdot)\in\mathcal{A}_b^{\times}$. Moreover if  $\alpha$ is a term  of the sequences $\boldsymbol{\alpha}$ or $\boldsymbol{\beta}$ whose denominator is not divisible by $p$, then $\underline{t^{(r)}\alpha}$ depends only of the class of  $t^{(r)}$ in $\mathbb{Z}/D\mathbb{Z}$ which is that of $\varpi_D(p^{-r}b)$ when  $t\in\Omega_b$. Indeed, if $\langle\alpha\rangle=1$, then $\underline{t^{(r)}\alpha}=1$ and if $\langle\alpha\rangle=k/N\neq 1$, where $N$ is a divisor of $D$, then $N\underline{t^{(r)}\alpha}=N\{t^{(r)}\langle\alpha\rangle\}=\varpi_N(t^{(r)}k)$. For all $t\in\Omega_b$ and all $r\in\mathbb{N}$, we have $p^{r}t^{(r)}\equiv b\mod D$, so that $\varpi_N(t^{(r)}k)=\varpi_N(bp^{-r}k)$. It follows that $\Lambda_{b,r}(n)$ depends only on  $b$, $r$ and $n$. By Proposition \ref{propo magie}, we have
\begin{align*}
v_p\big(\Lambda_{b,r}(n)\big)=v_p\left(C^n\frac{(\underline{t^{(r)}\boldsymbol{\alpha}})_n}{(\underline{t^{(r)}\boldsymbol{\beta}})_n}\right)&=\sum_{\ell=1}^{\infty}\widetilde{\Delta}_{r,t}^{p,\ell}\left(\left\{\frac{n}{p^{\ell}}\right\}\right)-\lambda_p\frac{\mathfrak{s}_p(n)}{p-1}+n\left\{\frac{\lambda_p}{p-1}\right\}\\
&=\frac{1}{\varphi(p^{\nu})}\sum_{\ell=1}^\infty\underset{\gcd(c,p)=1}{\sum_{c=1}^{p^{\nu}}}\Delta_{r,t}^{c,\ell}\left(\left\{\frac{n}{p^{\ell}}\right\}\right)+n\left\{\frac{\lambda_p}{p-1}\right\}.
\end{align*}

In the sequel, we assume that $p$ divides $d_{\boldsymbol{\alpha},\boldsymbol{\beta}}$. Let us now show that if $n\geq 1$, then  $v_p\big(\Lambda_{b,r}(n)\big)\geq 1$. Let $\alpha$ be a term of the sequences   $\underline{t^{(r)}\boldsymbol{\alpha}}$ or $\underline{t^{(r)}\boldsymbol{\beta}}$ whose denominator is divisible by $p$.  By \eqref{omega fin}, the number of elements  $\mathfrak{D}_{p_{c,\ell}}(\alpha)$, $\ell\geq 1$, $c\in\{1,\dots,p^\nu\}$, $\gcd(c,p)=1$, that satisfy $\{n/p^{\ell}\}\geq\mathfrak{D}_{p_{c,\ell}}(\alpha)$ is equal to $\varphi(p^\nu)\mathfrak{s}_p(n)/(p-1)$. In particular, if $n\geq 1$, then there  exist at least one $\ell\geq 1$ and one $c\in\{1,\dots,p^\nu\}$, $\gcd(c,p)=1$, such that  $\{n/p^\ell\}\geq\mathfrak{D}_{p_{c,\ell}}(\alpha)$. 

Thus, there exists one term $\alpha'\in(0,1)$ of the sequence $\underline{t^{(r)}\boldsymbol{\alpha}}$ or $\underline{t^{(r)}\boldsymbol{\beta}}$ such that $\Delta_{r,t}^{c,\ell}\big(\{n/p^\ell\}\big)=\Delta_{r,t}^{c,\ell}\big(\mathfrak{D}_{p_{c,\ell}}(\alpha')\big)$. By Lemma \ref{valeurs Delta}, we obtain  $\Delta_{r,t}^{c,\ell}\big(\{n/p^\ell\}\big)=\xi_{\langle t^{(r)}\boldsymbol{\alpha}\rangle,\langle t^{(r)}\boldsymbol{\beta}\rangle}(a,a\alpha')$, where  $a\in\{1,\dots,d_{\boldsymbol{\alpha},\boldsymbol{\beta}}\}$ satisfies $p_{c,\ell}a\equiv 1\mod d_{\boldsymbol{\alpha},\boldsymbol{\beta}}$. Since $\alpha'\notin\mathbb{Z}$, we have $m_{\boldsymbol{\alpha},\boldsymbol{\beta}}(a)\preceq a\alpha'\prec a$ and by Lemma \ref{H transfert}, Assertion $H_{\langle t^{(r)}\boldsymbol{\alpha}\rangle,\langle t^{(r)}\boldsymbol{\beta}\rangle}$ holds, so that  $\Delta_{r,t}^{c,\ell}\big(\{n/p^\ell\}\big)\geq 1$. Hence,  $v_p\big(\Lambda_{b,r}(n)\big)\geq 1$. 

Moreover, if $\boldsymbol{\beta}\in\mathbb{Z}^r$, then $\lambda_p\leq -1$ and the functions $\widetilde{\Delta}_{r,t}^{p,\ell}$ are positive on $[0,1)$. It follows that  
$$
v_p\big(\Lambda_{b,r}(n)\big)\geq-\lambda_p\frac{\mathfrak{s}_p(n)}{p-1}+n\left\{\frac{\lambda_p}{p-1}\right\}\geq-\frac{\lambda_p}{p-1}+\left\{\frac{\lambda_p}{p-1}\right\}\geq-\left\lfloor\frac{\lambda_p}{p-1}\right\rfloor.
$$
This completes the proof of Lemma \ref{Lemma decomp Q}.
\end{proof}

In the sequel, we set $\mathcal{K}_b:=\mathcal{A}_b^\ast$ if  $p$ does not divide $d_{\boldsymbol{\alpha},\boldsymbol{\beta}}$. If $p$ divides $d_{\boldsymbol{\alpha},\boldsymbol{\beta}}$, we set 
$$
\mathcal{K}_b:=\begin{cases}
p^{-1-\lfloor\lambda_p/(p-1)\rfloor}\mathcal{A}_b\textup{ if $\boldsymbol{\beta}\in\mathbb{Z}^r$;}\\
\mathcal{A}_b\textup{ if $\boldsymbol{\beta}\notin\mathbb{Z}^r$, $\mathfrak{m}$ is odd and  $p=2$;}\\
\mathcal{A}_b\textup{ if $\boldsymbol{\beta}\notin\mathbb{Z}^r$ and  $p-1\nmid\lambda_p$;}\\ 
\mathcal{A}_b^\ast \textup{ otherwise.}
\end{cases}.
$$
By Lemma \ref{Lemma decomp Q}, for all $r\in\mathbb{N}$,  
$$
\big(t\in\Omega_b\mapsto F_{\underline{t^{(r)}\boldsymbol{\alpha}},\underline{t^{(r)}\boldsymbol{\beta}}}(C'z)\big)\in 1+z\mathcal{K}_b[[z]]
$$ 
is an invertible formal series in $\mathcal{K}_b[[z]]$. Hence, to prove Theorem~\ref{theo expand}, it is enough to prove that the function
\begin{equation}\label{eq form}
t\in\Omega_b\mapsto G_{\underline{t^{(1)}\boldsymbol{\alpha}},\underline{t^{(1)}\boldsymbol{\beta}}}(C'z^p)F_{\underline{t\boldsymbol{\alpha}},\underline{t\boldsymbol{\beta}}}(C'z)-pG_{\underline{t\boldsymbol{\alpha}},\underline{t\boldsymbol{\beta}}}(C'z)F_{\underline{t^{(1)}\boldsymbol{\alpha}},\underline{t^{(1)}\boldsymbol{\beta}}}(C'z^p)
\end{equation}
is in $p\mathcal{K}_b[[z]]$.
\medskip

For all $a\in\{0,\dots,p-1\}$ and all  $K\in\mathbb{N}$, the $a+Kp$-th coefficient of the formal series \eqref{eq form} is 
$$
t\in\Omega_b\mapsto\Phi_t(a+Kp):=\sum_{i=1}^r\big(\Phi_{\alpha_i,t}(a+Kp)-\Phi_{\beta_i,t}(a+Kp)\big),
$$
where 
$$
\Phi_{\alpha,t}(a+Kp):=\sum_{j=0}^K\mathcal{Q}_{0,t}(a+jp)\mathcal{Q}_{1,t}(K-j)(H_{\underline{t^{(1)}\alpha}}(K-j)-pH_{\underline{t\alpha}}(a+jp)\big).
$$
It is sufficient to show that, for all terms $\alpha$ of the sequences $\boldsymbol{\alpha}$ and $\boldsymbol{\beta}$, for all $a\in\{0,\dots,p-1\}$ and all $K\in\mathbb{N}$, we have  
\begin{equation}\label{eq form suite}
\Phi_{\alpha,\cdot}(a+Kp)\in p\mathcal{K}_b.
\end{equation}
If $a+Kp=0$, then  $\Phi_{\alpha,\cdot}(0)$ is obviously the null map. In the sequel, we assume that $a+Kp\neq 0$, so that for all $j\in\{0,\dots,K\}$, we have $a+jp\geq 1$ or $K-j\geq 1$.

\medskip

If $p$ divides $d_{\boldsymbol{\alpha},\boldsymbol{\beta}}$ and if $\alpha$ is a term of the sequences $\boldsymbol{\alpha}$ or $\boldsymbol{\beta}$ whose denominator is  divisible by $p$, then for all $n,r\in\mathbb{N}$ and all $t\in\Omega_b$, we have 
$$
H_{\underline{t^{(r)}\alpha}}(n)=\sum_{k=0}^{n-1}\frac{d(\alpha)}{d(\alpha)(\underline{t^{(r)}\alpha}+k)},
$$
yielding $\big(t\in\Omega_b\mapsto H_{\underline{t^{(r)}\alpha}}(n)\big)\in p\mathcal{A}_b$. By Lemma \ref{Lemma decomp Q}, for all $n,r\in\mathbb{N}$, $n\geq 1$, we have $\mathcal{Q}_{r,\cdot}(n)\in\Lambda_{b,r}(n)\mathcal{A}_b$ with 
$$
\Lambda_{b,r}(n)\in\begin{cases}
p^{-\lfloor\lambda_p/(p-1)\rfloor}\mathbb{Z}_p\textup{ if  $\boldsymbol{\beta}\in\mathbb{Z}^r$;}\\
p\mathbb{Z}_p\textup{ otherwise.}
\end{cases}. 
$$
Hence, we have $\big(t\in\Omega_b\mapsto\Phi_{\alpha,t}(a+Kp)\big)\in p^2\mathcal{K}_b\subset p\mathcal{K}_b$, as expected.
\medskip

It remains to deal with the case when the denominator of  $\alpha$ is not divisible by  $p$. We fix an element  $\alpha\in\mathbb{Z}_p$ of the sequences  $\boldsymbol{\alpha}$ or $\boldsymbol{\beta}$ in the proof of \eqref{eq form suite}. We recall that $\underline{t\alpha}$ is independent of  $t\in\Omega_b$ because  $\alpha\in\mathbb{Z}_p$. By \cite[Lemma $4.1$]{Dwork}, for all  $j\in\{0,\dots,K\}$, we have  
$$
pH_{\underline{t\alpha}}(a+jp)\equiv pH_{\underline{t\alpha}}(jp)+\frac{\rho(a,\underline{t\alpha})}{\mathfrak{D}_p(\underline{t\alpha})+j}\mod p\mathbb{Z}_p,
$$
where we recall that for all $x\in\mathbb{Q}\cap\mathbb{Z}_p$, we have 
$$
\rho(a,x)=\begin{cases}\textup{$0$ if $a\leq p\mathfrak{D}_p(x)-x$;}\\ \textup{$1$ if $a>p\mathfrak{D}_p(x)-x$.}\end{cases}.
$$
Moreover,   
$$
H_{\underline{t\alpha}}(jp)=\sum_{k=0}^{jp-1}\frac{1}{\underline{t\alpha}+k}=\frac{1}{p}\sum_{k=0}^{j-1}\frac{1}{\mathfrak{D}_p(\underline{t\alpha})+k}+\underset{i\neq p\mathfrak{D}_p(\underline{t\alpha})-\underline{t\alpha}}{\sum_{i=0}^{p-1}}\sum_{k=0}^{j-1}\frac{1}{\underline{t\alpha}+i+kp},
$$
so that $pH_{\underline{t\alpha}}(jp)\equiv H_{\mathfrak{D}_p(\underline{t\alpha})}(j)\mod p\mathbb{Z}_p$. Writing $\langle\alpha\rangle=k/N$ as an irreducible fraction, we obtain 
\begin{equation}\label{transition}
\mathfrak{D}_p(\underline{t\alpha})=\frac{\varpi_N(Np^{-1}\underline{t\alpha})}{N}=\frac{\varpi_N\big(p^{-1}\varpi_N(bk)\big)}{N}=\frac{\varpi_N(p^{-1}bk)}{N}=\underline{t^{(1)}\alpha}.
\end{equation}
Hence, 
\begin{equation}\label{result Harmo}
pH_{\underline{t\alpha}}(a+jp)\equiv H_{\underline{t^{(1)}\alpha}}(j)+\frac{\rho(a,\underline{t\alpha})}{\mathfrak{D}_p(\underline{t\alpha})+j}\mod p\mathbb{Z}_p.
\end{equation}
We now use the following fact \eqref{double eq}, to be proved in Section~\ref{demo double eq}. For all $j\in\{0,\dots,K\}$, we have 
\begin{equation}\label{double eq}
\left(t\in\Omega_b\mapsto\frac{\rho(a,\underline{t\alpha})}{\mathfrak{D}_p(\underline{t\alpha})+j}\mathcal{Q}_{0,t}(a+jp)\mathcal{Q}_{1,t}(K-j)\right)\in p\mathcal{K}_b.
\end{equation}

The notation $f(t)\equiv g(t)\mod I$ where $I$ is an ideal of  $\mathcal{A}_b$ means that there exists $h\in I$ such that $f-g=h|^{\mathbb{Q}_p}$ with $f:t\in\Omega_b\mapsto f(t)\in\mathbb{Q}_p$ and $g:t\in\Omega_b\mapsto g(t)\in\mathbb{Q}_p$. Using \eqref{result Harmo} and \eqref{double eq} in the definition of $\Phi_{\alpha,\cdot}(a+Kp)$, we obtain  
\begin{align*}
\Phi_{\alpha,t}(a+Kp)&\equiv\sum_{j=0}^K\mathcal{Q}_{0,t}(a+jp)\mathcal{Q}_{1,t}(K-j)\big(H_{\underline{t^{(1)}\alpha}}(K-j)-H_{\underline{t^{(1)}\alpha}}(j)\big)\\
&\equiv-\sum_{j=0}^KH_{\underline{t^{(1)}\alpha}}(j)\Big(\mathcal{Q}_{0,t}(a+jp)\mathcal{Q}_{1,t}(K-j)-\mathcal{Q}_{0,t}\big(a+(K-j)p\big)\mathcal{Q}_{1,t}(j)\Big),
\end{align*}
modulo $p\mathcal{K}_b$.

\subsubsection{Proof of Equation \eqref{double eq}}\label{demo double eq}

For this, we prove several results that will be used again in the proof of Theorem \ref{theo expand}.

\begin{Lemma}\label{H reduc}
Let $a\in\{0,\dots,p-1\}$, $m\in\mathbb{N}$ and $x \in \mathbb Z_p\cap \mathbb{Q}\cap(0,1]$. If $\rho(a,x)=1$, then for all $\ell\in\big\{1,\dots,1+v_p\big(\mathfrak{D}_p(x)+m\big)\big\}$, we have  $\big\{(a+mp)/p^{\ell}\big\}\geq\mathfrak{D}_p^{\ell}(x)$.
\end{Lemma}

\begin{proof}
We write $m=\sum_{j=0}^{\infty}m_jp^j$ with $m_j\in\{0,\dots,p-1\}$ and we fix some $\ell$ in $\big\{1,\dots,1+v_p\big(\mathfrak{D}_p(x)+m\big)\big\}$. Then, 
$$
\left\{\frac{a+mp}{p^{\ell}}\right\}=\frac{a+p\sum_{j=0}^{\ell-2}m_jp^j}{p^{\ell}}.
$$
We have $\mathfrak{D}_p(x)+m\in p^{\ell-1}\mathbb{Z}_p$ and thus 
$$
\mathfrak{D}_p(x)+m-\sum_{j=\ell-1}^{\infty}m_jp^j\in p^{\ell-1}\mathbb{Z}_p,
$$
so that
$$
p\mathfrak{D}_p(x)+p\sum_{j=0}^{\ell-2}m_jp^j-p^{\ell}\mathfrak{D}_p^{\ell}(x)\in p^{\ell}\mathbb{Z},
$$
because $p\mathfrak{D}_p(x)-p^{\ell}\mathfrak{D}_p^{\ell}(x)\in\mathbb{Z}$. We obtain 
$$
\frac{p\mathfrak{D}_p(x)+p\sum_{j=0}^{\ell-2}m_jp^j}{p^{\ell}}-\mathfrak{D}_p^{\ell}(x)\in\mathbb{Z}.
$$
Moreover $\mathfrak{D}_p^{\ell}(x)\in(0,1]$ and 
$$
0<\frac{p\mathfrak{D}_p(x)+p\sum_{j=0}^{\ell-2}m_jp^j}{p^{\ell}}\leq\frac{p+p(p^{\ell-1}-1)}{p^{\ell}}\leq 1,
$$
so that
$$
\frac{p\mathfrak{D}_p(x)+p\sum_{j=0}^{\ell-2}m_jp^j}{p^{\ell}}-\mathfrak{D}_p^{\ell}(x)=0.
$$
We have  $\rho(a,x)=1$ hence $a>p\mathfrak{D}_p(x)-x$ \textit{i.e.} $a\geq p\mathfrak{D}_p(x)-x+1$ and $a\geq p\mathfrak{D}_p(x)$. It follows that
$$
\frac{a+p\sum_{j=0}^{\ell-2}m_jp^j}{p^{\ell}}\geq \mathfrak{D}_p^{\ell}(x).
$$
\end{proof}

For all $c\in\{1,\dots,p^{\nu}\}$ not divisible by $p$ and all $\ell,r\in\mathbb{N}$, $\ell\geq 1$, we define $\tau(r,\ell)$ as the smallest of the numbers  $\mathfrak{D}_{p_{c,\ell}}\big(\underline{t^{(r)}\alpha}\big)$, where $\alpha$ runs through the elements of the sequences  $\boldsymbol{\alpha}$ and $\boldsymbol{\beta}$ whose denominator is not divisible by $p$. Since $\underline{t^{(r)}\alpha}\in\mathbb{Z}_p$, the number $\mathfrak{D}_{p_{c,\ell}}(\underline{t^{(r)}\alpha})$ does not depend on $c$ and thus $\tau(r,\ell)$ neither. Moreover, since $\alpha\in\mathbb{Z}_p$, the rational number $\underline{t^{(r)}\alpha}$ does not depend on $t\in\Omega_b$ and thus $\tau(r,\ell)$ neither. We define $\mathbf{1}_{r,\ell}$ as the characteristic function of the interval $\big[\tau(r,\ell),1\big)$. For all $m,r\in\mathbb{N}$, we set 
$$
\mu_r(m):=\sum_{\ell=1}^{\infty}\mathbf{1}_{r,\ell}\left(\left\{\frac{m}{p^{\ell}}\right\}\right)\in\mathbb{N}\quad\textup{and}\quad g_r(m):=p^{\mu_r(m)}.
$$
Similarly, the function $g_r$ does not depend on $t\in\Omega_b$.

\begin{Lemma}\label{Lambda vs mu}

Let $r,\ell,n\in\mathbb{N}$, $\ell\geq 1$, be such that $\{n/p^{\ell}\}\geq\tau(r,\ell)$. Then for all $t\in\Omega_b$ and all $c\in\{1,\dots,p^\nu\}$ not divisible by $p$, we have  $\Delta_{r,t}^{c,\ell}\big(\{n/p^{\ell}\}\big)\geq 1$. In particular for all $n\in\mathbb{N}$, we have  
$$
v_p\big(\Lambda_{b,r}(n)\big)\geq v_p\big(g_r(n)\big)+n\left\{\frac{\lambda_p}{p-1}\right\}.
$$
If $\boldsymbol{\beta}\in\mathbb{Z}^r$, then for all $n\in\mathbb{N}$, $n\geq 1$, we have 
$$
v_p\big(\Lambda_{b,r}(n)\big)\geq v_p\big(g_r(n)\big)-\left\lfloor\frac{\lambda_p}{p-1}\right\rfloor.
$$
\end{Lemma}

\begin{proof} Let $r,\ell,n\in\mathbb{N}$, $\ell\geq 1$, such that $\{n/p^{\ell}\}\geq\tau(r,\ell)$. Let $c\in\{1,\dots,p^\nu\}$ not divisible by $p$. There exists an element $\alpha_c$ of the sequences  $\boldsymbol{\alpha}$ or $\boldsymbol{\beta}$ such that  $\Delta_{r,t}^{c,\ell}\big(\{n/p^\ell\}\big)=\Delta_{r,t}^{c,\ell}\big(\mathfrak{D}_{p_{c,\ell}}(\underline{t^{(r)}\alpha_c})\big)$ with $\mathfrak{D}_{p_{c,\ell}}(\underline{t^{(r)}\alpha_c})\leq\{n/p^\ell\}<1$. Hence   $\underline{t^{(r)}\alpha_c}<1$. By Lemma \ref{valeurs Delta}, we obtain 
$$
\Delta_{r,t}^{c,\ell}\left(\left\{\frac{n}{p^\ell}\right\}\right)=\xi_{\underline{t^{(r)}\boldsymbol{\alpha}},\underline{t^{(r)}\boldsymbol{\beta}}}\big(a,a\underline{t^{(r)}\alpha_c}\big),
$$
where  $a\in\{1,\dots,d_{\boldsymbol{\alpha},\boldsymbol{\beta}}\}$ satisfies $p_{c,\ell}a\equiv 1\mod d_{\boldsymbol{\alpha},\boldsymbol{\beta}}$. We also have  $m_{\underline{t^{(r)}\boldsymbol{\alpha}},\underline{t^{(r)}\boldsymbol{\beta}}}(a)\preceq a\underline{t^{(r)}\alpha_c}\prec a$ and by Lemma \ref{H transfert}, Assertion $H_{\langle t^{(r)}\boldsymbol{\alpha}\rangle,\langle t^{(r)}\boldsymbol{\beta}\rangle}$ holds. Hence, $\Delta_{r,t}^{c,\ell}\big(\{n/p^\ell\}\big)\geq 1$. By Lemma \ref{Lemma decomp Q}, we have 
$$
v_p\big(\Lambda_{b,r}(n)\big)=\frac{1}{\varphi(p^\nu)}\sum_{\ell=1}^{\infty}\underset{\gcd(c,p)=1}{\sum_{c=1}^{p^\nu}}\Delta_{r,t}^{c,\ell}\left(\left\{\frac{n}{p^\ell}\right\}\right)+n\left\{\frac{\lambda_p}{p-1}\right\},
$$ 
so that 
$$
v_p\big(\Lambda_{b,r}(n)\big)\geq v_p\big(g_r(n)\big)+n\left\{\frac{\lambda_p}{p-1}\right\}.
$$

Let us now assume that $\boldsymbol{\beta}\in\mathbb{Z}^r$. If we have  $1>\{n/p^\ell\}\geq\tau(r,\ell)$, then there exists an element $\alpha$ of  $\boldsymbol{\alpha}$ whose denominator is not divisible by $p$ and such that $\{n/p^{\ell}\}\geq\mathfrak{D}_{p_{c,\ell}}(\underline{t^{(r)}\alpha})$ for some $c\in\{1,\dots,p^\nu\}$ not divisible by $p$. The denominator of  $\underline{t^{(r)}\alpha}$ divides $D$ and $p_{c,\ell}\equiv p^\ell\mod D$ hence we have  $\mathfrak{D}_{p_{c,\ell}}(\underline{t^{(r)}\alpha})=\mathfrak{D}_p^\ell(\underline{t^{(r)}\alpha})$, which yields $\widetilde{\Delta}_{r,t}^{p,\ell}\big(\{n/p^\ell\}\big)\geq 1$. By Lemma \ref{Lemma decomp Q}, for all $n\in\mathbb{N}$, $n\geq 1$, we have 
\begin{align*}
v_p\big(\Lambda_{b,r}(n)\big)&=\sum_{\ell=1}^{\infty}\widetilde{\Delta}_{r,t}^{p,\ell}\left(\left\{\frac{n}{p^\ell}\right\}\right)-\lambda_p\frac{\mathfrak{s}_p(n)}{p-1}+n\left\{\frac{\lambda_p}{p-1}\right\}\\
&\geq \mu_r(n)-\frac{\lambda_p}{p-1}+\left\{\frac{\lambda_p}{p-1}\right\}\geq v_p\big(g_r(n)\big)+\left\lfloor\frac{\lambda_p}{p-1}\right\rfloor,
\end{align*}
because  $\lambda_p\leq 0$. This proves Lemma \ref{Lambda vs mu}.
\end{proof}

We are now in position to prove \eqref{double eq}.

\begin{proof}[Proof of  \eqref{double eq}]

If $\rho(a,\underline{t\alpha})=0$ then \eqref{double eq} holds. We can thus assume that $\rho(a,\underline{t\alpha})=1$, \textit{i.e.} that $a>p\mathfrak{D}_p(\underline{t\alpha})-\underline{t\alpha}$. In particular, we have  $\underline{t\alpha}<1$ and $a\geq 1$. For all $j\in\{0,\dots,K\}$, we have $a+jp\geq 1$ hence by Lemma \ref{Lambda vs mu}, 
$$
\mathcal{Q}_{0,\cdot}(a+jp)\in g_0(a+jp)\mathcal{K}_b.
$$
It follows that it is sufficient to show that
\begin{equation}\label{double eq with g}
\frac{\rho(a,\underline{t\alpha})}{\mathfrak{D}_p(\underline{t\alpha})+j}g_0(a+jp)\in p\mathbb{Z}_p.
\end{equation}

By Lemma \ref{H reduc} with $\underline{t\alpha}$ instead of  $x$ and $j$ instead $m$, we obtain, for all $j\in\{0,\dots,K\}$ and all $\ell\in\big\{1,\dots,1+v_p\big(\mathfrak{D}_p(\underline{t\alpha})+j\big)\big\}$, that  $\big\{(a+jp)/p^\ell\big\}\geq\mathfrak{D}_p^{\ell}\big(\underline{t\alpha}\big)=\mathfrak{D}_{p_{c,\ell}}(\underline{t\alpha})$ because $\underline{t\alpha}\in\mathbb{Z}_p$. We obtain $\big\{(a+jp)/p^\ell\big\}\geq\tau(0,\ell)$, thus 
$$
v_p\big(g_0(a+jp)\big)=\sum_{\ell=1}^\infty\mathbf{1}_{r,\ell}\left(\left\{\frac{a+jp}{p^\ell}\right\}\right)\geq v_p\big(\mathfrak{D}_p(\langle t\alpha\rangle)+j\big)+1,
$$ 
and this completes the proof of \eqref{double eq with g} and also that of \eqref{double eq}.
\end{proof}

\subsubsection{A combinatorial lemma}

We  now use a combinatorial identity due to Dwork (see \cite[Lemma $4.2$, p. $308$]{Dwork}) that enables us to write
\begin{multline*}
\sum_{j=0}^KH_{\underline{t^{(1)}\alpha}}(j)\Big(\mathcal{Q}_{0,t}(a+jp)\mathcal{Q}_{1,t}(K-j)-\mathcal{Q}_{1,t}(j)\mathcal{Q}_{0,t}\big(a+(K-j)p\big)\Big)\\
=\sum_{s=0}^r\sum_{m=0}^{p^{r+1-s}-1}W_t(a,K,s,p,m),
\end{multline*}
where $r$ is such that $K<p^r$, 
$$
W_t(a,K,s,p,m):=\big(H_{\underline{t^{(1)}\alpha}}(mp^s)-H_{\underline{t^{(1)}\alpha}}(\lfloor m/p\rfloor p^{s+1})\big)S_t(a,K,s,p,m) 
$$
and
$$
S_t(a,K,s,p,m)=\sum_{j=mp^s}^{(m+1)p^s-1}\Big(\mathcal{Q}_{0,t}(a+jp)\mathcal{Q}_{1,t}(K-j)-\mathcal{Q}_{1,t}(j)\mathcal{Q}_{0,t}\big(a+(K-j)p\big)\Big),
$$
where, for all $r\in\mathbb{N}$, we set $\mathcal{Q}_{r,t}(n)=0$ if $n<0$. Thus, to complete the proof, it is enough to show that for all $s,m\in\mathbb{N}$, we have  $\big(t\in\Omega_b\mapsto W_t(a,K,s,p,m)\big)\in p\mathcal{K}_b$. If $m=0$, this is obvious. We now assumes that $m\geq 1$.
\medskip

We write $m=k+qp$ with $k\in\{0,\dots,p-1\}$ and $q\in\mathbb{N}$, so that  $mp^s=kp^s+qp^{s+1}$ and $\lfloor m/p\rfloor p^{s+1}=qp^{s+1}$. By \cite[Lemma $4.1$]{Dwork}, we obtain
$$
H_{\underline{t^{(1)}\alpha}}(mp^s)-H_{\underline{t^{(1)}\alpha}}(\lfloor m/p\rfloor p^{s+1})\equiv\frac{1}{p^{s+1}}\frac{\rho\big(k,\mathfrak{D}_p^s(\underline{t^{(1)}\alpha})\big)}{\mathfrak{D}_p^{s+1}(\underline{t^{(1)}\alpha})+q}\mod\frac{1}{p^s}\mathbb{Z}_p.
$$
Let us show that for all $s,m\in\mathbb{N}$, $m\geq 1$, we have 
\begin{equation}\label{H vs g}
g_{s+1}(m)\frac{\rho\big(k,\mathfrak{D}_p^s(\underline{t^{(1)}\alpha})\big)}{\mathfrak{D}_p^{s+1}(\underline{t^{(1)}\alpha})+q}\in p\mathbb{Z}_p.
\end{equation}
If $\rho\big(k,\mathfrak{D}_p^s(\underline{t^{(1)}\alpha})\big)=0$, this is clear. Let us assume that $\rho\big(k,\mathfrak{D}_p^s(\underline{t^{(1)}\alpha})\big)=1$. Since $\underline{t^{(1)}\alpha}\in\mathbb{Z}_p$, Eq. \eqref{transition} yields $\mathfrak{D}_p^s(\underline{t^{(1)}\alpha})=\underline{t^{(s+1)}\alpha}$ and  $\mathfrak{D}_p^{s+1}(\underline{t^{(1)}\alpha})=\mathfrak{D}_p(\underline{t^{(s+1)}\alpha})$.

Using Lemma \ref{H reduc} with $\underline{t^{(s+1)}\alpha}$ for $x$, $k$ for $a$ and $q$ for $m$, we get that, for all  $\ell\in\big\{1,\dots,1+v_p\big(\mathfrak{D}_p(\underline{t^{(s+1)}\alpha})+q\big)\big\}$, we have  $\{m/p^\ell\}\geq\mathfrak{D}_p^{\ell}(\underline{t^{(s+1)}\alpha})=\mathfrak{D}_{p_{c,\ell}}(\underline{t^{(s+1)}\alpha})$ because $\underline{t^{(s+1)}\alpha}\in\mathbb{Z}_p$. We obtain $\{m/p^\ell\}\geq\tau(s+1,\ell)$ and 
$$
v_p\big(g_{s+1}(m)\big)=\sum_{\ell=1}^\infty\mathbf{1}_{s+1,\ell}\left(\left\{\frac{m}{p^\ell}\right\}\right)\geq v_p\big(\mathfrak{D}_p(\langle t^{(s+1)}\alpha\rangle)+q)+1,
$$ 
which finishes the proof of \eqref{H vs g}.
\medskip

By \eqref{H vs g}, for all $s,m\in\mathbb{N}$, $m\geq 1$, we have 
$$
\big(H_{\underline{t^{(1)}\alpha}}(mp^s)-H_{\underline{t^{(1)}\alpha}}(\lfloor m/p\rfloor p^{s+1})\big)p^{s+1}g_{s+1}(m)\in p\mathbb{Z}_p.
$$
Hence, to complete the proof of Theorem \ref{theo expand}, it is enough to show that, for all $s,m\in\mathbb{N}$, $m\geq 1$, we have 
\begin{equation}\label{application pour S}
\big(t\in\Omega_b\mapsto S_t(a,K,s,p,m)\big)\in p^{s+1}g_{s+1}(m)\mathcal{K}_b.
\end{equation}
We do this in the next section.

\subsubsection{Application of Theorem \ref{congruences formelles}}

To prove \eqref{application pour S}, we will use Theorem \ref{congruences formelles} with the ring  $\mathbb{Z}_p$ for $\mathcal{Z}$ and the $\mathbb{Z}_p$-algebra $\mathcal{A}$ defined as follows:
\begin{itemize}
\item $\mathcal{A}=\mathcal{A}_b$ if ($\boldsymbol{\beta}\in\mathbb{Z}^r$ or $p-1\nmid\lambda_p$) or if ($p=2$ and $\mathfrak{m}$ is odd); 
\item $\mathcal{A}=\mathcal{A}_b^\ast$ otherwise.
\end{itemize}
A map $f\in\mathcal{A}_b^\ast$ is regular if and only if for all $t\in\Omega_b$, we have  $f(t)\neq 0$. Moreover, we have $\mathcal{A}_b\subset\mathcal{A}_b^\ast$.

In particular, by Lemma \ref{Lemma decomp Q} and Assertion $(2)$ of Lemma \ref{Lemma alg\`ebres 2}, for all $r,m\in\mathbb{N}$, the map $\mathcal{Q}_{r,\cdot}(m)$ is a regular element of $\mathcal{A}_b$. In the sequel, for all $r,m\in\mathbb{N}$, we set $\mathbf{A}_r(m):=\mathcal{Q}_{r,\cdot}(m)$ and we define a function $\mathbf{g}_r$ as follows:
\begin{itemize}
\item If $\boldsymbol{\beta}\in\mathbb{Z}^r$ and $p\mid d_{\boldsymbol{\alpha},\boldsymbol{\beta}}$, then $\mathbf{g}_r(0)=1$ and $\mathbf{g}_r(m)=\frac{1}{p}\Lambda_{b,r}(m)$ for $m\geq 1$;
\item If $\boldsymbol{\beta}\notin\mathbb{Z}^r$ or $p\nmid d_{\boldsymbol{\alpha},\boldsymbol{\beta}}$, then $\mathbf{g}_r=g_r$. 
\end{itemize}
We recall that if $m\geq 1$ and if $p\mid d_{\boldsymbol{\alpha},\boldsymbol{\beta}}$, then for all $r\in\mathbb{N}$, we have $\Lambda_{b,r}(m)\in p\mathbb{Z}_p$. Hence,   the maps $\mathbf{g}_r$ take their values in $\mathbb{Z}_p$.

We will show in the next sections that the sequences $\big(\mathbf{A}_r\big)_{r\geq 0}$ and $(\mathbf{g}_r)_{r\geq 0}$ satisfy Hypothesis $(i)$, $(ii)$ and $(iii)$ of Theorem \ref{congruences formelles}. Thus, for all $m,s\in\mathbb{N}$, $m\geq 1$, we will obtain that
$$
S_\cdot(a,K,s,p,m)\in
\begin{cases}p^s\Lambda_{b,s+1}(m)\mathcal{A}_b\textup{ if $\boldsymbol{\beta}\in\mathbb{Z}^r$ and $p\mid d_{\boldsymbol{\alpha},\boldsymbol{\beta}}$;}\\
p^{s+1}g_{s+1}(m)\mathcal{A}_b\textup{ if $\boldsymbol{\beta}\notin\mathbb{Z}^r$ and $p-1\nmid\lambda_p$;}\\
p^{s+1}g_{s+1}(m)\mathcal{A}_b\textup{ if $\boldsymbol{\beta}\notin\mathbb{Z}^r$, $p=2$ and $\mathfrak{m}$ is odd;}\\
p^{s+1}g_{s+1}(m)\mathcal{A}_b^\ast\textup{ otherwise.}
\end{cases}
$$
because if $p\nmid d_{\boldsymbol{\alpha},\boldsymbol{\beta}}$, then  $\mathcal{A}_b=\mathcal{A}_b^\ast$.

Proceeding this way, we will obtain \eqref{application pour S}. Indeed, the only non-obvious case is the one for which $\boldsymbol{\beta}\in\mathbb{Z}^r$ and $p\mid d_{\boldsymbol{\alpha},\boldsymbol{\beta}}$. But in this case, by Lemma \ref{Lambda vs mu}, we have
$$
p^s\Lambda_{b,s+1}(m)\mathcal{A}_b\in p^{s+1}p^{-1-\left\lfloor\frac{\lambda_p}{p-1}\right\rfloor}g_{s+1}(m)\mathcal{A}_b=p^{s+1}g_{s+1}(m)\mathcal{K}_b.
$$

In the next sections, we check the various hypothesis of Theorem \ref{congruences formelles}.

\subsubsection{Verification of Conditions $(i)$ and $(ii)$ of Theorem \ref{congruences formelles}}\label{verif 1 2}

For all $r\geq 0$, the map $\mathcal{Q}_{r,\cdot}(0)$ is constant on $\Omega_b$ with  value $1$, and thus it is invertible in $\mathcal{A}_b$. 
\medskip

By Lemmas \ref{Lemma decomp Q} and \ref{Lambda vs mu}, for all $m\in\mathbb{N}$, we have $\mathcal{Q}_{r,\cdot}(m)\in g_r(m)\mathcal{A}_b$ and $\mathcal{Q}_{r,\cdot}(m)\in \Lambda_{b,r}(m)\mathcal{A}_b$ so that in all these cases we have $\mathcal{Q}_{r,\cdot}(m)\in \mathbf{g}_r(m)\mathcal{A}_b$. This shows that Conditions $(i)$ and $(ii)$ of Theorem \ref{congruences formelles} hold.

\subsubsection{Verification of Condition $(iii)$ of Theorem \ref{congruences formelles}}

For all $r\in\mathbb{N}$, we set
$$
\mathcal{N}_r:=\bigcup_{t\geq 1}\bigg(\Big\{n\in\{0,\dots,p^t-1\}\,:\,\forall\ell\in\{1,\dots,t\},\,\{n/p^{\ell}\}\geq\tau(r,\ell)\Big\}\times\big\{t\big\}\bigg).
$$
We apply Theorem \ref{congruences formelles} with the sequence  $\mathcal{N}:=(\mathcal{N}_r)_{r\geq 0}$. We observe that for all $r,\ell\in\mathbb{N}$, $\ell\geq 1$, we have $\tau(r,\ell)>0$ and hence if  $(n,t)\in\mathcal{N}_r$, then $n\geq 1$. Moreover, in the sequel, we will often use the fact that for all $h\in\mathbb{N}$, all $c\in\{1,\dots,p^\nu\}$ not divisible by $p$ and all $t\in\Omega_b$, we have  
\begin{equation}\label{\'echange r l}
\tau(r,\ell+h)=\tau(r+h,\ell),\quad\widetilde{\Delta}_{r,t}^{p,\ell+h}=\widetilde{\Delta}_{r+h,t}^{p,\ell}\quad\textup{and}\quad\Delta_{r,t}^{c,\ell+h}=\Delta_{r+h,t}^{c,\ell}. 
\end{equation}
Indeed, let $\alpha$ be a term of the sequences $\boldsymbol{\alpha}$ or $\boldsymbol{\beta}$. Writing $\langle\alpha\rangle=k/N$ as an irreducible fraction, we obtain
$$
\mathfrak{D}_{p_{c,\ell+h}}(\langle t^{(r)}\alpha\rangle)=\frac{\varpi_N(p_{c,\ell+h}^{-1}t^{(r)}k)}{N}=\frac{\varpi_N(p_{c,\ell}^{-1}t^{(r+h)}k)}{N}=\mathfrak{D}_{p_{c,\ell}}(\langle t^{(r+h)}\alpha\rangle),
$$
so that $\tau(r,\ell+h)=\tau(r+h,\ell)$ and $\Delta_{r,t}^{c,\ell+h}=\Delta_{r+h,t}^{c,\ell}$. Furthermore, if $\alpha\in\mathbb{Z}_p$, then by  \eqref{transition}, we have 
$$
\mathfrak{D}_p^{\ell+h}(\underline{t^{(r)}\alpha})=\mathfrak{D}_p^{\ell}\big(\mathfrak{D}_p^h(\underline{t^{(r)}\alpha})\big)=\mathfrak{D}_p^{\ell}(\underline{t^{(r+h)}\alpha}),
$$
which yields $\widetilde{\Delta}_{r,t}^{p,\ell+h}=\widetilde{\Delta}_{r+h,t}^{p,\ell}$.

\subsubsection{Verification of Condition $(b)$ of Theorem \ref{congruences formelles}}

Let $r,m\in\mathbb{N}$ and $(n,u)\in\mathcal{N}_r$. We want to show that $\mathbf{g}_r(n+mp^{u})\in p^{u}\mathbf{g}_{r+u}(m)\mathbb{Z}_p$. 
We need to distinguish two cases.
\medskip

$\bullet$ If $\boldsymbol{\beta}\in\mathbb{Z}^r$ and $p\mid d_{\boldsymbol{\alpha},\boldsymbol{\beta}}$, then
\begin{multline*}
v_p\big(\Lambda_{b,r}(n+mp^u)\big)=\sum_{\ell=1}^{\infty}\widetilde{\Delta}^{p,\ell}_{r,t}\left(\left\{\frac{n+mp^{u}}{p^{\ell}}\right\}\right)-\lambda_p\frac{\mathfrak{s}_p(n+mp^u)}{p-1}+(n+mp^u)\left\{\frac{\lambda_p}{p-1}\right\}\\
> \sum_{\ell=1}^u\widetilde{\Delta}^{p,\ell}_{r,t}\left(\left\{\frac{n}{p^{\ell}}\right\}\right)+\sum_{\ell=u+1}^{\infty}\widetilde{\Delta}^{p,\ell}_{r,t}\left(\left\{\frac{n+mp^u}{p^{\ell}}\right\}\right)-\lambda_p\frac{\mathfrak{s}_p(m)}{p-1}+m\left\{\frac{\lambda_p}{p-1}\right\},
\end{multline*}
because $\lambda_p\leq -1$ and $n\geq 1$. Since $(n,u)\in\mathcal{N}_r$, for all $\ell\in\{1,\dots,u\}$, we have $\{n/p^\ell\}\geq\tau(r,\ell)$ and thus 
$$
v_p\big(\Lambda_{b,r}(n+mp^u)\big)> u+\sum_{\ell=u+1}^{\infty}\widetilde{\Delta}^{p,\ell}_{r,t}\left(\left\{\frac{n+mp^u}{p^{\ell}}\right\}\right)-\lambda_p\frac{\mathfrak{s}_p(m)}{p-1}+m\left\{\frac{\lambda_p}{p-1}\right\}.
$$
We set $m=\sum_{k=0}^{\infty}m_kp^k$, where $m_k\in\{0,\dots,p-1\}$ is $0$ for all but a finite number of $k$'s. For all $\ell\geq u+1$, 
$$
\left\{\frac{n+mp^u}{p^{\ell}}\right\}=\frac{n+p^u\left(\sum_{k=0}^{\ell-u-1}m_kp^k\right)}{p^{\ell}}\geq\frac{p^u\left(\sum_{k=0}^{\ell-u-1}m_kp^k\right)}{p^{\ell}}=\left\{\frac{m}{p^{\ell-u}}\right\}.
$$
Moreover, since $\langle\boldsymbol{\beta}\rangle=(1,\dots,1)$, the map $\widetilde{\Delta}^{p,\ell}_{r,t}$ is non-decreasing on $[0,1)$ and we obtain that
$$
v_p\big(\Lambda_{b,r}(n+mp^u)\big)> u+\sum_{\ell=u+1}^{\infty}\widetilde{\Delta}^{p,\ell}_{r,t}\left(\left\{\frac{m}{p^{\ell-u}}\right\}\right)-\lambda_p\frac{\mathfrak{s}_p(m)}{p-1}+m\left\{\frac{\lambda_p}{p-1}\right\}.
$$
But we have
$$
\sum_{\ell=u+1}^{\infty}\widetilde{\Delta}^{p,\ell}_{r,t}\left(\left\{\frac{m}{p^{\ell-u}}\right\}\right)=
\sum_{\ell=1}^{\infty}\widetilde{\Delta}^{p,\ell+u}_{r,t}\left(\left\{\frac{m}{p^{\ell}}\right\}\right)=
\sum_{\ell=1}^{\infty}\widetilde{\Delta}^{p,\ell}_{r+u,t}\left(\left\{\frac{m}{p^{\ell}}\right\}\right),
$$
which yields $v_p\big(\Lambda_{b,r}(n+mp^u)\big)> u+v_p\big(\Lambda_{b,r+u}(m)\big)$ and thus  
$$
v_p\big(\Lambda_{b,r}(n+mp^u)\big)\geq u+v_p\big(\Lambda_{b,r+u}(m)\big)+1.
$$
Since $n\geq 1$, we have $\mathbf{g}_r(n+mp^u)=\frac{1}{p}\Lambda_{b,r}(n+mp^u)$ and we obtain 
$$
v_p\big(\mathbf{g}_r(n+mp^u)\big)\geq u+v_p\big(\Lambda_{b,r+u}(m)\big)\geq u+v_p\big(\mathbf{g}_{r+u}(m)\big),
$$
as expected.
\medskip

$\bullet$ If $\boldsymbol{\beta}\notin\mathbb{Z}^r$ or $p\nmid d_{\boldsymbol{\alpha},\boldsymbol{\beta}}$, then we have to show that $g_r(n+mp^u)\in p^ug_{r+u}(m)\mathbb{Z}_p$. We have
\begin{align}
v_p\big(g_r(n+mp^{u})\big)&=\sum_{\ell= 1}^{\infty}{\mathbf 1}_{r,\ell}\left(\left\{\frac{n+mp^{u}}{p^{\ell}}\right\}\right)\notag\\
&\geq \sum_{\ell=1}^u{\mathbf 1}_{r,\ell}\left(\left\{\frac{n}{p^{\ell}}\right\}\right)+\sum_{\ell=u+1}^{\infty}{\mathbf 1}_{r,\ell}\left(\left\{\frac{n+mp^u}{p^{\ell}}\right\}\right)\notag\\
&\geq u+\sum_{\ell=u+1}^{\infty}{\mathbf 1}_{r,\ell}\left(\left\{\frac{n+mp^u}{p^{\ell}}\right\}\right)\label{rrr2},
\end{align}
because $(n,u)\in\mathcal{N}_r$. Hence, for all $\ell\in\{1,\dots,u\}$, we have $\{n/p^\ell\}\geq\tau(r,\ell)$. Furthermore, for all $h\in\mathbb{N}$, we have  $\tau(r,\ell+h)=\tau(r+h,\ell)$ and consequently
$$
\sum_{\ell=u+1}^{\infty}{\mathbf 1}_{r,\ell}\left(\left\{\frac{n+p^um}{p^{\ell}}\right\}\right)\geq\sum_{\ell=u+1}^{\infty}{\mathbf 1}_{r,\ell}\left(\left\{\frac{m}{p^{\ell-u}}\right\}\right)=\sum_{\ell=1}^{\infty}{\mathbf 1}_{r+u,\ell}\left(\left\{\frac{m}{p^{\ell}}\right\}\right)=v_p\big(g_{r+u}(m)\big).
$$
Together with \eqref{rrr2}, we obtain 
$g_r(n+mp^u)\in p^ug_{r+u}(m)\mathbb{Z}_p$.

\subsubsection{Verification of Condition $(a_2)$ of Theorem \ref{congruences formelles}}

Let $r,s,m\in\mathbb{N}$, $u\in\Psi_{\mathcal{N}}(r,s)$ and $v\in\{0,\dots,p-1\}$ be such that $v+up\notin\Psi_{\mathcal{N}}(r-1,s+1)$. It is enough to show that
\begin{equation}\label{montr a2}
\mathbf{g}_r(v+up)\frac{\mathcal{Q}_{r+1,\cdot}(u+mp^s)}{\mathcal{Q}_{r+1,\cdot}(u)}\in p^{s+1}\mathbf{g}_{r+s+1}(m)\mathcal{A}_b.
\end{equation}

We will first provide a few important properties concerning the set 
$\Psi_{\mathcal{N}}(r,s)$.

\begin{Lemma}\label{prop Psi}
Let $r\in\mathbb{Z}$, $r\geq -1$ and $s\in\mathbb{N}$. Then $\Psi_{\mathcal{N}}(r,s)$ is the set of the  $u\in\{0,\dots,p^s-1\}$ such that $\{u/p^s\}<\tau(r+1,s)$. Moreover, for all $u\in\Psi_{\mathcal{N}}(r,s)$ and all $\ell\geq s$, we have  $\{u/p^{\ell}\}<\tau(r+1,\ell)$ and, for all $m\in\mathbb{N}$, we have
$$
\frac{\mathcal{Q}_{r+1,\cdot}(u+mp^s)}{\mathcal{Q}_{r+1,\cdot}(u)}\in 2^{\iota mp^s}p^{\left\{\frac{\lambda_p}{p-1}\right\}m(p^s-1)}\Lambda_{b,r+s+1}(m)\mathcal{A}_b.
$$
\end{Lemma} 

By Lemma \ref{prop Psi}, to show \eqref{montr a2} and thus to complete the verification of Condition $(a_2)$, it is enough to show that $v_p\big(\mathbf{g}_r(v+up)\big)\geq s+1$. 
\medskip

We have $v+up\notin\Psi_{\mathcal{N}}(r-1,s+1)$ hence there exist $(n,t)\in\mathcal{N}_{r+s-t+1}$, $t\leq s+1$ and $j\in\{0,\dots,p^{s+1-t}-1\}$ such that  $v+up=j+p^{s+1-t}n$. Since $u\in\Psi_{\mathcal{N}}(r,s)$, we necessarily have  $s+1-t=0$, so that $(v+up,s+1)\in\mathcal{N}_{r}$, \textit{i.e.}, for all $\ell\in\{1,\dots,s+1\}$, we have $\big\{(v+up)/p^\ell\big\}\geq\tau(r,\ell)$ and thus $g_{r}(v+up)\in p^{s+1}\mathbb{Z}_p$. Furthermore, if $\boldsymbol{\beta}\in\mathbb{Z}^r$ and $p\mid d_{\boldsymbol{\alpha},\boldsymbol{\beta}}$, then since $v+up\geq 1$, we have $\mathbf{g}_r(v+up)=\frac{1}{p}\Lambda_{b,r}(v+up)\mathbb{Z}_p$ and by Lemma \ref{Lambda vs mu}, we obtain
$$
v_p\big(\mathbf{g}_r(v+up)\big)\geq v_p\big(g_r(v+up)\big)-1-\left\lfloor\frac{\lambda_p}{p-1}\right\rfloor\geq s+1,
$$  
because $\lambda_p\leq -1$. This completes the verification 
modulo Lemma \ref{prop Psi}.

\begin{proof}[Proof of Lemma \ref{prop Psi}]

We first show that $\Psi_{\mathcal{N}}(r,s)$ is the set of the  $u\in\{0,\dots,p^s-1\}$ such that $\{u/p^s\}<\tau(r+1,s)$. If $s=0$, then $\Psi_{\mathcal{N}}(r,0)=\{0\}$ and  $\tau(r+1,0)>0$ thus this is obvious. We can then assume that $s\geq 1$. Let $u\in\{0,\dots,p^s-1\}$, $u=\sum_{k=0}^{s-1}u_kp^k$, with $u_k\in\{0,\dots,p-1\}$. It is sufficient to prove that the following assertions are equivalent.
\begin{enumerate}
\item We have  $\left\{u/p^s\right\}\geq\tau(r+1,s)$.
\item There exist $(n,t)\in\mathcal{N}_{r+s-t+1},t\leq s$ and $j\in\{0,\dots,p^{s-t}-1\}$ such that $u=j+p^{s-t}n$.
\end{enumerate}

Proof of $(2)\Rightarrow(1)$: we have 
$$
\left\{\frac{u}{p^s}\right\}=\frac{u}{p^s}=\frac{j+p^{s-t}n}{p^s}\geq\frac{n}{p^t}=\left\{\frac{n}{p^t}\right\}.
$$
Moreover, by definition of the sequence $\mathcal{N}$, we have  $\{n/p^t\}\geq\tau(r+s-t+1,t)=\tau(r+1,s)$ and hence  $\left\{u/p^s\right\}\geq\tau(r+1,s)$. 
\medskip

Proof of $(1)\Rightarrow(2)$: for all $s\geq 1$, we denote by $\mathcal{B}_s$ the assertion : ``For all $u\in\{0,\dots,p^s-1\}$ and all $r\in\mathbb{N}$ such that $\{u/p^s\}\geq\tau(r+1,s)$, there exists $i\in\{0,\dots,s-1\}$ such that $\left(\sum_{k=i}^{s-1}u_kp^{k-i},s-i\right)\in\mathcal{N}_{r+i+1}$.''
\medskip
It is enough to show by induction on $s$ that for all $s\geq 1$, $\mathcal{B}_s$ holds. 
\medskip

If $s=1$, then, for all $u\in\{0,\dots,p-1\}$ and all $r\in\mathbb{N}$ such that $\{u/p\}\geq\tau(r+1,1)$, we have $(u,1)\in\mathcal{N}_{r+1}$. Hence,  $\mathcal{B}_1$ holds. 
\medskip

Let $s\geq 2$ be such that $\mathcal{B}_1,\dots,\mathcal{B}_{s-1}$ hold, let $u\in\{0,\dots,p^s-1\}$ and $r\in\mathbb{N}$ be such that $\{u/p^s\}\geq\tau(r+1,s)$. We further assume that $(u,s)\notin\mathcal{N}_{r+1}$. Hence, there exists $\ell\in\{1,\dots,s\}$ such that 
$$
a_{\ell}:=\frac{\sum_{k=0}^{\ell-1}u_kp^k}{p^{\ell}}=\left\{\frac{u}{p^{\ell}}\right\}<\tau(r+1,\ell).
$$ 
We necessarily have $\ell\in\{1,\dots,s-1\}$. We write
$$
\left\{\frac{u}{p^s}\right\}=\frac{u}{p^s}=\frac{p^{\ell}a_{\ell}+p^{\ell}\sum_{k=\ell}^{s-1}u_kp^{k-\ell}}{p^{s}}=\frac{a_{\ell}}{p^{s-\ell}}+\frac{\sum_{k=\ell}^{s-1}u_kp^{k-\ell}}{p^{s-\ell}}.
$$
Since $\{u/p^s\}\geq\tau(r+1,s)$, we obtain that
$$
\sum_{k=\ell}^{s-1}u_kp^{k-\ell}\geq p^{s-\ell}\tau(r+1,s)-a_{\ell}>p^{s-\ell}\tau(r+1,s)-\tau(r+1,\ell),
$$
or 
$$
\sum_{k=\ell}^{s-1}u_kp^{k-\ell}>p^{s-\ell}\tau(r+\ell+1,s-\ell)-\tau(r+\ell+1,0).
$$
Let $\alpha$ be an element of the sequences $\widetilde{\boldsymbol{\alpha}}$ or  $\widetilde{\boldsymbol{\beta}}$ such that $\tau(r+\ell+1,s-\ell)=\mathfrak{D}_{p}^{s-\ell}(\underline{t^{(r+\ell+1)}\alpha})$. Then, we have $\tau(r+\ell+1,0)\leq \underline{t^{(r+\ell+1)}\alpha}$ and thus
\begin{equation}\label{ineg entier}
\sum_{k=\ell}^{s-1}u_kp^{k-\ell}>p^{s-\ell}\mathfrak{D}_{p}^{s-\ell}(\underline{t^{(r+\ell+1)}\alpha})-\underline{t^{(r+\ell+1)}\alpha}.
\end{equation}
Both sides of inequality \eqref{ineg entier} are integers, so that 
$$
\sum_{k=\ell}^{s-1}u_kp^{k-\ell}\geq p^{s-\ell}\mathfrak{D}_{p}^{s-\ell}(\underline{t^{(r+\ell+1)}\alpha})-\underline{t^{(r+\ell+1)}\alpha}+1\geq p^{s-\ell}\mathfrak{D}_{p}^{s-\ell}(\underline{t^{(r+\ell+1)}\alpha}).
$$
It follows that 
$$
\frac{\sum_{k=\ell}^{s-1}u_kp^{k-\ell}}{p^{s-\ell}}\geq\mathfrak{D}_{p}^{s-\ell}(\underline{t^{(r+\ell+1)}\alpha})=\tau(r+\ell+1,s-\ell).
$$
By $\mathcal{B}_{s-\ell}$, there exists  $i\in\{0,\dots,s-\ell-1\}$ such that $\left(\sum_{k=\ell+i}^{s-1}u_kp^{k-\ell-i},s-\ell-i\right)\in\mathcal{N}_{r+\ell+i+1}$. Hence there exists $j\in\{\ell,\dots,s-1\}$ such that $\left(\sum_{k=j}^{s-1}u_kp^{k-j},s-j\right)\in\mathcal{N}_{r+j+1}$, which proves Assertion $\mathcal{B}_s$ and finishes the induction on $s$.
\medskip

The equivalence of Assertions $(1)$ and $(2)$ is now proved and we have 
$$
\Psi_{\mathcal{N}}(r,s)=\big\{u\in\{0,\dots,p^s-1\}\,:\,\{u/p^s\}<\tau(r+1,s)\big\}.
$$
Let $u\in\Psi_{\mathcal{N}}(r,s)$. Let us prove that for all $\ell\geq s$, we have $\{u/p^\ell\}<\tau(r+1,\ell)$.
\medskip

To  get a contradiction,  let us assume that there exists  $\ell\geq s$ such that $\{u/p^\ell\}\geq\tau(r+1,\ell)$. Let $\alpha$ be an element of the sequences  $\widetilde{\boldsymbol{\alpha}}$ or $\widetilde{\boldsymbol{\beta}}$ such that $\tau(r+1,\ell)=\mathfrak{D}_p^{\ell}(\underline{t^{(r+1)}\alpha})$. We obtain that \begin{align*}
\left\{\frac{u}{p^s}\right\}=p^{\ell-s}\left\{\frac{u}{p^\ell}\right\}&\geq p^{\ell-s}\mathfrak{D}_p^{\ell}(\underline{t^{(r+1)}\alpha})\\
&\geq p^{\ell-s}\mathfrak{D}_p^{\ell}(\underline{t^{(r+1)}\alpha})-\mathfrak{D}_p^{s}(\underline{t^{(r+1)}\alpha})+\mathfrak{D}_p^{s}(\underline{t^{(r+1)}\alpha})\\
&\geq\mathfrak{D}_p^{s}(\underline{t^{(r+1)}\alpha})\\
&\geq\tau(r+1,s),
\end{align*}
which is a contradiction. Hence, for all $\ell\geq s$, we have $\{u/p^\ell\}<\tau(r+1,\ell)$.
\medskip

To complete the proof of Lemma \ref{prop Psi}, it remains to prove that, for all $u\in\Psi_{\mathcal{N}}(r,s)$ and all $m\in\mathbb{N}$, we have 
\begin{equation}\label{end Lemma prop Psi}
\frac{\mathcal{Q}_{r+1,\cdot}(u+mp^s)}{\mathcal{Q}_{r+1,\cdot}(u)}\in 2^{\iota mp^s}p^{\left\{\frac{\lambda_p}{p-1}\right\}m(p^s-1)}\Lambda_{b,r+s+1}(m)\mathcal{A}_b.
\end{equation}
By  Lemma \ref{Lemma decomp Q}, we have 
$$
\frac{\mathcal{Q}_{r+1,\cdot}(u+mp^s)}{\mathcal{Q}_{r+1,\cdot}(u)}\in 2^{\iota mp^s}\frac{\Lambda_{b,r+1}(u+mp^s)}{\Lambda_{b,r+1}(u)}\mathcal{A}_b^{\times},
$$
with 
\begin{multline}\label{interVal 1}
v_p\left(\frac{\Lambda_{b,r+1}(u+mp^s)}{\Lambda_{b,r+1}(u)}\right)\\
=\sum_{\ell=1}^{\infty}\left(\widetilde{\Delta}_{r+1,t}^{p,\ell}\left(\left\{\frac{u+mp^s}{p^{\ell}}\right\}\right)-\widetilde{\Delta}_{r+1,t}^{p,\ell}\left(\left\{\frac{u}{p^{\ell}}\right\}\right)\right)-\lambda_p\frac{\mathfrak{s}_p(m)}{p-1}+mp^s\left\{\frac{\lambda_p}{p-1}\right\}\\
=\sum_{\ell=s+1}^{\infty}\widetilde{\Delta}_{r+1,t}^{p,\ell}\left(\left\{\frac{u+mp^s}{p^{\ell}}\right\}\right)-\lambda_p\frac{\mathfrak{s}_p(m)}{p-1}+mp^s\left\{\frac{\lambda_p}{p-1}\right\},
\end{multline}
because,  for all $\ell\in\{1,\dots,s\}$, we have $\{u/p^{\ell}\}=\big\{(u+mp^s)/p^\ell\big\}$ and, for all $\ell\geq s+1$, we have $\{u/p^\ell\}<\tau(r+1,\ell)$, thus $\widetilde{\Delta}_{r+1,t}^{p,\ell}\big(\{u/p^\ell\}\big)=0$. Let us show that, for all $\ell\geq s+1$, we have 
\begin{equation}\label{interVal 2}
\widetilde{\Delta}_{r+1,t}^{p,\ell}\left(\left\{\frac{u+mp^s}{p^\ell}\right\}\right)=\widetilde{\Delta}_{r+s+1,t}^{p,\ell-s}\left(\left\{\frac{m}{p^{\ell-s}}\right\}\right).
\end{equation}
Let $\alpha$ be an element of the sequences  $\boldsymbol{\alpha}$ or  $\boldsymbol{\beta}$ whose denominator is not divisible by $p$. To prove \eqref{interVal 2}, it is enough to show that, for all $\ell\geq s+1$, we have 
\begin{equation}\label{interVal 3}
\left\{\frac{u+mp^s}{p^\ell}\right\}\geq\mathfrak{D}_p^\ell(\underline{t^{(r+1)}\alpha})\Longleftrightarrow\left\{\frac{m}{p^{\ell-s}}\right\}\geq\mathfrak{D}_p^{\ell-s}(\underline{t^{(r+s+1)}\alpha}). 
\end{equation}
We write $m=\sum_{k=0}^{\infty}m_kp^k$ with $m_k\in\{0,\dots,p-1\}$. Then, we have 
$$
\left\{\frac{u+mp^s}{p^\ell}\right\}=\frac{u+\sum_{k=0}^{\ell-s-1}m_kp^{k+s}}{p^{\ell}}=\frac{u}{p^\ell}+\left\{\frac{m}{p^{\ell-s}}\right\}.
$$
We observe that $\mathfrak{D}_p^{\ell-s}(\underline{t^{(r+s+1)}\alpha})=\mathfrak{D}_p^\ell(\underline{t^{(r+1)}\alpha})$, so that 
$$
\left\{\frac{m}{p^{\ell-s}}\right\}\geq\mathfrak{D}_p^{\ell-s}(\underline{t^{(r+s+1)}\alpha})\Longrightarrow\left\{\frac{u+mp^s}{p^\ell}\right\}\geq\mathfrak{D}_p^\ell(\underline{t^{(r+1)}\alpha}). 
$$ 
Moreover, we have
\begin{align*}
\left\{\frac{u+mp^s}{p^\ell}\right\}\geq\mathfrak{D}_p^\ell(\underline{t^{(r+1)}\alpha})&\Longrightarrow\frac{u}{p^\ell}+\left\{\frac{m}{p^{\ell-s}}\right\}\geq\mathfrak{D}_p^\ell(\underline{t^{(r+1)}\alpha})\\
&\Longrightarrow p^{\ell-s}\left\{\frac{m}{p^{\ell-s}}\right\}\geq p^{\ell-s}\mathfrak{D}_p^\ell(\underline{t^{(r+1)}\alpha})-\frac{u}{p^s}\\
&\Longrightarrow p^{\ell-s}\left\{\frac{m}{p^{\ell-s}}\right\}> p^{\ell-s}\mathfrak{D}_p^\ell(\underline{t^{(r+1)}\alpha})-\mathfrak{D}_p^s(\underline{t^{(r+1)}\alpha})\\
&\Longrightarrow p^{\ell-s}\left\{\frac{m}{p^{\ell-s}}\right\}\geq p^{\ell-s}\mathfrak{D}_p^\ell(\underline{t^{(r+1)}\alpha})-\mathfrak{D}_p^s(\underline{t^{(r+1)}\alpha})+1\\
&\Longrightarrow\left\{\frac{m}{p^{\ell-s}}\right\}\geq\mathfrak{D}_p^{\ell-s}(\underline{t^{(r+s+1)}\alpha}).
\end{align*}
Equivalence \eqref{interVal 3} is thus proved and we have \eqref{interVal 2}. Using  \eqref{interVal 2} in \eqref{interVal 1}, we obtain 
\begin{align*}
v_p\left(\frac{\Lambda_{b,r+1}(u+mp^s)}{\Lambda_{b,r+1}(u)}\right)&=\sum_{\ell=s+1}^{\infty}\widetilde{\Delta}_{r+s+1,t}^{p,\ell-s}\left(\left\{\frac{m}{p^{\ell-s}}\right\}\right)-\lambda_p\frac{\mathfrak{s}_p(m)}{p-1}+mp^s\left\{\frac{\lambda_p}{p-1}\right\}\\
&=\sum_{\ell=1}^{\infty}\widetilde{\Delta}_{r+s+1,t}^{p,\ell}\left(\left\{\frac{m}{p^{\ell}}\right\}\right)-\lambda_p\frac{\mathfrak{s}_p(m)}{p-1}+mp^s\left\{\frac{\lambda_p}{p-1}\right\}\\
&= v_p\big(\Lambda_{b,r+s+1}(m)\big)+m(p^s-1)\left\{\frac{\lambda_p}{p-1}\right\}.
\end{align*}
This completes the proof of Lemma \ref{prop Psi}.
\end{proof}

\subsubsection{Verification of Conditions $(a)$ and $(a_1)$ of Theorem \ref{congruences formelles}}

Let us fix  $r\in\mathbb{N}$. For all $s\in\mathbb{N}$, all $v\in\{0,\dots,p-1\}$ and all  $u\in\Psi_{\mathcal{N}}(r,s)$, we set $\theta_{r,s}(v+up):=\mathcal{Q}_{r,\cdot}(v+up)$ if $v+up\notin\Psi_{\mathcal{N}}(r-1,s+1)$, and $\theta_{r,s}(v+up):=\mathbf{g}_r(v+up)$ otherwise.
\medskip

The aim of this section is to prove the following fact: for all  $s,m\in\mathbb{N}$, all  $v\in\{0,\dots,p-1\}$ and all  $u\in\Psi_{\mathcal{N}}(r,s)$, we have 
\begin{equation}\label{Hypoth\`ese (iii)}
\theta_{r,s}(v+up)\left(\frac{\mathcal{Q}_{r,\cdot}(v+up+mp^{s+1})}{\mathcal{Q}_{r,\cdot}(v+up)}-\frac{\mathcal{Q}_{r+1,\cdot}(u+mp^{s})}{\mathcal{Q}_{r+1,\cdot}(u)}\right)\in p^{s+1}\mathbf{g}_{r+s+1}(m)\mathcal{A}.
\end{equation}
This will prove Conditions $(a)$ and $(a_1)$ of Theorem \ref{congruences formelles}. Indeed, by Lemmas \ref{Lemma decomp Q} and \ref{Lambda vs mu}, for all $v\in\{0,\dots,p-1\}$ and all  $u\in\Psi_{\mathcal{N}}(r,s)$, we have 
$$
\mathcal{Q}_{r,\cdot}(v+up)\in\Lambda_{b,r}(v+up)\mathcal{A}\subset\mathbf{g}_r(v+up)\mathcal{A}.
$$
Hence, Congruence \eqref{Hypoth\`ese (iii)} implies Condition $(a)$ of Theorem \ref{congruences formelles}. Moreover, by definition of $\theta_{r,s}$, when  $v+up\in\Psi_{\mathcal{N}}(r-1,s+1)$, Congruence \eqref{Hypoth\`ese (iii)} implies Condition $(a_1)$ of Theorem~\ref{congruences formelles}.
\medskip

If $m=0$, then we have \eqref{Hypoth\`ese (iii)}. In the sequel, we assume that $m\geq 1$ and we split the proof of \eqref{Hypoth\`ese (iii)} into four distinct cases.
\medskip

$\bullet$ Case $1$: we assume that $v+up\notin\Psi(r-1,s+1)$.
\medskip

We then have $\theta_{r,s}(v+up)=\mathcal{Q}_{r,\cdot}(v+up)\in\Lambda_{b,r}(v+up)\mathcal{A}_b$. Let us show that $\Lambda_{b,r}(v+up)\in p^{s+1}\mathbb{Z}_p$. We have 
$$
v_p\big(\Lambda_{b,r}(v+up)\big)=\frac{1}{\varphi(p^\nu)}\sum_{\ell=1}^\infty\underset{\gcd(c,p)=1}{\sum_{c=1}^{p^\nu}}\Delta_{r,t}^{c,\ell}\left(\left\{\frac{v+up}{p^\ell}\right\}\right)+(v+up)\left\{\frac{\lambda_p}{p-1}\right\}.
$$ 
Since $v+up\notin\Psi(r-1,s+1)$ and $u\in\Psi(r,s)$, we obtain that $(v+up,s+1)\in\mathcal{N}_r$ and, for all $\ell\in\{1,\dots,s+1\}$, we have $\{(v+up)/p^{\ell}\}\geq\tau(r,\ell)$. It follows that 
$$
\frac{1}{\varphi(p^{\nu})}\sum_{\ell=1}^{s+1}\underset{\gcd(c,p)=1}{\sum_{c=1}^{p^{\nu}}}\Delta_{r,t}^{c,\ell}\left(\left\{\frac{v+up}{p^{\ell}}\right\}\right)\geq s+1
$$
and $v_p\big(\Lambda_{b,r}(v+up)\big)\geq s+1$ because the functions $\Delta_{r,t}^{c,\ell}$ are positive on $[0,1)$.

Since $u\in\Psi(r,s)$, Lemma \ref{prop Psi} yields
$$
\mathcal{Q}_{r,\cdot}(v+up)\frac{\mathcal{Q}_{r+1,\cdot}(u+mp^s)}{\mathcal{Q}_{r+1,\cdot}(u)}\in p^{s+1}\Lambda_{b,r+s+1}(m)\mathcal{A}_b\subset p^{s+1}\mathbf{g}_{r+s+1}(m)\mathcal{A}_b.
$$
Thus, to show  \eqref{Hypoth\`ese (iii)}, it it enough to show 
\begin{equation}\label{Hypoth\`ese (iii) 2}
\mathcal{Q}_{r,\cdot}(v+up+mp^{s+1})\in p^{s+1}\mathbf{g}_{r+s+1}(m)\mathcal{A}_b.
\end{equation}
By  Lemma \ref{Lemma decomp Q}, we have 
\begin{multline*}
v_p\big(\Lambda_{b,r}(v+up+mp^{s+1})\big)=\frac{1}{\varphi(p^{\nu})}\sum_{\ell=1}^{\infty}\underset{\gcd(c,p)=1}{\sum_{c=1}^{p^{\nu}}}\Delta_{r,t}^{c,\ell}\left(\left\{\frac{v+up+mp^{s+1}}{p^{\ell}}\right\}\right)\\
+(v+up+mp^{s+1})\left\{\frac{\lambda_p}{p-1}\right\},
\end{multline*}
hence 
\begin{multline*}
v_p\big(\Lambda_{b,r}(v+up+mp^{s+1})\big)\geq s+1+\frac{1}{\varphi(p^{\nu})}\sum_{\ell=s+2}^{\infty}\underset{\gcd(c,p)=1}{\sum_{c=1}^{p^{\nu}}}\Delta_{r,t}^{c,\ell}\left(\left\{\frac{v+up+mp^{s+1}}{p^{\ell}}\right\}\right)\\
+m\left\{\frac{\lambda_p}{p-1}\right\}.
\end{multline*}
\medskip

If $\boldsymbol{\beta}\in\mathbb{Z}^r$, then the functions $\Delta_{r,t}^{c,\ell}$ are non-decreasing on $[0,1)$ and, by \eqref{\'echange r l}, for all $\ell\geq s+2$, we obtain 
\begin{align*}
\sum_{\ell=s+2}^{\infty}\underset{\gcd(c,p)=1}{\sum_{c=1}^{p^{\nu}}}\Delta_{r,t}^{c,\ell}\left(\left\{\frac{v+up+mp^{s+1}}{p^{\ell}}\right\}\right)&\geq\sum_{\ell=s+2}^{\infty}\underset{\gcd(c,p)=1}{\sum_{c=1}^{p^{\nu}}}\Delta_{r,t}^{c,\ell}\left(\left\{\frac{mp^{s+1}}{p^{\ell}}\right\}\right)\\
&\geq \sum_{\ell=1}^{\infty}\underset{\gcd(c,p)=1}{\sum_{c=1}^{p^{\nu}}}\Delta_{r,t}^{c,\ell+s+1}\left(\left\{\frac{m}{p^\ell}\right\}\right)\\
&\geq \sum_{\ell=1}^{\infty}\underset{\gcd(c,p)=1}{\sum_{c=1}^{p^{\nu}}}\Delta_{r+s+1,t}^{c,\ell}\left(\left\{\frac{m}{p^\ell}\right\}\right).
\end{align*}
Consequently if $\boldsymbol{\beta}\in\mathbb{Z}^r$, then 
$$
v_p\big(\Lambda_{b,r}(v+up+mp^{s+1})\big)\geq s+1+v_p\big(\Lambda_{b,r+s+1}(m)\big)\geq s+1+v_p\big(\mathbf{g}_{r+s+1}(m)\big),
$$
as expected.
\medskip

On the other hand, if $\boldsymbol{\beta}\notin\mathbb{Z}^r$, then we observe that, for all $\ell\in\mathbb{N}$, $\ell\geq 1$, we have 
\begin{align*}
\left\{\frac{m}{p^\ell}\right\}\geq\tau(r+s+1,\ell)&\Longrightarrow \left\{\frac{mp^{s+1}}{p^{\ell+s+1}}\right\}\geq\tau(r,\ell+s+1)\\
&\Longrightarrow \left\{\frac{v+up+mp^{s+1}}{p^{\ell+s+1}}\right\}\geq\tau(r,\ell+s+1)\\
&\Longrightarrow\frac{1}{\varphi(p^\nu)}\underset{\gcd(c,p)=1}{\sum_{c=1}^{p^\nu}}\Delta_{r,t}^{c,\ell+s+1}\left(\left\{\frac{v+up+mp^{s+1}}{p^{\ell+s+1}}\right\}\right)\geq 1,
\end{align*}
so that 
$$
\frac{1}{\varphi(p^\nu)}\sum_{\ell=s+2}^{\infty}\underset{\gcd(c,p)=1}{\sum_{c=1}^{p^\nu}}\Delta_{r,t}^{c,\ell}\left(\left\{\frac{v+up+mp^{s+1}}{p^\ell}\right\}\right)\geq v_p\big(g_{r+s+1}(m)\big)
$$
and thus $v_p\big(\Lambda_{b,r}(v+up+mp^{s+1})\big)\geq s+1+v_p\big(\mathbf{g}_{r+s+1}(m)\big)$, as expected. Hence  \eqref{Hypoth\`ese (iii) 2} is proved, which finishes the proof of \eqref{Hypoth\`ese (iii)} when  $v+up\notin\Psi(r-1,s+1)$.
\medskip

$\bullet$ Case $2$: we assume that $v+up\in\Psi(r-1,s+1)$ and that  $p-1\nmid\lambda_p$.
\medskip

We have $\theta_{r,s}(v+up)=\mathbf{g}_r(v+up)$, $\mathcal{A}=\mathcal{A}_b$ and we have to show that 
$$
\mathbf{g}_{r}(v+up)\left(\frac{\mathcal{Q}_{r,\cdot}(v+up+mp^{s+1})}{\mathcal{Q}_{r,\cdot}(v+up)}-\frac{\mathcal{Q}_{r+1,\cdot}(u+mp^s)}{\mathcal{Q}_{r+1,\cdot}(u)}\right)\in p^{s+1}\mathbf{g}_{r+s+1}(m)\mathcal{A}_b.
$$ 
By  Lemma \ref{prop Psi}, 
$$
\frac{\mathcal{Q}_{r+1,\cdot}(u+mp^s)}{\mathcal{Q}_{r+1,\cdot}(u)}\in p^{\left\{\frac{\lambda_p}{p-1}\right\}m(p^s-1)}\Lambda_{b,r+s+1}(m)\mathcal{A}_b
$$
and 
$$
\frac{\mathcal{Q}_{r,\cdot}(v+up+mp^{s+1})}{\mathcal{Q}_{r,\cdot}(v+up)}\in p^{\left\{\frac{\lambda_p}{p-1}\right\}m(p^{s+1}-1)}\Lambda_{b,r+s+1}(m)\mathcal{A}_b.
$$
Since $p-1\nmid\lambda_p$ and $m\geq 1$, we have 
$$
\left\{\frac{\lambda_p}{p-1}\right\}m(p^s-1)\geq m\frac{p^s-1}{p-1}\geq s\quad\textup{and}\quad \left\{\frac{\lambda_p}{p-1}\right\}m(p^{s+1}-1)\geq s+1.
$$
Thus, we obtain 
$$
\mathbf{g}_r(v+up)\frac{\mathcal{Q}_{r,\cdot}(v+up+mp^{s+1})}{\mathcal{Q}_{r,\cdot}(v+up)}\in p^{s+1}\Lambda_{b,r+s+1}(m)\mathcal{A}_b\subset p^{s+1}\mathbf{g}_{r+s+1}(m)\mathcal{A}_b,
$$
because $\mathbf{g}_r(v+up)\in\mathbb{Z}_p$. It remains to show that 
\begin{equation}\label{Reste 2}
\mathbf{g}_r(v+up)\frac{\mathcal{Q}_{r+1,\cdot}(u+mp^s)}{\mathcal{Q}_{r+1,\cdot}(u)}\in p^{s+1}\mathbf{g}_{r+s+1}(m)\mathcal{A}_b.
\end{equation}
By  Lemma \ref{Lambda vs mu}, 
$$
v_p\big(\Lambda_{b,r+s+1}(m)\big)\geq v_p\big(g_{r+s+1}(m)\big)+m\left\{\frac{\lambda_p}{p-1}\right\}
$$ 
and thus, since $p-1\nmid\lambda_p$ and $m\geq 1$, we obtain that $\Lambda_{b,r+s+1}(m)\in pg_{r+s+1}(m)\mathbb{Z}_p$. Hence, we have $\Lambda_{b,r+s+1}(m)\in p\mathbf{g}_{r+s+1}(m)\mathbb{Z}_p$, as well as   \eqref{Reste 2} because $\mathbf{g}_r(v+up)\in\mathbb{Z}_p$. 
\medskip

$\bullet$ Case $3$: we assume that $v+up\in\Psi(r-1,s+1)$, $\boldsymbol{\beta}\notin\mathbb{Z}^r$, $p=2$ and that $\mathfrak{m}$ is odd.
\medskip

We have $\theta_{r,s}(v+up)=\mathbf{g}_r(v+up)=g_r(v+up)$, $\mathcal{A}=\mathcal{A}_b$ and we have to show
\begin{equation}\label{Reste 3}
g_{r}(v+up)\left(\frac{\mathcal{Q}_{r,\cdot}(v+up+mp^{s+1})}{\mathcal{Q}_{r,\cdot}(v+up)}-\frac{\mathcal{Q}_{r+1,\cdot}(u+mp^s)}{\mathcal{Q}_{r+1,\cdot}(u)}\right)\in p^{s+1}g_{r+s+1}(m)\mathcal{A}_b.
\end{equation}
By  Lemma \ref{prop Psi}, we have 
$$
\frac{\mathcal{Q}_{r+1,\cdot}(u+mp^s)}{\mathcal{Q}_{r+1,\cdot}(u)}\in 2^{mp^s}\Lambda_{b,r+s+1}(m)\mathcal{A}_b
$$
and 
$$
\frac{\mathcal{Q}_{r,\cdot}(v+up+mp^{s+1})}{\mathcal{Q}_{r,\cdot}(v+up)}\in 2^{mp^{s+1}}\Lambda_{b,r+s+1}(m)\mathcal{A}_b.
$$
Moreover, we have $m2^s\geq s+1$ and $m2^{s+1}\geq s+1$ because $m\geq 1$. Since $\Lambda_{b,r+s+1}(m)\in g_{r+s+1}(m)\mathbb{Z}_p$ and  $g_r(v+up)\in\mathbb{Z}_p$, we get \eqref{Reste 3}.
\medskip

$\bullet$ Case $4$: we assume that $v+up\in\Psi(r-1,s+1)$, $p-1$ divides $\lambda_p$ and that, if $p=2$ and $\boldsymbol{\beta}\notin\mathbb{Z}^r$, then  $\mathfrak{m}$ is even.
\medskip 

We set
$$
X_{r,s}(v,u,m):=\frac{\mathcal{Q}_{r,\cdot}(v+up)}{\mathcal{Q}_{r+1,\cdot}(u)}\frac{\mathcal{Q}_{r+1,\cdot}(u+mp^{s})}{\mathcal{Q}_{r,\cdot}(v+up+mp^{s+1})}.
$$ 
Assertion \eqref{Hypoth\`ese (iii)} is satisfied if and only if  for all  $s,m\in\mathbb{N}$, all $v\in\{0,\dots,p-1\}$ and all  $u\in\Psi_{\mathcal{N}}(r,s)$, we have 
\begin{equation}\label{eq 3.13}
\mathbf{g}_{r}(v+up)\big(X_{r,s}(v,u,m)-1\big)\frac{\mathcal{Q}_{r,\cdot}(v+up+mp^{s+1})}{\mathcal{Q}_{r,\cdot}(v+up)}\in p^{s+1}\mathbf{g}_{r+s+1}(m)\mathcal{A}.
\end{equation}
The following lemma will give the conclusion.

\begin{Lemma}\label{d\'efinition de Y}
We assume that $p-1$ divides $\lambda_p$ and that, if $p=2$ and  $\boldsymbol{\beta}\notin\mathbb{Z}^r$, then  $\mathfrak{m}$ is even. Then,
\begin{enumerate}
\item For all $r,s\in\mathbb{N}$, all $v\in\{0,\dots,p-1\}$, all $u\in\Psi_{\mathcal{N}}(r,s)$ and all  $m\in\mathbb{N}$, there exists  $Y_{r,s}(v,u,m)\in\mathbb{Z}_p$ independent of  $t\in\Omega_b$ such that 
$$
X_{r,s}(v,u,m)\in\begin{cases} Y_{r,s}(v,u,m)(1+p^{s}\mathcal{A}_b)\textup{ if  $\boldsymbol{\beta}\in\mathbb{Z}^r$ and $p\mid d_{\boldsymbol{\alpha},\boldsymbol{\beta}}$ ;}\\ Y_{r,s}(v,u,m)(1+p^{s+1}\mathcal{A}_b^\ast)\textup{ otherwise;}\end{cases}.
$$ 
\item If there exists  $j\in\{1,\dots,s+1\}$ such that $\big\{(v+up)/p^j\big\}<\tau(r,j)$, then we have $Y_{r,s}(v,u,m)\in 1+p^{s-j+2}\mathbb{Z}_p$.
\end{enumerate}
\end{Lemma}

Since $v+up\in\Psi(r-1,s+1)$, Lemma \ref{prop Psi} implies that $\big\{(v+up)/p^{s+1}\big\}<\tau(r,s+1)$. Let $j_0$ be the smallest $j\in\{1,\dots,s+1\}$ such that $\big\{(v+up)/p^j\big\}<\tau(r,j)$. By  Lemma \ref{d\'efinition de Y} applied with $j_0$, we obtain that $Y_{r,s}(v,u,m)\in 1+p^{s-j_0+2}\mathbb{Z}_p$ and that  
$$
X_{r,s}(v,u,m)\in\begin{cases}1+p^{s-j_0+1}\mathcal{A}_b\textup{ if  $\boldsymbol{\beta}\in\mathbb{Z}^r$ et $p\mid d_{\boldsymbol{\alpha},\boldsymbol{\beta}}$;}\\ 1+p^{s-j_0+2}\mathcal{A}_b^\ast\textup{ otherwise.}\end{cases}.
$$
Hence, Lemma \ref{prop Psi} yields
$$
\big(X_{r,s}(v,u,m)-1\big)\frac{\mathcal{Q}_{r,\cdot}(v+up+mp^{s+1})}{\mathcal{Q}_{r,\cdot}(v+up)}
\in p^{s-j_0+2}\mathbf{g}_{r+s+1}(m)\times 
\begin{cases}\mathcal{A}_b\textup{ if $\boldsymbol{\beta}\in\mathbb{Z}^r$ and  $p\mid d_{\boldsymbol{\alpha},\boldsymbol{\beta}}$;}\\ \mathcal{A}_b^\ast\textup{ otherwise.}\end{cases}.
$$
Therfore to prove \eqref{eq 3.13}, it is enough to show that $\mathbf{g}_r(v+up)\in p^{j_0-1}\mathbb{Z}_p$. If $v+up=0$, then we have $j_0=1$ and the conclusion is clear. We can thus assume that $v+up\geq 1$. But for all $j\in\{1,\dots,j_0-1\}$, we have $\big\{(v+up)/p^j\big\}\geq\tau(r,j)$, hence $v_p\big(g_r(v+up)\big)\geq j_0-1$.  Furthermore, if  $\boldsymbol{\beta}\in\mathbb{Z}^r$ and if $p\mid d_{\boldsymbol{\alpha},\boldsymbol{\beta}}$, we have $\lambda_p\leq -1$ and, by  Lemma \ref{Lambda vs mu}, we have 
$$
\mathbf{g}_r(v+up)=\frac{\Lambda_{b,r}(v+up)}{p}\in g_r(v+up)\mathbb{Z}_p\subset p^{j_0-1}\mathbb{Z}_p,
$$
as expected.
\medskip

To complete the proof of \eqref{eq 3.13} and that of Theorem \ref{theo expand}, it remains to prove  Lemma \ref{d\'efinition de Y}.

\begin{proof}[Proof of Lemma \ref{d\'efinition de Y}]
We will show that Lemma \ref{d\'efinition de Y} holds with 
$$
Y_{r,s}(v,u,m):=\frac{\prod_{\beta_i\in\mathbb{Z}_p}\left(1+\frac{mp^s}{\underline{t^{(r+1)}\beta_i}+u}\right)^{\rho(v,\underline{t^{(r)}\beta_i})}}{\prod_{\alpha_i\in\mathbb{Z}_p}\left(1+\frac{mp^s}{\underline{t^{(r+1)}\alpha_i}+u}\right)^{\rho(v,\underline{t^{(r)}\alpha_i})}}.
$$

By  Lemma $1$ of  \cite[Dwork]{cycles}, if  $\alpha$ is an element of the sequences  $\boldsymbol{\alpha}$ or $\boldsymbol{\beta}$ whose denominator is not divisible by $p$, then  for all $v\in\{0,\dots,p-1\}$, all $s,m\in\mathbb{N}$ and all  $u\in\{0,\dots,p^s-1\}$, we have 
\begin{equation}\label{part free p}
\frac{(\alpha)_{v+up+mp^{s+1}}\big(\mathfrak{D}_p(\alpha)\big)_u}{\big(\mathfrak{D}_p(\alpha)\big)_{u+mp^s}(\alpha)_{v+up}}\in\big((-p)^{p^s}\varepsilon_{p^s}\big)^m\left(1+\frac{mp^s}{\mathfrak{D}_p(\alpha)+u}\right)^{\rho(v,\alpha)}(1+p^{s+1}\mathbb{Z}_p),
\end{equation}
where $\varepsilon_{k}=-1$ if $k=2$, and $\varepsilon_{k}=1$ otherwise. 

Similarly, using Dwork's method, we will show that if $\alpha$ is an element of the sequences  $\boldsymbol{\alpha}$ or $\boldsymbol{\beta}$ whose denominator is  divisible by $p$, then for all $v\in\{0,\dots,p-1\}$, all $r,s,m\in\mathbb{N}$ and all  $u\in\{0,\dots,p^s-1\}$, we have 
\begin{equation}\label{part p}
\left(t\in\Omega_b\mapsto d(\alpha)^{m\varphi(p^{s+1})}\frac{(\underline{t^{(r)}\alpha})_{v+up+mp^{s+1}}(\underline{t^{(r+1)}\alpha})_u}{(\underline{t^{(r+1)}\alpha})_{u+mp^s}(\underline{t^{(r)}\alpha})_{v+up}}\right)\in
\varepsilon_{p^s}'(\alpha)^m(1+p^{s+1}\mathcal{A}_b^\ast),
\end{equation}
where $\varepsilon_k'(\alpha)=\varepsilon_k$ if $v_p\big(d(\alpha)\big)=1$ and  $\varepsilon_k'(\alpha)=1$ otherwise.
\medskip 

We first show that \eqref{part free p} and \eqref{part p} imply the validity of Assertion  $(1)$ of  Lemma \ref{d\'efinition de Y}. Indeed, by \eqref{part free p}, we obtain 
\begin{multline}\label{joint 0.1}
\frac{\Lambda_{b,r}(v+up+mp^{s+1})\Lambda_{b,r+1}(u)}{\Lambda_{b,r+1}(u+mp^s)\Lambda_{b,r}(v+up)}\\
\in\left(C\frac{\prod_{\beta_i\notin\mathbb{Z}_p}d(\beta_i)}{\prod_{\alpha_i\notin\mathbb{Z}_p}d(\alpha_i)}\right)^{m\varphi(p^{s+1})}\frac{\big((-p)^{p^s}\varepsilon_{p^s}\big)^{m\lambda_p}}{Y_{r,s}(v,u,m)}(1+p^{s+1}\mathbb{Z}_p).
\end{multline} 
We write 
$$
C\frac{\prod_{\beta_i\notin\mathbb{Z}_p}d(\beta_i)}{\prod_{\alpha_i\notin\mathbb{Z}_p}d(\alpha_i)}=\sigma p^{-\left\lfloor\frac{\lambda_p}{p-1}\right\rfloor}=\sigma p^{-\frac{\lambda_p}{p-1}},
$$
with $\sigma\in\mathbb{Z}_p^{\times}$, so that 
$$
\left(C\frac{\prod_{\beta_i\notin\mathbb{Z}_p}d(\beta_i)}{\prod_{\alpha_i\notin\mathbb{Z}_p}d(\alpha_i)}\right)^{m\varphi(p^{s+1})}\in p^{-mp^s\lambda_p}(1+p^{s+1}\mathbb{Z}_p).
$$
We thus have  
\begin{equation}\label{joint 0.2}
\left(C\frac{\prod_{\beta_i\notin\mathbb{Z}_p}d(\beta_i)}{\prod_{\alpha_i\notin\mathbb{Z}_p}d(\alpha_i)}\right)^{m\varphi(p^{s+1})}\big((-p)^{p^s}\varepsilon_{p^s}\big)^{m\lambda_p}\in (-1)^{mp^s\lambda_p}\varepsilon_{p^s}^{m\lambda_p}(1+p^{s+1}\mathbb{Z}_p)\subset\varepsilon_{p^s}^{m\lambda_p}(1+p^{s+1}\mathbb{Z}_p),
\end{equation}
because  $-1\in\mathbb{Z}_p^{\times}$ and $\varphi(p^{s+1})=p^s(p-1)$ divides $mp^s\lambda_p$. Using \eqref{joint 0.2} in \eqref{joint 0.1}, we obtain that
\begin{equation}\label{joint 0.3}
\frac{\Lambda_{b,r+1}(u)\Lambda_{b,r}(v+up+mp^{s+1})}{\Lambda_{b,r}(v+up)\Lambda_{b,r+1}(u+mp^s)}\\
\in\frac{\varepsilon_{p^s}^{m\lambda_p}}{Y_{r,s}(v,u,m)}(1+p^{s+1}\mathbb{Z}_p).
\end{equation}

By \eqref{part p}, we also obtain 
$$
\frac{\mathfrak{R}_{r+1}(u+mp^s,\cdot)\mathfrak{R}_{r}(v+up,\cdot)}{\mathfrak{R}_{r}(v+up+mp^{s+1},\cdot)\mathfrak{R}_{r+1}(u,\cdot)}\in \left(\frac{\prod_{\beta_i\notin\mathbb{Z}_p}\varepsilon_{p^s}'(\beta_i)}{\prod_{\alpha_i\notin\mathbb{Z}_p}\varepsilon_{p^s}'(\alpha_i)}\right)^{m}(1+p^{s+1}\mathcal{A}_b^\ast).
$$
If $p^s\neq 2$, then, for any element $\alpha\notin\mathbb{Z}_p$ of  $\boldsymbol{\alpha}$ or $\boldsymbol{\beta}$, we have $\varepsilon_{p^s}'(\alpha)=\varepsilon_{p^s}=1$. If $p^s=2$ and if the number of elements $\alpha$ of $\boldsymbol{\alpha}$ and $\boldsymbol{\beta}$ that satisfy $v_2\big(d(\alpha)\big)\geq 2$ is even, then since $r=s$, we have 
$$
\frac{\prod_{\beta_i\notin\mathbb{Z}_p}
\varepsilon_{p^s}'(\beta_i)}{\prod_{\alpha_i\notin\mathbb{Z}_p}\varepsilon_{p^s}'(\alpha_i)}=(-1)^{\lambda_2}=\varepsilon_2^{\lambda_2}. 
$$
Moreover, we have $p\mathcal{A}_b^\ast\subset\mathcal{A}_b$ and $\varepsilon_{p^s},\varepsilon_{p^s}'(\alpha)\in 1+p^s\mathbb{Z}_p$. It follows that we obtain  
\begin{equation}\label{joint 1.3}
\frac{\mathfrak{R}_{r+1}(u+mp^s,\cdot)\mathfrak{R}_{r}(v+up,\cdot)}{\mathfrak{R}_{r}(v+up+mp^{s+1},\cdot)\mathfrak{R}_{r+1}(u,\cdot)}\in
\begin{cases}
\textup{$1+p^{s}\mathcal{A}_b$ if $\boldsymbol{\beta}\in\mathbb{Z}^r$ and $p\mid d_{\boldsymbol{\alpha},\boldsymbol{\beta}}$;}\\
\textup{$\varepsilon_{p^s}^{m\lambda_p}(1+p^{s+1}\mathcal{A}_b^\ast)$ otherwise.}
\end{cases}.
\end{equation}
By \eqref{joint 0.3} and \eqref{joint 1.3}, we obtain 
$$
X_{r,s}(v,u,m)\in Y_{r,s}(v,u,m)\times\begin{cases}
\textup{$(1+p^{s}\mathcal{A}_b)$ if $\boldsymbol{\beta}\in\mathbb{Z}^r$ and $p\mid d_{\boldsymbol{\alpha},\boldsymbol{\beta}}$;}\\
\textup{$(1+p^{s+1}\mathcal{A}_b^\ast)$ otherwise.}
\end{cases}. 
$$
To finish the proof of Assertion $(1)$ of Lemma \ref{d\'efinition de Y}, we have to prove  \eqref{part p}.
\medskip

Let $\alpha$ be an element of $\boldsymbol{\alpha}$ or $\boldsymbol{\beta}$ whose  denominator is divisible by $p$. For all $s,m\in\mathbb{N}$ and all $u\in\{0,\dots,p^s-1\}$, we set
$$
\mathfrak{q}_r(u,s,m):=t\in\Omega_b\mapsto d(\alpha)^{mp^{s}}\frac{(\underline{t^{(r)}\alpha})_{u+mp^{s}}}{(\underline{t^{(r)}\alpha})_{u}}=
\prod_{k=0}^{mp^{s}-1}\big(d(\alpha)\underline{t^{(r)}\alpha}
+d(\alpha)u+d(\alpha)k\big).
$$
Hence, proving \eqref{part p} amounts to proving that  
$$
\frac{\mathfrak{q}_r(v+up,s+1,m)}{\mathfrak{q}_{r+1}(u,s,m)}\in\varepsilon_{p^s}'(\alpha)^m(1+p^{s+1}\mathcal{A}_b^\ast).
$$
As functions of $t$, we have 
\begin{align}
\mathfrak{q}_r(u,s,m)(t)&=\prod_{i=0}^{p^{s}-1}\prod_{j=0}^{m-1}
\big(d(\alpha)\underline{t^{(r)}\alpha}+d(\alpha)u+d(\alpha)i
+d(\alpha)jp^{s}\big)\notag\\
&\equiv\prod_{i=0}^{p^{s}-1}\big(d(\alpha)\underline{t^{(r)}\alpha}
+d(\alpha)u+d(\alpha)i\big)^m\mod p^{s+1}\mathcal{A}_b\notag\\
&\equiv\prod_{i=0}^{p^{s}-1}\big(d(\alpha)\underline{t^{(r)}\alpha}
+d(\alpha)i\big)^m\mod p^{s+1}\mathcal{A}_b\notag.
\end{align}
Since $d(\alpha)$ is divisible by $p$, we obtain that, for all $i\in\{0,\dots,p^s-1\}$, the map  $t\in\Omega_b\mapsto d(\alpha)\underline{t^{(r)}\alpha}+d(\alpha)i$ is invertible in $\mathcal{A}_b$ and thus 
$$
\mathfrak{q}_r(u,s,m)\in\mathfrak{q}_r(0,s,1)^m(1+p^{s+1}
\mathcal{A}_b).
$$
Hence proving \eqref{part p} amounts to proving that, for all $s\in\mathbb{N}$, we have 
\begin{equation}\label{7.18 equiv 1}
\frac{\mathfrak{q}_r(0,s+1,1)}{\mathfrak{q}_{r+1}(0,s,1)}\in\varepsilon_{p^s}'(\alpha)(1+p^{s+1}\mathcal{A}_b^\ast).
\end{equation}
\medskip

$\bullet$ Case $1$: we assume that $s=0$.
\medskip

As functions of $t$, we have 
$$
\frac{\mathfrak{q}_r(0,1,1)(t)}{\mathfrak{q}_{r+1}(0,0,1)(t)}\in\frac{\big(d(\alpha)\underline{t^{(r)}\alpha}\big)^p}{d(\alpha)\underline{t^{(r+1)}\alpha}}(1+p\mathcal{A}_b)
$$
and 
$$
t^{(r)}\equiv\varpi_{p^\nu}\left(\frac{t}{D}\right)D+\varpi_D\left(\frac{b}{p^{\nu+r}}\right)p^\nu\mod p^\nu D.
$$
Hence with $\langle\alpha\rangle:=\kappa/d(\alpha)$, we obtain the existence of  $\eta(r,t)\in\mathbb{Z}$ such that
$$
d(\alpha)\underline{t^{(r)}\alpha}=\varpi_{p^\nu}\left(\frac{t\kappa}{D}\right)D+\varpi_D\left(\frac{b\kappa}{p^{\nu+r}}\right)p^\nu+d(\alpha)\eta(r,t).
$$

Moreover by Assertions $(2)$, $(4)$ and $(5)$ of Lemma \ref{Lemma alg\`ebres 2}, the maps  $t\in\Omega_b\mapsto d(\alpha)\underline{t^{(r)}\alpha}$ and  $f:t\in\Omega_b\mapsto\varpi_{p^\nu}(t\kappa/D)D$ are in $\mathcal{A}_b^\times$.  Thus $t\in\Omega_b\mapsto d(\alpha)\eta(r,t)$ is in  $\mathcal{A}_b$ and  $t\in\Omega_b\mapsto d(\alpha)\eta(r,t)/p$ is in $\mathcal{A}_b^{\ast}$ because $p$ divides $d(\alpha)$. It follows that  
\begin{equation}\label{transfo}
\big(t\in\Omega_b\mapsto d(\alpha)\underline{t^{(r)}\alpha}\big)\in f(1+p\mathcal{A}_b^\ast).
\end{equation}
We obtain 
$$
\frac{\mathfrak{q}_r(0,1,1)}{\mathfrak{q}_{r+1}(0,0,1)}\in f^{p-1}(1+p\mathcal{A}_b^\ast)\subset\big(1+p(\mathfrak{E}_1\circ f)\big)(1+p\mathcal{A}_b^\ast)\subset 1+p\mathcal{A}_b^\ast,
$$ 
as expected, where the final inclusion is obtained \textit{via} Assertion $(3)$ of  Lemma \ref{Lemma alg\`ebres 1}.
\medskip

$\bullet$ Case $2$: we assume that $s\geq 1$. 
\medskip

If $s\geq 1$, then 
\begin{align}
\prod_{i=0}^{p^s-1}\big(d(\alpha)\underline{t^{(r)}\alpha}
+d(\alpha)i\big)&=\prod_{j=0}^{p^{s-1}-1}\prod_{a=0}^{p-1}
\big(d(\alpha)\underline{t^{(r)}\alpha}+d(\alpha)j+d(\alpha)
ap^{s-1}\big)\label{pour s s+1}\\
&\equiv\prod_{j=0}^{p^{s-1}-1}\big(d(\alpha)\underline{t^{(r)}
\alpha}+d(\alpha)j\big)^p\mod p^s\mathcal{A}_b.\label{pour s+1 s+1}
\end{align}
Using \eqref{pour s+1 s+1} with $s+1$ for $s$, we obtain 
$$
\mathfrak{q}_r(0,s+1,1)\in\mathfrak{q}_r(0,s,1)^p(1+p^{s+1}\mathcal{A}_b)
$$
and thus 
\begin{equation}\label{rang s+1}
\mathfrak{q}_r(0,s+1,1)\in\big(d(\alpha)\underline{t^{(r)}\alpha}\big)^{p^{s+1}}(1+p^{s+1}\mathcal{A}_b).
\end{equation}
\medskip

We set $P(x):=x^p-x\in\mathbb{Z}_p[x]$. For all $a\in\{0,\dots,p-1\}$, we have $a^p-a\equiv 0\mod p\mathbb{Z}_p$. Since $P'(x)=px^{p-1}-1$, for all $a\in\{0,\dots,p-1\}$, we have $v_p\big(P'(a)\big)=0$ and, by Hensel's lemma (see \cite{Robert}), there exists  a root $w_a$ of $P$ in $\mathbb{Z}_p$ such that $w_a\equiv a\mod p\mathbb{Z}_p$. Consequently, for all $x\in\mathbb{Z}_p$ and all $s\in\mathbb{N}$, $s\geq 1$, we have
\begin{align}
\prod_{a=0}^{p-1}\big(x+d(\alpha)ap^{s-1}\big)&\equiv\prod_{i=0}^{p-1}\big(x-d(\alpha)w_ip^{s-1}\big)\mod p^{s+1}\mathbb{Z}_p\notag\\
&\equiv x^p-\big(d(\alpha)p^{s-1}\big)^{p-1}x\mod p^{s+1}\mathbb{Z}_p.\label{pol cong}
\end{align}
If $p\neq 2$, then $\big(d(\alpha)p^{s-1}\big)^{p-1}x\in p^{s+1}\mathbb{Z}_p$ thus, by \eqref{pour s s+1}, for all $s\in\mathbb{N}$, $s\geq 1$, we obtain 
$$
\mathfrak{q}_{r+1}(0,s,1)\in\prod_{j=0}^{p^{s-1}-1}\big(d(\alpha)\underline{t^{(r+1)}\alpha}+d(\alpha)j\big)^p(1+p^{s+1}\mathcal{A}_b),
$$
hence $\mathfrak{q}_{r+1}(0,s,1)\in\mathfrak{q}_{r+1}(0,s-1,1)^p(1+p^{s+1}\mathcal{A}_b)$ and 
$$
\mathfrak{q}_{r+1}(0,s,1)\in\big(d(\alpha)\underline{t^{(r+1)}\alpha}\big)^{p^s}(1+p^{s+1}\mathcal{A}_b).
$$
By \eqref{rang s+1} and \eqref{transfo}, we obtain the existence of  $f_1,f_2\in\mathcal{A}_b^\ast$ such that
\begin{multline*}
\frac{\mathfrak{q}_r(0,s+1,1)}{\mathfrak{q}_{r+1}(0,s,1)}\in f^{\varphi(p^{s+1})}\frac{(1+pf_1)^{p^{s+1}}}{(1+pf_2)^{p^s}}(1+p^{s+1}\mathcal{A}_b)\\
\subset\big(1+p^{s+1}(\mathfrak{E}_{s+1}\circ f)\big)(1+p^{s+1}\mathcal{A}_b^\ast)\subset 1+p^{s+1}\mathcal{A}_b^\ast,
\end{multline*}
which proves \eqref{7.18 equiv 1} when $p\neq 2$ because in this case we have $\varepsilon_{p^s}'(\alpha)=1$.
\medskip

Let us now assume $p=2$. Then by \eqref{pour s s+1} and \eqref{pol cong}, for all $s\in\mathbb{N}$, $s\geq 1$, we obtain 
$$
\mathfrak{q}_{r+1}(0,s,1)\in\prod_{j=0}^{2^{s-1}-1}\big(d(\alpha)\underline{t^{(r+1)}\alpha}+d(\alpha)j\big)^2\left(1-\frac{d(\alpha)2^{s-1}}{d(\alpha)\underline{t^{(r+1)}\alpha}+d(\alpha)j}\right)(1+2^{s+1}\mathcal{A}_b).
$$
Since $2$ divides $d(\alpha)$, we have 
\begin{align*}
\prod_{j=0}^{2^{s-1}-1}\left(1-\frac{d(\alpha)2^{s-1}}{d(\alpha)\underline{t^{(r+1)}\alpha}+d(\alpha)j}\right)&=
\prod_{j=0}^{2^{s-1}-1}\left(1-\frac{d(\alpha)2^{s-1}}{1+2\mathfrak{E}_1\big(d(\alpha)\underline{t^{(r+1)}\alpha}+d(\alpha)j\big)}\right)\\
&\equiv\prod_{j=0}^{2^{s-1}-1}\big(1-d(\alpha)2^{s-1}\big)\mod 2^{s+1}\mathcal{A}_b^\ast\\
&\equiv 1-d(\alpha)2^{2s-2}\mod 2^{s+1}\mathcal{A}_b^\ast,
\end{align*}
with $1-d(\alpha)2^{2s-2}\equiv 1\mod 2^{s+1}$ if $s\geq 2$ or  $v_2\big(d(\alpha)\big)\geq 2$, and $1-d(\alpha)2^{2s-2}\equiv -1\mod 4$ if  $s=v_2\big(d(\alpha)\big)=1$. It follows that 
$$
\mathfrak{q}_{r+1}(0,s,1)\in\varepsilon_{2^s}'(\alpha)\prod_{j=0}^{2^{s-1}-1}
\big(d(\alpha)\underline{t^{(r+1)}\alpha}+d(\alpha)j\big)^2(1+2^{s+1}
\mathcal{A}_b^\ast),
$$
\textit{i.e.} $\mathfrak{q}_{r+1}(0,s,1)\in\varepsilon_{2^s}'(\alpha)\mathfrak{q}_{r+1}(0,s-1,1)^2(1+2^{s+1}\mathcal{A}_b^\ast)$ and thus 
$$
\mathfrak{q}_{r+1}(0,s,1)\in\varepsilon_{2^s}'(\alpha)\big(d(\alpha)\underline{t^{(r+1)}\alpha}\big)^{2^s}(1+2^{s+1}\mathcal{A}_b^\ast).
$$
By \eqref{rang s+1} and \eqref{transfo}, we obtain the existence of  $f_1,f_2\in\mathcal{A}_b^\ast$ such that
\begin{multline*}
\frac{\mathfrak{q}_r(0,s+1,1)}{\mathfrak{q}_{r+1}(0,s,1)}\in\frac{1}{\varepsilon_{2^s}'(\alpha)} f^{\varphi(2^{s+1})}\frac{(1+2f_1)^{2^{s+1}}}{(1+2f_2)^{2^s}}(1+2^{s+1}\mathcal{A}_b^\ast)\\
\subset\varepsilon_{2^s}'(\alpha)\big(1+2^{s+1}(\mathfrak{E}_{s+1}\circ f)\big)(1+2^{s+1}\mathcal{A}_b^\ast)\subset\varepsilon_{2^s}'(\alpha)(1+2^{s+1}\mathcal{A}_b^\ast),
\end{multline*}
which proves \eqref{7.18 equiv 1} and completes the proof of  $(1)$ of  Lemma \ref{d\'efinition de Y}. 
\medskip

Let us now prove Assertion $(2)$ of Lemma \ref{d\'efinition de Y}. We have 
$$
Y_{r,s}(v,u,m)=\frac{\prod_{\beta_i\in\mathbb{Z}_p}\left(1+\frac{mp^s}{\underline{t^{(r+1)}\beta_i}+u}\right)^{\rho(v,\underline{t^{(r)}\beta_i})}}{\prod_{\alpha_i\in\mathbb{Z}_p}\left(1+\frac{mp^s}{\underline{t^{(r+1)}\alpha_i}+u}\right)^{\rho(v,\underline{t^{(r)}\alpha_i})}}.
$$
Let $j\in\{1,\dots,s+1\}$ be such that $\big\{(v+up)/p^j\big\}<\tau(r,j)$. We set  $u=\sum_{k=0}^\infty u_kp^k$. For all elements $\alpha\in\mathbb{Z}_p$ of the sequences $\boldsymbol{\alpha}$ or $\boldsymbol{\beta}$, we have 
\begin{align*}
\left\{\frac{v+up}{p^j}\right\}<\tau(r,j)&\Longrightarrow v+p\sum_{k=0}^{j-2}u_kp^k<p^j\mathfrak{D}_p^j(\underline{t^{(r)}\alpha})\\
&\Longrightarrow v+p\sum_{k=0}^{j-2}u_kp^k\leq p^j\mathfrak{D}_p^j(\underline{t^{(r)}\alpha})-\underline{t^{(r)}\alpha}\\
&\Longrightarrow v+p\sum_{k=0}^{j-2}u_kp^k\leq \sum_{k=0}^{j-1}p^k\big(p\mathfrak{D}_p^{k+1}(\underline{t^{(r)}\alpha})-\mathfrak{D}_p^k(\underline{t^{(r)}\alpha})\big)\\
&\Longrightarrow\left(\rho(v,\underline{t^{(r)}\alpha})=0\quad\textup{or}\quad \sum_{k=0}^{j-2}u_kp^k<p^{j-1}
\mathfrak{D}_p^j(\underline{t^{(r)}\alpha})
-\mathfrak{D}_p(\underline{t^{(r)}\alpha})\right)\\
&\Longrightarrow \left(\rho(v,\underline{t^{(r)}\alpha})=0\quad\textup{or}\quad \sum_{k=0}^{j-2}u_kp^k<p^{j-1}\mathfrak{D}_p^{j-1}(\underline{t^{(r+1)}\alpha})-\underline{t^{(r+1)}\alpha}\right)\\
&\Longrightarrow \left(\rho(v,\underline{t^{(r)}\alpha})=0\quad\textup{or}\quad v_p(u+\underline{t^{(r+1)}\alpha})\leq j-2\right)\\
&\Longrightarrow\left(1+\frac{mp^s}{\underline{t^{(r+1)}\alpha}+u}\right)^{\rho(v,\underline{t^{(r)}\alpha})}\in 1+p^{s-j+2}\mathbb{Z}_p,
\end{align*}
as expected. This completes the proof of Lemma \ref{d\'efinition de Y} and that of 
Theorem \ref{theo expand}.
\end{proof}

\section{Proof of Assertion $(1)$ of Theorem \ref{Criterion}}\label{section proof pos}

We shall prove the more precise following statement.

\begin{propo}\label{propo positive}
Let $\boldsymbol{\alpha}$ and $\boldsymbol{\beta}$ be tuples of parameters in $\mathbb{Q}\setminus\mathbb{Z}_{\leq 0}$ such that $\langle\boldsymbol{\alpha}\rangle$ and $\langle\boldsymbol{\beta}\rangle$ are disjoint. Let $a\in\{1,\dots,d_{\boldsymbol{\alpha},\boldsymbol{\beta}}\}$ coprime to $d_{\boldsymbol{\alpha},\boldsymbol{\beta}}$ be such that, for all $x\in\mathbb{R}$, we have $\xi_{\boldsymbol{\alpha},\boldsymbol{\beta}}(a,x)\geq 0$. Then, all the Taylor coefficients at the origin of $q_{\langle a\boldsymbol{\alpha}\rangle,\langle a\boldsymbol{\beta}\rangle}(z)$ are positive, but its constant term which is $0$.
\end{propo}

To prove Proposition \ref{propo positive}, we follow the method used by Delaygue in \cite[section $10.3$]{Delaygue0}, itself inspired by the work of Krattenthaler-Rivoal in \cite{TanguyPos}. We state three lemmas which enable us to prove Proposition \ref{propo positive}.

\begin{Lemma}[Lemma $2.1$ in \cite{TanguyPos}]\label{Lemma pos 1}
Let $a(z)=\sum_{n=0}^\infty a_nz^n\in\mathbb{R}[[z]]$, $a_0=1$, be such that all Taylor coefficients at the origin of $\mathfrak{a}(z)=1-1/a(z)$ are nonnegative. Let $b(z)=\sum_{n=0}^\infty a_nh_nz^n$ where $(h_n)_{n\geq 0}$ is a nondecreasing sequence of nonnegative real numbers. Then, all Taylor coefficients at the origin of $b(z)/a(z)$ are non-negative.

Furthermore, if all Taylor coefficients of $a(z)$ and $\mathfrak{a}(z)$ are positive (excepted the constant term of $\mathfrak{a}(z)$) and if $(h_n)_{n\geq 0}$ is an increasing sequence, then all Taylor coefficients at the origin of $b(z)/a(z)$ are positive, except its constant term if $h_0=0$.
\end{Lemma}

The following lemma is a refined version of Kaluza's Theorem \cite[Satz $3$]{Kaluza}. Initially, Satz~$3$ did not cover the case $a_{n+1}a_{n-1}>a_n^2$.

\begin{Lemma}[Lemma $2.2$ in \cite{TanguyPos}]\label{Lemma pos 2}
Let $a(z)=\sum_{n=0}^\infty a_nz^n\in\mathbb{R}[[z]]$, $a_0=1$, be such that $a_1>0$ and $a_{n+1}a_{n-1}\geq a_n^2$ for all positive integers $n$. Then, all Taylor coefficients of $\mathfrak{a}(z)=1-1/a(z)$ are nonnegative.

Furthermore, if we have $a_{n+1}a_{n-1}>a_n^2$ for all positive integers $n$, then all Taylor coefficients of $\mathfrak{a}(z)$ are positive (except its constant term).
\end{Lemma}

By Lemmas \ref{Lemma pos 1} and \ref{Lemma pos 2}, to prove Proposition \ref{propo positive}, it suffices to prove the following result.

\begin{Lemma}\label{Lemma pos 3}
Let $\boldsymbol{\alpha}=(\alpha_1,\dots,\alpha_r)$ and $\boldsymbol{\beta}=(\beta_1,\dots,\beta_s)$ be tuples of parameters in $\mathbb{Q}\setminus\mathbb{Z}_{\leq 0}$ such that $\langle\boldsymbol{\alpha}\rangle$ and $\langle\boldsymbol{\beta}\rangle$ are disjoint. Let $a\in\{1,\dots,d_{\boldsymbol{\alpha},\boldsymbol{\beta}}\}$ coprime to $d_{\boldsymbol{\alpha},\boldsymbol{\beta}}$ be such that, for all $x\in\mathbb{R}$, we have $\xi_{\boldsymbol{\alpha},\boldsymbol{\beta}}(a,x)\geq 0$. Then, for all positive integers $n$, we have 
$$
\mathcal{Q}_{\langle a\boldsymbol{\alpha}\rangle,\langle a\boldsymbol{\beta}\rangle}(n+1)\mathcal{Q}_{\langle a\boldsymbol{\alpha}\rangle,\langle a\boldsymbol{\beta}\rangle}(n-1)>\mathcal{Q}_{\langle a\boldsymbol{\alpha}\rangle,\langle a\boldsymbol{\beta}\rangle}(n)^2.
$$
Furthermore, $\big(\sum_{i=1}^rH_{\langle a\alpha_i\rangle}(n)-\sum_{j=1}^sH_{\langle a\beta_j\rangle}(n)\big)_{n\geq 0}$ is an increasing sequence.
\end{Lemma}

To  prove Lemma \ref{Lemma pos 3}, we first prove the following lemma that we also use in the proof of Assertion $(2)$ of Theorem  \ref{Criterion}.

\begin{Lemma}\label{Lemma sauts}
Let $\boldsymbol{\alpha}=(\alpha_1,\dots,\alpha_r)$ and $\boldsymbol{\beta}=(\beta_1,\dots,\beta_s)$ be tuples of parameters in $\mathbb{Q}\setminus\mathbb{Z}_{\leq 0}$ such that $\langle\boldsymbol{\alpha}\rangle$ and $\langle\boldsymbol{\beta}\rangle$ are disjoint. Let $a\in\{1,\dots,d_{\boldsymbol{\alpha},\boldsymbol{\beta}}\}$ be coprime to $d_{\boldsymbol{\alpha},\boldsymbol{\beta}}$. Let $\gamma_1,\dots,\gamma_t$ be rational numbers such that $\langle a\gamma_1\rangle<\cdots<\langle a\gamma_t\rangle$ and such that $\big\{\langle a\gamma_1\rangle,\dots,\langle a\gamma_t\rangle\big\}$ is the set of the numbers $\langle a\gamma\rangle$ when $\gamma$ describes all the elements of $\boldsymbol{\alpha}$ and $\boldsymbol{\beta}$. For all $i\in\{1,\dots,t\}$, we define $m_i:=\#\big\{1\leq j\leq r\,:\,\langle a\alpha_j\rangle=\langle a\gamma_i\rangle\big\}-\#\big\{1\leq j\leq s\,:\,\langle a\beta_j\rangle=\langle a\gamma_i\rangle\big\}$.

Assume that, for all $x\in\mathbb{R}$, we have $\xi_{\boldsymbol{\alpha},\boldsymbol{\beta}}(a,x)\geq 0$. Then, for all $i\in\{1,\dots,t\}$ and all $b\in\mathbb{R}$, $b\geq 0$, we have 
$$
\sum_{k=1}^{i}\frac{m_k}{\langle a\gamma_k\rangle+b}>0\quad\textup{and}\quad\prod_{k=1}^{i}\left(1+\frac{1}{\langle a\gamma_k\rangle+b}\right)^{m_k}>1.
$$
\end{Lemma}

\begin{proof}[Proof of Lemma \ref{Lemma sauts}]
First, observe that by Proposition \ref{propo reduction}, for all $j\in\{1,\dots,t\}$, we have 
$$
\sum_{i=1}^jm_i=\xi_{\langle a\boldsymbol{\alpha}\rangle,\langle a\boldsymbol{\beta}\rangle}(1,\langle a\gamma_j\rangle)\geq 0. 
$$
Furthermore, since $\langle a\boldsymbol{\alpha}\rangle$ and $\langle a\boldsymbol{\beta}\rangle$ are disjoint, for all $i\in\{1,\dots,t\}$, we have $m_i\neq 0$. In particular, we obtain that $m_1\geq 1$. It follows that we have 
$$
\frac{m_1}{\langle a\gamma_1\rangle+b}>0\quad\textup{and}\quad \left(1+\frac{1}{\langle a\gamma_1\rangle+b}\right)^{m_1}>1.
$$ 
Now assume that $t\geq 2$. We shall prove by induction on $i$ that, for all $i\in\{2,\dots,t\}$, we have
\begin{equation}\label{induction sauts}
\sum_{k=1}^i\frac{m_k}{\langle a\gamma_k\rangle+b}>\frac{\sum_{k=1}^im_k}{\langle a\gamma_i\rangle+b}\quad\textup{and}\quad\prod_{k=1}^i\left(1+\frac{1}{\langle a\gamma_k\rangle+b}\right)^{m_k}>\left(1+\frac{1}{\langle a\gamma_i\rangle+b}\right)^{\sum_{k=1}^im_k}.
\end{equation}
We have $\langle a\gamma_1\rangle<\langle a\gamma_2\rangle$ and $m_1>0$ thus we get
$$
\frac{m_1}{\langle a\gamma_1\rangle+b}+\frac{m_2}{\langle a\gamma_2\rangle+b}>\frac{m_1+m_2}{\langle a\gamma_2\rangle+b}
$$
and
$$
\left(1+\frac{1}{\langle a\gamma_1\rangle+b}\right)^{m_1}\left(1+\frac{1}{\langle a\gamma_2\rangle+b}\right)^{m_2}>\left(1+\frac{1}{\langle a\gamma_2\rangle+b}\right)^{m_1+m_2},
$$
so that \eqref{induction sauts} holds for $i=2$. We now assume that $t\geq 3$ and let $i\in\{2,\dots,t-1\}$ be such that \eqref{induction sauts} holds. We obtain that
\begin{equation}\label{gasel1}
\sum_{k=1}^{i+1}\frac{m_k}{\langle a\gamma_k\rangle+b}>\frac{\sum_{k=1}^im_k}{\langle a\gamma_i\rangle+b}+\frac{m_{i+1}}{\langle a\gamma_{i+1}\rangle+b}
\end{equation}
and
\begin{equation}\label{gasel2}
\prod_{k=1}^{i+1}\left(1+\frac{1}{\langle a\gamma_k\rangle+b}\right)^{m_k}>\left(1+\frac{1}{\langle a\gamma_i\rangle+b}\right)^{\sum_{k=1}^im_k}\left(1+\frac{1}{\langle a\gamma_{i+1}\rangle+b}\right)^{m_{i+1}}.
\end{equation}
Since $\langle a\gamma_i\rangle<\langle a\gamma_{i+1}\rangle$ and $\sum_{k=1}^im_k\geq 0$, we obtain that
$$
\frac{\sum_{k=1}^im_k}{\langle a\gamma_i\rangle+b}\geq\frac{\sum_{k=1}^im_k}{\langle a\gamma_{i+1}\rangle+b}\quad\textup{and}\quad\left(1+\frac{1}{\langle a\gamma_i\rangle+b}\right)^{\sum_{k=1}^im_k}\geq\left(1+\frac{1}{\langle a\gamma_{i+1}\rangle+b}\right)^{\sum_{k=1}^im_k},
$$
which, together with \eqref{gasel1} and \eqref{gasel2}, finishes the induction on $i$. By  \eqref{induction sauts} together with $\sum_{k=1}^tm_k\geq 0$, this completes the proof of Lemma \ref{Lemma sauts}.
\end{proof}

We can now prove Lemma \ref{Lemma pos 3} and hence complete the proof of Proposition \ref{propo positive} and of Assertion $(1)$ of Theorem  \ref{Criterion}.

\begin{proof}[Proof of Lemma \ref{Lemma pos 3}]
Throughout this proof, we use the notations defined in Lemma \ref{Lemma sauts}. For all nonnegative integers $n$, we have
\begin{align*}
\frac{\mathcal{Q}_{\langle a\boldsymbol{\alpha}\rangle,\langle a\boldsymbol{\beta}\rangle}(n+1)}{\mathcal{Q}_{\langle a\boldsymbol{\alpha}\rangle,\langle a\boldsymbol{\beta}\rangle}(1)\mathcal{Q}_{\langle a\boldsymbol{\alpha}\rangle,\langle a\boldsymbol{\beta}\rangle}(n)}&=\frac{1}{\mathcal{Q}_{\langle a\boldsymbol{\alpha}\rangle,\langle a\boldsymbol{\beta}\rangle}(1)}\cdot\frac{\prod_{i=1}^r(\langle a\alpha_i\rangle+n)}{\prod_{j=1}^s(\langle a\beta_j\rangle+n)}\\
&=\frac{\prod_{i=1}^r(1+n/\langle a\alpha_i\rangle)}{\prod_{j=1}^s(1+n/\langle a\beta_j\rangle)}\\
&=\prod_{k=1}^t\left(1+\frac{n}{\langle a\gamma_k\rangle}\right)^{m_k}.
\end{align*}
We deduce that for all positive integers $n$, we obtain 
\begin{align*}
\frac{\mathcal{Q}_{\langle a\boldsymbol{\alpha}\rangle,\langle a\boldsymbol{\beta}\rangle}(n+1)\mathcal{Q}_{\langle a\boldsymbol{\alpha}\rangle,\langle a\boldsymbol{\beta}\rangle}(n-1)}{\mathcal{Q}_{\langle a\boldsymbol{\alpha}\rangle,\langle a\boldsymbol{\beta}\rangle}(n)^2}&=\prod_{k=1}^t\left(\frac{1+n/\langle a\gamma_k\rangle}{1+(n-1)/\langle a\gamma_k\rangle}\right)^{m_k}\\
&=\prod_{k=1}^t\left(1+\frac{1}{\langle a\gamma_k\rangle+n-1}\right)^{m_k}>1,
\end{align*}
where the last inequality is obtained by Lemma \ref{Lemma sauts} with $n-1$ instead of $b$.
\medskip

Furthermore, for all $n\in\mathbb{N}$, we have
\begin{align*}
\sum_{i=1}^rH_{\langle a\alpha_i\rangle}(n+1)-\sum_{j=1}^sH_{\langle a\beta_j\rangle}(n+1)&-\left(\sum_{i=1}^rH_{\langle a\alpha_i\rangle}(n)-\sum_{j=1}^sH_{\langle a\beta_j\rangle}(n)\right)\\
&=\sum_{i=1}^r\frac{1}{\langle a\alpha_i\rangle+n}-\sum_{j=1}^s\frac{1}{\langle a\beta_j\rangle+n}\\
&=\sum_{k=1}^t\frac{m_k}{\langle a\gamma_k\rangle+n}>0,
\end{align*}
where the last inequality is obtained by Lemma \ref{Lemma sauts} with $n$ instead of $b$. It follows that $\big(\sum_{i=1}^rH_{\alpha_i}(n)-\sum_{j=1}^sH_{\beta_j}(n)\big)_{n\geq 0}$ is an increasing sequence and Lemma \ref{Lemma pos 3} is proved.
\end{proof}

\section{Proof of Assertion $(3)$ of Theorem \ref{Criterion}}\label{section reformulation}

Throughout this section, we fix two tuples $\boldsymbol{\alpha}$ and $\boldsymbol{\beta}$ of parameters in $\mathbb{Q}\setminus\mathbb{Z}_{\leq 0}$ with same length such that $\langle\boldsymbol{\alpha}\rangle$ and $\langle\boldsymbol{\beta}\rangle$ are disjoint. Furthermore, we assume that $H_{\boldsymbol{\alpha},\boldsymbol{\beta}}$ holds, that is, for all $a\in\{1,\dots,d_{\boldsymbol{\alpha},\boldsymbol{\beta}}\}$ coprime to $d_{\boldsymbol{\alpha},\boldsymbol{\beta}}$ and all $x\in\mathbb{R}$ satisfying $m_{\boldsymbol{\alpha},\boldsymbol{\beta}}(a)\preceq x\prec a$, we have $\xi_{\boldsymbol{\alpha},\boldsymbol{\beta}}(a,x)\geq 1$. We will also use the notations defined at the beginning of Section \ref{subsection demo theo expand}.

\subsection{A $p$-adic reformulation of Assertion $(3)$ of Theorem \ref{Criterion}}

To  prove Assertion $(3)$ of Theorem \ref{Criterion}, we have to prove that
\begin{equation}\label{MAINaim}
\exp\left(\frac{S_{\boldsymbol{\alpha},\boldsymbol{\beta}}(C_{\boldsymbol{\alpha},\boldsymbol{\beta}}'z)}{\mathfrak{n}_{\boldsymbol{\alpha},\boldsymbol{\beta}}}\right)\in\mathbb{Z}[[z]].
\end{equation}
A classical method to prove the integrality of the Taylor coefficients of exponential of a power series is to reduce the problem to a $p$-adic one for all primes $p$ and to use Dieudonn\'e-Dwork's lemma as follows. Assertion \eqref{MAINaim} holds if and only if, for all primes $p$, we have
\begin{equation}\label{MAINaimp}
\exp\left(\frac{S_{\boldsymbol{\alpha},\boldsymbol{\beta}}(C_{\boldsymbol{\alpha},\boldsymbol{\beta}}'z)}{\mathfrak{n}_{\boldsymbol{\alpha},\boldsymbol{\beta}}}\right)\in\mathbb{Z}_p[[z]].
\end{equation}
Let us recall that we have
$$
S_{\boldsymbol{\alpha},\boldsymbol{\beta}}(z)=\underset{\gcd(a,d)=1}{\sum_{a=1}^d}\frac{G_{\langle a\boldsymbol{\alpha}\rangle,\langle a\boldsymbol{\beta}\rangle}(z)}{F_{\langle a\boldsymbol{\alpha}\rangle,\langle a\boldsymbol{\beta}\rangle}(z)}\in z\mathbb{Q}[[z]],
$$
with $d=d_{\boldsymbol{\alpha},\boldsymbol{\beta}}$. By Corollary \ref{cor exp} applied to \eqref{MAINaimp}, we obtain that \eqref{MAINaim} holds if and only if, for all primes $p$, we have
\begin{equation}\label{MAINaimp2}
S_{\boldsymbol{\alpha},\boldsymbol{\beta}}(C_{\boldsymbol{\alpha},\boldsymbol{\beta}}'z^p)-pS_{\boldsymbol{\alpha},\boldsymbol{\beta}}(C_{\boldsymbol{\alpha},\boldsymbol{\beta}}'z)\in p\mathfrak{n}_{\boldsymbol{\alpha},\boldsymbol{\beta}}\mathbb{Z}_p[[z]].
\end{equation} 

The map $t\mapsto t^{(1)}$ is a permutation of the elements of  $\{1,\dots,d_{\boldsymbol{\alpha},\boldsymbol{\beta}}\}$ coprime to $d_{\boldsymbol{\alpha},\boldsymbol{\beta}}$. Hence, we have 
$$
S_{\boldsymbol{\alpha},\boldsymbol{\beta}}(C'z^p)-pS_{\boldsymbol{\alpha},\boldsymbol{\beta}}(C'z)=\underset{\gcd(t,d)=1}{\sum_{t=1}^d}\left(\frac{G_{\underline{t^{(1)}\boldsymbol{\alpha}},\underline{t^{(1)}\boldsymbol{\beta}}}}{F_{\underline{t^{(1)}\boldsymbol{\alpha}},\underline{t^{(1)}\boldsymbol{\beta}}}}(C'z^p)-p\frac{G_{\underline{t\boldsymbol{\alpha}},\underline{t\boldsymbol{\beta}}}}{F_{\underline{t\boldsymbol{\alpha}},\underline{t\boldsymbol{\beta}}}}(C'z)\right),
$$
with $d=d_{\boldsymbol{\alpha},\boldsymbol{\beta}}$ and $C'=C_{\boldsymbol{\alpha},\boldsymbol{\beta}}'$. By Theorem \ref{theo expand}, we obtain 
\begin{align*} 
S_{\boldsymbol{\alpha},\boldsymbol{\beta}}(C_{\boldsymbol{\alpha},\boldsymbol{\beta}}'z^p)-pS_{\boldsymbol{\alpha},\boldsymbol{\beta}}(C_{\boldsymbol{\alpha},\boldsymbol{\beta}}'z)&= p\underset{\gcd(b,D)=1}{\sum_{b=1}^D}\sum_{t\in\Omega_b}\sum_{k=0}^{\infty}R_{k,b}(t)z^k\\
&=p\underset{\gcd(b,D)=1}{\sum_{b=1}^D}\sum_{k=0}^{\infty}\left(\sum_{t\in\Omega_b}R_{k,b}(t)\right)z^k,
\end{align*}
with $R_{k,b}\in\mathcal{A}_b^\ast$ and, moreover if  $p$ divides  $d_{\boldsymbol{\alpha},\boldsymbol{\beta}}$, then we have 
$$
R_{k,b}\in\begin{cases}
p^{-1-\lfloor\lambda_p/(p-1)\rfloor}\mathcal{A}_b\textup{ if $\boldsymbol{\beta}\in\mathbb{Z}^r$;}\\
\mathcal{A}_b\textup{ if $\boldsymbol{\beta}\notin\mathbb{Z}^r$ and  $p-1\nmid\lambda_p$;}\\
\mathcal{A}_b\textup{ if $\boldsymbol{\beta}\notin\mathbb{Z}^r$, $\mathfrak{m}_{\boldsymbol{\alpha},\boldsymbol{\beta}}$ is odd and $p=2$}.
\end{cases}.
$$
By point $(7)$ of Lemma \ref{Lemma alg\`ebres 2}, we have 
\begin{equation}\label{explik}
\sum_{t\in\Omega_b}R_{k,b}(t)\in\mathfrak{n}_{\boldsymbol{\alpha},\boldsymbol{\beta}}\mathbb{Z}_p.
\end{equation}

Indeed, if $p$ does not divide $d_{\boldsymbol{\alpha},\boldsymbol{\beta}}$, then  $p$ does not divide $\mathfrak{n}_{\boldsymbol{\alpha},\boldsymbol{\beta}}$ and $R_{k,b}(t)\in\mathbb{Z}_p$. Let us now assume that $p$ divides $d_{\boldsymbol{\alpha},\boldsymbol{\beta}}$ so that  $\nu\geq 1$. 

If $\boldsymbol{\beta}\in\mathbb{Z}^r$, then  we have $v_p(\mathfrak{n}_{\boldsymbol{\alpha},\boldsymbol{\beta}})=\nu-2-\lfloor\lambda_p/(p-1)\rfloor$. If  $\boldsymbol{\beta}\notin\mathbb{Z}^r$ and if $p-1\nmid \lambda_p$, then  we have $p\neq 2$ and  $v_p(\mathfrak{n}_{\boldsymbol{\alpha},\boldsymbol{\beta}})=\nu-1$. Let us now assume that $\boldsymbol{\beta}\notin\mathbb{Z}^r$ and that $p-1\mid\lambda_p$. If $p\neq 2$ then  $v_p(\mathfrak{n}_{\boldsymbol{\alpha},\boldsymbol{\beta}})=0$ or $\nu-2$. On the other hand, if $p=2$, then either $\mathfrak{m}_{\boldsymbol{\alpha},\boldsymbol{\beta}}$ is even and  $v_2(\mathfrak{n}_{\boldsymbol{\alpha},\boldsymbol{\beta}})=0$ or $\nu-2$, or  $\mathfrak{m}_{\boldsymbol{\alpha},\boldsymbol{\beta}}$ is odd and $v_2(\mathfrak{n}_{\boldsymbol{\alpha},\boldsymbol{\beta}})=\nu-1$. 

It follows that in all cases, we have \eqref{MAINaimp2} and Assertion $(3)$ of Theorem \ref{Criterion} is proved.

\section{Proof of Assertion $(2)$ of Theorem \ref{Criterion}}\label{section last}

Let $\boldsymbol{\alpha}$ and $\boldsymbol{\beta}$ be tuples of parameters in $\mathbb{Q}\setminus\mathbb{Z}_{\leq 0}$ such that $\langle\boldsymbol{\alpha}\rangle$ and $\langle\boldsymbol{\beta}\rangle$ are disjoint (this is equivalent to the irreducibility of $\mathcal{L}_{\boldsymbol{\alpha},\boldsymbol{\beta}}$) and such that $F_{\boldsymbol{\alpha},\boldsymbol{\beta}}$ is $N$-integral. Assertion $(3)$ of Theorem \ref{Criterion} implies Assertion $(iii)\Rightarrow(i)$ of Theorem \ref{Criterion}. Indeed, it suffices to prove the following result.

\begin{propo}
Let $f(z)\in 1+z\mathbb{Q}[[z]]$ be an $N$-integral power series and let $a$ be a positive integer. Then $f(z)^{1/a}$ is an $N$-integral power series.
\end{propo}

\begin{proof}
We write $f(z)=1+zg(z)$ with $g(z)\in\mathbb{Q}[[z]]$. Thus, we obtain that
$$
f(z)^{1/a}=1+\sum_{n=1}^\infty(-1)^n\frac{(-1/a)_n}{n!}z^ng(z)^n.
$$
Since $f(z)$ is $N$-integral, there exists $C\in\mathbb{N}$ such that $g(Cz)\in\mathbb{Z}[[z]]$. Furthermore, by Theorem A applied with $\boldsymbol{\alpha}=(-1/a)$ and $\boldsymbol{\beta}=(1)$, we obtain that there exists $K\in\mathbb{N}$ such that, for all $n\in\mathbb{N}$, we have
$$
K^n\frac{(-1/a)_n}{n!}\in\mathbb{Z}.
$$
It follows that $f(CKz)^{1/a}\in\mathbb{Z}[[z]]$, \textit{i.e.} $f(z)^{1/a}$ is $N$-integral.
\end{proof}

Furthermore, by definition, we have $(ii)\Rightarrow(i)$ of Theorem  \ref{Criterion}. Thus, we only have to prove that $(i)\Rightarrow(iii)$, $(i)\Rightarrow (ii)$ and that if $(i)$ holds, then we have either $\boldsymbol{\alpha}=(1/2)$ and $\boldsymbol{\beta}=(1)$ or there are at least two elements equal to $1$ in $\langle\boldsymbol{\beta}\rangle$. Throughout this section, we assume that $(i)$ holds, \textit{i.e.} that $q_{\boldsymbol{\alpha},\boldsymbol{\beta}}$ is $N$-integral. Furthermore, for all $n\in\mathbb{N}$, we set
$$
\mathcal{Q}_{\boldsymbol{\alpha},\boldsymbol{\beta}}(n):=\frac{(\alpha_1)_n\cdots(\alpha_r)_n}{(\beta_1)_n\cdots(\beta_s)_n}.
$$

\subsection{Proof of Assertion $(iii)$ of Theorem \ref{Criterion}}\label{section proof (iii)}

The aim of this section is to prove that $r=s$, that $H_{\boldsymbol{\alpha},\boldsymbol{\beta}}$ holds and that, for all $a\in\{1,\dots,d_{\boldsymbol{\alpha},\boldsymbol{\beta}}\}$ coprime to $d_{\boldsymbol{\alpha},\boldsymbol{\beta}}$, we have $q_{\boldsymbol{\alpha},\boldsymbol{\beta}}(z)=q_{\langle a\boldsymbol{\alpha}\rangle,\langle a\boldsymbol{\beta}\rangle}(z)$. Since $F_{\boldsymbol{\alpha},\boldsymbol{\beta}}$ and $q_{\boldsymbol{\alpha},\boldsymbol{\beta}}$ are $N$-integral, there exists $C\in\mathbb{Q}\setminus\{0\}$ such that
$$
F_{\boldsymbol{\alpha},\boldsymbol{\beta}}(Cz)\in\mathbb{Z}[[z]]\quad\textup{and}\quad q_{\boldsymbol{\alpha},\boldsymbol{\beta}}(Cz)=\exp\left(\frac{G_{\boldsymbol{\alpha},\boldsymbol{\beta}}(Cz)}{F_{\boldsymbol{\alpha},\boldsymbol{\beta}}(Cz)}\right)\in\mathbb{Z}[[z]].
$$
Thus, for almost all primes $p$, we have
\begin{equation}\label{HypoMirrorp}
F_{\boldsymbol{\alpha},\boldsymbol{\beta}}(z)\in\mathbb{Z}_p[[z]]\quad\textup{and}\quad\exp\left(\frac{G_{\boldsymbol{\alpha},\boldsymbol{\beta}}(z)}{F_{\boldsymbol{\alpha},\boldsymbol{\beta}}(z)}\right)\in\mathbb{Z}_p[[z]].
\end{equation}
We shall use Dieudonn\'e-Dwork's lemma in order to get rid of the exponential map in \eqref{HypoMirrorp}.

Let $p$ be a prime  such that \eqref{HypoMirrorp} holds. By Corollary \ref{cor exp} applied to \eqref{HypoMirrorp}, we obtain that
$$
\frac{G_{\boldsymbol{\alpha},\boldsymbol{\beta}}(z^p)}{F_{\boldsymbol{\alpha},\boldsymbol{\beta}}(z^p)}-p\frac{G_{\boldsymbol{\alpha},\boldsymbol{\beta}}(z)}{F_{\boldsymbol{\alpha},\boldsymbol{\beta}}(z)}\in p z\mathbb{Z}_p[[z]].
$$
Since $F_{\boldsymbol{\alpha},\boldsymbol{\beta}}(z)\in\mathbb{Z}_p[[z]]$, we get
\begin{equation}\label{ref series}
G_{\boldsymbol{\alpha},\boldsymbol{\beta}}(z^p)F_{\boldsymbol{\alpha},\boldsymbol{\beta}}(z)-pG_{\boldsymbol{\alpha},\boldsymbol{\beta}}(z)F_{\boldsymbol{\alpha},\boldsymbol{\beta}}(z^p)\in pz\mathbb{Z}_p[[z]].
\end{equation}
In the sequel of the proof of Assertion $(2)$ of Theorem \ref{Criterion}, we use several times that \eqref{ref series} holds for almost all primes $p$.

\subsubsection{Proof of $r=s$}\label{section r=s}

We give a proof by contradiction assuming that $r\neq s$. Since $F_{\boldsymbol{\alpha},\boldsymbol{\beta}}$ is $N$-integral, Christol's criterion ensures that, for all $a\in\{1,\dots,d_{\boldsymbol{\alpha},\boldsymbol{\beta}}\}$ coprime to $d_{\boldsymbol{\alpha},\boldsymbol{\beta}}$ and all $x\in\mathbb{R}$, we have $\xi_{\boldsymbol{\alpha},\boldsymbol{\beta}}(a,x)\geq 0$. In particular, since $r-s$ is the limit of $\xi_{\boldsymbol{\alpha},\boldsymbol{\beta}}(1,n)$ when $n\in\mathbb{Z}$ tends to $-\infty$, we obtain that $r-s\geq 1$. For all $n\in\mathbb{N}$, we write $A_n$ for the assertion
$$
\sum_{i=1}^rH_{\alpha_i}(n)-\sum_{j=1}^sH_{\beta_j}(n)=0.
$$
First, we prove by induction on $n$ that $A_n$ is true for all $n\in\mathbb{N}$. 
\medskip

Assertion $A_0$ holds. Let $n$ be a positive integer such that, for all integer $k$, $0\leq k<n$, $A_k$ holds.
The coefficient $\Phi_p(np)$ of $z^{np}$ in \eqref{ref series} belongs to $p\mathbb{Z}_p$ and is equal to
$$
\sum_{j=0}^n\mathcal{Q}_{\boldsymbol{\alpha},\boldsymbol{\beta}}(jp)\mathcal{Q}_{\boldsymbol{\alpha},\boldsymbol{\beta}}(n-j)\left(\sum_{i=1}^r\big(H_{\alpha_i}(n-j)-pH_{\alpha_i}(jp)\big)-\sum_{i=1}^s\big(H_{\beta_i}(n-j)-pH_{\beta_i}(jp)\big)\right).
$$
By induction, we obtain that
\begin{multline*}
\Phi_p(np)=\mathcal{Q}_{\boldsymbol{\alpha},\boldsymbol{\beta}}(n)\left(\sum_{i=1}^rH_{\alpha_i}(n)-\sum_{i=1}^sH_{\beta_i}(n)\right)\\
-p\sum_{j=1}^n\mathcal{Q}_{\boldsymbol{\alpha},\boldsymbol{\beta}}(jp)\mathcal{Q}_{\boldsymbol{\alpha},\boldsymbol{\beta}}(n-j)\left(\sum_{i=1}^rH_{\alpha_i}(jp)-\sum_{i=1}^sH_{\beta_i}(jp)\right).
\end{multline*}

Furthermore, according to Lemma \ref{Lemma structure}, there exists a constant $M_{\boldsymbol{\alpha},\boldsymbol{\beta}}>0$ such that, for all $x\in[0,1/M_{\boldsymbol{\alpha},\boldsymbol{\beta}}[$, all primes $p$ not dividing $d_{\boldsymbol{\alpha},\boldsymbol{\beta}}$ and all $\ell\in\mathbb{N}$, $\ell\geq 1$, we have $\Delta_{\boldsymbol{\alpha},\boldsymbol{\beta}}^{p,\ell}(x)=0$. Hence, for almost all primes $p$ and all $j\in\{1,\dots,n\}$, we have
\begin{equation}\label{combi0}
v_p\big(\mathcal{Q}_{\boldsymbol{\alpha},\boldsymbol{\beta}}(jp)\big)=\sum_{\ell=1}^{\infty}\Delta_{\boldsymbol{\alpha},\boldsymbol{\beta}}^{p,\ell}\left(\frac{jp}{p^\ell}\right)=\Delta_{\boldsymbol{\alpha},\boldsymbol{\beta}}^{p,1}(j)+\sum_{\ell=1}^\infty\Delta_{\boldsymbol{\alpha},\boldsymbol{\beta}}^{p,\ell+1}\left(\frac{j}{p^\ell}\right)=\Delta_{\boldsymbol{\alpha},\boldsymbol{\beta}}^{p,1}(j)=j(r-s).
\end{equation}

According to Lemma \ref{Lemma infty}, for almost all primes $p$ and all the elements $\alpha$ in $\boldsymbol{\alpha}$ or $\boldsymbol{\beta}$, we have $\mathfrak{D}_p(\alpha)=\mathfrak{D}_p(\langle\alpha\rangle)$, so that $\mathfrak{D}_p(\alpha)=\langle \omega\alpha\rangle$ where $\omega\in\{1,\dots,d_{\boldsymbol{\alpha},\boldsymbol{\beta}}\}$ satisfies $\omega p\equiv 1\mod d_{\boldsymbol{\alpha},\boldsymbol{\beta}}$. Thus we get
\begin{align*}
pH_\alpha(jp)&=p\sum_{k=0}^{p-1}\sum_{i=0}^{j-1}\frac{1}{\alpha+k+ip}\\
&=H_{\mathfrak{D}_p(\alpha)}(j)+p\underset{k\neq p\mathfrak{D}_p(\alpha)-\alpha}{\sum_{k=0}^{p-1}}\sum_{i=0}^{j-1}\frac{1}{\alpha+k+ip}\in H_{\langle \omega\alpha\rangle}(j)+p\mathbb{Z}_p,
\end{align*}
which leads to
\begin{equation}\label{combi1}
p\left(\sum_{i=1}^rH_{\alpha_i}(jp)-\sum_{i=1}^sH_{\beta_i}(jp)\right)\equiv\sum_{i=1}^rH_{\langle \omega\alpha_i\rangle}(j)-\sum_{i=1}^sH_{\langle \omega\beta_i\rangle}(j)\mod p\mathbb{Z}_p.
\end{equation}
Furthermore, for almost all primes $p$, we have
$$
\left\{\sum_{i=1}^rH_{\langle \omega\alpha_i\rangle}(j)-\sum_{i=1}^sH_{\langle \omega\beta_i\rangle}(j)\,:\,1\leq j\leq n,\,1\leq \omega\leq d_{\boldsymbol{\alpha},\boldsymbol{\beta}},\,\gcd(\omega,d_{\boldsymbol{\alpha},\boldsymbol{\beta}})=1\right\}\subset\mathbb{Z}_p,
$$
which, together with \eqref{combi0} and \eqref{combi1}, gives us that
$$
-p\mathcal{Q}_{\boldsymbol{\alpha},\boldsymbol{\beta}}(jp)\mathcal{Q}_{\boldsymbol{\alpha},\boldsymbol{\beta}}(n-j)\left(\sum_{i=1}^rH_{\alpha_i}(jp)-\sum_{i=1}^sH_{\beta_i}(jp)\right)\in p^{r-s}\mathbb{Z}_p,
$$
for almost all primes $p$ and all $j\in\{1,\dots,n\}$. In addition, 
for almost all primes $p$, we have
$$
\mathcal{Q}_{\boldsymbol{\alpha},\boldsymbol{\beta}}(n)\left(\sum_{i=1}^rH_{\alpha_i}(n)-\sum_{j=1}^sH_{\beta_j}(n)\right)\in\mathbb{Z}_p^\times\cup\{0\}\quad\textup{and}
\quad\mathcal{Q}_{\boldsymbol{\alpha},\boldsymbol{\beta}}(n)\neq 0.
$$
Since $\Phi_p(np)\in p\mathbb{Z}_p$ and $r-s\geq 1$, we obtain that $A_n$ holds, which finishes the induction on $n$.
\medskip

It follows that for all $n\in\mathbb{N}$, we obtain that
$$
\sum_{i=1}^r\frac{1}{\alpha_i+n}-\sum_{i=1}^s\frac{1}{\beta_i+n}=\sum_{i=1}^r\big(H_{\alpha_i}(n+1)-H_{\alpha_i}(n)\big)-\sum_{i=1}^s\big(H_{\beta_i}(n+1)-H_{\beta_i}(n)\big)=0,
$$
contradicting that $\boldsymbol{\alpha}$ and $\boldsymbol{\beta}$ are disjoint since
$$
\sum_{i=1}^r\frac{1}{\alpha_i+X}-\sum_{i=1}^s\frac{1}{\beta_i+X}\in\mathbb{Q}(X)
$$
must be a non-trivial rational fraction in this case. Thus we have $r=s$ as expected.
\hfill $\square$

\subsubsection{Proof of $H_{\boldsymbol{\alpha},\boldsymbol{\beta}}$}\label{section proof H}

Let us recall that, since $F_{\boldsymbol{\alpha},\boldsymbol{\beta}}$ is $N$-integral, for all $a\in\{1,\dots,d_{\boldsymbol{\alpha},\boldsymbol{\beta}}\}$ coprime to $d_{\boldsymbol{\alpha},\boldsymbol{\beta}}$ and all $x\in\mathbb{R}$, we have $\xi_{\boldsymbol{\alpha},\boldsymbol{\beta}}(a,x)\geq 0$. We give a proof by contradiction of $H_{\boldsymbol{\alpha},\boldsymbol{\beta}}$ assuming that there exist $a\in\{1,\dots,d_{\boldsymbol{\alpha},\boldsymbol{\beta}}\}$ coprime to $d_{\boldsymbol{\alpha},\boldsymbol{\beta}}$ and $x_0\in\mathbb{R}$ such that $m_{\boldsymbol{\alpha},\boldsymbol{\beta}}(a)\preceq x_0\prec a$ and $\xi_{\boldsymbol{\alpha},\boldsymbol{\beta}}(a,x_0)=0$. Let $\alpha$ and $\beta$ be such that
$$
a\beta=\max\big(\{a\gamma\,:\,a\gamma\preceq x_0,\,\gamma\textup{ is in $\boldsymbol{\alpha}$ or $\boldsymbol{\beta}$}\},\preceq\big)
$$
and
$$
 a\alpha=\min\big(\{a\gamma\,:\,x_0\prec a\gamma,\,\gamma\textup{ equals $1$ or is in $\boldsymbol{\alpha}$ or $\boldsymbol{\beta}$}\},\preceq\big).
$$ 
It follows that for all $x\in\mathbb{R}$ satisfying $a\beta\preceq x\prec a\alpha$, we have $\xi_{\boldsymbol{\alpha},\boldsymbol{\beta}}(a,x)=0$. Observe that, since $\langle\boldsymbol{\alpha}\rangle$ and $\langle\boldsymbol{\beta}\rangle$ are disjoint, $\langle a\boldsymbol{\alpha}\rangle$ and $\langle a\boldsymbol{\beta}\rangle$ are also disjoint, thus $\beta$ is a component of $\boldsymbol{\beta}$ and $\alpha$ equals $1$ or is an element of $\boldsymbol{\alpha}$ because $\xi_{\boldsymbol{\alpha},\boldsymbol{\beta}}(a,\cdot)$ is nonnegative on $\mathbb{R}$. 
\medskip 

Let us write $\mathfrak{P}_{\boldsymbol{\alpha},\boldsymbol{\beta}}(a)$ for the set of all primes $p$ such that $ap\equiv 1\mod d_{\boldsymbol{\alpha},\boldsymbol{\beta}}$. For all large enough $p\in\mathfrak{P}_{\boldsymbol{\alpha},\boldsymbol{\beta}}(a)$, Lemma \ref{Lemma infty} gives us that $\mathfrak{D}_p(\alpha)=\mathfrak{D}_p(\langle\alpha\rangle)=\langle a\alpha\rangle$ and $\mathfrak{D}_p(\beta)=\langle a\beta\rangle$. On the one hand, if $\langle a\beta\rangle<\langle a\alpha\rangle$, then, for almost all $p\in\mathfrak{P}_{\boldsymbol{\alpha},\boldsymbol{\beta}}(a)$, we obtain that
$$
\mathfrak{D}_p(\alpha)+\frac{\lfloor 1-\alpha\rfloor}{p}-\mathfrak{D}_p(\beta)-\frac{\lfloor 1-\beta\rfloor}{p}\geq\frac{1}{d_{\boldsymbol{\alpha},\boldsymbol{\beta}}}+\frac{\lfloor 1-\alpha\rfloor}{p}-\frac{\lfloor 1-\beta\rfloor}{p}\geq\frac{1}{p}.
$$
On the other hand, if $\langle a\beta\rangle=\langle a\alpha\rangle$ and $\beta>\alpha$, then we have $\langle\beta\rangle=\langle\alpha\rangle$ so $\beta\geq 1+\alpha$ and
$$
\mathfrak{D}_p(\alpha)+\frac{\lfloor 1-\alpha\rfloor}{p}-\mathfrak{D}_p(\beta)-\frac{\lfloor 1-\beta\rfloor}{p}=\frac{\lfloor 1-\alpha\rfloor}{p}-\frac{\lfloor 1-\beta\rfloor}{p}\geq\frac{1}{p}.
$$

In both cases, we obtain that, for almost all $p\in\mathfrak{P}_{\boldsymbol{\alpha},\boldsymbol{\beta}}(a)$, there exists $v_p\in\{0,\dots,p-1\}$ such that
$$
\mathfrak{D}_p(\beta)+\frac{\lfloor 1-\beta\rfloor}{p}\leq\frac{v_p}{p}<\mathfrak{D}_p(\alpha)+\frac{\lfloor 1-\alpha\rfloor}{p},
$$
which, together with Lemma \ref{valeurs Delta}, gives us that $\Delta_{\boldsymbol{\alpha},\boldsymbol{\beta}}^{p,1}(v_p/p)=0$ for all large enough $p\in\mathfrak{P}_{\boldsymbol{\alpha},\boldsymbol{\beta}}(a)$. Furthermore, by  Lemma \ref{Lemma structure}, for almost all $p\in\mathfrak{P}_{\boldsymbol{\alpha},\boldsymbol{\beta}}(a)$ and all $\ell\in\mathbb{N}$, $\ell\geq 1$, $\Delta_{\boldsymbol{\alpha},\boldsymbol{\beta}}^{p,\ell}$ vanishes on $[0,1/p]$ so that 
$$
v_p\big(\mathcal{Q}_{\boldsymbol{\alpha},\boldsymbol{\beta}}(v_p)\big)=\sum_{\ell=1}^\infty\Delta_{\boldsymbol{\alpha},\boldsymbol{\beta}}^{p,\ell}\left(\frac{v_p}{p^\ell}\right)=\Delta_{\boldsymbol{\alpha},\boldsymbol{\beta}}^{p,1}\left(\frac{v_p}{p}\right)=0,
$$
\textit{i.e.} $\mathcal{Q}_{\boldsymbol{\alpha},\boldsymbol{\beta}}(v_p)\in\mathbb{Z}_p^\times$. Now looking at the coefficient of $z^{v_p}$ in \eqref{ref series}, one obtains that
$$
-p\mathcal{Q}_{\boldsymbol{\alpha},\boldsymbol{\beta}}(v_p)\sum_{i=1}^r\big(H_{\alpha_i}(v_p)-H_{\beta_i}(v_p)\big)\in p\mathbb{Z}_p.
$$
To  get a contradiction, we shall prove that, for all large enough $p\in\mathfrak{P}_{\boldsymbol{\alpha},\boldsymbol{\beta}}(a)$, we have
\begin{equation}\label{Harmo0}
p\left(\sum_{i=1}^rH_{\alpha_i}(v_p)-\sum_{i=1}^rH_{\beta_i}(v_p)\right)\in\mathbb{Z}_p^\times.
\end{equation}
Indeed, for all elements $\gamma$ of $\boldsymbol{\alpha}$ or $\boldsymbol{\beta}$ and all large enough $p\in\mathfrak{P}_{\boldsymbol{\alpha},\boldsymbol{\beta}}(a)$, we have
\begin{align*}
pH_{\gamma}(v_p)=p\sum_{k=0}^{v_p-1}\frac{1}{\gamma+k}&\equiv \frac{\rho(v_p,\gamma)}{\mathfrak{D}_p(\gamma)}\mod p\mathbb{Z}_p\\
&\equiv\frac{\rho(v_p,\gamma)}{\langle a\gamma\rangle}\mod p\mathbb{Z}_p.
\end{align*}
Furthermore, we have 
\begin{align*}
\rho(v_p,\gamma)=1\Longleftrightarrow v_p\geq p\mathfrak{D}_p(\gamma)-\gamma+1&\Longleftrightarrow v_p\geq p\mathfrak{D}_p(\gamma)+\lfloor 1-\gamma\rfloor\\
&\Longleftrightarrow \frac{v_p}{p}\geq \mathfrak{D}_p(\gamma)+\frac{\lfloor 1-\gamma\rfloor}{p},
\end{align*}
because $p\mathfrak{D}_p(\gamma)-\gamma\in\mathbb{Z}$ which leads to $v_p\geq p\mathfrak{D}_p(\gamma)+\lfloor 1-\gamma\rfloor\Rightarrow v_p\geq p\mathfrak{D}_p+\lfloor 1-\gamma\rfloor+\{1-\gamma\}$. Thus, by Lemma \ref{valeurs Delta}, for all large enough $p\in\mathfrak{P}_{\boldsymbol{\alpha},\boldsymbol{\beta}}(a)$, we have $\rho(v_p,\gamma)=1$ if $a\gamma\preceq a\beta$ and $\rho(v_p,\gamma)=0$ otherwise. 

Now, let $\gamma_1,\cdots,\gamma_t$ be rational numbers such that $\langle a\gamma_1\rangle<\cdots<\langle a\gamma_t\rangle$ and such that $\{\langle a\gamma_1\rangle,\dots,\langle a\gamma_t\rangle\}$ is the set of the numbers $\langle a\gamma\rangle$ when $\gamma$ describes all the elements of $\boldsymbol{\alpha}$ and $\boldsymbol{\beta}$ satisfying $a\gamma\preceq a\beta$. For all $i\in\{1,\dots,t\}$, we define 
$$
m_i:=\#\big\{1\leq j\leq r\,:\,\langle a\alpha_j\rangle=\langle a\gamma_i\rangle\big\}-\#\big\{1\leq j\leq r\,:\,\langle a\beta_j\rangle=\langle a\gamma_i\rangle\big\}.
$$ 
Then, we obtain that
\begin{align*}
p\left(\sum_{i=1}^rH_{\alpha_i}(v_p)-\sum_{j=1}^rH_{\beta_j}(v_p)\right)&\equiv \sum_{i=1}^r\frac{\rho(v_p,\alpha_i)}{\langle a\alpha_i\rangle}-\sum_{j=1}^r\frac{\rho(v_p,\beta_j)}{\langle a\beta_j\rangle}\mod p\mathbb{Z}_p\\
&\equiv\sum_{i=1}^{t}\frac{m_i}{\langle a\gamma_i\rangle}\mod p\mathbb{Z}_p.
\end{align*}

For almost all primes $p$, we have $\sum_{i=0}^t(m_i/\langle a\gamma_i\rangle)\in\mathbb{Z}_p^\times\cup\{0\}$, thus to prove \eqref{Harmo0}, it suffices to prove that 
$$
\sum_{i=1}^t\frac{m_i}{\langle a\gamma_i\rangle}\neq 0,
$$
which follows by Lemma \ref{Lemma sauts} applied with $b=0$. This finishes the proof of $H_{\boldsymbol{\alpha},\boldsymbol{\beta}}$.
\hfill $\square$

\subsubsection{Last step in the proof of Assertion $(iii)$ of Theorem \ref{Criterion}}\label{section Lemma Dwork}

To finish the proof of Assertion~$(iii)$ of Theorem~\ref{Criterion}, it remains to prove that, for all $a\in\{1,\dots,d_{\boldsymbol{\alpha},\boldsymbol{\beta}}\}$ coprime to $d_{\boldsymbol{\alpha},\boldsymbol{\beta}}$, we have $q_{\boldsymbol{\alpha},\boldsymbol{\beta}}(z)=q_{\langle a\boldsymbol{\alpha}\rangle,\langle a\boldsymbol{\beta}\rangle}(z)$. For that purpose, we shall use Dwork's results presented in \cite{Dwork} on the integrality of Taylor coefficients at the origin of power series similar to $q_{\boldsymbol{\alpha},\boldsymbol{\beta}}$. We remind the reader that, by Sections \ref{section r=s} and \ref{section proof H}, we have $r=s$ and $H_{\boldsymbol{\alpha},\boldsymbol{\beta}}$ holds. 

More precisely, we prove the following lemma which shows that, under these assumptions, we can apply Dwork's result \cite[Theorem $4.1$]{Dwork} for almost all primes.

\begin{Lemma}\label{Lemma Dwork}
Let $\boldsymbol{\alpha}$ and $\boldsymbol{\beta}$ be two tuples of parameters in $\mathbb{Q}\setminus\mathbb{Z}_{\leq 0}$ with the same numbers of elements. If $\langle\boldsymbol{\alpha}\rangle$ and $\langle\boldsymbol{\beta}\rangle$ are disjoint (this is equivalent to the irreducibility of $\mathcal{L}_{\boldsymbol{\alpha},\boldsymbol{\beta}}$) and if $H_{\boldsymbol{\alpha},\boldsymbol{\beta}}$ holds, then for almost all primes $p$ not dividing $d_{\boldsymbol{\alpha},\boldsymbol{\beta}}$, we have 
$$
\frac{G_{\mathfrak{D}_p(\boldsymbol{\alpha}),\mathfrak{D}_p(\boldsymbol{\beta})}(z^p)}{F_{\mathfrak{D}_p(\boldsymbol{\alpha}),\mathfrak{D}_p(\boldsymbol{\beta})}(z^p)}-p\frac{G_{\boldsymbol{\alpha},\boldsymbol{\beta}}(z)}{F_{\boldsymbol{\alpha},\boldsymbol{\beta}}(z)}\in p\mathbb{Z}_p[[z]].
$$
\end{Lemma}

\begin{Remark}
Lemma \ref{Lemma Dwork} in combination with Lemma \ref{H transfert} gives us that $S_{\boldsymbol{\alpha},\boldsymbol{\beta}}(z)\in pz\mathbb{Z}_p[[z]]$ for almost all primes $p$. 
\end{Remark}

\begin{proof}
If $p$ is a prime not dividing $d_{\boldsymbol{\alpha},\boldsymbol{\beta}}$, then the elements of $\boldsymbol{\alpha}$ and $\boldsymbol{\beta}$ lie in $\mathbb{Z}_p$ and
$$
\frac{G_{\mathfrak{D}_p(\boldsymbol{\alpha}),\mathfrak{D}_p(\boldsymbol{\beta})}(z^p)}{F_{\mathfrak{D}_p(\boldsymbol{\alpha}),\mathfrak{D}_p(\boldsymbol{\beta})}(z^p)}-p\frac{G_{\boldsymbol{\alpha},\boldsymbol{\beta}}(z)}{F_{\boldsymbol{\alpha},\boldsymbol{\beta}}(z)}\in\mathbb{Q}_p[[z]].
$$
Furthermore, $\boldsymbol{\alpha}$ and $\boldsymbol{\beta}$ have the same number  of elements so that Lemma \ref{Lemma Dwork} follows from the conclusion of Dwork's Theorem \cite[Theorem $4.1$]{Dwork}. In the sequel of this proof, we check that $\boldsymbol{\alpha}$ and $\boldsymbol{\beta}$ satisfy hypothesis of \cite[Theorem $4.1$]{Dwork} for almost all primes $p$, and we use the notations defined in Section \ref{section comparison 1}. For a given fixed prime $p$ not dividing $d_{\boldsymbol{\alpha},\boldsymbol{\beta}}$, hypothesis of \cite[Theorem $4.1$]{Dwork} read
\begin{itemize}
\item[$(v)$] for all $i\in\{1,\dots,r'\}$ and all $k\in\mathbb{N}$, we have $\mathfrak{D}_p^k(\beta_i)\in\mathbb{Z}_p^\times$;
\item[$(vi)$] for all $a\in[0,p[$ and all $k\in\mathbb{N}$, we have either $N^k_{p,\boldsymbol{\alpha}}(a)=N^k_{p,\boldsymbol{\beta}}(a+)=0$ or $N^k_{p,\boldsymbol{\alpha}}(a)-N^k_{p,\boldsymbol{\beta}}(a+)\geq 1$.
\end{itemize}
\medskip

If $p$ is a large enough prime, then, by Lemma \ref{Lemma infty}, for all $i\in\{1,\dots,r'\}$, we have $\mathfrak{D}_p(\beta_i)=\mathfrak{D}_p(\langle\beta_i\rangle)$ so that 
\begin{equation}\label{lab further 1}
\mathfrak{D}_p(\beta_i)\in\left\{\frac{1}{d_{\boldsymbol{\alpha},\boldsymbol{\beta}}},\dots,\frac{d_{\boldsymbol{\alpha},\boldsymbol{\beta}}-1}{d_{\boldsymbol{\alpha},\boldsymbol{\beta}}},1\right\}\subset\mathbb{Z}_p^\times.
\end{equation}
Thus, for all large enough primes $p$, $\boldsymbol{\beta}$ satisfies Assertion $(v)$.
\medskip

Let $\alpha$ and $\beta$ be elements of $\boldsymbol{\alpha}$ and $\boldsymbol{\beta}$. First, we prove that, for all large enough primes $p$ we have
\begin{equation}\label{clef}
p\mathfrak{D}_p(\alpha)-\alpha\leq p\mathfrak{D}_p(\beta)-\beta\Longleftrightarrow \omega\alpha\preceq \omega\beta,
\end{equation}
where $\omega\in\{1,\dots,d_{\boldsymbol{\alpha},\boldsymbol{\beta}}\}$ satisfies $\omega p\equiv 1\mod d_{\boldsymbol{\alpha},\boldsymbol{\beta}}$. Assume that $p$ is large enough so that, by Lemma \ref{Lemma infty}, we get $\mathfrak{D}_p(\alpha)=\langle\omega\alpha\rangle$ and $\mathfrak{D}_p(\beta)=\langle\omega\beta\rangle$. In particular, we obtain that
$$
\mathfrak{D}_p(\alpha)=\mathfrak{D}_p(\beta)\quad\textup{or}\quad\big|\mathfrak{D}_p(\alpha)-\mathfrak{D}_p(\beta)\big|\geq\frac{1}{d_{\boldsymbol{\alpha},\boldsymbol{\beta}}}.
$$
Thus, for all large enough primes $p$, we have
\begin{align*}
p\mathfrak{D}_p(\alpha)-\alpha\leq p\mathfrak{D}_p(\beta)-\beta &\Longleftrightarrow\mathfrak{D}_p(\alpha)-\mathfrak{D}_p(\beta)\leq\frac{\alpha-\beta}{p}\\
&\Longleftrightarrow\Big(\mathfrak{D}_p(\alpha)<\mathfrak{D}_p(\beta)\quad\textup{or}\quad\big(\mathfrak{D}_p(\alpha)=\mathfrak{D}_p(\beta)\quad\textup{and}\quad\alpha\geq\beta\big)\Big)\\
&\Longleftrightarrow\omega\alpha\preceq\omega\beta,
\end{align*}
as expected. Now, we observe that if $N^k_{p,\boldsymbol{\beta}}(a+)=0$, then Assertion $(vi)$ is trivial, so we can assume that $N^k_{p,\boldsymbol{\beta}}(a+)\geq 1$. We set $\boldsymbol{\beta}':=(\beta_1,\dots,\beta_{r'})$. Let us write $\theta_p^k(x)$ for $p\mathfrak{D}_p^{k+1}(x)-\mathfrak{D}_p^k(x)$ and let $\gamma$ be the component of $\boldsymbol{\alpha}$ or $\boldsymbol{\beta}'$ such that $\theta_p^k(\gamma)$ is the largest element of
$$
\Big\{\theta_p^k(\alpha_i)\,:\,1\leq i\leq r,\,\theta_p^k(\alpha_i)<a\Big\}\bigcup\Big\{\theta_p^k(\beta_j)\,:\,1\leq j\leq r',\,\theta_p^k(\beta_j)\leq a\Big\}.
$$

Since $\langle\boldsymbol{\alpha}\rangle$ and $\langle\boldsymbol{\beta}\rangle$ are disjoint, $\mathfrak{D}_p^k(\boldsymbol{\alpha})$ and $\mathfrak{D}_p^k(\boldsymbol{\beta}')$ are also disjoint and, according to \eqref{clef}, $\theta_p^k(\boldsymbol{\alpha})$ and $\theta_p^k(\boldsymbol{\beta}')$ are disjoint. It follows that $N^k_{p,\boldsymbol{\alpha}}(a)-N^k_{p,\boldsymbol{\beta}}(a+)$ is equal to
\begin{multline*}
\#\big\{1\leq i\leq r\,:\,\theta_p^k(\alpha_i)\leq\theta_p^k(\gamma)\big\}-\#\big\{1\leq i\leq r'\,:\,\theta_p^k(\beta_j)\leq\theta_p^k(\gamma)\big\}\\
=\#\big\{1\leq i\leq r\,:\,\omega\mathfrak{D}_p^k(\alpha_i)\preceq\omega\mathfrak{D}_p^k(\gamma)\big\}-\#\big\{1\leq i\leq r'\,:\,\omega\mathfrak{D}_p^k(\beta_j)\preceq\omega\mathfrak{D}_p^k(\gamma)\big\}.
\end{multline*}

If $k=0$, then we obtain that $\omega\mathfrak{D}_p^k(\alpha)\preceq\omega\mathfrak{D}_p^k(\gamma)\Leftrightarrow\omega\alpha\preceq\omega\gamma$ with $m_{\boldsymbol{\alpha},\boldsymbol{\beta}}(\omega)\preceq\omega\gamma\prec\omega$ since $\gamma\neq 1$. Indeed, if $\gamma$ is an element of $\boldsymbol{\beta}'$ then $\gamma\neq 1$, else $\gamma$ is an element of $\boldsymbol{\alpha}$ and $\theta_p^k(\gamma)<a$ so that $\gamma\neq 1$. Thus we have $N^0_{p,\boldsymbol{\alpha}}(a)-N^0_{p,\boldsymbol{\beta}}(a+)=\xi_{\boldsymbol{\alpha},\boldsymbol{\beta}}(\omega,\omega\gamma)$ and, by $H_{\boldsymbol{\alpha},\boldsymbol{\beta}}$, we get $N^0_{p,\boldsymbol{\alpha}}(a)-N^0_{p,\boldsymbol{\beta}}(a+)\geq 1$ as expected.
\medskip

If $k\geq 1$, then, for all elements $\alpha$ of $\boldsymbol{\alpha}$ and $\boldsymbol{\beta}'$, we have $\mathfrak{D}_p^k(\alpha)=\langle\omega^k\alpha\rangle$ and $\big\langle\omega\mathfrak{D}_p^k(\alpha)\big\rangle=\big\langle\omega\langle\omega^k\alpha\rangle\big\rangle=\langle\omega^{k+1}\alpha\rangle$. We deduce that we have $\omega\mathfrak{D}_p^k(\alpha)\preceq\omega\mathfrak{D}_p^k(\gamma)\Leftrightarrow\langle\omega^{k+1}\alpha\rangle\leq\langle\omega^{k+1}\gamma\rangle$ because 
$$
\langle\omega^{k+1}\alpha\rangle=\langle\omega^{k+1}\gamma\rangle\Longleftrightarrow\langle\alpha\rangle=\langle\gamma\rangle\Longleftrightarrow\langle\omega^k\alpha\rangle=\langle\omega^k\gamma\rangle. 
$$
If $\langle\gamma\rangle<1$, then $\langle\omega^{k+1}\gamma\rangle<1$ and we obtain that
$$
N^k_{p,\boldsymbol{\alpha}}(a)-N^k_{p,\boldsymbol{\beta}}(a+)=\xi_{\boldsymbol{\alpha},\boldsymbol{\beta}}(\omega^{k+1},\langle\omega^{k+1}\gamma\rangle+)\geq 1.
$$
On the contrary, if $\langle\gamma\rangle=1$, then we get $N^k_{p,\boldsymbol{\alpha}}(a)-N^k_{p,\boldsymbol{\beta}}(a+)=r-r'$. Note that $r'<r$ since there is at least one element of $\boldsymbol{\beta}$ equal to $1$. Indeed, according to $H_{\boldsymbol{\alpha},\boldsymbol{\beta}}$, if $x\in\mathbb{R}$ satisfies $m_{\boldsymbol{\alpha},\boldsymbol{\beta}}(1)\preceq x\prec 1$, then we have $\xi_{\boldsymbol{\alpha},\boldsymbol{\beta}}(1,x)\geq 1$. Since $\langle\boldsymbol{\alpha}\rangle$ and $\langle\boldsymbol{\beta}\rangle$ are disjoint, we have $\langle m_{\boldsymbol{\alpha},\boldsymbol{\beta}}(1)\rangle<1$ so that $m_{\boldsymbol{\alpha},\boldsymbol{\beta}}(1)\preceq 2\prec 1$ and
$$
1\leq\xi_{\boldsymbol{\alpha},\boldsymbol{\beta}}(1,2)=\#\{1\leq i\leq r\,:\,\alpha_i\neq 1\}-\#\{1\leq j\leq r\,:\,\beta_j\neq 1\}.
$$
We deduce that there is at least one $j\in\{1,\dots,r\}$ such that $\beta_j=1$ and we obtain that
$$
N^k_{p,\boldsymbol{\alpha}}(a)-N^k_{p,\boldsymbol{\beta}}(a+)=r-r'\geq 1,
$$
as expected. Thus Assertion $(vi)$ holds and Lemma \ref{Lemma Dwork} is proved. 
\end{proof}

Now we fix $a\in\{1,\dots,d_{\boldsymbol{\alpha},\boldsymbol{\beta}}\}$ coprime to $d_{\boldsymbol{\alpha},\boldsymbol{\beta}}$. For all large enough primes $p\in\mathfrak{P}_{\boldsymbol{\alpha},\boldsymbol{\beta}}(a)$ and all the elements $\alpha$ of $\boldsymbol{\alpha}$ or $\boldsymbol{\beta}$, we have $\mathfrak{D}_p(\alpha)=\langle a\alpha\rangle$. By  Lemma~\ref{Lemma Dwork}, we obtain that, for almost all primes $p\in\mathfrak{P}_{\boldsymbol{\alpha},\boldsymbol{\beta}}(a)$, we have
$$
\frac{G_{\langle a\boldsymbol{\alpha}\rangle,\langle a\boldsymbol{\beta}\rangle}(z^p)}{F_{\langle a\boldsymbol{\alpha}\rangle,\langle a\boldsymbol{\beta}\rangle}(z^p)}-p\frac{G_{\boldsymbol{\alpha},\boldsymbol{\beta}}(z)}{F_{\boldsymbol{\alpha},\boldsymbol{\beta}}(z)}\in p\mathbb{Z}_p[[z]].
$$
Furthermore, since $q_{\boldsymbol{\alpha},\boldsymbol{\beta}}(z)$ is $N$-integral, for almost all primes $p$, we have
$$
\frac{G_{\boldsymbol{\alpha},\boldsymbol{\beta}}(z^p)}{F_{\boldsymbol{\alpha},\boldsymbol{\beta}}(z^p)}-p\frac{G_{\boldsymbol{\alpha},\boldsymbol{\beta}}(z)}{F_{\boldsymbol{\alpha},\boldsymbol{\beta}}(z)}\in p\mathbb{Z}_p[[z]].
$$
Thus, for almost all primes $p\in\mathfrak{P}_{\boldsymbol{\alpha},\boldsymbol{\beta}}(a)$, we obtain that
$$
\frac{G_{\langle a\boldsymbol{\alpha}\rangle,\langle a\boldsymbol{\beta}\rangle}(z^p)}{F_{\langle a\boldsymbol{\alpha}\rangle,\langle a\boldsymbol{\beta}\rangle}(z^p)}-\frac{G_{\boldsymbol{\alpha},\boldsymbol{\beta}}(z^p)}{F_{\boldsymbol{\alpha},\boldsymbol{\beta}}(z^p)}\in p\mathbb{Z}_p[[z]],
$$
which leads to
$$
\frac{G_{\langle a\boldsymbol{\alpha}\rangle,\langle a\boldsymbol{\beta}\rangle}(z)}{F_{\langle a\boldsymbol{\alpha}\rangle,\langle a\boldsymbol{\beta}\rangle}(z)}-\frac{G_{\boldsymbol{\alpha},\boldsymbol{\beta}}(z)}{F_{\boldsymbol{\alpha},\boldsymbol{\beta}}(z)}\in p\mathbb{Z}_p[[z]].
$$
By Dirichlet's theorem, there are infinitely many primes in $\mathfrak{P}_{\boldsymbol{\alpha},\boldsymbol{\beta}}(a)$ so that we have
$$
\frac{G_{\langle a\boldsymbol{\alpha}\rangle,\langle a\boldsymbol{\beta}\rangle}(z)}{F_{\langle a\boldsymbol{\alpha}\rangle,\langle a\boldsymbol{\beta}\rangle}(z)}=\frac{G_{\boldsymbol{\alpha},\boldsymbol{\beta}}(z)}{F_{\boldsymbol{\alpha},\boldsymbol{\beta}}(z)},
$$
which implies that $q_{\boldsymbol{\alpha},\boldsymbol{\beta}}(z)=q_{\langle a\boldsymbol{\alpha}\rangle,\langle a\boldsymbol{\beta}\rangle}(z)$ as expected. This finishes the proof of Assertion~$(iii)$ of Theorem \ref{Criterion}.

\subsection{Proof of Assertion $(ii)$ of Theorem \ref{Criterion}}

We have to prove that $q_{\boldsymbol{\alpha},\boldsymbol{\beta}}(C_{\boldsymbol{\alpha},\boldsymbol{\beta}}'z)\in\mathbb{Z}[[z]]$. By Section \ref{section proof (iii)}, Assertion $(iii)$ of Theorem \ref{Criterion} holds, \textit{i.e.} we have $r=s$, $H_{\boldsymbol{\alpha},\boldsymbol{\beta}}$ holds and, for all $a\in\{1,\dots,d_{\boldsymbol{\alpha},\boldsymbol{\beta}}\}$ coprime to $d_{\boldsymbol{\alpha},\boldsymbol{\beta}}$, we have $q_{\langle a\boldsymbol{\alpha}\rangle,\langle a\boldsymbol{\beta}\rangle}(z)=q_{\boldsymbol{\alpha},\boldsymbol{\beta}}(z)$ so that
\begin{equation}\label{TheSame}
\frac{G_{\langle a\boldsymbol{\alpha}\rangle,\langle a\boldsymbol{\beta}\rangle}(z)}{F_{\langle a\boldsymbol{\alpha}\rangle,\langle a\boldsymbol{\beta}\rangle}(z)}=\frac{G_{\boldsymbol{\alpha},\boldsymbol{\beta}}(z)}{F_{\boldsymbol{\alpha},\boldsymbol{\beta}}(z)}.
\end{equation}
By Theorem \ref{theo expand} in combination with \eqref{TheSame}, we obtain that
$$
\frac{G_{\boldsymbol{\alpha},\boldsymbol{\beta}}}{F_{\boldsymbol{\alpha},\boldsymbol{\beta}}}(C_{\boldsymbol{\alpha},\boldsymbol{\beta}}'z^p)-p\frac{G_{\boldsymbol{\alpha},\boldsymbol{\beta}}}{F_{\boldsymbol{\alpha},\boldsymbol{\beta}}}(C_{\boldsymbol{\alpha},\boldsymbol{\beta}}'z)\in p\mathbb{Z}_p[[z]],
$$
so that, according to Corollary \ref{cor exp}, we have $q_{\boldsymbol{\alpha},\boldsymbol{\beta}}(C_{\boldsymbol{\alpha},\boldsymbol{\beta}}'z)\in\mathbb{Z}_p[[z]]$. Since $p$ is an arbitrary prime;  we get $q_{\boldsymbol{\alpha},\boldsymbol{\beta}}(C_{\boldsymbol{\alpha},\boldsymbol{\beta}}'z)\in\mathbb{Z}[[z]]$, as expected.

\subsection{Last step in the proof of Assertion $(2)$ of Theorem \ref{Criterion}}

To complete the proof of Assertion $(2)$ of Theorem \ref{Criterion} and hence that of Theorem \ref{Criterion}, we have to prove that we have either $\boldsymbol{\alpha}=(1/2)$ and $\boldsymbol{\beta}=(1)$, or $r\geq 2$ and there are at least two $1$'s in $\boldsymbol{\beta}$. We shall distinguish two cases.
\medskip

$\bullet$ Case $1$: We assume that $r=1$.
\medskip

As already proved at the end of the proof of Lemma \ref{Lemma Dwork}, there is at least one element of $\boldsymbol{\beta}$ equal to $1$. Thus we obtain that $\boldsymbol{\beta}=(1)$. We write $\boldsymbol{\alpha}=(\alpha)$. Since Assertion $(iii)$ of Theorem \ref{Criterion} holds, for all $a\in\{1,\dots,d(\alpha)\}$ coprime to $d(\alpha)$, we have $G_{\langle a\boldsymbol{\alpha}\rangle,\langle a\boldsymbol{\beta}\rangle}(z)/F_{\langle a\boldsymbol{\alpha}\rangle,\langle a\boldsymbol{\beta}\rangle}(z)=G_{\boldsymbol{\alpha},\boldsymbol{\beta}}(z)/F_{\boldsymbol{\alpha},\boldsymbol{\beta}}(z)$, \textit{i.e.}
\begin{equation}\label{series=}
F_{\boldsymbol{\alpha},\boldsymbol{\beta}}(z)G_{\langle a\boldsymbol{\alpha}\rangle,\langle a\boldsymbol{\beta}\rangle}(z)=F_{\langle a\boldsymbol{\alpha}\rangle,\langle a\boldsymbol{\beta}\rangle}(z)G_{\boldsymbol{\alpha},\boldsymbol{\beta}}(z).
\end{equation}
Now looking at the coefficient of $z$ in power series involved in \eqref{series=}, one obtains that
$$
\langle a\alpha\rangle\left(\frac{1}{\langle a\alpha\rangle}-1\right)=\alpha\left(\frac{1}{\alpha}-1\right).
$$
We deduce that for all $a\in\{1,\dots,d(\alpha)\}$ coprime to $d(\alpha)$, we have $\langle a\alpha\rangle=\alpha$. Thus we get that
\begin{align*}
\left\{\frac{\kappa}{d(\alpha)}\,:\,1\leq\kappa\leq d(\alpha),\,\gcd\big(\kappa,d(\alpha)\big)=1\right\}&=\big\{\langle a\alpha\rangle\,:\,1\leq a\leq d(\alpha),\,\gcd\big(a,d(\alpha)\big)=1\big\}\\
&=\{\alpha\},
\end{align*}
which implies that $\alpha=1/2$ as expected.
\medskip

$\bullet$ Case $2$: We assume that $r\geq 2$.
\medskip

We already know that there is at least one element of $\boldsymbol{\beta}$ equal to $1$. Since $\langle\boldsymbol{\alpha}\rangle$ and $\langle\boldsymbol{\beta}\rangle$ are disjoint, for all the elements $\alpha$ of $\boldsymbol{\alpha}$, we have $\langle\alpha\rangle<1$. Furthermore, for all $a\in\{1,\dots,d_{\boldsymbol{\alpha},\boldsymbol{\beta}}\}$ coprime to $d_{\boldsymbol{\alpha},\boldsymbol{\beta}}$, we have
\begin{align*}
\xi_{\langle\boldsymbol{\alpha}\rangle,\langle\boldsymbol{\beta}\rangle}(a,1-)&=\#\big\{1\leq i\leq r\,:\,\langle\alpha_i\rangle\neq 1\}-\#\{1\leq i\leq r\,:\,\langle\beta_i\rangle\neq 1\big\}\\
&=r-\#\{1\leq i\leq r\,:\,\langle\beta_i\rangle\neq 1\big\}.
\end{align*}
It follows that we have to prove that $\xi_{\langle\boldsymbol{\alpha}\rangle,\langle\boldsymbol{\beta}\rangle}(a,1-)\geq 2$.

Let $\gamma$ be an element of $\boldsymbol{\alpha}$ or $\boldsymbol{\beta}$ with the largest exact denominator. Then, there exists $a\in\{1,\dots,d_{\boldsymbol{\alpha},\boldsymbol{\beta}}\}$ coprime to $d_{\boldsymbol{\alpha},\boldsymbol{\beta}}$ such that $\langle a\gamma\rangle=1/d(\gamma)$. By  $H_{\boldsymbol{\alpha},\boldsymbol{\beta}}$ in combination with Lemma \ref{H transfert}, we obtain that $H_{\langle\boldsymbol{\alpha}\rangle,\langle\boldsymbol{\beta}\rangle}$ holds. In addition, we have $\big\langle a\langle\gamma\rangle\big\rangle=\langle a\gamma\rangle=1/d(\gamma)$ so that $\xi_{\langle\boldsymbol{\alpha}\rangle,\langle\boldsymbol{\beta}\rangle}\big(a,1/d(\gamma)+\big)\geq 1$. Since $\langle a\boldsymbol{\alpha}\rangle$ and $\langle a\boldsymbol{\beta}\rangle$ are disjoint and have elements larger than or equal to $1/d(\gamma)$, we obtain that $\gamma$ is a component of $\boldsymbol{\alpha}$.

Furthermore, there exists $a\in\{1,\dots,d_{\boldsymbol{\alpha},\boldsymbol{\beta}}\}$ coprime to $d_{\boldsymbol{\alpha},\boldsymbol{\beta}}$ such that 
$$
\langle a\gamma\rangle=\frac{d(\gamma)-1}{d(\gamma)}=:\kappa.
$$
Thus $\kappa$ is the largest element distinct from $1$ in $\langle a\boldsymbol{\alpha}\rangle$ and $\langle a\boldsymbol{\beta}\rangle$, and we obtain that $\xi_{\langle\boldsymbol{\alpha}\rangle,\langle\boldsymbol{\beta}\rangle}(a,\kappa+)=\xi_{\langle\boldsymbol{\alpha}\rangle,\langle\boldsymbol{\beta}\rangle}(a,1-)$. If $\langle m_{\langle\boldsymbol{\alpha}\rangle,\langle\boldsymbol{\beta}\rangle}(a)\rangle=\kappa$, then all the elements of $\langle\boldsymbol{\beta}\rangle$ are equal to $1$ and the result is proved. Otherwise, we have $\langle m_{\langle\boldsymbol{\alpha}\rangle,\langle\boldsymbol{\beta}\rangle}(a)\rangle<\kappa$ so that $\xi_{\langle\boldsymbol{\alpha}\rangle,\langle\boldsymbol{\beta}\rangle}(a,\kappa-)\geq 1$. Since $\gamma$ is an element of $\boldsymbol{\alpha}$, we obtain that $\xi_{\langle\boldsymbol{\alpha}\rangle,\langle\boldsymbol{\beta}\rangle}(a,\kappa+)\geq 2$ as expected. This finishes the proof of Assertion $(2)$ of Theorem  \ref{Criterion} and thus the one of Theorem \ref{Criterion}.

\address{E. Delaygue, Institut Camille Jordan, Universit\'e Claude Bernard Lyon 1, 43 boulevard du 11 Novembre 1918, 69622 Villeurbanne cedex, France. Email: delaygue@math.univ-lyon1.fr}

\address{T. Rivoal, Institut Fourier, CNRS et Universit\'e Grenoble 1, CNRS UMR 5582, 100 rue des maths, BP 74, 38402 St Martin d'H\`eres cedex, France. Email: Tanguy.Rivoal@ujf-grenoble.fr}

\address{J. Roques, Institut Fourier, Universit\'e Grenoble 1, CNRS UMR 5582, 100 rue des maths, BP 74, 38402 St Martin d'H\`eres cedex, France. Email: Julien.Roques@ujf-grenoble.fr}

\end{document}